\documentclass[letterpaper, 11pt, openany]{book}

    \usepackage[headings]{fullpage}
    \usepackage{amssymb, amsthm, amsfonts, amsmath, graphicx}
    \usepackage{marginnote}
    \usepackage{geometry,pb-diagram}
    \usepackage{epstopdf}
    \usepackage{xspace}
    \usepackage{tikz}
    \usepackage{chngcntr}
    \usepackage[all]{xy}
    \usepackage{pb-xy}
    \usepackage[colorlinks=true, pdfstartview=FitV, linkcolor=blue, citecolor=blue, urlcolor=blue]{hyperref}

\usepackage{tabmac}
\usepackage{pstricks,pst-node}

\psset{unit=1pt} \psset{arrowsize=4pt 1} \psset{linewidth=1pt}
\countdef\x=23 \countdef\y=24 \countdef\z=25 \countdef\t=26

    \newtheorem{theorem}{Theorem}\numberwithin{theorem}{section}
    \newtheorem{prop}[theorem]{Proposition}
\newtheorem{proposition}[theorem]{Proposition}
    \newtheorem{cor}[theorem]{Corollary}
\newtheorem{corollary}[theorem]{Corollary}
    \newtheorem{lemma}[theorem]{Lemma}
    \newtheorem{lem}[theorem]{Lemma}
    \newtheorem{conj}[theorem]{Conjecture}
	\newtheorem{conjecture}[theorem]{Conjecture}

    \theoremstyle{definition}
    \newtheorem{definition}[theorem]{Definition}
    \newtheorem{algorithm}[theorem]{Algorithm}
    \newtheorem{example}[theorem]{Example}
\newtheorem{exercise}[theorem]{Exercise}
\newtheorem{prob}[theorem]{Problem}
    \newtheorem{remark}[theorem]{Remark}
    \usepackage{sagetex}
    \newtheorem{sagedemo}[theorem]{Sage Example}
    \setkeys{sageexample}{input style = SageInput, output style = SageOutput, output=draft text}

    \numberwithin{equation}{section}

    \newcommand{\Sagecombinat}{{\sc Sage-Combinat}\xspace}
    \newcommand{\Sage}{{\sc Sage}\xspace}

    \definecolor{Gray}{rgb}{0.9,0.9,0.9} 
    
    \newcommand{\BBaf}{\mathbb{B}_{\mathrm{af}}}
    \newcommand{\diag}{\mathrm{dg}}
    \newcommand{\hook}{\mathrm{hook}}
    \newcommand{\charge}{\mathrm{charge}}
    \newcommand{\shape}{\mathrm{shape}}
    \newcommand{\weight}{\mathrm{weight}}
    \newcommand{\height}{\mathrm{height}}
    \newcommand{\la}{\lambda}
    \newcommand{\La}{{\Lambda}}
    
    \newcommand{\QQ}{{\mathbb Q}}
    \newcommand{\ZZ}{{\mathbb Z}}
    
    \newcommand{\NN}{{\mathbb N}}
    \newcommand{\CC}{{\mathbb C}}
    \newcommand{\BB}{{\mathbb B}}
    \newcommand{\KK}{{\mathbb K}}
    \newcommand{\AAA}{{\mathbb A}}
    \newcommand{\F}{{\mathcal F}}
    \newcommand{\M}{{\mathcal M}}
    \newcommand{\PP}{{\mathcal P}}
    \newcommand{\LL}{{\mathcal L}}
    \newcommand{\Sop}{{\bf S}}
    \newcommand{\Bop}{{\bf B}}
    \newcommand{\coeff}{{\Big|}}
    \renewcommand{\gray}{\color{Gray}}
    \newcommand{\mfp}{{\mathfrak p}}
    \newcommand{\mfc}{{\mathfrak c}}
    \newcommand{\mfa}{{\mathfrak a}}
    \newcommand{\res}{\mathrm{res}}
    \newcommand{\RSK}{\mathrm{RSK}}
    \newcommand{\spin}{\mathrm{spin}}
    \newcommand{\sK}{{\bf K}}
    \newcommand{\At}{{\tilde A}}
    \newcommand{\lleft}{{\rm left}}
    \renewcommand{\bot}{{\rm bot}}
    \newcommand{\nequiv}{{~\equiv\!\!\!\!\!/~~~}}
    \newcommand{\cover}{{-\!\!\!\!\!\!>}}
    \newcommand{\TAB}{{\vskip .1in \noindent { $A^{(k)}$: }}}
    \newcommand{\OPER}{{\vskip .1in \noindent { $\At^{(k)}$: }}}
    \newcommand{\STRONG}{{\vskip .1in \noindent {$s^{(k)}$: }}}
    \newcommand{\WEAK}{{\vskip .1in \noindent {$\tilde{s}^{(k)}$: }}}
    \newcommand{\STRONGWEAK}{{\vskip .1in \noindent{$s^{(k)}$, $\tilde{s}^{(k)}$: }}}
    \newcommand{\SSYT}{\mathrm{SSYT}}
    \newcommand{\hh}{\tilde{\bf h}}
    \newcommand{\FF}{\tilde{F}}
    \renewcommand{\SS}{\FF}
    \renewcommand{\ss}{{\bf s}}
    \newcommand{\uu}{{A}}
    \newcommand{\sort}{\mathrm{sort}}
    \newcommand{\pchoose}[2]{\begin{pmatrix}#1\\ #2\end{pmatrix}}
    
    \newcommand{\Strong}{\mathrm{Strong}}
    \newcommand{\Weak}{\mathrm{Weak}}
    \newcommand{\Int}{\mathrm{Int}}
    \newcommand{\Des}{\mathrm{Des}}
    \newcommand{\arm}{\mathrm{arm}}

\newdimen\squaresize \squaresize=9pt
\newdimen\thickness \thickness=0.4pt

\def\square#1{\hbox{\vrule width \thickness
     \vbox to \squaresize{\hrule height \thickness\vss
        \hbox to \squaresize{\hss#1\hss}
     \vss\hrule height\thickness}
\unskip\vrule width \thickness}
\kern-\thickness}

\def\vsquare#1{\vbox{\square{$#1$}}\kern-\thickness}
\def\blk{\omit\hskip\squaresize}
\def\noir{\vrule height\squaresize width\squaresize}%
\def\gris{\gray{\vrule height\squaresize width\squaresize}}

\def\young#1{
\vbox{\smallskip\offinterlineskip
\halign{&\vsquare{##}\cr #1}}}

\def\thisbox#1{\kern-.09ex\fbox{#1}}
\def\downbox#1{\lower1.200em\hbox{#1}}

\renewcommand{\appendix}{
   \par
   \setcounter{section}{0}%
   \renewcommand{\thesection}{\Alph{section}}%
}

\begin{document}

    \begin{center}
        {\huge {\bf $k$-Schur functions}}\\[5mm]
        {\huge {\bf and}}\\[5mm]
        {\huge {\bf affine Schubert calculus}}
   \end{center}
   
        \bigskip\medskip
        
    \begin{flushleft}
        first released September 2012, last updated October 2013
    \end{flushleft}

    \vfill
    
    \vspace{3cm}

    \begin{center}
    \includegraphics[width=5in]{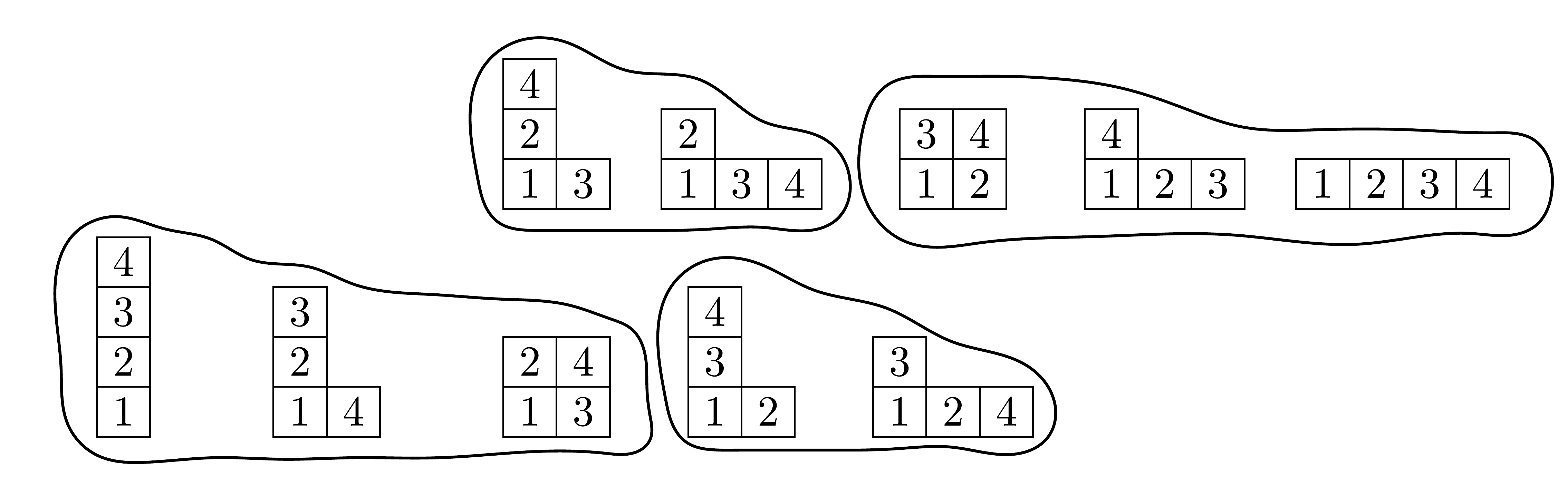}
    \end{center}

    \vfill

    \begin{flushright}
        {\large
        Thomas Lam~~~\\
        Luc Lapointe~~~\\
        Jennifer Morse~~~\\
        Anne Schilling~~~\\
        Mark Shimozono~~~\\
        Mike Zabrocki~~~\\
        }
    \end{flushright}

    \bigskip
    \bigskip
    \bigskip

    \noindent Copyright \copyright\ 2012 by the authors.\\
    These lecture notes may be reproduced in their entirety for
    non-commercial purposes.
    
    \tableofcontents
    \def\Gr{\mathrm{Gr}}
\def\tF{\tilde F}
\def\A{{\mathbb A}}
\def\Fl{\mathrm{Fl}}
\def\af{\mathrm{af}}
\chapter{Introduction}

Affine Schubert calculus is a subject that lies at the crossroads of combinatorics, 
geometry, and representation theory.  Its modern development is motivated by two seemingly 
unrelated directions.  One is the introduction of $k$-Schur functions in the study of 
Macdonald polynomial positivity, a mostly combinatorial branch of symmetric function 
theory.  The other direction is the study of the Schubert bases of the (co)homology 
of the affine Grassmannian, an algebro-topological formulation of a problem in 
enumerative geometry.

Classical Schubert calculus is a branch of enumerative algebraic geometry concerned 
with problems of the form:
\begin{quotation}
How many lines $L$ in 3-space intersect four fixed lines $L_1,L_2,L_3,L_4$?
\end{quotation}
In general, lines are replaced by affine linear subspaces, and conditions on the dimensions of intersections are imposed.  When $L_1,L_2,L_3,L_4$ are in generic position, the answer to the above problem is two; this is a pleasant surprise, since in linear algebra one expects to find $0,1$, or $\infty$ solutions.  Schubert~\cite{Schubert:1879} studied such ``Schubert problems'' in the 19th century.  
At the turn of the 20th century, Hilbert posed as his fifteenth problem the 
rigorous foundation of Schubert's enumerative calculus.  
Subsequent developments in geometry and topology converted
such Schubert problems into problems of computation in the 
cohomology ring $H^*(\Gr(k,n))$ of the Grassmannian
$\Gr(k,n)$ of $k$-planes in $n$-space.  The problems were 
reduced to finding structure constants, now called
{\it Littlewood-Richardson coefficients}~\cite{LR:1934},
of a certain ``Schubert basis'' for $H^*(\Gr(k,n))$.  

The explicit realization of these computations using the theory of Schur functions played 
an important role in transforming Schubert calculus into a contemporary theory 
that stretches into many fields.  The Schur functions $s_\lambda$ form a 
basis for the symmetric function space $\Lambda$ and at the turn of the
century, it was discovered that they match irreducible representations of 
the symmetric group. Later, a deep connection between Schur functions
and the geometry of 
Grassmannians was established when it was shown that the Schubert structure
constants exactly equal coefficients in the product of Schur functions in 
$\Lambda$.  The rich combinatorial backbone of the theory of Schur functions,
including the Robinson--Schensted algorithm, jeu-de-taquin, the plactic 
monoid (see for example~\cite{Sagan}), 
crystal bases~\cite{Nakashima:1993}, and puzzles~\cite{KnutsonTao:2003},
now underlies Schubert calculus and in particular produces
a direct formula for the Littlewood--Richardson coefficients.
The influence of Schur functions on the geometry of Grassmannians 
provoked the broadening of Schubert calculus to other studies
ranging from representation theory to physics.

A trend in Schubert calculus is to generalize the classical setup in
two basic directions: (1) to vary the underlying geometric object being 
considered by replacing the Grassmannian by the flag variety, or more generally 
by a partial flag variety of a Kac--Moody group, and (2) to vary the algebraic 
structure considered by replacing cohomology by equivariant cohomology, 
$K$-theory, quantum cohomology, or other algebraic invariants.  Our interest
is in the case when the Grassmannian is replaced by infinite-dimensional 
spaces $\Gr_G$ known as affine Grassmannians.

Investigations of the quantum cohomology rings of flag varieties led
Peterson~\cite{Pet} to begin a systematic study in this direction
for any complex simple simply-connected algebraic group $G$.
Applying work of Kostant and Kumar~\cite{KK} on the topology of Kac-Moody flag varieties, 
Peterson showed that the equivariant homology $H_T(\Gr_G)$ is isomorphic to a subalgebra
of Kostant and Kumar's nilHecke ring. Moreover, he proved that the 
Littlewood--Richardson coefficients of $H_T(\Gr_G)$ could be identified 
with the 3-point Gromov--Witten invariants of the flag variety of $G$.
A classical result of Quillen \cite{Q} establishes that the affine Grassmannian $\Gr_G$ 
is itself homotopy-equivalent to the group $\Omega K$ of based loops into the 
maximal compact subgroup $K\subset G$.  This places $\Gr_G$ in a unique position 
amongst the homogeneous spaces of all Kac--Moody groups.  It endows $H_T(\Gr_G)$ 
with the structure of a Hopf algebra, and is also partly responsible for the important 
position that the affine Grassmannian has in geometric representation theory.

The aim of this book is to present ongoing work developing a 
theory of {\it affine Schubert calculus} in the spirit of
classical Schubert calculus; here the Grassmannian is replaced 
by the affine Grassmannian.  As with Schubert calculus, topics 
under the umbrella of affine Schubert calculus are vast, but now
it is the combinatorics of a family of polynomials 
called $k$-Schur functions that underpins the theory.

The theory of $k$-Schur functions originated in the apparently unrelated 
study of {\it Macdonald polynomials}.  Macdonald polynomials 
are symmetric functions over $\mathbb Q(q,t)$ that possess 
remarkable properties; the proofs of which have inspired deep 
work in many areas (e.g. double affine Hecke algebras \cite{Che},
quantum relativistic systems \cite{Ruij},
Hilbert schemes of points in the plane \cite{Haiman:2001}).
Macdonald conjectured in the late 80's that the coefficients expressing Macdonald polynomials
in terms of the Schur basis lie in $\mathbb N[q,t]$.  
Since then, the Macdonald/Schur transition coefficients have
been intensely studied from a combinatorial, representation theoretic, 
and algebro-geometric perspective.

In one such study~\cite{LLM:2003}, Lapointe, Lascoux, and Morse found 
computational evidence for a family of new bases for subspaces 
$\Lambda_k^t$ in a filtration $\Lambda^{t}_1\subset\Lambda^{t}_2\subset\dots
\subset\Lambda^t_{\infty}$ of  $\Lambda$.  Conjecturally, the
star feature of each basis was the property that Macdonald polynomials 
expand positively in terms of it, giving a remarkable factorization 
for the Macdonald/Schur transition matrices over $\mathbb N[q,t]$.  
Pursuant investigations of these bases led to
various conjecturally equivalent characterizations and the
discovery that they refined the very aspects of Schur functions 
that make them so fundamental and wide-reaching.
As such, they are now generically called {\it $k$-Schur functions}.

The role of $k$-Schur functions in affine Schubert
calculus emerged over a number of years.  The springboard was a
realization that the combinatorial backbone of $k$-Schur theory 
lies in the setting of the type-$A$ affine Weyl group.
Generalizing the classical theory of Schur functions,
Pieri rules, Young's lattice, the Cauchy identity, tableaux, 
and Stanley symmetric functions were refined using $k$-Schur
functions~\cite{LM:2005,LM:2007,Lam:2006}. 
These are naturally described in terms of 
posets of elements in $\tilde A_k$.
For example, the number of monomial terms in an entry of the  Macdonald/$k$-Schur 
matrix equals the number of reduced expressions for an element 
in $\tilde A_k$.

The combinatorial exploration fused into a geometric one when
the $k$-Schur functions were connected to the quantum cohomology 
of Grassmannians.
Lapointe and Morse~\cite{LM:2008} showed that each
Gromov--Witten invariant for the quantum cohomology of Grassmannians 
exactly equals a $k$-Schur coefficient in the product of
{$k$-Schur functions} in $\Lambda$.  
A basis of {\it dual (or affine) $k$-Schur functions} was also 
introduced in~\cite{LM:2008}.  In response to questions
about the geometric role for dual $k$-Schur functions and
the significance of the complete set of $k$-Schur coefficients,
Morse and Shimozono conjectured that the 
Schubert bases for cohomology and homology of the affine 
Grassmannian $\Gr$ are given by the
dual $k$-Schur functions and the $k$-Schur functions, respectively. 
Lam proved the conjectures in~\cite{Lam:2008}.  Since then,
the synthesis of affine Schubert calculus and $k$-Schur function theory
has produced a subject involving prolific research in mathematics,
computer science, and physics.

\bigskip

This book arose from an NSF funded Focused Research Group entitled 
``Affine Schubert Calculus: Combinatorial, geometric, 
physical, and computational aspects'', which involved Thomas Lam, 
Luc Lapointe, Jennifer Morse, Anne Schilling, Mark Shimozono, 
Nicolas M. Thi\'ery, and Mike Zabrocki as active participants among others.  
Our exposition here grew out of several lecture series given 
at a summer school on `Affine Schubert Calculus' organized by 
Anne Schilling and Mike Zabrocki 
and held in July 2010 at the Fields Institute in Toronto.  

We give the story in three parts, through varying lenses.
Chapter~\ref{chapter.k schur primer} presents the origins and
early work on $k$-Schur functions, emphasizing the symmetric 
function setting and the combinatorics therein.
The computational aspects are highlighted and illustrated with
examples in \Sage~\cite{sage,sage-combinat}.  More information 
about the open-source computer algebra system \Sage is given in 
Appendix~\ref{appendix:sage}.  Chapter~\ref{chapter.stanley symmetric 
functions} is Thomas Lam's synopsis of his summer school lectures 
entitled ``affine Stanley symmetric functions''.
This chapter explains the combinatorial connections 
between Stanley symmetric functions and $k$-Schur functions 
via the algebraic constructions of nilCoxeter and nilHecke rings.  
Some of the latter constructions are presented for arbitrary root systems.
Chapter~\ref{chapter.affineSchubert} is Mark Shimozono's 
synopsis of his lectures on ``generalizations to other affine group types''.
This chapter presents the nilHecke ring in the very general 
Kac-Moody setting and develops some of the geometric connections.  
The general construction is then applied to the situation of the affine Grassmannian.
Each chapter is self-contained and can in principle be read
independently of the others.

\bigskip

Let us outline the contents of this book.
As discussed, the origin of $k$-Schur functions was in a study of 
Macdonald polyomials where they are characterized as symmetric functions 
that depend on one additional parameter $t$.  However, the bulk
of our presentation lies in the $t=1$ setting.  Although the general 
case is needed for implications in representation and Macdonald theory,
the proven combinatorial and geometric properties largely center around
this special case of $k$-Schur functions.

Extensive computer experimentation led to many conjectured
properties of the $k$-Schur functions.  Most notable is the
$k$-Pieri rule for $k$-Schur functions, allowing one to express 
the product of a $k$-Schur function with a homogeneous 
symmetric function in terms of $k$-Schur functions.  
Chapter~\ref{chapter.k schur primer} starts by laying the
combinatorial foundation needed to describe the $k$-Pieri rule 
including partitions, cores, and the affine Weyl group of type-$A$.
Then, for fixed $k$ and for $t=1$, the $k$-Schur functions are presented 
as the family of symmetric functions which satisfy this Pieri rule.  
These functions form a  basis of a subalgebra of the ring of 
symmetric functions.  
The dual basis lies in a Hopf-dual algebra which may be 
realized as a quotient of the ring of symmetric functions.   
Chapter~\ref{chapter.k schur primer} studies 
the $k$-Schur functions and their duals as symmetric functions, 
including a detailed summary of the 
{\it weak and strong tableaux} for which the $k$-Schur functions 
and their duals are the generating functions.  
Section~\ref{sec:strongweakduality} of this chapter includes an 
account of the affine insertion algorithm 
of~\cite{LLMS:2006}, which explains how the generating functions for 
strong tableaux ($k$-Schur functions) 
are known to be dual to the generating function for weak tableaux 
(dual $k$-Schur functions).

For arbitrary $t$, the $k$-Schur functions span a subspace of the ring of symmetric 
functions which is closed under the coproduct operation.  
It was in this setting that the $k$-Schur functions originally
arose.  They were first defined as a sum of the usual Schur functions
over a combinatorially defined collection of tableaux known as a $k$-atom.  
Lapointe, Lascoux, and Morse~\cite{LLM:2003} conjectured that the Macdonald 
symmetric functions expand positively in terms of $k$-Schur functions.  
An obvious difficulty with this approach is a missing algebraic connection 
that could be used to connect Macdonald symmetric functions with the 
combinatorics of $k$-atoms.
A second definition of the $k$-Schur functions was given in terms of 
symmetric function operators and followed
in subsequent research~\cite{LM:2003,LM2:2003}.
Chapter~\ref{chapter.k schur primer}, Section~\ref{sec:threedef}
discusses these definitions as well as several others
that are conjecturally equivalent.
Section~\ref{section.properties} is used to give a list of 
mostly conjectural properties of $k$-Schur functions
and an account of what is known (to date) about the status of these conjectures.  

Throughout Chapter~\ref{chapter.k schur primer} we have included 
examples of computations with \Sage in order to demonstrate examples 
of the formulas, but also to show how to use the functions that have 
been written by developers and incorporated into \Sage.  These examples
will hopefully both inspire and encourage exploration so that readers 
can generate further data and 
make new conjectures about $k$-Schur functions and their duals.

Chapter~\ref{chapter.stanley symmetric functions} then goes into more
depth about  $k$-Schur functions in the setting of the nil-Coxeter algebra.
In the early 1980s, Stanley~\cite{Sta} became 
interested in the enumeration of the reduced words in 
the symmetric group.  This led him to define a family of symmetric functions $\{F_w \mid w \in S_n\}$ now 
known as {\it Stanley symmetric functions}.  In \cite{Lam:2006}, Lam showed that the dual $k$-Schur functions 
were a special case of the affine Stanley symmetric functions $\tF_w$, analogues of Stanley's symmetric 
functions for the affine symmetric group. 

In earlier work of Fomin and Stanley~\cite{FS}, it was shown that some of the main properties of 
Stanley symmetric functions could be obtained systematically from the nilCoxeter algebra of the 
symmetric group.  This algebra is the associated graded algebra of the group algebra $\CC[S_n]$ 
with respect to the length filtration.  The affine nilCoxeter algebra played the same role 
for affine symmetric functions, and this provided an algebraic tool to study $k$-Schur functions 
and their duals.  This interplay between algebra, combinatorics and symmetric functions is the 
main theme of Chapter~\ref{chapter.stanley symmetric functions}.  The connection to the nilHecke
ring of Kostant and Kumar~\cite{KK} is also explained and parts of the theory is carried 
out in the case of an arbitrary Weyl group.

Chapter~\ref{chapter.affineSchubert} puts the preceding chapters in a more geometric context, and 
begins with a careful development of Kostant and Kumar's nilHecke ring $\A$~\cite{KK}.  The nilHecke
ring can roughly be described as the smash product of the nilCoxeter algebra and a polynomial ring 
and it was introduced to study the torus equivariant cohomology of Kac--Moody partial flag varieties.
This ring acts as divided difference operators on the equivariant cohomology.

Peterson~\cite{Pet} studied the equivariant homology $H_T(\Gr_G)$ of the affine Grassmannian $\Gr_G$ 
of the complex simple simply-connected algebraic group $G$ as a Hopf algebra with the following idea: 
applying the homotopy equivalences $\Gr \simeq \Omega K$ and $\Fl_\af \simeq LK/T_{\mathbb R}$ the natural 
inclusion $\Omega K \hookrightarrow LK/T_{\mathbb R}$ gives rise to an action of $H_T(\Gr_G)$ on $H_T(\Fl_\af)$. 
(Here $\Fl_\af$ denotes the affine flag variety of $G$.)  This action can be described in terms of divided 
difference operators, giving an injection $j: H_T(\Gr_G) \to \A$.  
Peterson's work is given a thorough treatment in 
Chapter~\ref{chapter.affineSchubert}, Section~\ref{sec:affineGrassmannian}.

Using the natural relation between the nilCoxeter algebra and the nilHecke ring, in~\cite{Lam:2006} Lam 
confirmed a conjecture of Morse and Shimozono identifying polynomial representatives for the Schubert 
classes of the affine Grassmannian as 
the $k$-Schur functions in homology and the dual $k$-Schur functions in cohomology.  The algebraic part 
of this result is established in Chapter~\ref{chapter.stanley symmetric functions}, Theorems \ref{thm:aB} and \ref{thm:aBHopf}.

We now discuss various generalizations of (dual) $k$-Schur functions, which are symmetric-function 
versions of the Schubert bases of the dual Hopf algebras of the homology $H_*(\Gr_{SL_n})$ and 
cohomology $H^*(\Gr_{SL_n})$ of the type A affine Grassmannian $\Gr_{SL_n}$. For the affine 
Grassmannians $\Gr_G$ for $G$ of classical type, analogous symmetric functions have been 
defined~\cite{LSS:C,Pon}. However, only for the analogues of dual $k$-Schur functions, is an explicit 
monomial expansion known. The classical type analogues of $k$-Schur functions are only defined
by duality and little is known about their combinatorics.

There is an equivariant or ``double'' analogue of $k$-Schur functions,
called $k$-double Schur functions \cite{LS:kdoubleSchur},
which are to $k$-Schur functions what double Schubert polynomials are to Schubert polynomials.
These are symmetric functions for the Schubert bases of the equivariant homology $H_T(\Gr_{SL_n})$
and $H^T(\Gr_{SL_n})$ for the ``small torus'' $T$, a maximal torus in $G$ (as opposed to 
the maximal torus in the affine Kac-Moody group).
The $k$-Schur functions are recovered from their double analogues by setting some
variables to zero. Aside from setting up the correct symmetric function rings
and bases, the only combinatorial result in this context is a Pieri rule
for $H_T(\Gr_{SL_n})$.

Essentially all of the general theory presented here has an analogue in $K$-theory,
which carries more information than (co)homology. Passing from the $k$-Schur function to its 
$K$-theoretic analogue, is like passing from a Schubert polynomial to a Grothendieck polynomial. For the affine Grassmannian,
as in (co)homology one again has a pair of dual Hopf algebras, but one obtains two \textit{pairs} of
dual bases; both algebras have a structure sheaf basis and an ideal sheaf basis.
Kostant and Kumar developed the
torus-equivariant $K$-theory of Kac-Moody homogeneous spaces \cite{KK:K}
and Peterson's theory can be carried out in $K$-theory as well \cite{LSS}.
In particular Peterson's $j$-basis (see Chapter \ref{chapter.affineSchubert}, Section \ref{SS:jbasis}), 
which is defined algebraically using a leading term condition
for an expansion in the divided difference basis, has an analogue (called the $k$-basis in \cite{LSS})
that corresponds to ideal sheaves of Schubert varieties in the affine Grassmannian.
Peterson's ``quantum equals affine'' theorem \cite{Pet,LS:QH} 
(see Chapter \ref{chapter.affineSchubert}, Section \ref{SS:quantumequalsaffine}) has an analogue in $K$-theory: 
the structure sheaves of opposite Schubert varieties in the quantum $K$-theory $QK^T(G/B)$ of finite-dimensional 
flag varieties $G/B$, appear to multiply in the same way as the structure sheaves of the Schubert varieties in
the $K$-homology $K_T(\Gr_G)$ of the affine Grassmannian \cite{LaLiMiSh}. To establish this connection
one must prove a conjectural Chevalley formula of Lenart and Postnikov \cite[Conjecture 17.1]{LP:2007} for quantum K-theory.

\section*{Acknowledgements}

This material is based upon work supported by the National Science Foundation under Grant No. DMS-0652641.
``Any opinions, findings, and conclusions or recommendations expressed in this
material are those of the author(s)
and do not necessarily reflect the views of the National Science Foundation.''
We are grateful to the Fields Institute in Toronto for helping organize and support the summer school and
workshop on ``Affine Schubert calculus''.

We would like to thank Tom Denton and Karola M\'esz\'aros for helpful comments and additions
on Chapter~\ref{chapter.k schur primer}, and Jason Bandlow, Chris Berg, Nicolas M. Thi\'ery as well as many
other \Sage developers for their help with the \Sage implementations.

\appendix

\section{Appendix: \Sage}
\label{appendix:sage}

\Sage~\cite{sage} is a completely open source
general purpose mathematical software system, which appeared under the
leadership of William Stein (University of Washington) and has
developed explosively within the last five years. It is similar to
{\sc Maple}, {\sc MuPAD}, {\sc Mathematica}, {\sc Magma}, and up to
some point {\sc Matlab}, and is based on the popular Python
programming language.  \Sage has gained strong momentum in the
mathematics community far beyond its initial focus in number theory,
in particular in the field of combinatorics, see~\cite{sage-combinat}.

Tutorials and instructions on how to install \Sage can be found at the main \Sage website
{\tt http://www.sagemath.org/}. For example, for the basic \Sage syntax and programming
tricks see {\tt http://www.sagemath.org/doc/tutorial/programming.html}.

Many aspects related to $k$-Schur functions and symmetric functions in general
have been implemented in \Sage and in fact are still being developed as an on-going project.
Throughout the text we provide many examples on how to use \Sage to do calculations related
to $k$-Schur functions. Further information about the latest code and developments can be
obtained from the \Sagecombinat website~\cite{sage-combinat}.
\chapter{Primer on $k$-Schur Functions}
\label{chapter.k schur primer}

\counterwithout{section}{chapter}
\counterwithin{equation}{section}
\setcounter{secnumdepth}{3}
\setcounter{tocdepth}{3}

\begin{center}
{\sc Jennifer Morse}
\footnote{The author was supported by NSF gant DMS--0652641, DMS--0652668,
DMS--1001898.}
{\sc , Anne Schilling} 
\footnote{The author was supported by NSF grants DMS--0652641, DMS--0652652, DMS--1001256,
and OCI--1147247.}
{\sc and Mike Zabrocki}
\footnote{The author was supported by NSERC} \\
{\tt morsej@math.drexel.edu}, {\tt anne@math.ucdavis.edu}, {\tt zabrocki@mathstat.yorku.ca}\\
\mbox{}\\
based on lectures by Luc Lapointe and Jennifer Morse\\
{\tt lapointe@inst-mat.utalca.cl} and {\tt morsej@math.drexel.edu}\\
\end{center}

The purpose of this chapter is to outline some of the results and open problems
related to $k$-Schur functions, mostly in the setting of symmetric function theory.
This chapter roughly follows the outline of several talks given by  Luc Lapointe and Jennifer Morse at 
a conference titled ``Affine Schubert Calculus'' held in July of 2010 at the Fields Institute in 
Toronto~\footnote{see {\tt http://www.fields.utoronto.ca/programs/scientific/10-11/schubert/}}.

In addition it presents many examples based on code written in \Sage~\cite{sage,sage-combinat}
by Jason Bandlow, Nicolas M. Thi\'ery, the last two authors, and many other \Sage developers.
The following presentation is intended to give both an idea of the origins of
the $k$-Schur functions as well as the current ideas and 
computational tools which have been most
productive for demonstrating their properties.  

We will present almost no proofs in this
chapter, but rather refer to the original articles for detailed arguments. Instead the concepts
are illustrated with many \Sage examples to highlight how to discover and experiment with many
of the still open conjectures related to $k$-Schur functions.  The purpose behind most of the
\Sage examples is to demonstrate the formulas with examples and to give the commands
that would allow a first time user of \Sage to be able to use the functions to
generate data that they might need for their own research.

Section~\ref{section.background} reviews much of the combinatorial background
of $k$-Schur theory including partitions, cores, (partial) orders on the affine symmetric 
group, and some symmetric function theory.  This section also sets up the
combinatorial backdrop needed to give the Pieri rules for $k$-Schur functions 
and their duals.
In Section~\ref{sec:Pierirules}, we define a parameterless ($t=1$) 
family of $k$-Schur functions using an analogue of the Pieri rule for 
Schur functions~\cite{LM:2005}. This definition is used to
relate $k$-Schur functions to the geometry and to Stanley symmetric functions 
discussed in Chapter~\ref{chapter.stanley symmetric functions}.
We also give the dual Pieri rule \cite{LLMS:2006} which
gives rise to a monomial expansion of the $k$-Schur functions.  
The Pieri and dual Pieri rule motivate the definition of weak 
and strong order tableaux. 

In Section~\ref{sec:threedef}, we present four conjecturally equivalent
definitions of the $k$-Schur functions for generic $t$.  Some are known 
to be equivalent when $t=1$.
The first definition of $k$-Schur functions appeared
in a paper by Lapointe, Lascoux and Morse~\cite{LLM:2003} and 
is purely combinatorial in nature;
defined as a sum over certain classes of tableaux called atoms.  
Lapointe and Morse~\cite{LM:2003} followed this paper by defining symmetric 
functions which were defined by
algebraic operations instead of a sum over combinatorial objects.  
The last two definitions of the $k$-Schur functions with a generic parameter $t$ 
are defined along lines similar to the parameterless $k$-Schur functions,
but now a $t$-statistic is introduced on weak (resp. strong) order 
tableaux.

In Section~\ref{section.properties} we present many of the properties 
of $k$-Schur functions 
and outline what is known about which property for each of the definitions. 
This is followed by Section~\ref{section.directions} which contains further research
directions and many conjectures that remain to be resolved 
(and hence the content is likely to change in
the future)! Section~\ref{sec:strongweakduality} explains the 
duality between strong and weak
order in terms of a $k$-analogue of the Robinson--Schensted--Knuth algorithm, 
which gives rise to an affine insertion algorithm.  We present 
part of this algorithm by giving a bijection
between permutations and pairs of tableaux. 
Finally in Section~\ref{sec:kshapeplus} some details
about the branching from $k$ to $(k+1)$-Schur functions are given.

\begin{section}{Background and notation}
\label{section.background}
\begin{subsection}{Partitions and cores} \label{subsec:parts}
A {\it partition} $\la = (\la_1, \la_2, \ldots, \la_{\ell(\la)})$ of $m$ 
is a sequence of weakly decreasing positive integers which sum to $m = \la_1 + \la_2 + \cdots + \la_{\ell(\la)}$.
The value of $m$ is called the {\it size} of the partition and this will be denoted by $|\la|$.
The entries of the partition are called the parts and the number of parts of the
partition is denoted by $\ell(\la)$.  As a general convention, 
if $i > \ell(\la)$ then $\la_i = 0$ and the definition of symmetric functions 
(which turn out to be indexed by partitions) given later in this section
respects this convention.  The statistic $n(\la) = \sum_{i=1}^{\ell(\la)} (i-1)\la_i$ on 
partitions has a value between $0$ and $m(m-1)/2$ for partitions of $m$
and this will arise in the definitions of symmetric functions.
A partition $\la$ is called $k$-{\it bounded} if $\la_1 \leq k$.
The notation $\lambda \vdash m$ indicates that $\lambda$ is a partition of $m$ and generally we reserve 
the symbols $\la$, $\mu$, $\nu$ to denote partitions.

A partition will be identified with its {\it Young (or Ferrers) diagram}.  This is a diagram consisting of
square cells arranged in left justified rows stacked on top of each other with the largest row with
$\la_1$ cells on the bottom.  (This convention is also called the French notation; when stacking the rows 
with the largest row at the top is called the English convention).
Alternatively, a Young diagram is a collection of cells in the first quadrant
of the $(x,y)$-plane with $\diag(\la) = \{ (i,j): 1 \leq i \leq \ell(\la)\hbox{ and }1 \leq j \leq \la_i\}$ 
represented as boxes in the Cartesian plane so that the upper right hand
corner of a cell has coordinate which is in this collection.  For consistency with
other references we have chosen that the first coordinate represents the row and the second
coordinate represents the column (each beginning at 1 for the first row and column).  
For an example the Young diagram for the partition $\la = (4,3,3,3,2,2,1)$ is drawn
in Example~\ref{example.partition}.

There is a partial order on partitions that arises naturally in symmetric functions
when ordering basis elements.  For two partitions $\la, \mu$ such that $|\la| = |\mu|$, 
we say that $\la \leq \mu$ if $\sum_{i=1}^r \la_i \leq \sum_{i=1}^r \mu_i$ for all $r \geq 1$. 
This is usually referred to as the {\it dominance order} on partitions.

The {\it conjugate} of a partition $\la$ is the sequence $\la' = (\la_1', \la_2', \ldots, \la_{\la_1}')$
where $\la_r' = \#\{ i : \la_i \geq r \}$.  Alternatively, this can be seen on Young diagrams by reflecting
the diagram in the $x=y$ line of the coordinate plane so that $\diag(\la') = \{ (j,i): (i,j) \in \diag(\la)\}$.
For example in Example~\ref{example.partition} below, $\la' = (7,6,4,1)$ for the partition $\la = (4,3,3,3,2,2,1)$.

For many uses we will need to refer to the number of parts of a partition of a given size $i$
and this will be denoted by $m_i(\la) = \#\{ j : \la_j=i\}$.  The quantity 
\begin{equation}\label{zladef}
z_\la = \prod_{i \geq 1} m_i(\la)! \;\; i^{m_i(\la)}
\end{equation}
is the size of the stabilizer
of a permutation $\sigma \in S_m$, the symmetric group on $m=|\lambda|$ letters,
whose cycle type is $\la$ under the conjugation
action of $S_m$.  That is, if $\sigma$ has cycle type
$\la$, then $z_\la = \#\{ \tau \in S_m : \tau \sigma \tau^{-1} = \sigma \}$.
Since we know that all permutations with the same cycle type are conjugate,
the number of permutations with cycle type $\la$ is equal to $m!/z_\la$.

Each cell in a partition $\la$ has a {\it hook length} which consists of the number of cells in the column
above and in the row to the right (including the cell itself).  
Namely, for a cell $(i,j) \in \diag(\la)$,
the hook length of the cell is $\hook_\la(i,j) = \la_i + \la_j' -i-j+1$.  In Example~\ref{example.partition} below
$\hook_{(4,3,3,3,2,2,1)}(3,2) = 5 = \la_3 + \la_2'-3-2+1$.

For a partition $\la$ with $\la_1 \leq k$, define the {\it $k$-split} of $\la$ as a sequence of partitions
(which will be denoted by $\la^{\rightarrow k}$)
recursively.  If $\la_1 + \ell(\la)-1 \leq k$, then $\la^{\rightarrow k} = (\la)$.  Otherwise,
\begin{equation} \label{equation.k split}
\la^{\rightarrow k} = ((\la_1, \la_2, \ldots, \la_{k-\la_1+1}), (\la_{k-\la_1+2}, \la_{k-\la_1+3}, \ldots, \la_{\ell(\la)})^{\rightarrow k})~.
\end{equation}
In other words, the $k$-split of a partition is found by successively splitting off parts of the partition with hook $k$,
starting with the first part, until that is no longer possible.

\begin{example} \label{example.partition}
The Young diagram for the partition $\la = (4,3,3,3,2,2,1)$ is the diagram on the left
and its conjugate partition $\la' = (7,6,4,1)$ is the diagram in the center.
$$\squaresize=9pt\young{\cr&\cr&\cr&&\cr&&\cr&&\cr&&&\cr} \hskip .75in
\young{\cr&&&\cr&&&&&\cr&&&&&&\cr}\hskip .75in
\young{\cr&\gris\cr&\gris\cr&\gris&\cr&\noir&\gris\cr&&\cr&&&\cr}$$
The diagram on the right is the Young diagram for the partition $\la=(4,3,3,3,2,2,1)$ with the cells that are in 
the hook of the cell $(3,2)$ shaded in.  In this case $\hook_{\la}(3,2) = 5$.
The $4$-split of $\la$ is $\la^{\rightarrow 4}=((4),(3,3), (3,2), (2,1))$ and the $5$-split is $\la^{\rightarrow 5} = ((4,3),(3,3,2),(2,1))$.
\end{example}

We will use the realization of the Young diagram as the set of cells in our
notation and define $\la \subseteq \mu$ if $\diag(\la) \subseteq \diag(\mu)$.  This forms
a lattice, also known as the {\it Young lattice}, on set of partitions and the cover relation is given by 
$\la \rightarrow \mu$ if $\la \subseteq \mu$ and $|\la|+1 = |\mu|$.  The lattice is graded by the size of the
partition and the first 6 levels of the infinite Hasse diagram are shown in Figure~\ref{partHasse}.

\begin{figure}[h]
\includegraphics[width=6.0in]{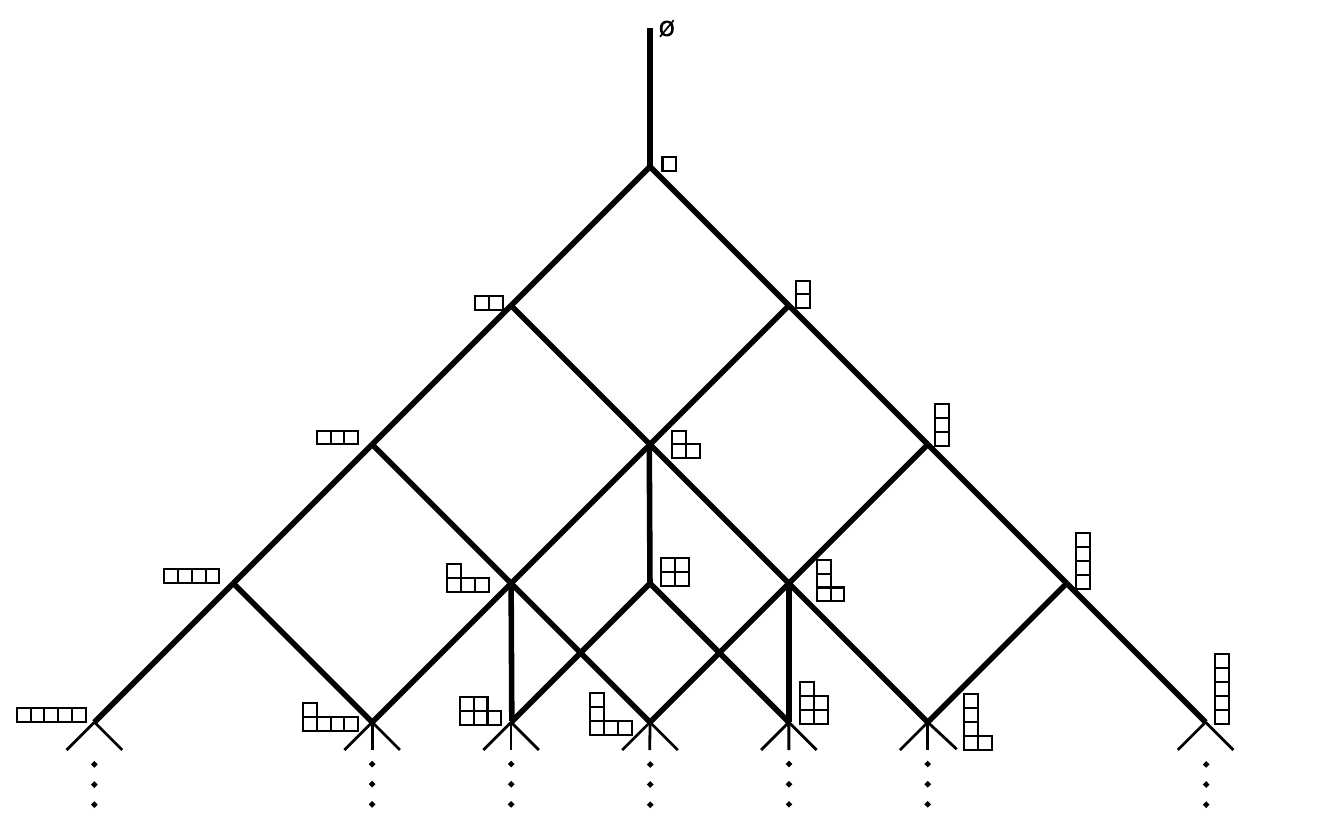}
\caption{The Young lattice of partitions (up to those of size 5) 
ordered by inclusion .}\label{partHasse}
\end{figure}

There are several special types of containments of partitions that will arise in this discussion.
If $\la \subseteq \mu$, then $\mu\slash\la$ is called a {\it skew partition} and it will represent the
cells which are in $\diag(\mu) \slash \diag(\la)$, with the $\slash$ here representing the difference of sets.
We call $\mu\slash\la$ {\it connected} if for any two cells there is
a sequence of cells in $\mu\slash\la$ from one to the other where consecutive cells share an edge.
We say that $\mu\slash\la$ is a {\it horizontal} ({\it vertical}) strip if 
there is at most one cell in each column (row)
of $\mu\slash\la$.  The skew partition $\mu\slash\la$ is called a {\it ribbon} if it does not contain any
$2\times2$ subset of cells.

\begin{sagedemo}
We now demonstrate how to access partitions and their properties in the open source computer algebra
system \Sage (see Appendix~\ref{appendix:sage}). We begin by listing all partitions of 4:
\begin{sageexample}
    sage: P = Partitions(4); P
    Partitions of the integer 4
    sage: P.list()
    [[4], [3, 1], [2, 2], [2, 1, 1], [1, 1, 1, 1]]
\end{sageexample}
\Sage has list comprehension so that the last line could have also been written as
\begin{sageexample}
    sage: [p for p in P]
    [[4], [3, 1], [2, 2], [2, 1, 1], [1, 1, 1, 1]]
\end{sageexample}
We can check how two partitions $\la$ and $\mu$ relate in the dominance order
\begin{sageexample}
    sage: la=Partition([2,2]); mu=Partition([3,1])
    sage: mu.dominates(la)
    True
\end{sageexample}
and draw the entire Hasse diagram
\begin{sageexample}
    sage: ord = lambda x,y: y.dominates(x)
    sage: P = Poset([Partitions(6), ord], facade=True)
    sage: H = P.hasse_diagram()
\end{sageexample}
\begin{sageexample}[test = false]
    sage: view(H)
\end{sageexample}
which outputs the graph. Here we used the python syntax for a function, which
is {\tt lambda x : f(x)} for a function that maps $x$ to $f(x)$.
We can also compute the conjugate of a partition, its $k$-split
\begin{sageexample}
    sage: la=Partition([4,3,3,3,2,2,1])
    sage: la.conjugate()
    [7, 6, 4, 1]
    sage: la.k_split(4)
    [[4], [3, 3], [3, 2], [2, 1]]
\end{sageexample}
and create skew partitions
\begin{sageexample}
    sage: p = SkewPartition([[2,1],[1]])
    sage: p.is_connected()
    False
\end{sageexample}
\end{sagedemo}
\end{subsection}

\begin{subsection}{Bounded partitions, cores, and affine Grassmannian elements}
\label{section:partitionscores}

We will see that $k$-Schur functions are symmetric functions indexed by 
$k$-bounded partitions and consequently, the underlying combinatorial 
framework we need often comes out of a refinement of classical ideas 
in the theory of partitions.  As it happens, the set of $k$-bounded 
partitions is in bijection with several different sets of natural 
combinatorial objects and often the $k$-Schur function setting is better
expressed in those terms.  To this end, we begin with a discussion 
of several other examples of possible indexing sets.  

As with the $k$-bounded partitions, we are interested in another special
subset of partitions.  In particular, an $r$-{\it core} is a shape where 
none of its cells have a hook-length equal to $r$.  We denote the set of all
$r$-cores by $\mathcal{C}_r$. When we consider a partition 
as a core, the notion of size differs from the usual notion (where size
counts the number of cells in the shape).  In contrast, 
the relevant notion of size on a $(k+1)$-core is to count 
only the number of cells which have a hook-length smaller than $k+1$.
We call this the {\it length} of the core. For a
$(k+1)$-core $\kappa$, its length will
be denoted by $|\kappa|_{k+1}$ or simply $|\kappa|$ 
if it is clear from the context that $\kappa$ is viewed as a $(k+1)$-core.
As $k\to \infty$, this becomes the usual size of the
partition. Later in this section, we will see that the length
is related to the length of elements in the affine symmetric group. 
Now, we give the connection between cores and bounded partitions.

\begin{prop}\label{bijcoresparts} (\cite[Theorem 7]{LM:2005})
There is a bijection between the set of $(k+1)$-cores $\kappa$ with 
$|\kappa|_{k+1}=m$ and partitions
$\la \vdash m$ with $\la_1 \leq k$.
\end{prop}

The bijection from $(k+1)$-cores to $k$-bounded partitions is
$$
\mfp: \kappa \mapsto \la\,,
$$
defined by setting
\begin{equation}\label{bijectiondir1}
\la_i = \#\{ (i,j) \in \kappa : \hook_{\kappa}(i,j) \leq k\}~.
\end{equation}

\begin{example} \label{example.core}
The partition $(12,8,5,5,2,2,1)$ on the left is a $5$-core
since there are no cells in its Ferrers diagram with hook-length
equal to $5$.
Equation~\eqref{bijectiondir1} tells us how to 
applying $\mfp$ to this core to obtain a $4$-bounded partition;
delete each cell in the diagram for the $5$-core 
whose hook-length exceeds $5$ and then slide all
remaining cells to the left. 
\squaresize=9pt
$$\young{\cr&\cr&\cr&&&&\cr&&&&\cr&&&&&&&\cr&&&&&&&&&&&\cr}\hskip .5in
\young{\cr&\cr&\cr\gris&\gris&&&\cr\gris&\gris&&&\cr
\gris&\gris&\gris&\gris&\gris&&&\cr\gris&\gris&\gris&\gris&\gris&\gris&\gris&\gris&&&&\cr}\hskip .5in
\young{\cr&\cr&\cr&&\cr&&\cr&&\cr&&&\cr}$$
The first part of the resulting partition is at most $4.$
\end{example}

The other direction of the bijection is also not difficult.  Consider a $k$-bounded partition
and work from the smallest part of the partition to the largest and slide the cells to the
right until it is a $(k+1)$-core.
Here is a description of the procedure which can be followed with
Example~\ref{example.to_core}.
Start with the top row $\la_{\ell(\la)}$ of the $k$-bounded partition
$\la$ and successively move down a row. For a given row, calculate the hook lengths of its cells; if 
there is a cell with hook length greater than $k$, slide this row to the right until all cells have 
hook length less than or equal to $k$. Continue this process until all rows have been adjusted. 
The end result will be a $(k+1)$-core which we shall denote by $\mfc_k(\la)$ or just $\mfc(\la)$ 
if $k$ is clear from the context.

\begin{example}  \label{example.to_core}
The partition $(4,3,3,3,2,2,1)$ is a $4$-bounded partition.
Here we draw the successive slides of the rows until we reach a 5-core:
\squaresize=9pt
\[
	\young{\cr&\cr&\cr&&\cr&&\cr&&\cr&&&\cr} \rightarrow 
	\young{\cr&\cr&\cr\gris&\gris&&&\cr\gris&\gris&&&\cr\gris&\gris&&&\cr\gris&\gris&&&&\cr} \rightarrow 
	\young{\cr&\cr&\cr\gris&\gris&&&\cr\gris&\gris&&&\cr\gris&\gris&\gris&\gris&\gris&&&\cr\gris&\gris&\gris&\gris&\gris&&&&\cr} \rightarrow 
	\young{\cr&\cr&\cr\gris&\gris&&&\cr\gris&\gris&&&\cr\gris&\gris&\gris&\gris&\gris&&&\cr\gris&\gris&
	\gris&\gris&\gris&\gris&\gris&\gris&&&&\cr}
\]
\end{example}

\begin{sagedemo}
\label{bound2core}
Here is the way to compute the map $\mfc$ in \Sage:
\begin{sageexample}
    sage: la = Partition([4,3,3,3,2,2,1])
    sage: kappa = la.k_skew(4); kappa
    [12, 8, 5, 5, 2, 2, 1] / [8, 5, 2, 2]
\end{sageexample}
For the inverse $\mfp$ we write
\begin{sageexample}
    sage: kappa.row_lengths()
    [4, 3, 3, 3, 2, 2, 1]
\end{sageexample}
If one is only handed the 5-core $(12,8,5,5,2,2,1)$ instead of the skew partition, one can do the following:
\begin{sageexample}
    sage: tau = Core([12,8,5,5,2,2,1],5)
    sage: mu = tau.to_bounded_partition(); mu
    [4, 3, 3, 3, 2, 2, 1]
    sage: mu.to_core(4)
    [12, 8, 5, 5, 2, 2, 1]
\end{sageexample}
All $3$-cores of length $6$ can be listed as:
\begin{sageexample}
    sage: Cores(3,6).list()
    [[6, 4, 2], [5, 3, 1, 1], [4, 2, 2, 1, 1], [3, 3, 2, 2, 1, 1]]
\end{sageexample}
\end{sagedemo}

We now turn our attention to the third set of objects that is in bijection 
with the set of $k$-bounded partitions (and the set of $(k+1)$-cores).  
These come out of studying the type $A$ affine Weyl group and its
realization as the {\it affine symmetric group} $\tilde S_{n}$
given by generators $\{s_0,s_1,\ldots,s_{n-1}\}$ satisfying 
the relations 
\begin{align}
& s_i^2 = 1,\nonumber\\
& s_i s_{i+1} s_i = s_{i+1} s_i s_{i+1}, \label{eq:affsymgens}\\
& s_i s_j = s_j s_i\quad \hbox{ for }i-j \nequiv 0,1,n-1 \pmod{n}\nonumber
\end{align}
with all indices related $\pmod{n}$.
Hereafter, we shall reserve 
parameters $n$ and $k$ and we will set $n=k+1$ throughout.

There is a subset of the elements in $\tilde S_n$ that is particularly conducive 
to combinatorics in large part because it is in 
bijection with the set of $k$-bounded partitions and of $(k+1)$-cores.  
Note that the symmetric group $S_{n}$ generated by $\{ s_1, s_2, \ldots, s_{n-1} \}$ is
a subgroup, where the element $s_i$ represents the permutation which interchanges $i$ and $i+1$.
We will refer to the left cosets of ${\tilde S}_{n} / S_{n}$ as 
\textit{affine Grassmannian elements} and 
they will be identified with their minimal length coset representatives, that is, the elements of
$w \in {\tilde S}_{n}$ such that either $w=id$ or $s_0$ is the only 
elementary transposition such 
that $\ell(w s_0)<\ell(w)$.
\begin{remark}
The definition of affine Grassmannian elements are the special case of a
more general definition.  The $l$-Grassmannian
elements are the minimal length coset representatives of ${\tilde S}_{n} / S_{n}^l$ where
$S_n^l$ is the group generated by $\{ s_0, s_1, s_2, \ldots, s_{n-1} \} \backslash \{ s_l \}$ and
the affine Grassmannian elements are the $0$-Grassmannian elements.
Due to the cyclic symmetry of the affine type $A$ Dynkin diagram, these constructions are of course
all equivalent.
\end{remark}

\begin{sagedemo}
We can create the affine symmetric group and its generators in \Sage as
\begin{sageexample}
    sage: W = WeylGroup(["A",4,1])
    sage: S = W.simple_reflections()
    sage: [s.reduced_word() for s in S]
    [[0], [1], [2], [3], [4]]
\end{sageexample}
For a given element, we can ask for its reduced word or create it from
a word in the generators and ask whether it is Grassmannian:
\begin{sageexample}
    sage: w = W.an_element(); w
    [ 2  0  0  1 -2]
    [ 2  0  0  0 -1]
    [ 1  1  0  0 -1]
    [ 1  0  1  0 -1]
    [ 1  0  0  1 -1]
    sage: w.reduced_word()
    [0, 1, 2, 3, 4]
    sage: w = W.from_reduced_word([2,1,0])
    sage: w.is_affine_grassmannian()
    True
\end{sageexample}
\end{sagedemo}

\begin{prop} \cite{Lascoux:2001} \cite[Proposition 40]{LM:2005} 
\label{prop:affgrass2core}
There is a bijection between the collection of 
cosets $\tilde S_{k+1}/S_{k+1}$ whose minimal length representative 
has length $m$ and $(k+1)$-cores of length $m$.
\end{prop}

The bijection of Proposition~\ref{prop:affgrass2core} is defined by an
action of the affine symmetric group on cores.
It suffices to define the left action of the generators $s_i$ of the affine symmetric group
on $(k+1)$-cores.
The diagonal index or {\it content} of a cell $c = (i,j)$ in the diagram for a core is $j-i$.  
We will often instead 
be concerned with the {\it residue} of $c$, denoted by $\res(c)$, which is the diagonal index
mod $k+1$.  
We call a cell $c$ an addable corner of a partition $\mu$ 
if $\diag(\mu) \cup \{ c \}$ is the diagram for a partition and a cell $c$ is
a removable corner of $\mu$ if $\diag(\mu) \backslash \{ c \}$ is diagram for a partition.

\begin{definition} \cite{Lascoux:2001} \cite[Definition 18]{LM:2005} \label{def:siactioncore}
For $\kappa$ a $(k+1)$-core, let $s_i \cdot \kappa$ be the partition with
\begin{enumerate}
\item if there is at least one addable corner of
residue $i$, then the result is $\kappa$ with all addable corners of $\kappa$ of residue $i$ added,
\item if there is at least one removable corner of
residue $i$, then the result is $\kappa$ with all removable corners of $\kappa$ of residue $i$ removed,
\item otherwise, the result is $\kappa$.
\end{enumerate}
\end{definition}

\begin{example}
Consider the $5$-core, $(7,3,1)$.  If we draw its Ferrers diagrams and label each of the
cells with the content modulo $5$ we have the following diagram
\[ \squaresize=12pt\young{3\cr4&0&1\cr0&1&2&3&4&0&1\cr} \; .\]
This diagram has addable corners with residue $2$ and $4$ and removable corners with
residue $1$ and $3$ and all cells of residue $0$ are neither addable nor removable.  Therefore,
$s_2 \cdot (7,3,1) = (8,4,1,1)$, $s_4 \cdot (7,3,1) = (7,3,2)$, $s_1\cdot  (7,3,1) = (6,2,1)$,
$s_3 \cdot (7,3,1) = (7,3)$, $s_0 \cdot (7,3,1) = (7,3,1)$.
\end{example}

\begin{sagedemo}
In \Sage we can get the affine symmetric group action on cores as follows:
\begin{sageexample}
    sage: c = Core([7,3,1],5)
    sage: c.affine_symmetric_group_simple_action(2)
    [8, 4, 1, 1]
    sage: c.affine_symmetric_group_simple_action(0)
    [7, 3, 1]
\end{sageexample}

We can also check directly that the set of affine Grassmannian elements of given length are in bijection
with the corresponding cores:
\begin{sageexample}
    sage: k=4; length=3
    sage: W = WeylGroup(["A",k,1])
    sage: G = W.affine_grassmannian_elements_of_given_length(length)
    sage: [w.reduced_word() for w in G]
    [[2, 1, 0], [4, 1, 0], [3, 4, 0]]

    sage: C = Cores(k+1,length)
    sage: [c.to_grassmannian().reduced_word() for c in C]
    [[2, 1, 0], [4, 1, 0], [3, 4, 0]]
\end{sageexample}
\end{sagedemo}

The bijection of Proposition~\ref{prop:affgrass2core} is realized by taking a reduced 
word for an affine Grassmannian element and acting on the empty $(k+1)$-core.  The resulting
$(k+1)$-core is the image of the bijection.
\cite[Corollary 48]{LM:2005} then states
that the reverse bijection can be found by taking $\la = \mfp(\kappa)$ and
forming the element in the affine symmetric group $s_{\res(c_1)} s_{\res(c_2)} \cdots s_{\res(c_m)}$,
where $c_1, c_2, \ldots, c_m$ are the cells of $\diag(\la)$ read from the smallest row
to the largest with each read from right to left.

We denote the map which sends $(k+1)$-core $\kappa$ to the corresponding
affine Grassmannian element by $\mfa(\kappa) = w_\kappa$. Since $(k+1)$-cores
and $k$-bounded partitions are in bijection, we will also use the notation
$\mfa(\la) = w_\la$ to represent the map from a $k$-bounded partition $\la$ to
an affine Grassmannian element.

\begin{example}  Consider the reduced word,
\[
w = s_1 s_0 s_4 s_2 s_3 s_1 s_0 s_4 s_1 s_2 s_3 s_1 s_0 s_4 s_3 s_2 s_1 s_0~.
\]
We apply this word on the left on an empty $5$-core to build up the result.  The sequence of
applications builds the core as follows:
\[\scriptsize
\emptyset \rightarrow \young{0\cr} \rightarrow \young{&1\cr} \rightarrow \young{&&2\cr}
\rightarrow \young{&&&3\cr}\rightarrow \young{4\cr&&&&4\cr}
\rightarrow \young{&0\cr&&&&&0\cr}\rightarrow \young{&&1\cr&&&&&&1\cr}
\]
\[\scriptsize
\rightarrow \young{3\cr&&\cr&&&&&&\cr}
\rightarrow \young{2\cr\cr&&&2\cr&&&&&&&2\cr}
\rightarrow \young{1\cr\cr\cr&&&\cr&&&&&&&\cr}
\rightarrow \young{\cr\cr&4\cr&&&\cr&&&&&&&\cr}
\]
\[\scriptsize
\rightarrow \young{0\cr\cr\cr&&0\cr&&&\cr&&&&&&&\cr}
\rightarrow \young{\cr\cr\cr&&&1\cr&&&\cr&&&&&&&\cr}
\rightarrow \young{\cr\cr&3\cr&&&\cr&&&&3\cr&&&&&&&&3\cr}
\rightarrow \young{\cr&2\cr&\cr&&&&2\cr&&&&\cr&&&&&&&&\cr}
\]
\[\scriptsize
\rightarrow \young{4\cr\cr&\cr&&4\cr&&&&\cr&&&&&4\cr&&&&&&&&&4\cr}
\rightarrow \young{\cr\cr&\cr&&&0\cr&&&&\cr&&&&&&0\cr&&&&&&&&&&0\cr}
\rightarrow \young{\cr&1\cr&\cr&&&&1\cr&&&&\cr&&&&&&&1\cr&&&&&&&&&&&1\cr}
\]
and hence the resulting $5$-core is $\mfa^{-1}(w)=(12,8,5,5,2,2,1)$.

The reverse bijection comes from reading the residues of the corresponding
$4$-bounded partition $(4,3,3,3,2,2,1)$ from the smallest row to the largest row,
and from right to left within the rows.  For example:
$$\squaresize=12pt\young{4\cr0&1\cr1&2\cr2&3&4\cr3&4&0\cr4&0&1\cr0&1&2&3\cr}$$
is sent under $\mfa$ to the word
$$w' = s_4 s_1 s_0 s_2 s_1 s_4 s_3 s_2 s_0 s_4 s_3 s_1 s_0 s_4 s_3 s_2 s_1 s_0\,.$$
It is not difficult to show that $w$ is equivalent to $w'$.
\end{example}

\begin{sagedemo}
We can verify the previous example in \Sage.
\begin{sageexample}
    sage: la = Partition([4,3,3,3,2,2,1])
    sage: c = la.to_core(4); c
    [12, 8, 5, 5, 2, 2, 1]
    sage: W = WeylGroup(["A",4,1])
    sage: w = W.from_reduced_word([4,1,0,2,1,4,3,2,0,4,3,1,0,4,3,2,1,0])
    sage: c.to_grassmannian() == w
    True
\end{sageexample}
\end{sagedemo}

The affine symmetric group ${\tilde S}_n$ can also be thought of as the group of permutations
of ${\mathbb Z}$ with the property that for $w\in {\tilde S}_n$ we have $w( i + rn ) = w(i) + rn$ for all $r \in \ZZ$
with the additional property that $\sum_{i=1}^n w(i)-i = 0$.
We can choose the convention that the elements $s_i \in {\tilde S}_{n}$ for $0\le i\le n-1$
act on $\ZZ$ by $s_i( i+rn ) = i+1+rn$, $s_i( i+1+rn ) = i+rn$, and $s_i( j ) = j$ for 
$j \nequiv i,i+1~\pmod{n}$.  While the elements $s_i$ generate the group,
there is also the notion of a general transposition $t_{ij}$ which generalizes this notion 
by interchanging $i$ and $j~\pmod{n}$.  Take integers $i < j$ with
$i \nequiv j~\pmod{n}$ and $v = \lfloor (j-i)/n \rfloor$, then $t_{i,i+1} = s_i$ and for
$j-i>1$,
\[
t_{ij} = s_i s_{i+1} s_{i+2} \cdots s_{j-v-2} s_{j-v-1} s_{j-v-2} s_{j-v-3} \cdots s_{i+1} s_i
\]
where all of the indices of the $s_m$ are taken $\pmod{n}$.
For $j>i$, we set $t_{ij} = t_{ji}$. The $t_{ij}$ generalize the elements
$s_i$ by their action $t_{ij}( i +rn ) = j + rn$, $t_{ij}( j+rn) = i+rn$
and $t_{ij}( \ell ) = \ell$ for $\ell \nequiv i,j~\pmod{n}$.
It is not hard to show that
$$t_{ij} w = w t_{w^{-1}(i) w^{-1}(j)},$$
which allows us to define a left as well as a right action on affine permutations.
Here we state the results in terms of the left action.

The elements $w \in {\tilde S}_n$ are determined by the action of $w$ on the values $1$ through $n$ since
this determines the action on all of ${\mathbb Z}$ by $w(i+rn) = w(i) + rn$.  
If $w$ is represented in two line notation,
\begin{center}
\begin{tabular}{ccccccc}
$\cdots$&$-2$&$-1$&$0$&$1$&$2$&$\cdots$\\
$\cdots$&$w(-2)$&$w(-1)$&$w(0)$&$w(1)$&$w(2)$&$\cdots$
\end{tabular}
\end{center}
then $t_{ij} w$ is obtained from $w$ by exchanging $i+rn$ and $j+rn$ in the lower row of the two line
notation for $w$.
We also have that $w t_{ij}$ is obtained from $w$ by exchanging $w(i+rn)$ and $w(j+rn)$.
An element is affine Grassmannian if $w(1) < w(2) < \cdots < w(n)$.  
The tuple of values $[ w(1), w(2), \ldots, w(n)]$ is referred to as the {\it window notation} for $w$.

There is a close relationship between the action of $w^{-1}$ on integers and on cores.  
This relationship is best demonstrated by an example before we give the precise statement.

\begin{example}
Start with $n = k+1 = 3$ and the affine Grassmannian element 
$w = s_1 s_0 s_2 s_1 s_0$ that has 
window notation $[w(1), w(2), w(3)] = [-2, 0, 8]$. We can use this to determine that
$w^{-1}$ is given in two line notation as
\begin{center}
\begin{tabular}{cccccccccccccccccc}
&$\cdots$&$-7$&$-6$&$-5$&$-4$&$-3$&$-2$&$-1$&$0$&$1$&$2$&$3$&$4$&$5$&$6$&$7$&$\cdots$\cr
$w^{-1}=$&$\cdots$&$-12$&$-4$&$-2$&$-9$&$-1$&$1$&$-6$&$2$&$4$&$-3$&$5$&$7$&$0$&$8$&$10$&$\cdots$\cr
\end{tabular}
\end{center}
Now we notice in this notation that $w^{-1}(d)\leq 0$ for all integers $d\leq-3$ (and for
any affine permutation there will always be an integer $d'$ such that $w^{-1}(d) \leq 0$ for all $d\leq d'$).
We also see that for all integers $d\geq 6$ we have $w^{-1}(d)>0$
(and for any affine permutation there will always be a value $D'$ such that $w^{-1}(d) >0$ for all $d \geq D'$). 

Now consider the integers $-2 \leq d \leq 5$
(which are the integers strictly between these two values $d'=-3$ and $D'=6$).
Reading from left to right in the two line notation, we construct a path consisting of East
and South steps where for each $d$ such that $w^{-1}(d) \leq 0$ we place a South step, and for each
$d$ such that $w^{-1}(d) > 0$ we place an East step.  In this example we are looking at the sequence
\begin{center}
\begin{tabular}{cccccccc}
$-2$&$-1$&$0$&$1$&$2$&$3$&$4$&$5$\\
$1$&$-6$&$2$&$4$&$-3$&$5$&$7$&$0$
\end{tabular}
\end{center}
in order to create this path.
The way we have chosen our $d'$ and $D'$ the first
step of this path will always be East and the last step will always be South.  
In this example we have the path:
\begin{center}
\includegraphics[width=1in]{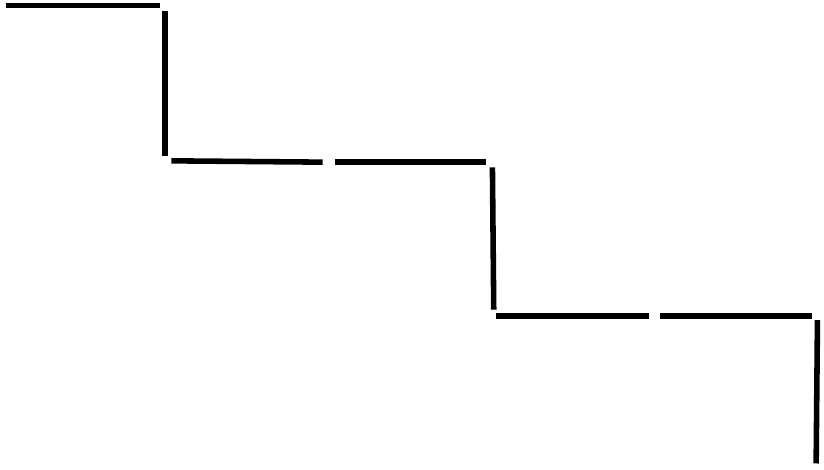}
\end{center}
\noindent
and it is the outline for the diagram of a $3$-core:
{\squaresize=14pt
$$\young{\cr&&\cr&&&&\cr}~.$$}
It is the case that $\mfc(2,2,1) = (5,3,1)$ and 
corresponds to the affine permutations 
$w = s_1 s_0 s_2 s_1 s_0 = \mfa(2,2,1)$ by
the other bijections.
\end{example}

\begin{sagedemo}
We can demonstrate the previous example in \Sage. Using the class
{\sc AffinePermutationGroup} allows to input an affine permutation in
window notation:
\begin{sageexample}
sage: A = AffinePermutationGroup(["A",2,1])
sage: w = A([-2,0,8])
sage: w.reduced_word()
[1, 0, 2, 1, 0]
sage: w.to_core()
[5, 3, 1]
\end{sageexample}
\end{sagedemo}

The action of $t_{ij}$ on the two line notation for $w$ can be translated into the action
for $t_{ij}$ on the two line notation for $w^{-1}$.  Since we have $(t_{ij} w)^{-1} = w^{-1} t_{ij}$,
then the same action of left multiplication by $t_{ij}$ on $w$ has the effect of exchanging
$w^{-1}(i+rn)$ and $w^{-1}(j+rn)$ in the two line notation for $w^{-1}$.  Similarly, right 
multiplication by $t_{ij}$ on $w$ has the effect of exchanging the values of $i+rn$ and $j+rn$ in
the two line notation for $w^{-1}$.

As in our example above, the two line notation for $w^{-1}$ keeps track of the outline of the $(k+1)$-core representing the affine permutation.  The two line notation for $w^{-1}$ 
represents an infinite path where the $0$ and negative values represent South steps and the positive values represent East steps.  There is some point $d'$ for which all steps before are South and another point $D'$ where all steps after are East.  Between $d'$ and $D'$ there is a path which traces the outline of a $(k+1)$-core.  Notice that when $w = w^{-1}$ is the identity, then $d' = 0$ and $D'=1$
and the path between these two points represents the empty core.

We can work out precisely what the action of the transpositions $s_i$ and $t_{ij}$
are on this path and we will
see that left multiplication by $s_i$ on $w$ has the same effect on this
path as the action that $s_i$ has on the $(k+1)$-core given in Definition~\ref{def:siactioncore}.
If left multiplication on $w$ by an $s_i$ increases the length by $1$, then on the sequence
of values of $w^{-1}(i)$ this
has the effect of interchanging a negative value which lies to the left of a positive value.
On the path consisting of South steps for negative values and East steps for positive values,
interchanging a South step that comes just before an East step
has the effect of adding a cell on the path representing the core.

Although we have limited ourselves to affine permutations, $k$-bounded
partitions and $k+1$-cores, we note that there are are other useful ways to describe 
the set such as with abaci or bit sequences that we do not discuss.

\end{subsection}

\begin{subsection}{Weak order and horizontal chains}
\label{subsection.weak order}

Because our indexing set comes from a quotient of $\tilde S_{k+1}$,
and every Coxeter system is naturally equipped with the weak and 
the strong (Bruhat) orders, the close study of
the weak and strong order posets on  $\tilde S_{k+1}$ is called for.
Here we examine posets that are isomorphic to the weak subposet 
on affine Grassmannian elements and whose
vertices are given by the set 
of $k$-bounded partitions and by the set of $(k+1)$-cores.

The (left) weak order is defined by saying that $w$ is less than or equal to $v$ in $\tilde{S}_{k+1}$
if and only if there is some $u\in \tilde S_{k+1}$ such that $uw=v$ and $\ell(u)+\ell(w)=\ell(v)$.
We denote a cover in (left) weak order by $\rightarrow_{k}$, that is,
$w \rightarrow_{k} v$ if and only if there is some $s_i\in \tilde S_{k+1}$
such that $s_iw=v$ and $\ell(w)+1=\ell(v)$.
The affine Weyl group of type $A$ forms a lattice under inclusion 
in the weak order (this is known due to results of Waugh \cite{Waugh:1999}).

Now given the bijection between $\tilde S_{k+1}/S_{k+1}$ and
$k$-bounded partitions or $(k+1)$-cores, it is natural to question 
how weak order $\rightarrow_k$ is characterized on these other sets.
On the set of $(k+1)$-cores, the weak order relation can 
be framed in terms of the action of the elements $s_i$.
This result can be found in \cite{ Lascoux:2001, Leeuwen:1999} and
is restated in a similar form in \cite[Lemma 8.6]{LLMS:2006}.

\begin{prop} If $\kappa$ and $\tau$ are $(k+1)$-cores
with $|\kappa|_{k+1} = |\tau|_{k+1}+1$, then $\tau \rightarrow_k \kappa$
if and only if there exists an $i$ such that $\kappa = s_i \cdot \tau$.
\end{prop}

\begin{remark}
It follows that $\rightarrow_k$ can also be characterized on 
$(k+1)$-cores by
$\tau \rightarrow_k \kappa$ if and only if all cells in 
$\kappa/\tau$ have the same $k+1$-residue.
\end{remark}

\begin{figure}[h]
\begin{center}
\includegraphics[width=4in]{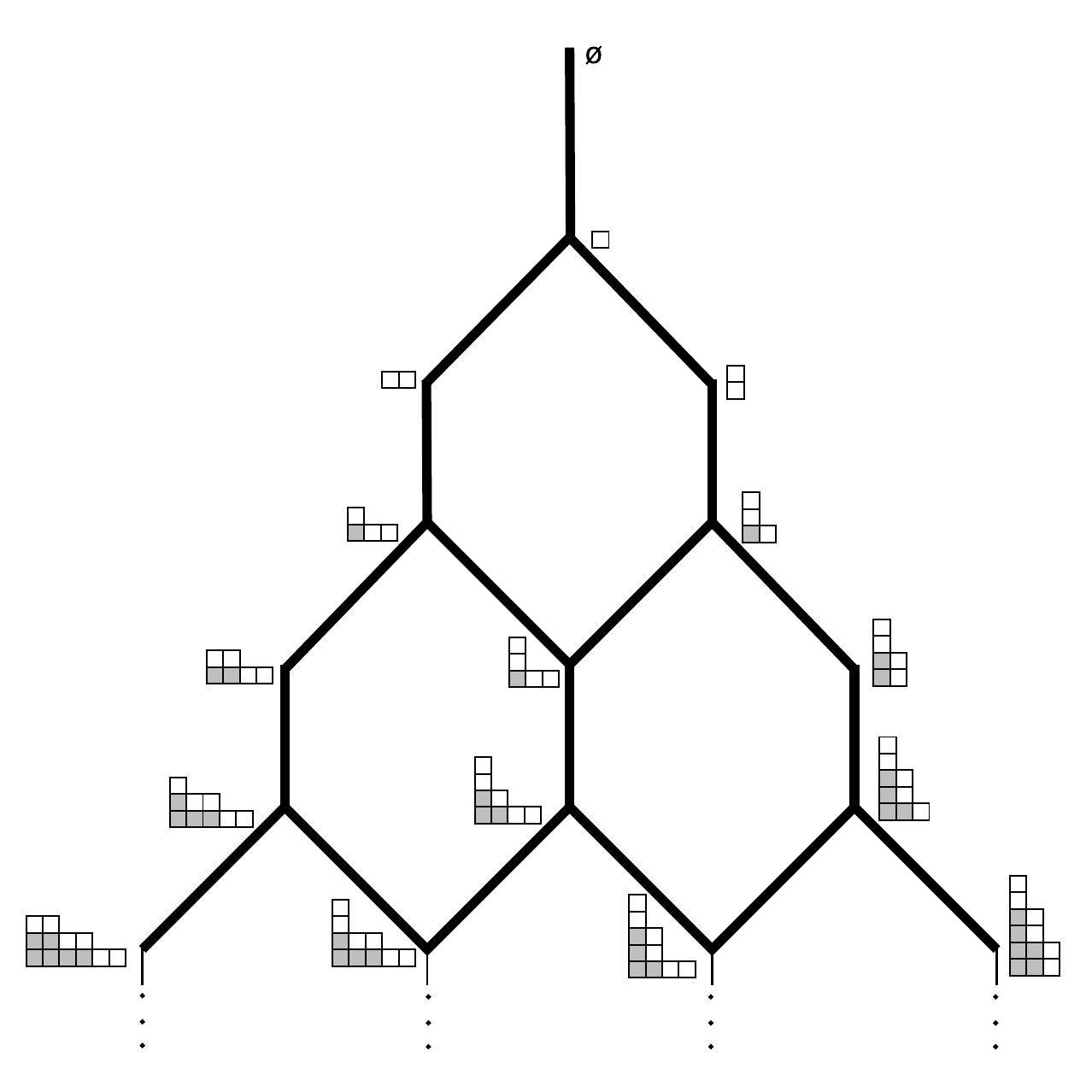}

\caption{
The lattice of 3-cores (up to those of length 6), which correspond to the $2$-bounded partitions,
ordered by the weak order.}\label{partkeq2Hasse}
\end{center}
\end{figure}

\begin{figure}[h]
\begin{center}
\includegraphics[width=6in]{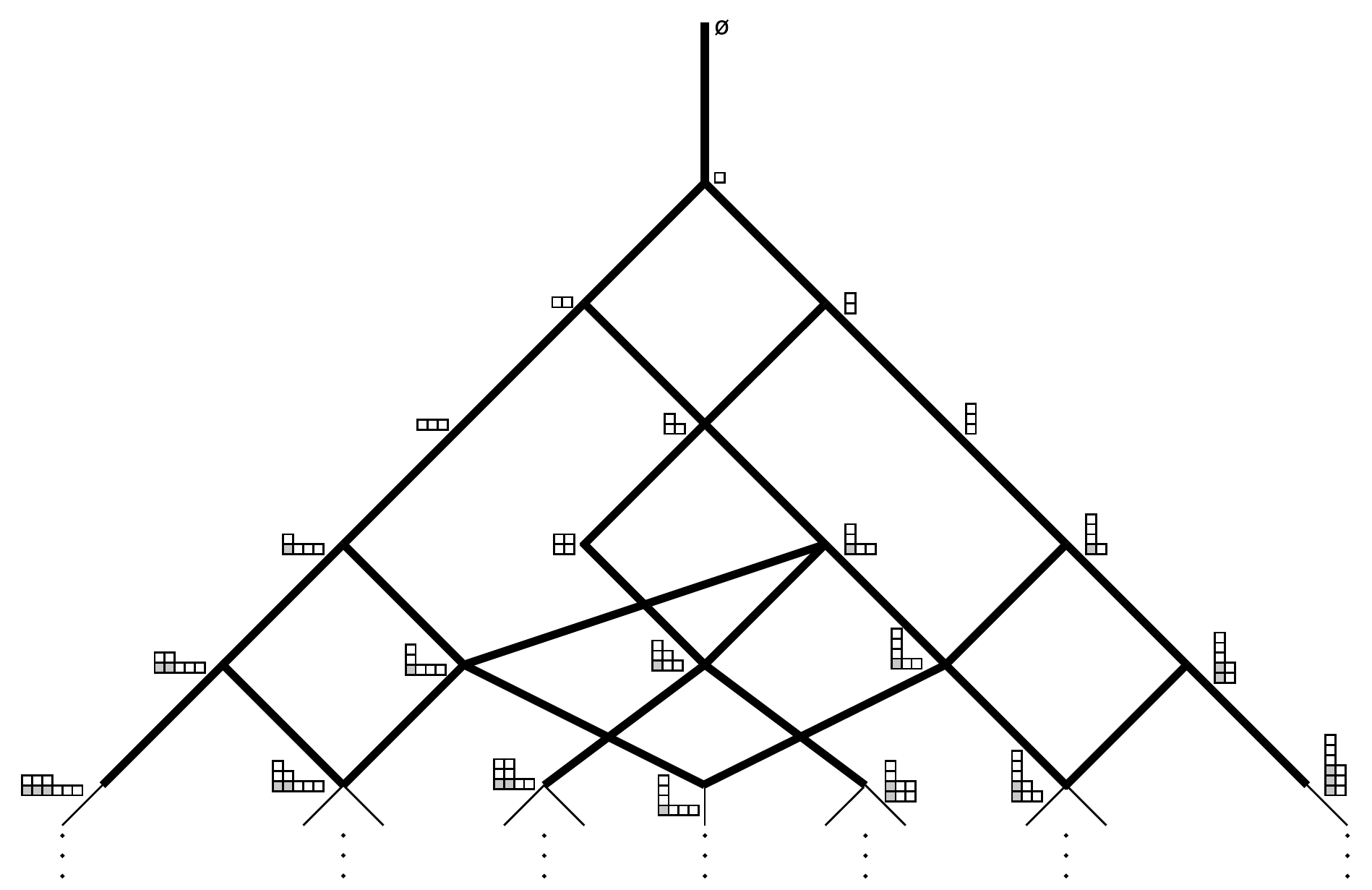}

\caption{
The lattice of 4-cores (up to those of length 6), which correspond to the $3$-bounded partitions,
ordered by the weak order.}\label{partkeq3Hasse}
\end{center}
\end{figure}

The characterization of the weak order poset on $\tilde S_{n}/S_n$
on the level of $k$-bounded partitions
is inspired by viewing the Young covering relation as
$\la \rightarrow \mu$ if $\la \subseteq \mu$ and $\la' \subseteq \mu'$ and 
$|\la|+1 = |\mu|$.  Of course,  $\la' \subseteq \mu'$ if and only if
$\la \subseteq \mu$.  However, working on the subset of $k$-bounded partitions, 
there is a generalization of conjugation under which this is not a 
superfluous condition.

\begin{definition} \label{def:kconjugate}
Let $\la$ be a $k$-bounded partition. 
Then the $k$-{\em conjugate} of $\la$ is defined as
\[
	\la^{\omega_k} := \mfp(\mfc(\la)').
\]
\end{definition}

\begin{sagedemo}
We can obtain the $4$-conjugate of the partition $(4,3,3,3,2,2,1)$ from the last partition in Example~\ref{example.to_core}
by reading off the column lengths of the unshaded boxes in each column. In \Sage:
\begin{sageexample}
    sage: la = Partition([4,3,3,3,2,2,1])
    sage: la.k_conjugate(4)
    [3, 2, 2, 2, 2, 1, 1, 1, 1, 1, 1, 1]
\end{sageexample}
\end{sagedemo}

\begin{prop} (\cite[Corollary 25]{LM:2005})
For $k$-bounded partitions $\lambda$ and $\mu$,
$\lambda \rightarrow_k \mu$ if and only if
$\lambda \subseteq \mu$,
$\lambda^{\omega_k}\subseteq\mu^{\omega_k}$, and $|\lambda|+1=|\mu|$.  
\end{prop}

We now turn our attention to a distinguished set of saturated chains.
Recall that standard tableaux can be viewed
as saturated chains in the Young lattice. This notion can be generalized
to semi-standard tableaux. A semi-standard tableau is an increasing
sequence of partitions in the Young lattice
such that two adjacent partitions in this sequence differ by a horizontal strip. 
Horizontal strips are skew shapes
with at most one cell in any column. As we will see in Section~\ref{sec:colstrict}
horizontal strips are fundamental in the formulation of the Pieri rule for 
Schur functions. The analogue of horizontal strips in the affine setting
was introduced in~\cite{LM:2005}.  As we will see in Section~\ref{sec:weaksection}
these will play a central role in the combinatorics of $k$-Schur functions.

Crudely, we define a {\em weak horizontal strip} of size $r \leq k$ to be
a horizontal strip $\kappa\slash\tau$ of $(k+1)$-cores $\kappa$ and $\tau$
such that there exists a saturated chain 
\begin{equation}\label{eq:weakhstrip}
\tau \rightarrow_k \tau^{(1)} \rightarrow_k \tau^{(2)} 
\rightarrow_k \cdots \rightarrow_k \tau^{(r)} = \kappa
\,.
\end{equation}
It is helpful to instead think of these strips as
the skew $\kappa\slash\tau$ of $(k+1)$-cores $\kappa$ and $\tau$,
where
\begin{align}
& \kappa/\tau\;\text{is a horizontal strip}\\
& |\kappa|_{k+1}=|\tau|_{k+1}+r\\
& \text{there are exactly $r$ residues in the set of cells of $\kappa/\tau$.}
\label{goodweakstrip}
\end{align}

We discussed how the weak order is naturally realized on 
$k$-bounded partitions and affine Grassmannian elements
and it is also worthwhile to rephrase the notion of strips 
in these different contexts.

In the $k$-bounded partition framework
(for example, useful in \cite{LLM:2003,LM:2003,LM2:2003,LM:2007,BandlowSchillingZabrocki}),
the notion of weak horizontal strip is defined
to be a horizontal strip $\mu/\la$ where $\mu^{\omega_k}/\la^{\omega_k}$
is a vertical strip.  This characterization is motivated by the
following result about weak horizontal strips.
\begin{prop} \label{th:hstrip} (\cite[Section 9]{LM:2005}
Let $\tau\subseteq \kappa$ be $(k+1)$-cores. 
Then $\kappa\slash\tau$ forms a weak horizontal 
strip if and only if 
$\mfp(\kappa)/\mfp(\tau)$ is a horizontal strip and
$\mfp(\kappa')/\mfp(\tau')$ is a vertical strip.
\end{prop}

It is important to note
that it is not sufficient to characterize $\kappa\slash\tau$ being a weak horizontal strip
by assuming that $\mfp(\kappa)/\mfp(\tau)$ is a horizontal strip.  A good example of this can
be observed in Figure \ref{partkeq3Hasse}. Consider the $4$-cores $\kappa = (4,1)$ and $\tau = (2,1)$.  
We note that $\mfp(\kappa) = (3,1)$ and $\mfp(\tau) = (2,1)$ and so even though we see that 
$\mfp(\kappa)/\mfp(\tau)$ is a horizontal strip, there does not exist a path from the
$4$-core $(2,1)$ to $(4,1)$ in the lattice.  If we consider the conjugate
partitions, we see that $\mfp(\kappa') = (1,1,1,1)$ and 
$\mfp(\tau') = (2,1)$ and so $\mfp(\kappa')/\mfp(\tau')$ is not a vertical strip.

Before we discuss the notion of weak horizontal strip in the framework of
affine permutations, we give an alternative description of the $k$-conjugate 
and the map $\mfc$ from $k$-bounded partitions to $(k+1)$-cores that was
communicated to us by Karola M\'esz\'aros~\cite{Meszaros:2011}.
Let $\la = (\la_1,\la_2,\ldots,\la_\ell)$ be a
$k$-bounded partition. Start from the longest part of $\la$, namely $\la_1$, and successively
connect a row of length $i$ to the $(k+1-i)$th row above it (in other words, skip $k-i$ rows). 
Call this a string. Repeat this with the next longest
part that is not yet part of a string until all parts of $\la$ are part of a string.
Denote by $\{\ell^{(j)}_1,\ell^{(j)}_2,\ldots\}$ the parts of $\la$ in the $j$th string where
$1\le j\le r$. Then
\begin{equation}
	(\la^{\omega_k})' = (\ell_1^{(1)}+\ell_2^{(1)} + \cdots , \ell_1^{(2)}+\ell_2^{(2)} + \cdots, \ldots,
	   \ell_1^{(r)} + \ell_2^{(r)} + \cdots).
\end{equation}
Similarly, the $i$th part of $\mfc(\la)$ is obtained by adding all elements in the string containing
$\la_i$ that are smaller or equal to $\la_i$ (or equivalently, all elements in the string of $\la_i$
above $\la_i$).

\begin{example}
Let $\la=(3,3,3,2,1)$ with $k=4$. Then the strings are as follows:
\begin{center}
\includegraphics[width=3cm]{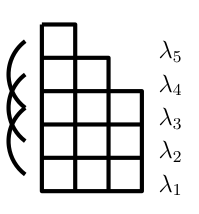}
\end{center}
and 
\begin{equation*}
\begin{split}
	(\la^{\omega_4})' &= (\la_1+\la_3+\la_5, \la_2+\la_4)=(7,5),\\
	\mfc(\la) &= (\la_1+\la_3+\la_5,\la_2+\la_4, \la_3+\la_5, \la_4, \la_5)=(7, 5, 4, 2, 1).
\end{split}
\end{equation*}
\end{example}

Lastly, we interpret horizontal chains in the framework of affine
permutations.  In particular, weak horizontal strips are 
the {\it cyclically decreasing} elements of $\tilde S_{k+1}/S_{k+1}$.
These are affine Grassmannian elements $w$ where
$w= s_{i_1}\cdots s_{i_\ell}$ for a sequence $i_1\cdots i_\ell$ such that
no number is repeated and $j$ precedes $j-1$ (taken modulo $k+1$)
when both $j,j-1\in \{i_1,\ldots, i_\ell\}$.
Then, based on the following proposition, a weak horizontal strip can be thought of as a pair of 
affine Grassmannian elements $v$ and $w$ where $uw=v$ 
and $\ell(u)+\ell(w)=\ell(v)$ for some cyclically decreasing $u$.
\begin{prop}
Let $\tau\subseteq \kappa$ be $(k+1)$-cores. 
Then $\kappa\slash\tau$ forms a weak horizontal 
strip if and only if $\kappa=s_{i_1}\cdots s_{i_\ell}\tau$
for some cyclically decreasing element $w = s_{i_1}\cdots s_{i_\ell}$.
\end{prop}
\end{subsection}

\begin{subsection}{Cores and the strong order of the affine symmetric group}
\label{subsection.strong order}

We are now ready to discuss the strong (Bruhat) order, starting from the
core viewpoint \cite{MisraMiwa:1990,Lascoux:2001}.
A {\em strong cover} is defined on $k+1$ cores by
$$
\tau \Rightarrow_k \kappa \iff
|\mfp(\tau)|+1=|\mfp(\kappa)|\quad\text{and}\quad \tau \subseteq \kappa\,.
$$
When a pair of cores satisfies $\tau\Rightarrow_i\kappa$,
their skew diagram has is made up of ribbons.
To be precise, the {\it head} of a connected ribbon
is the southeast most cell of the ribbon and the {\it tail} is the 
northwest most corner.  Then (see \cite[Proposition 9.5]{LLMS:2006}):
\begin{itemize}
\item each connected component of $\kappa/\tau$ is a ribbon and they are all identical translates
each other;
\item 
the residues of the heads
of the connected components must all be the same and must lie on consecutive positions of those residues 
(the term `consecutive' here means that if two heads are separated by a multiple of $k+1$ cells by a taxicab 
distance then there is one which is exactly $k+1$ distance that is in-between).
\end{itemize}

A {\em strong marked cover} is a strong cover along with a value $c$ which 
indicates the content of the head of one of the copies of the ribbons.
More precisely, we define a {\it marking} as a triple $(\kappa, \tau, c)$ where
 $\kappa,\tau$ are $(k+1)$-cores such
that $\tau \Rightarrow_{k} \kappa$ and $c$ is a number which is $j-i$ for the cell 
(the diagonal index of the cell)
at position $(i,j)$ of the south-east most cell of the connected component of $\kappa/\tau$ which
is marked.

\begin{example}
\label{Ex:strong4core}
Consider the $4$-cores 
$$\tau = (19, 16, 13, 10, 7, 7, 5, 5, 3, 3, 1, 1, 1) \Rightarrow_3 (22, 19, 16, 13, 10, 7, 5, 5, 3, 3, 1, 1, 1) = \kappa$$
which correspond to the skew diagram
$$\squaresize=9pt\young{\gris\cr\gris\cr\gris\cr\gris&\gris&\gris\cr\gris&\gris&\gris\cr\gris&\gris&\gris&\gris&\gris\cr\gris&\gris&\gris&\gris&\gris\cr\gris&\gris&\gris&\gris&\gris&\gris&\gris\cr\gris&\gris&\gris&\gris&\gris&\gris&\gris&&&\cr\gris&\gris&\gris&\gris&\gris&\gris&\gris&\gris&\gris&\gris&&&\cr\gris&\gris&\gris&\gris&\gris&\gris&\gris&\gris&\gris&\gris&\gris&\gris&\gris&&&\cr\gris&\gris&\gris&\gris&\gris&\gris&\gris&\gris&\gris&\gris&\gris&\gris&\gris&\gris&\gris&\gris&&&\cr\gris&\gris&\gris&\gris&\gris&\gris&\gris&\gris&\gris&\gris&\gris&\gris&\gris&\gris&\gris&\gris&\gris&\gris&\gris&&&\cr}$$
This is a relatively large example, where $\kappa/\tau$ contains $5$ different copies of $3$ cells in a
row.  The marking, $c$, can be any one of the $5$ values representing the content
of the rightmost cell in the connected component, $c \in \{ 21, 17, 13, 9, 5 \}$.
\end{example}

\begin{figure}[h]
\begin{center}
\includegraphics[width=4in]{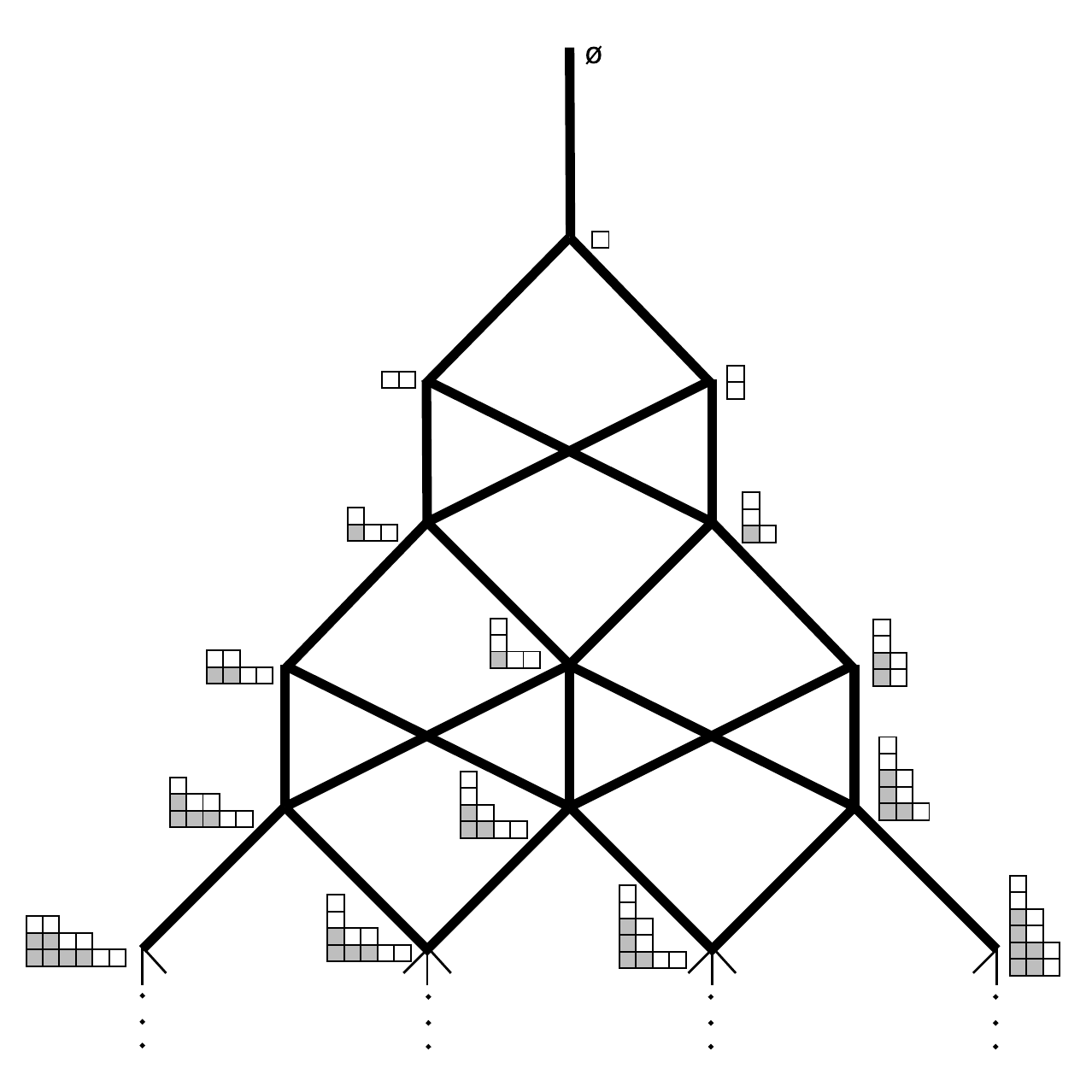}

\caption{
The poset of 3-cores (up to those of length 6), ordered by the strong order (no markings)}\label{strongkeq2Hasse}
\end{center}
\end{figure}

\begin{figure}[h]
\begin{center}
\includegraphics[width=6in]{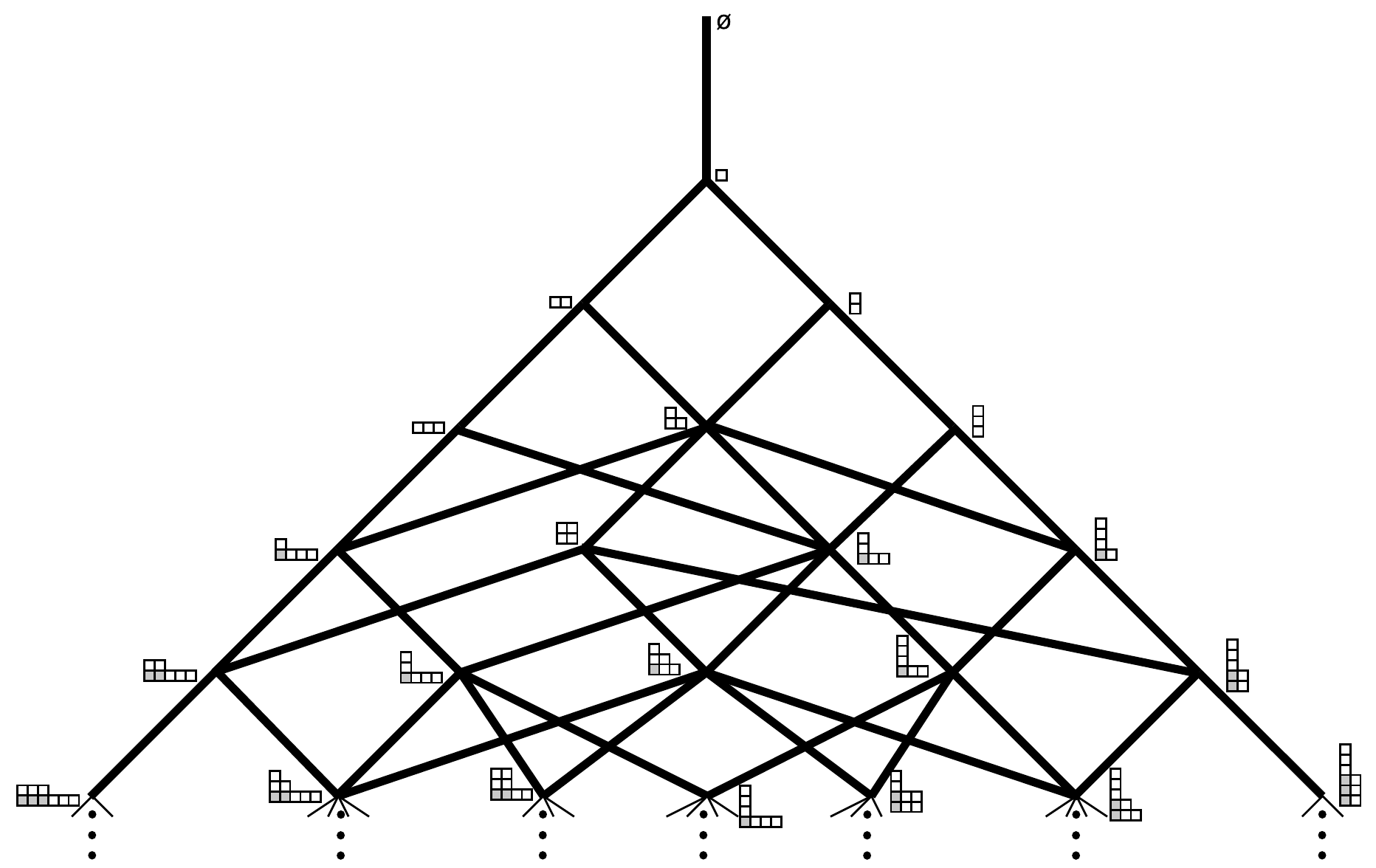}

\caption{The poset of 4-cores (up to those of length 6), ordered by the strong order (no markings) }\label{strongkeq3Hasse}
\end{center}\label{fig:strongkeq3}
\end{figure}

\begin{example}
The diagram for the poset
of the strong order at $k=2$ is given in Figure \ref{strongkeq2Hasse}
and it can be read off the diagram that
the strong covers of $(3,1,1)$ with respect to this order are $(5,3,1)$, 
$(4,2,1,1)$ and $(3,2,2,1,1)$ because $(3,1,1)$ is contained in each of these $3$-cores.
Note that it is not true that there is containment of the corresponding $2$-bounded partitions
since $\mfp(3,1,1) = (2,1,1)$ and $\mfp(3,2,2,1,1) = (1,1,1,1,1)$.

In particular, if $\tau = (3,1,1)$ and $\kappa = (5,3,1)$, then $(3,1,1) \Rightarrow_2 (5,3,1)$ because
$(5,3,1)/(3,1,1)$ consists of two copies of a connected horizontal strip.  
This means that $( (5,3,1)/(3,1,1), 1 )$ and $( (5,3,1)/(3,1,1), 4 )$ are strong
marked covers since the cells
$(2,3)$ and $(1,5)$ are the coordinates of the heads of the horizontal strips
and they have content $1$ and $4$ respectively.
These are represented by the skew tableaux
$$\squaresize=9pt\young{\gris\cr\gris&&*\cr\gris&\gris&\gris&&\cr}\hskip .3in
\young{\gris\cr\gris&&\cr\gris&\gris&\gris&&*\cr}~.$$
\end{example}

Although we started with a discussion of strong order in 
the setting of $k+1$-cores, it comes from ordering elements
$u,w$ of the Coxeter system $\tilde S_{k+1}$ by
$$
w\Rightarrow_k u \iff t_{ij} w=u \quad\text{and}\quad
\ell(w)+1=\ell(u)\,.
$$
The notion of marked covers can be interpreted in this framework as well.

\begin{prop}\label{prop:strongcovertransposition} \cite[Section 2.3]{LLMS:2006}
Let $\tau, \kappa$ be two $(k+1)$-cores such that $\tau \Rightarrow_k \kappa$ 
and assume there is a marking of $\kappa/\tau$ at diagonal $j-1$ and $i$ 
is the diagonal index of the tail of the marked ribbon.
Let $w$ be the affine Grassmannian permutation corresponding to $\tau$ 
and $u$ the affine Grassmannian
permutation corresponding to $\kappa$.  Then we have
\begin{itemize}
\item $w^{-1}(i) \leq 0 < w^{-1}(j)$.
\item $t_{ij} w = u$.
\item The number of connected ribbons which are below the marked one is 
$(-w^{-1}(i) - a)/n$ where $a = -w^{-1}(i)~(mod~n)$ (the representative between 0 and $n$).
\item The number of connected ribbons which are above the marked one is 
$(w^{-1}(j) - b)/n$ where $b = w^{-1}(j)~(mod~n)$ and the total number
of connected components in $\kappa/\tau$ is $1+(-w^{-1}(i)+ w^{-1}(j) - a - b)/n$.
\item The number of cells in the ribbon is $j-i$.
\item The height of the ribbon is the number of $d$ such that $i\leq d<j$ such
that $w^{-1}(d)\leq0$.
\end{itemize}
\end{prop}

\begin{example}
Consider the $3$-core $\tau = (5,3,1)$ which corresponds to the reduced word
$w = s_1 s_0 s_2 s_1 s_0$ and which is also given by its action 
on the integers by $[w(1), w(2), w(3)] = [-2, 0, 8]$.  
Then $t_{-10} = t_{23} = t_{56} = s_2$ and $\kappa = 
t_{-10} \cdot (5,3,1) = (6,4,2)$.
\[\squaresize=9pt
\young{\gris&*\cr\gris&\gris&\gris&\cr\gris&\gris&\gris&\gris&\gris&\cr},
\hskip .2in
\young{\gris&\cr\gris&\gris&\gris&*\cr\gris&\gris&\gris&\gris&\gris&\cr} ,
\hskip .2in
\young{\gris&\cr\gris&\gris&\gris&\cr\gris&\gris&\gris&\gris&\gris&*\cr}
\]
The three markings are on diagonal $-1$ and $2$ and $5$.  
Moreover since
$w^{-1}(-1) = -6 \leq 0 < w^{-1}(0) = 2$, $w^{-1}(2) = -3 \leq 0 < w^{-1}(3) = 5$
and $w^{-1}(5) = 0 \leq 0 < w^{-1}(6) = 8$, these three transpositions satisfy
the conditions of the proposition for each of the three marked strong covers.
The right action is expressed as the element $w t_{-62} = w t_{-35} = w t_{08}$.

The number of connected components is equal to 
$1+(-w^{-1}(i)+w^{-1}(j)-a-b)/n$ which for $j=0$ and $i=-1$ amounts to
$1+(6+2-0-2)/3=3$.
\end{example}
\begin{example}
The $4$-core $\tau$ in Example~\ref{Ex:strong4core}
corresponds to the reduced word
$$w = s_1 s_2 s_0 s_3 s_1 s_2 s_0 s_3 s_1 s_2 s_1 s_0 s_3 s_1 s_2 s_1 s_0 s_3 s_0 s_1 s_2 s_1 s_0 s_3 
s_2 s_1 s_0$$
and $t_{19,22} w = t_{15,18} w = t_{11,14} w = t_{7,10} w = t_{3,6} w$ represent the
five strong covers with the left action and
$w t_{0,19} = w t_{-4,15} = w t_{-8,11} = w t_{-12,7} = w t_{-16,3}$ represent the five strong
covers with the right action.
\end{example}
\begin{remark} Recall that for two elements $w, w' \in S_{n}$ 
with $\ell(w') = \ell(w) + 1$, $w'$ is a cover of $w$ in the (left) weak 
order if 
$$s_i w = w'$$
for some simple transposition $s_i$
and $w'$ is a cover of $w$ in the strong (or Bruhat)
order if
$$t_{ij} w = w'$$
for some transposition $t_{ij}$.  
By analogy, since $\Rightarrow_k$ is a left multiplication by
an affine transposition (Proposition \ref{prop:strongcovertransposition}) 
and $\rightarrow_k$ is left multiplication by a simple affine
transposition on cores (from Definition \ref{def:siactioncore}), the cover relations are called strong and weak covers, respectively.
\end{remark}

We have included the Hasse diagrams for the weak and strong orders
for the poset of $3$ and $4$-cores up to those of length $6$ 
in Figures~\ref{partkeq2Hasse} through~\ref{strongkeq3Hasse}.  Note
that the strong order on cores does not form a lattice.

\begin{sagedemo}
We can produce the weak and strong covers of a given core
\begin{sageexample}
    sage: c = Core([3,1,1],3)
    sage: c.weak_covers()
    [[4, 2, 1, 1]]
    sage: c.strong_covers()
    [[5, 3, 1], [4, 2, 1, 1], [3, 2, 2, 1, 1]]
\end{sageexample}
as well as compare two $(k+1)$-cores with respect to weak and strong order
\begin{sageexample}
    sage: kappa = Core([4,1],4)
    sage: tau = Core([2,1],4)
    sage: tau.weak_le(kappa)
    False
    sage: tau.strong_le(kappa)
    True
\end{sageexample}
Figure~\ref{partkeq3Hasse} can be reproduced in \Sage via:
\begin{sageexample}
    sage: C = sum(([c for c in Cores(4,m)] for m in range(7)),[])
    sage: ord = lambda x,y: x.weak_le(y)
    sage: P = Poset([C, ord], cover_relations = False)
    sage: H = P.hasse_diagram()
    sage: view(H)        #optional
\end{sageexample}
\end{sagedemo}

As with the weak order poset, we are also concerned with certain
chains in the strong order poset.
A {\it strong marked horizontal strip} of size $r$ 
is a succession of strong marked covers,
\begin{equation} \label{equation.strong h-strip}
\kappa^{(0)} \Rightarrow_{k} \kappa^{(1)} \Rightarrow_{k} \kappa^{(2)} 
\Rightarrow_{k} \cdots
\Rightarrow_{k} \kappa^{(r)}
\,,
\end{equation}
where the markings $c_i$ associated to $\kappa^{(i-1)}\Rightarrow_k\kappa^{(i)}$ 
satisfy $c_1 < c_2 < \cdots < c_r$.

\begin{example} 
\squaresize = 6pt
Consider the following sequence of $4$-cores with the markings indicated on the diagram.
$$\young{&\cr} \Rightarrow_3 \young{\times\cr&\cr} \Rightarrow_3 \young{\times\cr\cr&&\cr} \Rightarrow_3 \young{\cr\cr&&&\times\cr}
\hbox{, or more compactly denoted }{\squaresize=10pt \small \young{2^\ast\cr1^\ast\cr&&2&3^\ast\cr}}~.$$
This is not a strong marked horizontal strip because the markings $c_1 = -1, c_2 = -2, c_3 = 3$
are not increasing. However, the same set
of $4$-cores with markings
$$\young{&\cr} \Rightarrow_3 \young{\times\cr&\cr} \Rightarrow_3 \young{\cr\cr&&\times\cr} \Rightarrow_3 \young{\cr\cr&&&\times\cr}
\hbox{, or more compactly denoted }{\squaresize=10pt \small \young{2\cr1^\ast\cr&&2^\ast&3^\ast\cr}}~.$$
as shown have $c_1 = -1 < c_2 = 2 < c_3 = 3$ and hence this is a strong marked horizontal strip.
\end{example}

\begin{remark}
\label{strongstripkinf}
It is worth pointing out that when $k$ is large, a succession of strong marked covers,
\begin{equation} 
\kappa^{(0)} \Rightarrow_{k} \kappa^{(1)} \Rightarrow_{k} \kappa^{(2)} 
\Rightarrow_{k} \cdots
\Rightarrow_{k} \kappa^{(r)}
\,,
\end{equation}
reduces to a saturated chain in Young's lattice
where the condition that $c_1<c_2<\cdots <c_r$ implies that
$\kappa^{(r)}/\kappa^{(0)}$ is a horizontal strip.
\end{remark}

\end{subsection}

\begin{subsection}{Symmetric functions} \label{sec:symfunc}
The ring of symmetric functions shall be defined as
\begin{equation}
\La = \QQ[ h_1, h_2, h_3, \ldots],
\end{equation}
the ring of polynomials in the generators $h_r$.
Here we are considering the ring $\La$ without reference to `variables' for which there
is a symmetric group action but we will now make explicit the
connection with symmetric polynomials and symmetric series.  

Let $\sigma$ be a permutation that acts on the variables $\{ x_1, x_2, x_3, \ldots, x_m \}$ by
$\sigma( x_i ) = x_{\sigma_i}$ and this action extends to polynomials.  We call a polynomial
$S_m$-invariant (or symmetric) if 
$\sigma f(x_1, x_2, \ldots, x_m) = f(x_1, x_2, \ldots, x_m)$ for all $\sigma \in S_m$.
The ring $\Lambda$ is identified with functions which are symmetric series in an infinite 
set of variables by setting
\begin{equation}
h_r[X] = \sum_{1 \leq i_1 \leq i_2 \leq \cdots \leq i_r} x_{i_1} x_{i_2} \cdots x_{i_r}
\end{equation}
and symmetric polynomials are then just a specialization of these series with a finite number
of variables, $h_r[X_m] = \sum_{1 \leq i_1 \leq i_2 \leq \cdots \leq i_r \leq m} x_{i_1} x_{i_2} \cdots x_{i_r}$.

For each partition $\lambda=(\la_1,\la_2,\ldots,\la_{\ell(\la)})$, we set 
\begin{equation}
h_{\lambda}[X] := h_{\la_1}[X] h_{\la_2}[X] \cdots h_{\la_{\ell(\la)}}[X]~.
\end{equation}
The set of these symmetric series forms a linear basis for an algebra isomorphic to $\La$.
We will consider various bases for $\La$. One such basis is the {\it monomial basis}
\begin{equation}
m_\lambda[X] = \sum_{\sort(\alpha) 
= \lambda} x_1^{\alpha_1} x_2^{\alpha_2} \cdots x_{\ell(\alpha)}^{\alpha_{\ell(\alpha)}}\; ,
\end{equation}
where the sum is over all sequences $\alpha$ such that if the parts are arranged in weakly decreasing
order the resulting sequence is the partition $\lambda$.  It is not hard to see from the
definitions that the generators of $\La$ are related to the monomial
symmetric functions by 
\begin{equation}\label{hominmonom}
h_r[X] = \sum_{\la \vdash r} m_\la[X]~.
\end{equation}
The $h_r$ are known as the {\it complete homogeneous generators} and the
their products will be referred to as the complete homogeneous or simply homogeneous basis.

The {\it power sum} generators are the elements 
\begin{equation} \label{prexpr}
p_r[X] = m_{(r)}[X] = \sum_{i \geq 1} x_i^r
\end{equation}
and the {\it elementary generators} are defined as
\begin{equation}
e_r[X] = m_{(1^r)}[X] = \sum_{1 \leq i_1 < i_2 < \cdots < i_r} x_{i_1} x_{i_2} \cdots x_{i_r}~.
\end{equation}

The monomials in these sets of generators,
$p_\la[X] := p_{\la_1}[X] p_{\la_2}[X] \cdots p_{\la_{\ell(\la)}}[X]$
and $e_\la[X] := e_{\la_1}[X] e_{\la_2}[X] \cdots e_{\la_{\ell(\la)}}[X]$ also
form bases for the space $\La$ indexed by partitions.

The variable $X$ in these symmetric functions is, for the moment, superfluous notation and
to make certain formulas more compact we will drop $[X]$ when it is implicit that it is there.  
However, we will later consider transformations on the ring of symmetric functions by adding
notation to $[X]$ and there will be times that the $[X]$ will be used to indicate that the expression it is
attached to is a symmetric function.

One place where it will be necessary to keep the reference to the variables explicit is
in the use of a few `plethystic' expressions involving parameters $q$ and $t$.  We may extend the
notation defined above to include all rational expressions in variables $q, t, x_1, x_2, x_3, \ldots$,
$E = E(x_1, x_2, \ldots; q, t)$, then
\begin{equation} \label{eq:plethnot}
p_r[E] = E(x_1^r, x_2^r, \ldots; q^r, t^r)~.
\end{equation}
Note that by using the expression $E$ as
the infinite sum $X = x_1 + x_2 + x_3 + \cdots$, the notation in Equation~\eqref{eq:plethnot}
is consistent with the expression in Equation~\eqref{prexpr}.  

In particular we will frequently use the notation $f\left[ \frac{X}{1-t}\right]$, $f[X(1-t)]$ 
and $f\left[X\frac{1-q}{1-t}\right]$
to represent the symmetric function $f[X]$ with $p_r[X]$ replaced with $p_r[X]/(1-t^r)$, $p_r[X](1-t^r)$ and
$p_r[X]\frac{1-q^r}{1-t^r}$ respectively.  This transformation is sometimes also denoted by
\begin{equation} \label{equation.theta}
\theta_{qt} f[X] = f\left[X\frac{1-q}{1-t}\right]~.
\end{equation} 

When we need to use a finite number of variables, the expression $X_m = x_1+x_2+x_3+ \cdots +x_m$ is
used to indicate that $p_r[X_m] = \sum_{i=1}^m x_i^r$.  Normally we consider symmetric functions
in an arbitrary alphabet which can be specialized appropriately and we assume
that there is an implicit $[X]$ following all symmetric function expressions 
where no variables are specified.

The three types of generators
are related by
\begin{align}
\sum_{i=0}^r (-1)^{r-i} h_{i} e_{r-i} &= 0\hskip .75in &r h_r = \sum_{i=1}^r h_{r-i} p_i\\
\sum_{i=0}^{r} (-1)^{r-i} i h_{i} e_{r-i}&= p_r \hskip .75in &r e_r = \sum_{i=1}^r (-1)^{i-1} e_{r-i} p_i\, .
\label{eq:perelation}
\end{align}
These relations are sufficient to express any one of the generators $\{ e_r, h_r, p_r \}$
in terms of another.

There is a scalar product on the ring of symmetric functions for which the monomial and homogeneous
symmetric functions are orthonormal.  We have also,
\begin{equation}\label{scalardual}
\left< h_\la, m_\mu \right> = \left< p_\la, p_\mu/z_\mu \right> 
 = \delta_{\la\mu}
 := \begin{cases} 1&\hbox{ if }\la = \mu,\\
  0&\hbox{ otherwise,}\end{cases}
\end{equation}
where $z_\mu$ is defined in Equation \eqref{zladef}.

One powerful use of this scalar product is that it allows us to compute
a single coefficient in the expansion of a symmetric function in terms
of these bases.  If $f \in \La$, then $\left< f, h_\mu\right>$ is
the coefficient of $m_\mu$ in $f$.  That is, if $f = \sum_{\gamma} c_\gamma m_\gamma$,
then
\begin{equation}\label{mmucoeff}
\left< f, h_\mu \right> = \sum_\gamma c_\gamma \left< m_\gamma, h_\mu \right> = c_\mu~.
\end{equation}
Similarly $\left< f, m_\mu \right>$ is the coefficient of $h_\mu$ in $f$ and $\left< f, p_\mu/z_\mu \right>$
is equal to the coefficient of $p_\mu$ in $f$.

\begin{sagedemo} \label{sageexample.hm}
We now show how to create various bases in \Sage and how to obtain the coefficients of a given symmetric
function using the computer. We begin by defining the homogeneous and monomial bases:
\begin{sageexample}
    sage: Sym = SymmetricFunctions(QQ)
    sage: h = Sym.homogeneous()
    sage: m = Sym.monomial()
\end{sageexample}
Then we define a symmetric function $f$ and expand it in terms of the monomial basis:
\begin{sageexample}
    sage: f = h[3,1]+h[2,2]
    sage: m(f)
    10*m[1, 1, 1, 1] + 7*m[2, 1, 1] + 5*m[2, 2] + 4*m[3, 1] + 2*m[4]
\end{sageexample}
There are several ways to obtain the coefficients of a given term. Both of the following yield the coefficient of
$m_{211}$ in $f$:
\begin{sageexample}
    sage: f.scalar(h[2,1,1])
    7
    sage: m(f).coefficient([2,1,1])
    7
\end{sageexample}
The order in which bases are multiplied and added determines which is the output basis.  For instance
to demonstrate \eqref{eq:perelation} we consider the following two equivalent expressions:
\begin{sageexample}
    sage: p = Sym.power()
    sage: e = Sym.elementary()
    sage: sum( (-1)**(i-1)*e[4-i]*p[i] for i in range(1,4) ) - p[4]
    4*e[4]
    sage: sum( (-1)**(i-1)*p[i]*e[4-i] for i in range(1,4) ) - p[4]
    1/6*p[1, 1, 1, 1] - p[2, 1, 1] + 1/2*p[2, 2] + 4/3*p[3, 1] - p[4]
\end{sageexample}
\end{sagedemo}
\end{subsection}

\begin{subsection}{Schur functions}
\label{sectionschur}
A combinatorial definition of the Schur functions 
is given in Section~\ref{sec:Pierirules} which we 
shall generalize to the $k$-Schur functions in Section~\ref{sec:weaksection} and the dual $k$-Schur 
functions in Section~\ref{sec:strongsection}.  
For now we start with two (equivalent) algebraic definitions.

\begin{definition} The Schur functions $s_\la$ are the unique basis of $\Lambda$ for which
\begin{enumerate}
\item $\left< s_\la, s_\mu \right> = \delta_{\lambda\mu}$ for any partitions $\la,\mu$;
\item
$s_\la = m_{\la} +$ terms of the form $r_{\la\mu} m_{\mu}$ 
for partitions $\mu$ of $|\la|$ with $\mu<\la$ in dominance order.
\end{enumerate} 
\end{definition}

We have chosen this as the definition of the Schur functions because it naturally generalizes to the Hall--Littlewood
and Macdonald symmetric functions (which we shall introduce in the next few pages).  
There are many formulas known for the Schur functions which
can either be taken as a defining relation or as a consequence.
In the next section we shall shift perspectives and consider 
the Schur functions as the family of symmetric functions 
which satisfy the Pieri rule (see Equation \eqref{eq:pierirule}).

\begin{example} \label{example.sn3}
In the following we abbreviate $s_{(2,1)}$ by $s_{21}$ if there is no confusion about the parts.
By definition we have that
$$s_{111} =      m_{111}$$
and we may proceed by calculating Gram-Schmidt orthonormalization.
For instance, since $s_{111} = m_{111} = e_3 = h_{111} - 2 h_{21} + h_3$ we have
$$s_{21} = m_{21} - \left< m_{21}, s_{111} \right> s_{111} = m_{21} + 2 m_{111}~.$$
If then $s_{21}$ is expanded in the homogeneous basis, we see that it is $s_{21} = h_{21}-h_3$.
Finally, to calculate $s_3$ we note that
$$s_{3} = m_3 - \left< m_3, s_{21} \right> s_{21} - \left< m_3, s_{111} \right> s_{111}
= m_3 + m_{21} + m_{111}~.$$
\end{example}

\begin{sagedemo} \label{sageexample.sp}
If we wanted to check Example~\ref{example.sn3} using \Sage, we could define $h$ and $m$ as in
Sage Example~\ref{sageexample.hm} and then run
\begin{sageexample}
    sage: Sym = SymmetricFunctions(QQ)
    sage: s = Sym.schur()
    sage: m = Sym.monomial()
    sage: h = Sym.homogeneous()
    sage: m(s[1,1,1])
    m[1, 1, 1]
    sage: h(s[1,1,1])
    h[1, 1, 1] - 2*h[2, 1] + h[3]
\end{sageexample}
etc.. We can also obtain the expansion of the Schur functions into power sum symmetric functions:
\begin{sageexample}
    sage: p = Sym.power()
    sage: s = Sym.schur()
    sage: p(s[1,1,1])
    1/6*p[1, 1, 1] - 1/2*p[2, 1] + 1/3*p[3]
    sage: p(s[2,1])
    1/3*p[1, 1, 1] - 1/3*p[3]
    sage: p(s[3])
    1/6*p[1, 1, 1] + 1/2*p[2, 1] + 1/3*p[3]
\end{sageexample}
and the following calculation shows that the Schur functions are orthogonal:
\begin{sageexample}
    sage: s[2,1].scalar(s[1,1,1])
    0
    sage: s[2,1].scalar(s[2,1])
    1
\end{sageexample}
\end{sagedemo}

Since the Schur functions form a basis of the ring of symmetric functions $\La$, a product of two 
Schur functions can again be expanded in terms of Schur functions:
\begin{equation} \label{equation.LR coefficients}
	s_\la s_\mu = \sum_{\nu} c_{\la \mu}^\nu s_\nu.
\end{equation}
It turns out that the coefficients $c_{\la \mu}^\nu$,
called {\it Littlewood--Richardson coefficients},
are nonnegative integer coefficients.
The famous \textit{Littlewood--Richardson} rule~\cite{LR:1934} 
provides a combinatorial expression for the coefficients $c_{\lambda \mu}^\nu$. 
It says that $c_{\lambda\mu}^\nu$ is equal to the number of 
semi-standard
tableaux of skew shape $\nu/\lambda$ and weight $\mu$ whose column reading 
word is Yamanouchi. Here $\lambda,\mu,\nu$ are partitions and a 
semi-standard tableau of shape $\nu/\lambda$ is a filling of the skew shape which is weakly increasing 
across rows and strictly increasing up columns. The weight of a
tableau or word is 
$\mu=(\mu_1,\mu_2,\ldots)$, where $\mu_i$ counts the number of $i$
in the tableau or word. Furthermore, a word is Yamanouchi if all 
right subwords have partition weight.

Remarkably, Schur functions and their Littlewood--Richardson 
coefficients tie into the study of the geometry of the Grassmannian 
$\Gr_{\ell n}$ (the manifold of $\ell$-dimensional subspaces of $\mathbb C^n$).
The cohomology ring of $\Gr_{\ell n}$
has a basis of Schubert classes $\sigma_\lambda$,
indexed by shapes $\lambda\in \mathcal P^{\ell n}$ 
contained in an $\ell \times (n-\ell)$ rectangle.  
The intersection numbers are encoded by the
structure constants of $H^*(\Gr_{\ell n})$ 
in the Schubert basis:
\begin{equation}
\label{struc}
\sigma_\lambda\cup\sigma_\mu = \sum_{\nu\in\mathcal P^{\ell n}}
c_{\lambda\mu}^\nu \sigma_\nu
\,.
\end{equation}
The explicit understanding of $H^*(\Gr_{\ell n})$ and of these intersections 
is gained by Schur functions. Letting $I=\langle e_{n-\ell+1},\ldots,e_n\rangle$
and $\La_{(\ell)} = \QQ[ h_1, h_2, h_3, \ldots, h_\ell ]$,
there is an isomorphism,
\begin{equation}
H^*(\Gr_{\ell n})\cong \Lambda_{(\ell)}/ I
\,,
\end{equation}
under which $\sigma_\lambda$ corresponds to $s_\lambda$.
Importantly, $s_\nu\in I$ when 
$\nu\not\in\mathcal P^{\ell n}$,
and thus the structure constants of $H^*(\Gr_{\ell n})$ 
are none other than the Littlewood--Richardson coefficients
\eqref{equation.LR coefficients} for Schur function products.

The ring of symmetric functions is also endowed with a {\it Hopf algebra structure}. A systematic study
of $\La$ from the Hopf algebra perspective is given in~\cite{Zelevinsky}, 
see also~\cite{Mac:1995}.
In particular, the coproduct on the Schur basis is given in terms of the Littlewood--Richardson
coefficients
\begin{equation} \label{equation.coproduct schur}
	\Delta(s_\nu) = \sum_{\la,\mu} c_{\la\mu}^\nu s_\la \otimes s_\mu.
\end{equation}

There is also an algebraic involution on $\La$ 
for which $\omega( h_\la ) = e_\la$ and $\omega( p_\la ) = (-1)^{|\la| - \ell(\la)} p_\la$
and $\omega( s_\la ) = s_{\la'}$.

With the scalar product it is natural to introduce the operation which is dual to multiplication.
That is, for a homogeneous symmetric function $f$ of degree $k$, multiplication by $f$ is an operation
which will raise the degree of a symmetric function by $k$, and the notation $f^\perp$ will represent
an operator that will lower a symmetric function by degree $k$ and its action is defined as
\begin{equation}
f^\perp(g) = 
\sum_{\la} \left< g, f s_\la \right> s_\la 
= \sum_{\la} \left< g, f h_\la \right> m_\la \;.
\end{equation}

The reason why we introduce these operators is that we find that we can define `creation operators'
for the Schur functions which allow us to show that the Schur functions satisfy the Pieri rule.  Set
$\Sop_m = \sum_{r \geq 0} (-1)^r h_{m+r} e_r^\perp$.  The sum is apparently infinite but we need only calculate
up to $r$ equal to the degree of the symmetric function it is acting on.  These operators have the property
that for a partition $\la = (\la_1,\la_2,\ldots,\la_\ell)$ and $m \geq \lambda_1$,
\begin{equation} \label{bernsteinops}
\Sop_{m} s_{\la}  = s_{(m,\la_1,\la_2, \ldots, \la_{\ell})}~.
\end{equation}

One reason these operators are particularly useful is that commutation rules such as
$p_r \Sop_m = \Sop_m p_r + \Sop_{m+r}$ and $e_r \Sop_m = \Sop_{m+1} e_{r-1} + \Sop_{m} e_r$ and
$h_r \Sop_m = \sum_{i=0}^r \Sop_{m+i} h_{r-i}$ can be used to show the
Murnaghan--Nakayama rule and the Pieri rules (respectively).  We will also use the operators $\Sop_m$
to express creation operators for the Hall--Littlewood symmetric functions and a means for computing
them.  In particular,
$$s_{\la} = \Sop_{\la_1} \Sop_{\la_2} \cdots \Sop_{\la_{\ell(\la)}}(1)~.$$
\end{subsection}

\begin{subsection}{Hall--Littlewood symmetric functions} \label{sec:HLsymfunc}
The Hall--Littlewood symmetric functions have a definition which is similar to that
of the Schur functions.  These functions form a basis of the ring of symmetric functions
over a field containing a parameter $t$.  We work in the fraction field over the polynomials in
the parameter $t$, and set
\begin{equation}
\La_{t} = \QQ(t)[h_1, h_2, h_3, \ldots ]~.
\end{equation}

The functions $Q'_\la[X;t]$ are defined as the family of symmetric functions satisfying
$$Q'_\la[X;t] = s_{\la} + \hbox{ terms of the form }r_{\la\mu}(t) s_\mu\hbox{ for }\mu > \la$$
and 
$$\left< Q'_\la[X;t], Q'_\mu[X;t] \right>_{t} = 0\hbox{ if }\la \neq \mu$$
where the scalar product $\left<~\cdot~,~\cdot~\right>_t$ is defined so that
\begin{equation} \label{equation.scalar_t}
	\left< p_\la, p_\mu \right>_t = z_\la \delta_{\la\mu} \prod_i (1-t^{\la_i})~.
\end{equation}
By this definition of the $t$-scalar product, we see that
$\left< p_\la,p_\mu \right>_t = \left< p_\la[X], p_\mu[X(1-t)] \right>$,
where on the right hand side we have used the usual scalar product from Equation \eqref{scalardual}.

We will mainly be using the scalar product from Equation \eqref{scalardual}, so define
$\{P_\la[X;t]\}_{\lambda \vdash n}$ to be the dual basis to the basis
$\{ Q'_\lambda[X;t] \}_{\la \vdash n}$. Since we know that
$$\left< Q'_\lambda[X;t], Q'_\mu[X(1-t);t] \right> = 0 \quad \text{if $\la \neq \mu$,}$$
we must have $P_\lambda[X;t] = c_\la Q'_\la[X(1-t);t]$ for some coefficients 
$c_\la$.  Explicit formulas for the coefficient $c_\la$ are known, but since
$$1 = \left< Q'_\la[X;t], P_\la[X;t] \right> = \left< Q'_\la[X;t], c_\la Q'_\la[X(1-t);t] \right>~,$$
it follows that $c_\la = \left< Q'_\la[X;t], Q'_\la[X(1-t);t] \right>^{-1}~.$

\begin{example} \label{example.HL}
We compute the Hall--Littlewood symmetric functions
for partitions of $3$ using this method to demonstrate how they might be implemented
in a computer program.

The triangularity relation shows that $Q'_{3} = s_3$.  Then $Q'_{21}$ is defined as
$$Q'_{21} = s_{21} - \frac{\left< s_{21}, Q'_3 \right>_t}{\left< Q'_3, Q'_3 \right>_t} Q'_3~.$$
By calculating from the expansion of the Schur functions in the power sums as in 
Example~\ref{sageexample.sp}, we have that $\left< s_{21}, s_3 \right>_t = t^2 - t$ and
$\left< s_3, s_3 \right>_t = 1-t$.  We conclude that $Q'_{21} = s_{21} + t s_3$.

Similarly we can compute the last Hall--Littlewood symmetric function of size $3$ by
the computation
$$Q'_{111} = s_{111} - \frac{\left< s_{111}, Q'_{21} \right>_t}{\left< Q'_{21}, Q'_{21} \right>_t} Q'_{21} - \frac{\left< s_{111}, Q'_3 \right>_t}{\left< Q'_3, Q'_3 \right>_t} Q'_3~.$$
Using the previous examples it is not difficult to compute
the scalar products $\left< s_{111}, s_{21} \right>_t = t^2 - t$, $\left< s_{111}, s_{3} \right>_t = t^2 - t^3$, 
$\left< s_{21}, s_{21} \right>_t = (1-t)(1-t+t^2)$.  From these computations we compute that
$$ Q'_{111} = s_{111} + (t+t^2) s_{21} + t^3 s_{3}~.$$
\end{example}

There are other ways of computing the Hall--Littlewood symmetric functions.  They
can also be defined by means of `creation' operators that generalize the creation
operators for the Schur functions.  Define
\begin{equation} \label{Bopdef}
\Bop_{m} = \sum_{i,j \geq 0} (-1)^i t^j h_{m+i+j} e_i^\perp h_j^\perp = \sum_{j \geq 0} t^j \Sop_{m+j} h_j^\perp ~.
\end{equation}
This family of operators \cite{Jing} has the property that if $m \geq \la_1$,
\begin{equation} \label{equation.B_on_Qp}
\Bop_{m}( Q'_{\la}[X;t]) = Q'_{(m,\la_1,\la_2, \ldots, \la_{\ell})}[X;t]~.
\end{equation}
These operators will play an important role in one of the definitions of
the $k$-Schur functions in Section \ref{sec:threedef}.

\begin{sagedemo}
In \Sage, Example~\ref{example.HL} can be checked as follows. Note that now we need to define the
Schur functions over the base ring of the Hall--Littlewood functions:
\begin{sageexample}
  sage: Sym = SymmetricFunctions(FractionField(QQ["t"]))
  sage: Qp = Sym.hall_littlewood().Qp()
  sage: Qp.base_ring()
  Fraction Field of Univariate Polynomial Ring in t over Rational Field
  sage: s = Sym.schur()
  sage: s(Qp[1,1,1])
  s[1, 1, 1] + (t^2+t)*s[2, 1] + t^3*s[3]
\end{sageexample}
Recall the map $\theta_{qt}(q,t)$ of Equation~\eqref{equation.theta} which sends $p_k \mapsto p_k(1-q^k)/(1-t^k)$. Then we can
transform the usual scalar product $\langle \cdot, \cdot \rangle$ to the scalar product for the Hall--Littlewood polynomials
$\langle ~\cdot~, ~\cdot~ \rangle_t$ of~\eqref{equation.scalar_t} by setting $t=0$ and replacing $q$ by $t$ in $\theta_{qt}(q,t)$, 
that is $\theta_{qt}(t,0)$. We can now check our previous computation for $\langle s_{21}, s_3 \rangle_t$ in \Sage as follows:
\begin{sageexample}
  sage: t = Qp.t
  sage: s[2,1].scalar(s[3].theta_qt(t,0))
  t^2 - t
\end{sageexample}
We can also check~\eqref{equation.B_on_Qp}:
\begin{sageexample}
  sage: s(Qp([1,1])).hl_creation_operator([3])
  s[3, 1, 1] + t*s[3, 2] + (t^2+t)*s[4, 1] + t^3*s[5]
  sage: s(Qp([3,1,1]))
  s[3, 1, 1] + t*s[3, 2] + (t^2+t)*s[4, 1] + t^3*s[5]
\end{sageexample}
\end{sagedemo}
\end{subsection}

\begin{subsection}{Macdonald symmetric functions}
\label{subsection.macdonald}
Finally we introduce the Macdonald symmetric functions which form a basis of the space
of the symmetric functions with two parameters
\begin{equation}
\La_{q,t} = \QQ(q,t)[h_1, h_2, h_3, \ldots]~.
\end{equation}
We provide here a definition of the Macdonald symmetric functions that generalizes those of the
Hall--Littlewood and Schur functions introduced in the previous sections.  The Macdonald symmetric
functions $H_\la[X;q,t]$ are defined so that they are the unique basis which has the property
that 
$$H_\la[X;q,t] = r_\la(q,t) s_{\la}[X/(1-q)] + \hbox{ terms of the form }r_{\la\mu}(q,t) s_\mu[X/(1-q)]$$
with $\mu > \la$ and 
$$\left< H_\la[X;q,t], H_\mu[X;q,t] \right>_{qt} = 0\hbox{ if }\la \neq \mu$$
where the scalar product $\left<~\cdot~, ~\cdot~ \right>_{qt}$ is defined so that
$$\left< p_\la, p_\mu \right>_{qt} = z_\la \delta_{\la\mu} \prod_i (1-q^{\la_i})(1-t^{\la_i})~.$$  In addition, the condition 
that $\left< H_\la[X;q,t], s_{(n)}[X] \right> =t^{n(\la)}$ determines
the correct scalar multiple of the elements.

\begin{sagedemo}
Here we show how to expand the Macdonald symmetric functions in terms of Schur functions for 
all partitions of 3:
\begin{sageexample}
  sage: Sym = SymmetricFunctions(FractionField(QQ["q,t"]))
  sage: Mac = Sym.macdonald()
  sage: H = Mac.H()
  sage: s = Sym.schur()
  sage: for la in Partitions(3):
  ....:     print "H", la, "=", s(H(la))
  H [3] = q^3*s[1, 1, 1] + (q^2+q)*s[2, 1] + s[3]
  H [2, 1] = q*s[1, 1, 1] + (q*t+1)*s[2, 1] + t*s[3]
  H [1, 1, 1] = s[1, 1, 1] + (t^2+t)*s[2, 1] + t^3*s[3]
\end{sageexample}
When $q=0$, the expansion is upper triangular with respect to dominance order.  In
particular, $H_\la[X;0,t] = Q'_\la[X;t]$.
\begin{sageexample}
  sage: Sym = SymmetricFunctions(FractionField(QQ["t"]))
  sage: Mac = Sym.macdonald(q=0)
  sage: H = Mac.H()
  sage: s = Sym.schur()
  sage: for la in Partitions(3):
  ....:    print "H",la, "=", s(H(la))
  H [3] = s[3]
  H [2, 1] = s[2, 1] + t*s[3]
  H [1, 1, 1] = s[1, 1, 1] + (t^2+t)*s[2, 1] + t^3*s[3]
  sage: Qp = Sym.hall_littlewood().Qp()
  sage: s(Qp[1, 1, 1])
  s[1, 1, 1] + (t^2+t)*s[2, 1] + t^3*s[3]
\end{sageexample}
and when $t=0$ it is lower triangular
\begin{sageexample}
  sage: Sym = SymmetricFunctions(FractionField(QQ["q"]))
  sage: Mac = Sym.macdonald(t=0)
  sage: H = Mac.H()
  sage: s = Sym.schur()
  sage: for la in Partitions(3):
  ....:    print "H",la, "=", s(H(la))
  H [3] = q^3*s[1, 1, 1] + (q^2+q)*s[2, 1] + s[3]
  H [2, 1] = q*s[1, 1, 1] + s[2, 1]
  H [1, 1, 1] = s[1, 1, 1]
\end{sageexample}
\end{sagedemo}

Macdonald~\cite{Mac:1988} introduced the symmetric functions that 
bear his name in 1988 
by expanding on the work of Kevin Kadell (see notes in~\cite[p. 387]{Mac:1995}).
Attraction to researching Macdonald polynomials grew from conjectures
giving combinatorial, geometric, and representation theoretic 
meaning to the Macdonald/Schur coefficients
\begin{equation}
K_{\lambda\mu}(q,t) := \left< H_\mu[X;q,t], s_\la \right>\,,
\end{equation}
usually referred to as the Macdonald--Kostka or $q,t$-Kostka 
coefficients.  From the definition that we have presented here, 
it is not even clear that these coefficients are polynomials in $q$ and $t$.
Nevertheless, Macdonald conjectured that they are in fact
positive sums of monomials
in $q$ and $t$; that is, $K_{\lambda\mu}(q,t)\in\mathbb N[q,t]$.
These coefficients have since been a matter of great interest.

When $q=0$, the {\it Kostka-Foulkes polynomials} 
$K_{\lambda\mu}(0,t)$ are the Hall-Littlewood/Schur
transition coefficients since
$H_\lambda[X;0,t]=Q'_\lambda[X;t]$.
Kostka--Foulkes polynomials appear in other
contexts including affine Kazhdan--Lusztig theory \cite{LT:2000}
and affine tensor product multiplicities~\cite{NY:1997, SchillingWarnaar:1999}.  Moreover, these
polynomials encode the dimensions of certain bigraded $S_n$-modules \cite{GP:1992}.
Lascoux and Sch\"utzenberger \cite{LS:1981} give an intrinsically
positive formula for $K_{\lambda\mu}(0,t)$ using tableaux which we
now describe.

Recall that a semi-standard tableau is a nested sequence of partitions
such that consecutive partitions form a horizontal strip, 
also identified by a filling of a partition with the integers so that 
the label of the integer $i$ indicates the cells which are added 
between the $(i-1)^{st}$ and $i^{th}$ partition in the sequence.
The horizontal strip condition ensures that the entries increase
strictly (resp. weakly) along columns (resp. rows).
A semi-standard tableau has {\it weight} $\mu$ when there are 
$\mu_i$ labels of $i$.  When $\mu$ forms a partition, the
tableau is said to have partition weight and when the weight is
$(1,1,\ldots,1)$, the tableau is called {\it standard}.
A statistic (non-negative integer) called 
{\it charge}, defined in \cite{LS:1981},
can be associated to each semi-standard tableau.
Here, we describe charge for any semi-standard
tableaux with partition weight.  

First consider the definition of charge on a standard tableau $T$.
Define the {\it index} $I$ of $T$, starting from $I_1=0$, by
\begin{equation}
I_r =
\begin{cases}
\label{convindex}
I_{r-1} +1 & \text{if $r$ is east of $r-1$ }\\
I_{r-1} & \text{otherwise}\,,
\end{cases}
\end{equation}
for $r=2,\ldots,n$.
The {charge} of $T$ is the sum of entries in $I(T)$.
The notion of charge is easily extended to a generic semi-standard 
tableau by successively computing the index of an appropriate
choice of $i$ cells containing the letters $1,2,\ldots,i$.

\begin{definition}
\label{convchoice}
From a specific $x$ in cell $c$ of a tableau $T$,
the desired choice of $x+1$ is the south-easternmost 
one lying above $c$.  If there are none above $c$,  
the choice is the south-easternmost $x+1$ in all of $T$.
\end{definition}

Consider now any semi-standard tableau $T$  with partition weight.
Starting from the rightmost 1 in $T$, use Definition~\ref{convchoice}
to distinguish a standard sequence of $i$ cells containing $1,2,\ldots,i$.
Compute the index and then delete all cells in this sequence.
Repeat the process on the remaining cells.
The total {\it charge} is defined to be the sum of all the index vectors.

\begin{example}\label{ex:convcharge}
$$
\footnotesize
\tableau[sbY]
{\tf 6\cr\tf 4&5 \cr\tf 3&4
\cr2&\tf 2&3&\tf 5\cr1&1&\tf 1&2&3&\tf 7\cr}
\qquad
\qquad
\tableau[sbY]
{\cr&\tf 5 \cr&\tf 4
\cr\tf 2&&3&\cr1&\tf 1&& 2&\tf 3&\cr}
\qquad
\qquad
\tableau[sbY]
{\cr& \cr&
\cr&&\tf 3&\cr \tf 1&&& \tf 2&&\cr}
\qquad
\qquad
$$
$$
\hspace{-2.5cm}
I=[0,0,0,0,1,1,2]\qquad
\quad
I=[0,0,1,1,1]\qquad
\qquad\quad
I=[0,1,1]
$$
so that the charge is 9.

We can check this in \Sage:
\begin{sageexample}
    sage: t = Tableau([[1,1,1,2,3,7],[2,2,3,5],[3,4],[4,5],[6]])
    sage: t.charge()
    9
\end{sageexample}
\end{example}

This given,
with the {\it shape} of a semi-standard tableau $T$ denoted by $\shape(T)$
and the charge denoted $\charge(T)$, it is proven  in \cite{LS:1981} that
\begin{equation}
\label{eq:HLcomb}
K_{\lambda\mu}(0,t)=
\sum_{\weight(T)= \mu\atop\shape(T)=\lambda} t^{\charge(T)}
\,.
\end{equation}
Despite having such concrete results for the $q=0$ case, it 
was a big effort just to establish polynomiality for general 
$K_{\lambda\mu}(q,t)$
\cite{GT:1996,KirillovNoumi:1999,Knop:1996,LapointeVinet:1996,Sahi:1996}.
The geometry of Hilbert schemes was finally needed to prove positivity 
\cite{Haiman:1999,Haiman:2001}, where Haiman completed his proof by showing 
that there is a 
representation theoretical model (often referred to as the ``$n!$ conjecture'' 
\cite{GH:1996}) for which these coefficients are formulas for graded 
multiplicities of occurrences of
irreducible representations.  
A formula in the spirit of \eqref{eq:HLcomb} still remains a mystery.

There are other ways to express the charge. For example in~\cite{LenartSchilling:2012}
a different formulation of charge is used, which comes from the Ram--Yip formula
for Macdonald polynomials~\cite{RamYip:2011}, and is related to the quantum Bruhat 
graph, which first arose in connection with the quantum cohomology of the flag 
variety~\cite{BFP,FultonWoodward:2004}. This point of view is dual to the description above
in the sense that the positions of the entries in a tableau $T$ are recorded by columns $b_1\cdots b_n$,
where $n$ is the value of the largest letter in $T$.
Column $b_i$ records the column positions of the letters $i$ in $T$, where the columns of $T$
are labelled from right to left.

We attach to $b_1\cdots b_n$ a reordered filling $c:=c_1\cdots c_n$ according to the following 
algorithm, which is based on the circular order $\prec_i$ on $[m]$ starting at $i$, namely 
$i\prec_i i+1\prec_i\cdots \prec_i m\prec_i 1\prec_i\cdots\prec_i i-1$. Here $m$ is the width of $T$.

\begin{algorithm}\label{algreord}\hfill\\
let $c_1:=b_1$;\\
for $j$ from $2$ to $n$ do\\
\indent for $i$ from $1$ to height of $b_j$ do\\
\indent\indent let $c_j(i):=\min\,(b_j\setminus\{c_j(1),\ldots,c_j(i-1)\},\,\prec_{c_{j-1}(i)})$\\
\indent end do;\\
end do;\\
return $c:=c_1\ldots c_n$.\\
\end{algorithm}
Then the charge of $T$ is
\begin{equation*}
  \charge(T) = \sum_{\gamma \in \Des(c)} \arm(\gamma),
\end{equation*}
where $\Des(c)$ are all boxes in $c$ which contain a descent with the box directly to the right and
$\arm(\gamma)$ is the arm length of the box $\gamma$, that is the number of boxes to the right of
$\gamma$.

\begin{example}
Let us consider the tableau $T$ of Example~\ref{ex:convcharge}. Then
\[
    \raisebox{0.5cm}{b =}\; \tableau[sbY]{6&6&6\cr 5&5&4&6&5\cr 4&3&2&5&3&6&1\cr}
    \qquad \raisebox{0.5cm}{\text{and}} \qquad
    \raisebox{0.5cm}{c =}\; \tableau[sbY]{\tf 6&3&4\cr 5&\tf 6&2&5&5\cr 4&5&6&\tf 6&3&\tf 6&1\cr}
\]
where the boxes with descents are in bold, so that $\charge(T)=2+3+3+1=9$.
\end{example}

\end{subsection}

\begin{subsection}{Empirical approach to $k$-Schur functions}

The $k$-Schur functions came out of a study of the
$q,t$-Kostka coefficients and this study~\cite{LLM:2003}
led to a refinement of Macdonald's positivity conjecture;
for any fixed integer $k>0$
and each $\lambda\in \mathcal P^k$ (a partition
where $\lambda_1\leq k$),
\begin{equation}
H_{\lambda}[X;q,t\, ] = \sum_{\mu\in\mathcal P^k}
K_{\mu \lambda}^{(k)}(q,t) \, A_{\mu}^{(k)}[X;t\, ] \quad\text{where}\quad
K_{\mu \lambda}^{(k)}(q,t) \in \mathbb N[q,t] \, ,
\label{mackkostka}
\end{equation}
for some family of polynomials defined
by certain sets of tableaux $\mathcal A^k_\mu$ as:
\begin{equation}
A_{\mu}^{(k)}[X;t] = \sum_{T\in \mathcal A^k_\mu}
t^{\charge(T)} \, s_{\shape(T)}\, \,.
\label{atom}
\end{equation}
In \cite{LLM:2003}, Lapointe, Lascoux, and Morse
conjectured that such $\{A^{(k)}_\mu[X;t]\}_{\mu_1\leq k}$ exists
and forms a basis for 
$$
\Lambda^t_{(k)}=\mathrm{span} \{H_\lambda[X;q,t]\}_{\lambda_1\leq k}\,.
$$
They also conjectured  that more remarkably, for any $k'>k$, 
\begin{equation}
\label{surepos}
A_\lambda^{(k)}[X;t] = \sum_{\mu}
B_{\lambda,\mu}^{(k,k')}(t)\, A_\mu^{(k')}[X;t]\quad\text{where
$B_{\lambda,\mu}^{(k,k')}(t)\in\mathbb N[t]$.}
\end{equation}
Given that $A_\mu^{(k)}[X;t]=s_\mu$ for $k\geq |\mu|$,
the decomposition \eqref{mackkostka} strengthens Macdonald's
conjecture.  Although these new bases arose in the context 
of Macdonald polynomials, pursuant work led to unexpected 
connections with geometry, physics, and
representation theory.  At the root, $\{A^{(k)}_\mu[X;t]\}$ generalizes
the very aspects of the Schur basis that make it so fundamental and
wide-reaching.  As such, the functions are called {\it $k$-Schur functions}.

Before we give any formal definition for 
$k$-Schur functions, we begin with some computational examples 
that demonstrate how this all came about.
\begin{example}\label{ex:atom2}
Consider the Hall--Littlewood symmetric functions that are spanned by partitions of $3$ and $4$ with
$\la_1 \leq 2$:
\begin{equation}\label{equation.HL expansion}
\begin{split}
Q'_{111}[X;t] &= s_{111} + (t+t^2) s_{21} + t^3 s_{3}\\
Q'_{21}[X;t] &=  s_{21} + t s_{3}
\end{split}
\end{equation}
and
\begin{align*}
Q'_{1111}[X;t] &= s_{1111} + (t+t^2+t^3) s_{211} + (t^2+t^4) s_{22} + (t^3+t^4+t^5) s_{31} + t^6 s_{4}\\
Q'_{211}[X;t] &= s_{211} + t s_{22} + (t+t^2) s_{31} + t^3 s_{4}\\
Q'_{22}[X;t] &=  s_{22} + t s_{31} + t^2 s_{4}~.
\end{align*}
The first clue of the existence of $2$-Schur functions is to notice that there is a set of symmetric
functions $A^{(2)}_\la[X;t]$ for $\la \vdash 3,4$ such that $\la_1 \leq 2$ and 
\begin{itemize}
\item they form a basis for this subset of Hall--Littlewood symmetric functions that have coefficients
that are non-negative polynomials in $t$ when expanded in the Schur basis;
\item when expanded in the Hall--Littlewood basis have a term $Q'_\lambda[X;t]$ and
all other terms are larger in dominance order;
\item have a leading term in the Schur basis which
is $s_\la$ and are in the linear span of Schur functions indexed by partitions which are larger
than $\la$ in dominance order;
\item if the involution $\omega$ is applied and $t$ is replaced by $1/t$, then the 2-Schur function 
is equal to another $2$-Schur function up to a power of $t$;
\item the Schur function indexed by the partition which is largest in dominance
order is a power of $t$ times $s_{(\la^{\omega_k})'}$ 
(indexed by the conjugate of the $k$-conjugate of $\lambda$);
\item for which the Macdonald symmetric functions
are positive when expressed in terms of these elements (have coefficients
which are polynomials in $q$ and $t$ with non-negative integer coefficients).
\end{itemize}
It turns out that these conditions are enough to characterize the $2$-Schur functions.
The triangularity conditions with respect to the Hall--Littlewood polynomials and the
Schur functions require that $A_{22}^{(2)}[X;t] = s_{22} + t s_{31} + t^2 s_4$.
Then again by triangularity and positivity we have $A_{211}^{(2)}[X;t] = s_{211} + t s_{31}$, and
finally $A_{1111}^{(2)}[X;t] = s_{1111} + t s_{211} + t^2 s_{22}$.
\end{example}

\begin{sagedemo}
The conditions involving the $k$-conjugate of the partition can easily be checked
\begin{sageexample}
    sage: la = Partition([2,2])
    sage: la.k_conjugate(2).conjugate()
    [4]
    sage: la = Partition([2,1,1])
    sage: la.k_conjugate(2).conjugate()
    [3, 1]
    sage: la = Partition([1,1,1,1])
    sage: la.k_conjugate(2).conjugate()
    [2, 2]
\end{sageexample}
as well as the statement about the application of $\omega$ composed with $t\mapsto 1/t$:
\begin{sageexample}
    sage: Sym = SymmetricFunctions(FractionField(QQ["t"]))
    sage: ks = Sym.kschur(2)
    sage: ks[2,2].omega_t_inverse()
    1/t^2*ks2[1, 1, 1, 1]
    sage: ks[2,1,1].omega_t_inverse()
    1/t*ks2[2, 1, 1]
    sage: ks[1,1,1,1].omega_t_inverse()
    1/t^2*ks2[2, 2]
\end{sageexample}
We can also check the positive expansion of the Macdonald polynomials:
\begin{sageexample}
    sage: Sym = SymmetricFunctions(FractionField(QQ["q,t"]))
    sage: H = Sym.macdonald().H()
    sage: ks = Sym.kschur(2)
    sage: ks(H[2,2])
    q^2*ks2[1, 1, 1, 1] + (q*t+q)*ks2[2, 1, 1] + ks2[2, 2]
    sage: ks(H[2,1,1])
    q*ks2[1, 1, 1, 1] + (q*t^2+1)*ks2[2, 1, 1] + t*ks2[2, 2]
    sage: ks(H[1,1,1,1])
    ks2[1, 1, 1, 1] + (t^3+t^2)*ks2[2, 1, 1] + t^4*ks2[2, 2]
\end{sageexample}
\end{sagedemo}

At $k=3$, the above conditions are no longer a complete characterization of the $k$-Schur functions
and it is difficult to guess what $A^{(3)}_\lambda$ is for some of the partitions
at $\lambda \vdash 7$.  We add as a last condition that the $k$-Schur functions
should be the `smallest' basis with the above properties.
The existence of such a basis alone is enough to recognize that there is
something remarkable about the $k$-Schur functions, because it is unusual to
see a basis for which the Macdonald symmetric functions expand positively.

\begin{example}  At $k=3$, the triangularity condition for the Hall--Littlewood
basis implies that,
\begin{align*}
A^{(3)}_{331}[X;t] = Q'_{331}[X;t] = s_{331} &+ t s_{421} + (t+t^2) s_{43}
+ t^2 s_{511}+ (t^2+t^3) s_{52}\\ & + (t^3+t^4) s_{61} + t^5 s_{7}
\end{align*}
and since 
\begin{align*}
Q'_{322}[X;t] = s_{322} &+ t s_{331} + (t+t^2) s_{421} + (t^2+t^3) s_{43}
+ t^3 s_{511}\\ &+ (t^2+t^3+t^4) s_{52} + (t^4+t^5) s_{61} + t^6 s_{7},
\end{align*}
we can use the triangularity and positivity properties from above that imply
$A^{(3)}_{322}[X;t]$ must be given by
$$A^{(3)}_{322}[X;t] = Q'_{322}[X;t] - t Q'_{331} = s_{322}+ t s_{421} + t^2 s_{52}~.$$
However, to try to determine $A^{(3)}_{3211}[X;t]$, we calculate
\begin{align*}
Q'_{3211}[X;t] = s_{3211} &+ t s_{322} + (t+t^2) s_{331} + t s_{4111} +
(t+t^2+t^3) s_{421}\\ &+ (t^2+t^3+t^4) s_{43}+ (t^2+t^3+t^4)s_{511} + (2t^3+t^4+t^5) s_{52}\\
&+ (t^4+t^5+t^6)s_{61} + t^7 s_{7}
\end{align*}
and it is difficult to tell from the conditions above whether
$A^{(3)}_{3211}$ should be $Q'_{3211}[X;t] - t^2 Q'_{331}[X;t]$ or $Q'_{3211}[X;t] - t Q'_{322}[X;t]$ 
or some other linear combination of terms.
\end{example}

We leave it as an exercise for the reader to calculate
$A_{111}^{(2)}[X;t]$, $A_{21}^{(2)}[X;t]$ from~\eqref{equation.HL expansion},
using the properties in Example~\ref{ex:atom2} as a characterization.  With the additional symmetric 
function $Q'_{31}[X;t] = s_{31} + t s_{4}$, it is a worthwhile exercise to determine the symmetric functions 
$A_{31}^{(3)}[X;t]$, $A_{22}^{(3)}[X;t]$, $A_{211}^{(3)}[X;t]$ and $A_{1111}^{(3)}[X;t]$ 
to see how it might be possible to define the $k$-Schur functions 
for small values by experimentation.

Once it is clear that these symmetric functions exist, it is a matter of
determining an algorithm or a formula for computing them.  
It was along these lines that Lapointe, Lascoux, and Morse discovered
$k$-Schur functions and their first formula appeared in \cite{LM:2003}
where a method is given for constructing the sets 
$\mathcal A_\lambda$ in \eqref{atom}.
Subsequently, various conjecturally equivalent definitions have
arisen, each having different benefits and detriments.
In the following section we present 
a construction for $k$-Schur functions when the parameter
$t$ is set to one and the parameter case is then
addressed in Section~\ref{sec:threedef}.

\end{subsection}

\begin{subsection}{Notes on references}

The combinatorics of affine permutations can be expressed in terms of 
combinatorial models other than  $k$-bounded partitions and $(k+1)$-cores 
such as abaci, windows, codes, and $k$-castles 
\cite{BJV:2009, BB, Denton:2012, ErikssonEriksson:1998, Lascoux:2001}.
Certain combinatorial aspects of $k$-Schur functions are best expressed 
in terms of $k$-bounded partitions whereas others are better suited 
to $(k+1)$-cores or affine permutations.  It could be that
other formulations for the index set are well suited for 
expressing properties which have not yet been discovered.

For symmetric function notation we generally follow the notation of Macdonald~\cite{Mac:1995}
with the addition of the use of plethystic notation (see Equation~\eqref{eq:plethnot}) to 
encode certain transformations
of alphabets.  The scalar product $\left<~\cdot~,~\cdot~\right>_{t}$ defined in Section
\ref{sec:HLsymfunc} and
$\left<~\cdot~,~\cdot~\right>_{qt}$ defined in Section \ref{subsection.macdonald}
are not the scalar products that are used in~\cite{Mac:1995}, 
but they are the scalar products needed in order
to define the Hall--Littlewood $Q'$-basis and the Macdonald $H$-basis 
that we will use in subsequent sections. 
Macdonald does not use the $H$-basis
in~\cite{Mac:1995}, however it is a transformation of the basis referred to as the integral basis (the $J$-basis)
and they are related by the transformation 
$J_\lambda[X;q,t] = H_\lambda[X(1-t); q,t] = \theta_{t0}( H_\lambda[X;q,t])$
(see, for instance, \cite[Eq. I.16]{GH:1996}).

The operators $\Sop_m$ defined in Equation~\eqref{bernsteinops} are usually referred to
as Bernstein operators.  They first appear in \cite[p. 69]{Zelevinsky} (see 
also~\cite[Example 29 Section I.5]{Mac:1995}).  
Their generalizations to creation operators for Hall--Littlewood symmetric
functions is due to Jing~\cite{Jing} (see also~\cite[Example 8 Section III.5]{Mac:1995}).
\end{subsection}
\end{section}

\begin{section}{From Pieri rules to $k$-Schur functions 
at $t=1$}\label{sec:Pierirules}

We first present the definition of the Schur functions
as a generating function for semi-standard tableaux.  This presentation
provides the context to show how the $k$-Schur functions and the 
dual $k$-Schur functions both generalize this definition of the Schur 
functions.  We carefully explain how the algebraic Pieri 
rule naturally gives rise to distinguished sequences of partitions 
and show how the combinatorics of 'weak' and 'strong' tableaux
captures these relations.

\begin{subsection}{Semi-standard tableaux and a monomial expansion of Schur functions}
\label{sec:colstrict}

Recall that the Schur functions satisfy the {\it Pieri rule} for multiplication of a Schur
function by a homogeneous symmetric function
\begin{equation}
\label{eq:pierirule1}
h_1 s_\la = \sum_{\mu: \la \rightarrow \mu} s_\mu
\end{equation}
and more generally,
\begin{equation}\label{eq:pierirule}
h_r s_\la = \sum_{\mu} s_\mu\; ,
\end{equation}
where the sum is over all partitions $\mu$ where $\la \subseteq \mu$, $|\mu| = |\la|+r$
and $\mu/\la$ is a horizontal strip.

Consider how we can use these rules now to expand a homogeneous symmetric function in
terms of the Schur basis.  A given $h_\la = h_{\la_1} h_{\la_2} \cdots h_{\la_\ell}$
may be seen as a sequence of operators that act on $1 = s_{\emptyset}$ which we act on, first
by $h_{\la_1}$, then $h_{\la_2}$, and so on.  We demonstrate this with the following
example.
\begin{example}  We expand the expression $h_{431}$ in terms of Schur functions by
successive applications of the Pieri rule:
\squaresize=5pt
\begin{align*}
h_{431} &=  s_{\young{&&&\cr}} h_{31} =
( s_{\young{\gris&\gris&\gris\cr&&&\cr}} + s_{\young{\gris&\gris\cr&&&&\gris\cr}} 
+ s_{\young{\gris\cr&&&&\gris&\gris\cr}} + s_{\young{&&&&\gris&\gris&\gris\cr}})h_{1}\\
&= \big( s_{\young{\gris\cr&&\cr&&&\cr}} + s_{\young{&&&\gris\cr&&&\cr}} +s_{\young{&&\cr&&&&\gris\cr}}+
s_{\young{\gris\cr&\cr&&&&\cr}}+ s_{\young{&&\gris\cr&&&&\cr}}+s_{\young{&\cr&&&&&\gris\cr}}\\
&\hskip .3in +s_{\young{\gris\cr\cr&&&&&\cr}}+s_{\young{&\gris\cr&&&&&\cr}}+s_{\young{\cr&&&&&&\gris\cr}}+
s_{\young{\gris\cr&&&&&&\cr}}+s_{\young{&&&&&&&\gris\cr}} \big)\cr
&= s_{\young{\cr&&\cr&&&\cr}} + s_{\young{&&&\cr&&&\cr}} + s_{\young{ \cr&\cr&&&&\cr}}
+2 s_{\young{&&\cr&&&&\cr}} + s_{\young{\cr\cr&&&&&\cr}} + 2 s_{\young{&\cr&&&&&\cr}} 
 + 2 s_{\young{\cr&&&&&&\cr}} \\
 &\hskip .3in +s_{\young{&&&&&&&\cr}} \; .
\end{align*}
\end{example}

Notice in the example that every term in the expansion of the product of $h$'s is represented
by a sequence of partitions which records how that term appears in the final expression.
For instance there are two terms representing $(5,3)$, one that arose from the sequence $\emptyset
\subseteq (4) \subseteq (5,2) \subseteq (5,3)$ and the other which arose from the sequence
$\emptyset \subseteq (4) \subseteq (4,3) \subseteq (5,3)$. A sequence 
\begin{equation}
	\la^{(0)} \subseteq \la^{(1)} \subseteq \cdots \subseteq\la^{(d)}\; ,
\end{equation}
where each $\la^{(i+1)}/\la^{(i)}$ for $0\leq i < d$ is a horizontal strip, is 
called a {\it semi-standard skew tableau}. When $\la^{(0)}$ is empty, the sequence is
called a {\it semi-standard tableau} (see also Section~\ref{subsection.macdonald}). 
We say that the {\it shape} of the tableau is $\la^{(d)}/\la^{(0)}$ 
(or $\la^{(d)}$ if $\la^{(0)}=\emptyset$) and the weight of the tableau is the sequence 
\begin{equation}\label{eq:content}
(|\la^{(1)}/\la^{(0)}|, |\la^{(2)}/\la^{(1)}|, \ldots, |\la^{(d)}/\la^{(d-1)}|)~.
\end{equation}
Note that for any partition there is precisely one semi-standard tableau that has
both shape and weight equal to $\la$.  If $\la$ and $\mu$ are both partitions of 
the same size, then there is at least one partition of shape $\la$ and weight $\mu$ 
if and only if $\la\geq\mu$.

A semi-standard tableau of shape $\la$ is usually thought of as
a filling of the partition diagram for $\la$ by placing a $1$ 
in each of the cells of
$\la^{(1)}/\la^{(0)}$, a $2$ in each of the cells of $\la^{(2)}/\la^{(1)}$, and more generally
each of the cells of $\la^{(i)}/\la^{(i-1)}$ is labelled with an $i$.
Note, iteration of the special case \eqref{eq:pierirule1} 
ensures that each $\la^{(i)}/\la^{(i-1)}$ contains exactly one cell.
These are the tableaux of weight $(1,1,\ldots,1)$ and they are
called {\it standard} tableaux.

\begin{example}\squaresize = 6pt
The coefficient of $s_{52}$ in $h_{421}$ is equal to $2$.  The reason for this is
that one term comes from
$$\young{&&&\cr} \subseteq \young{&\cr&&&\cr} \subseteq \young{&\cr&&&&\cr} 
\hbox{ which is represented by the diagram } \squaresize = 10pt\small\young{2&2\cr1&1&1&1&3\cr}$$
and the other comes from
$$\young{&&&\cr} \subseteq \young{\cr&&&&\cr} \subseteq \young{&\cr&&&&\cr} 
\hbox{ which is represented by the diagram } \squaresize = 10pt\small \young{2&3\cr1&1&1&1&2\cr}\;.$$
\end{example}

\begin{remark}
\label{def:standardizedtab}
Another useful alternative to the conventional definition of
tableaux is to instead define a semi-standard tableau
to be a standard tableau with certain conditions on
its reading word.  The reading word of a tableau
is obtained by taking the entries of $T$ from top to bottom
and left to right.  A tableau of weight $\alpha$ is then
a standard tableau having increasing reading
words in the alphabets
\begin{equation}
\label{readingwordcond}
\mathcal A_{\alpha,x} =
[1+\Sigma^{x-1}\alpha,\Sigma^x\alpha] \quad \text{where} \quad
\Sigma^x\alpha
 = \sum_{i\leq x}\alpha_i\,,
\end{equation}
for each $x=1,\ldots,\ell(\alpha)$.
For example, this observation is
central in the study of quasisymmetric functions.
\end{remark}

It is typical to represent the number of semi-standard tableaux of shape
$\la$ and weight $\mu$ by the symbol $K_{\la\mu}$.  These numbers are often referred
as the {\it Kostka} coefficients.
Based on the Pieri rule and a relatively straightforward proof by induction, we can generalize
roughly what we see in the example, namely  that for $\mu \vdash m$,
\begin{equation}
h_\mu = \sum_{\la \vdash m} K_{\la\mu} s_\la~.
\end{equation}

From this, the monomial expansion of a Schur function can be derived.  
In particular,
because $\left< s_\la, s_\mu \right> = \delta_{\la\mu}$, we can
conclude that
$\left< h_\mu, s_\la \right> =  K_{\la\mu}$.
The coefficient in the monomial expansion of a Schur function then pops out
by the duality $\langle m_\lambda,h_\mu\rangle=\delta_{\lambda\mu}$:
$$
s_\la= \sum_{\mu} d_{\mu\la}\, m_\mu
\quad\text{where}\quad
d_{\mu\la} 
= \langle h_\mu, \sum_\alpha d_{\alpha \la} m_\alpha\rangle
= \langle h_\mu, s_\la \rangle
=K_{\la\mu}
\,.
$$
This formula provides us with a combinatorial expansion of the Schur functions 
in the monomial basis.  We can also derive a formula in terms of variables.
Note that the construction of tableaux with weight $\alpha$ applies for
any composition $\alpha$.  Let 
$\SSYT(\lambda,\alpha)$ be the set of tableaux with weight $\alpha$ and shape 
$\lambda$.  There is an involution
\cite{BK,LS}  on this set that maps
a tableau of weight $\alpha$ to a tableau whose weight
is a permutation of $\alpha$.
Thus, given that $K_{\lambda\alpha}=K_{\lambda\sigma(\alpha)}$,
\begin{equation}\label{eq:schurtableauxexp}
s_\la[X_m] = \sum_{T} {\bf x}^T\;,
\end{equation}
where the sum is over all possible semi-standard tableaux of shape $\la$ with entries in $\{ 1, 2, \ldots, m \}$. 
Here ${\bf x}^T$ denotes $x_1^{\alpha_1} x_2^{\alpha_2} \cdots x_m^{\alpha_m}$, where 
$\alpha$ is the weight of $T$.

\begin{sagedemo}
We now demonstrate on how to produce all semi-standard tableaux 
of a given shape and weight
\begin{sageexample}
    sage: SemistandardTableaux([5,2],[4,2,1]).list()
    [[[1, 1, 1, 1, 2], [2, 3]], [[1, 1, 1, 1, 3], [2, 2]]]
\end{sageexample}
The Kostka matrix can be computed as follows:
\begin{sageexample}
    sage: P = Partitions(4)
    sage: P.list()
    [[4], [3, 1], [2, 2], [2, 1, 1], [1, 1, 1, 1]]
    sage: n = P.cardinality(); n
    5
    sage: K = matrix(QQ,n,n,
    ....:           [[SemistandardTableaux(la,mu).cardinality()
    ....:            for mu in P] for la in P])
    sage: K
    [1 1 1 1 1]
    [0 1 1 2 3]
    [0 0 1 1 2]
    [0 0 0 1 3]
    [0 0 0 0 1]
\end{sageexample}
\end{sagedemo}

The self duality of the Schur functions is a very
remarkable property.  In fact, the self duality and a condition on triangularity
can be taken as either the defining property for Schur functions or as property which easily
follows from the definition.  It is this duality that implies that $\left< h_\mu, s_\la \right>$
will be the coefficient of $s_\la$ in the expansion of $h_\mu$.  In the next two sections
we shall introduce two bases of a subalgebra/quotient algebra which follows from ideas in
this construction.  
\end{subsection}

\begin{subsection}{Weak tableaux and a monomial expansion of dual $k$-Schur functions} 
\label{sec:weaksection}

We will present the $k$-Schur functions as a basis of the space
\begin{equation}\label{kboundedspace}
\La_{(k)} = \QQ[ h_1, h_2, h_3, \ldots, h_k ]\,.
\end{equation}
It will develop that this basis has a role in $\La_{(k)}$ that 
in many ways is analogous to the role of the Schur basis
in the study of the symmetric function space $\Lambda$.
The algebra of $\La_{(k)}$
is no longer a self dual Hopf algebra.  
The algebra is dual however to a quotient of the ring of symmetric functions.  We define
\begin{equation}
\La^{(k)} = \La / {\big<} m_{\la} : \la_1 > k {\big>} \;.
\end{equation}

A basis of $\La_{(k)}$ are the elements $h_\la$ for partitions $\la$ 
where $\la_1 \leq k$.  
The elements $m_\la$ with $\la_1 \leq k$ may be chosen as representatives of the dual algebra.  
We will present two bases: the {\it $k$-Schur functions} $s^{(k)}_{\la}$ which form a basis for
$\La_{(k)}$ and the {\it dual $k$-Schur functions} $\SS^{(k)}_{\la}$ which are representative elements 
of the basis for the dual algebra $\La^{(k)}$.  We proceed by defining the basis of $k$-Schur functions 
and the other will be determined by duality.

Recall by Proposition~\ref{bijcoresparts} that
the $k$-bounded partitions of size $m$ are in bijection with the $(k+1)$-cores of 
length $m$ (or equivalently cores with $m$ cells with hook length of 
size less than or equal to $k$).  At this point, we keep in mind that
the $k$-bounded partition $\lambda$ that serves as an
index for either a $k$-Schur function or a dual $k$-Schur
function represents the shape of a $k+1$-core, $\mfc(\lambda)$.
Sometimes we will use the properties of the partition $\lambda$ and
other times we will use the corresponding core, $\mfc(\la)$.  
Later, we will interpret things in
the language of $k$-bounded partitions and affine Grassmannian
elements.

We will define the $k$-Schur functions at $t=1$ based on a
{\it weak Pieri rule} for $k$-Schur functions (which at
this point exists only as data on the computer because we deduced
what it must be from the atom definition in the last chapter).
This is the relation, for a $k$-bounded partition $\mu$ and $r\leq k$,
\begin{equation}\label{eq:kweakpieri}
h_r\, s^{(k)}_\mu = \sum_{\la} s^{(k)}_\la\; ,
\end{equation}
summing over $k$-bounded partitions $\la$ where $\mfc(\la)/\mfc(\mu)$ is a weak horizontal strip
of size $r$ (see Equation \eqref{eq:weakhstrip}).  

\begin{example} With $k=3$, to compute $h_1 s^{(3)}_{31}$, we find all $4$-cores
that cover $\mfc(3,1) = (4,1)$ in the weak order poset Figure~\ref{partkeq3Hasse}
and there are two, one of shape $\mfc(3,2) = (5,2)$ and one of shape $\mfc(3,1,1) = (4,1,1)$:
$$h_1 s^{(3)}_{31} = s^{(3)}_{32} + s^{(3)}_{311}~.$$
To compute $h_2 s^{(3)}_{31}$, we follow the weak horizontal chains
of length $2$ from $(4,1)$ in Figure~\ref{partkeq3Hasse} and notice that there are
two of shape $\mfc(3,3) = (6,3)$ and $\mfc(3,2,1) = (5,2,1)$:
$$h_2 s^{(3)}_{31} = s^{(3)}_{33} + s^{(3)}_{321}~.$$
Note that there is a length $2$ chain from $\mfc(3,1) = (4,1)$ to $\mfc(3,1,1,1) = (4,1,1,1)$, but
because $(4,1,1,1)/(4,1)$ is not a weak horizontal strip, this term is omitted.
There is only one horizontal chain of length $3$ from $\mfc(3,1) = (4,1)$ implying that
$$h_3 s^{(3)}_{31} = s^{(3)}_{331}~.$$
\end{example}

As we have shown for usual Schur functions, the iteration of the
Pieri rule can be used to inspire a family of tableaux where, 
in this case, their enumeration gives
the coefficient of $s^{(k)}_\la$ in the expansion 
of $h_\mu\,s_\emptyset^{(k)}$.
Consider the following example.

\begin{example}  \label{example.kSchur}
For $k=4$, we compute $h_{431}$ in terms of
$k$-Schur functions.  We will index the $k$-Schur function
by a diagram for a $5$-core with the added cells indicated by an
$\ast$.  The indexing $4$-bounded partition can be read off of these
diagrams by counting only the non-grey cells in each row.
\definecolor{dark-gray}{gray}{0.65}
\renewcommand{\gray}{\color{dark-gray}}
\squaresize=5pt
$$h_{431} = h_{31}\,s^{(4)}_{\young{&&&\cr}} = 
h_1\,s^{(4)}_{\young{\ast&\ast&\ast\cr\gris&\gris&\gris&&\ast&\ast&\ast\cr}}  = 
s^{(4)}_{\young{\ast\cr&&\cr\gris&\gris&\gris&&&&\cr}}
+s^{(4)}_{\young{&&&\ast\cr\gris&\gris&\gris&\gris&&&&\ast\cr}}~.$$

Let $k=6$ and compute an example with more terms:

\begin{align*}
h_{431} &= 
h_1
(s^{(6)}_{\young{\ast&\ast&\ast\cr&&&\cr}} + 
s^{(6)}_{\young{\ast&\ast\cr&&&&\ast\cr}} + 
s^{(6)}_{\young{\ast\cr\gris&&&&\ast&\ast&\ast\cr}}) \\
&=\big(s^{(6)}_{\young{\ast\cr&&\cr&&&\cr}}+
s^{(6)}_{\young{&&&\ast\cr&&&\cr}}+s^{(6)}_{\young{&&\cr&&&&\ast\cr}}\big) 
+ \big(s^{(6)}_{\young{\ast\cr&\cr\gris&&&&&\ast\cr}}+
s^{(6)}_{\young{&&\ast\cr&&&&\cr}}\big)
+ 
\big(s^{(6)}_{\young{\ast\cr\cr\gris&&&&&&\cr}} + 
s^{(6)}_{\young{&\ast\cr\gris&\gris&&&&&&\ast\cr}}\big)~.
\end{align*}
\end{example}

\renewcommand{\gray}{\color{Gray}}

The iteration imposes the conditions needed to characterize
a {\it weak tableau};  it must be a sequence of $(k+1)$-cores,
\begin{equation}
\emptyset = \kappa^{(0)} \subseteq \kappa^{(1)} \subseteq \cdots \subseteq \kappa^{(d)} = \mfc(\la)
\end{equation}
such that
$\kappa^{(i)}/\kappa^{(i-1)}$ is a weak horizontal strip.  
We say that its {\it shape} is $\mfc(\la)$ and 
define its weight $\alpha$ as we did in Equation~\eqref{eq:content},
but now using the size function on cores;
\begin{equation}
\alpha_i=|\kappa^{(i)}/\kappa^{(i-1)}|_{k+1}\,,\quad
\text{for $i=1,\ldots,d$}\,.
\end{equation}
Let us emphasize that an entry $\alpha_i$ of the weight 
does {\it not} record the number of times letter $i$ 
appears in the tableau.  In fact, using \eqref{goodweakstrip}, 
we see that instead, $\alpha_i$ records the number of distinct
residues used to label the cells of $\kappa^{(i)}/\kappa^{(i-1)}$.
We can more concisely describe weak tableaux in the following
terms:

\begin{prop}
\cite{LM:2005}
Let $\kappa$ be a $(k+1)$-core
and let $\alpha=(\alpha_1,\ldots,\alpha_d)$ be a composition of 
$|\kappa|_{k+1}$ with no part larger than $k$.  A weak tableau
of weight $\alpha$ is a semi-standard 
filling of shape $\kappa$ with letters $1,\ldots,d$ such that
the collection of cells filled with letter $i$ is labeled 
by exactly $\alpha_i$ distinct $(k+1)$-residues.
\end{prop}

\begin{example}
\label{Ex:weaktab}
For $k=6$, the weak tableaux of weight $(4,3,1)$ are
\begin{equation*}
{\squaresize=12pt
\young{3\cr2&2&2\cr1&1&1&1\cr}\hskip .2in
\young{2&2&2&3\cr1&1&1&1\cr}\hskip .2in
\young{2&2&2\cr1&1&1&1&3\cr} \hskip .2in
\young{2&2&3\cr1&1&1&1&2\cr}}
\end{equation*}
\begin{equation*}{\squaresize=12pt
\young{3\cr2&2\cr1&1&1&1&2&3\cr}\hskip .1in
\young{3\cr2\cr1&1&1&1&2&2&2\cr}\hskip .1in
\young{2&3\cr1&1&1&1&2&2&2&3\cr}~.
}
\end{equation*}
For $k=3$, we list all weak
tableaux of shape $(5,2,1)$ which can be extracted by 
looking at all successions of horizontal chains
(with non-increasing sizes) in Figure~\ref{partkeq3Hasse}.
\begin{equation*}\squaresize=12pt
\young{6\cr4&5\cr1&2&3&4&5\cr} 
\hskip .1in \young{5\cr4&6\cr1&2&3&4&6\cr}
\hskip .1in \young{4\cr3&6\cr1&2&4&5&6\cr} 
\hskip .1in \young{4\cr2&6\cr1&3&4&5&6\cr}
\end{equation*}
\begin{equation*}\squaresize=12pt
\young{5\cr3&4\cr1&1&2&3&4\cr} \hskip .1in 
\young{4\cr3&5\cr1&1&2&3&5\cr} \hskip .1in 
\young{3\cr2&5\cr1&1&3&4&5\cr} 
\hskip.5in
\young{4\cr2&3\cr1&1&2&2&3\cr} \hskip .1in \young{3\cr2&4\cr1&1&2&2&4\cr}
\end{equation*}
\begin{equation*}\squaresize=12pt
\young{3\cr2&3\cr1&1&2&2&3\cr} 
\hskip .5in 
\young{4\cr2&3\cr1&1&1&2&3\cr} \hskip .1in \young{3\cr2&4\cr1&1&1&2&4\cr}
\hskip .5in 
\young{3\cr2&2\cr1&1&1&2&2\cr}
\end{equation*}
\end{example}

\begin{sagedemo} \label{sageex:weaktab}
In \Sage we can list all weak $k$-tableaux for a given shape and weight. For example, the two
weak 6-tableaux of weight $(4,3,1)$ and shape $(5,3)$ can be obtained as follows:
\begin{sageexample}
    sage: T = WeakTableaux(6, [5,3], [4,3,1])
    sage: T.list()
    [[[1, 1, 1, 1, 3], [2, 2, 2]], [[1, 1, 1, 1, 2], [2, 2, 3]]]
\end{sageexample}
The 13 weak 3-tableaux of shape $(5,2,1)$ in Example~\ref{Ex:weaktab} can be obtained as
\begin{sageexample}
    sage: k = 3
    sage: c = Core([5,2,1], k+1)
    sage: la = c.to_bounded_partition(); la
    [3, 2, 1]
    sage: for mu in Partitions(la.size(), max_part = 3):
    ....:     T = WeakTableaux(k, c, mu)
    ....:     print "weight", mu
    ....:     print T.list()
    ....:
    weight [3, 3]
    []
    weight [3, 2, 1]
    [[[1, 1, 1, 2, 2], [2, 2], [3]]]
    weight [3, 1, 1, 1]
    [[[1, 1, 1, 2, 4], [2, 4], [3]], [[1, 1, 1, 2, 3], [2, 3], [4]]]
    weight [2, 2, 2]
    [[[1, 1, 2, 2, 3], [2, 3], [3]]]
    weight [2, 2, 1, 1]
    [[[1, 1, 2, 2, 4], [2, 4], [3]], [[1, 1, 2, 2, 3], [2, 3], [4]]]
    weight [2, 1, 1, 1, 1]
    [[[1, 1, 3, 4, 5], [2, 5], [3]], [[1, 1, 2, 3, 5], [3, 5], [4]],
     [[1, 1, 2, 3, 4], [3, 4], [5]]]
    weight [1, 1, 1, 1, 1, 1]
    [[[1, 3, 4, 5, 6], [2, 6], [4]], [[1, 2, 4, 5, 6], [3, 6], [4]],
     [[1, 2, 3, 4, 6], [4, 6], [5]], [[1, 2, 3, 4, 5], [4, 5], [6]]]
\end{sageexample}
\end{sagedemo}

For any $k$-bounded partition $\la$ and $k$-bounded partition $\mu$,
we denote the number of weak tableaux of shape $\mfc(\la)$ and weight $\mu$
by $K^{(k)}_{\la\mu}$.  These numbers, called
{\it (weak) $k$-Kostka coefficients}, satisfy an
important property
\begin{equation}
K_{\la\mu}^{(k)}=
\begin{cases}
1 & \text{ if $\la=\mu$}\\
0 & \text{ if $\la \not\geq  \mu$}\,, 
\end{cases}
\end{equation}
with respect to dominance order on partitions.
Thus,  the matrix of coefficients 
$||K_{\lambda\mu}^{(k)}||$ over all
 $k$-bounded partitions $\lambda$ and $\mu$ of the same size
is unitriangular and thus invertible.
It is with this in hand that we arrive at
a family of functions that satisfy the weak Pieri rule.
To be precise, $k$-Schur functions were characterized 
in \cite{LM:2007} by the system obtained by taking
\begin{equation} \label{weakKostka}
h_\mu = \sum_{\la \;:\; \la_1\leq k} K^{(k)}_{\la\mu} s^{(k)}_\la \; ,
\end{equation}
for all $k$-bounded partitions $\mu$.
In fact, this system defines the $k$-Schur 
basis $\{s_\la^{(k)}\}_{\la_1 \leq k}$
because the elements $h_\la$ for $\la_1 \leq k$ form a basis 
for the space $\La_{(k)}$ and the transition matrix is invertible
over the integers.

\begin{example}
For $k=6$, the weak tableaux in Example~\ref{Ex:weaktab} tell us that
$$
h_{431} = s^{(6)}_{431} + s^{(6)}_{44} + 2 s^{(6)}_{53}+s^{(6)}_{521} + 
s^{(6)}_{611} + s_{62}^{(6)}\,.
$$
\end{example}

In Section~\ref{sec:threedef}, various different notions of $k$-Schur
functions will be given.  We will work with a running example 
and do a computation to give an idea of how each of the notions 
can be implemented and to
demonstrate their relative difficulty.

\begin{example}\squaresize = 6pt
Let us calculate $s_{3211}^{(3)}$ in terms of homogeneous
symmetric functions.  The $k$-Schur function $s^{(k)}_\la$
can be computed recursively using the weak Pieri rule.  The product
$h_{\la_1} s^{(k)}_{(\la_2, \la_3, \ldots, \la_{\ell})} = 
s^{(k)}_{\la} +$ other terms which
are indexed by $k$ bounded partitions $\gamma$ where
$\la$  is smaller than $\gamma$ in dominance order.

We begin by noting that $h_3 s_{311}^{(3)} = s_{3211}^{(3)}$,
and we thus must expand $s_{211}^{(3)}$.  
This time, noting that $h_2 s^{(3)}_{11} = s^{(3)}_{211}$ 
we turn to the computation of $s^{(3)}_{11}$. Since $h_1 s^{(3)}_{1} = 
s^{(3)}_{11} + s^{(3)}_{2}$
and $s^{(3)}_{2} = h_2$, we have computed $s^{(3)}_{11} = h_{11} - h_{2}$ which can
be substituted back to obtain $s^{(3)}_{211} = h_{211} - h_{22}$.
We conclude that $s^{(3)}_{3211} = h_{3211} - h_{322}$.

We leave it to the reader as an exercise to compute at least a few other of the
$3$-Schur functions of size $7$ by hand to get a feel for the difficulty
of these computations.  Fortunately, the functions can be verified against
a calculation using \Sage.
\end{example}

\begin{sagedemo}
Here we give the expansion of the $k$-Schur functions for $k=3$ in terms 
of the homogeneous basis as we did in the previous example.
Our current setting is the $t=1$ case.
\begin{sageexample}
    sage: Sym = SymmetricFunctions(QQ)
    sage: ks = Sym.kschur(3,t=1)
    sage: h = Sym.homogeneous()
    sage: for mu in Partitions(7, max_part =3):
    ....:     print h(ks(mu))
    ....:
    h[3, 3, 1]
    h[3, 2, 2] - h[3, 3, 1]
    h[3, 2, 1, 1] - h[3, 2, 2]
    h[3, 1, 1, 1, 1] - 2*h[3, 2, 1, 1] + h[3, 3, 1]
    h[2, 2, 2, 1] - h[3, 2, 1, 1] - h[3, 2, 2] + h[3, 3, 1]
    h[2, 2, 1, 1, 1] - 2*h[2, 2, 2, 1] - h[3, 1, 1, 1, 1] 
        + 2*h[3, 2, 1, 1] + h[3, 2, 2] - h[3, 3, 1]
    h[2, 1, 1, 1, 1, 1] - 3*h[2, 2, 1, 1, 1] + 2*h[2, 2, 2, 1] 
        + h[3, 2, 1, 1] - h[3, 2, 2]
    h[1, 1, 1, 1, 1, 1, 1] - 4*h[2, 1, 1, 1, 1, 1] + 4*h[2, 2, 1, 1, 1] 
        + 2*h[3, 1, 1, 1, 1] - 4*h[3, 2, 1, 1] + h[3, 3, 1]
\end{sageexample}
\end{sagedemo}

We are now in the position to use duality to produce a second basis, 
this time for the algebra $\La^{(k)}$.  
Although we do not 
have a scalar product on the spaces $\La_{(k)}$
and $\La^{(k)}$ separately, we appeal to a pairing between the two spaces 
\begin{equation} \label{equation.k scalar}
	\langle ~\cdot~, ~\cdot~ \rangle\; : \; \Lambda_{(k)} \times \Lambda^{(k)} \rightarrow \QQ\;,
\end{equation}
where $h_\mu \in \La_{(k)}$ and $m_\la \in \La^{(k)}$ are dual elements
\begin{equation}
\left< h_\mu, m_\la \right> = \delta_{\la\mu}~.
\end{equation}
This equation is precisely Equation~\eqref{scalardual} 
for the scalar product on symmetric functions.

The dual $k$-Schur functions $\SS^{(k)}_\la$ were introduced
in~\cite{LM:2008} as the unique basis of the degree $m$ 
subspace of $\La^{(k)}$ (with $m\geq 1$) that is dual to the basis 
of $\{s^{(k)}_\la\}_{\lambda\vdash m, \lambda_1\leq k}$ under the pairing~\eqref{equation.k scalar}.
It follows from \eqref{weakKostka} that
the enumeration of weak tableaux gives their monomial expansion:
\begin{align}\label{eq:dualkschurmexp}
\SS^{(k)}_\la &= \sum_{\mu:\mu_1 \leq k} \left< h_\mu, \SS^{(k)}_\la \right> m_\mu\nonumber\\
&= \sum_{\mu:\mu_1 \leq k}\; \sum_{\gamma :\gamma_1\leq k} K^{(k)}_{\gamma\mu} 
\left< s^{(k)}_{\gamma}, \SS^{(k)}_\la \right> m_\mu\\
&= \sum_{\mu:\mu_1 \leq k} K^{(k)}_{\la\mu} m_\mu~.\nonumber
\end{align}
There is an involution on the set of weak tableaux of fixed shape $\mfc(\la)$ 
and weight $\alpha$ that sends these tableaux to the set of
weak tableaux of shape $\mfc(\lambda)$ and weight which 
is a permutation of $\alpha$.
Thus,
the dual $k$-Schur functions are the weight
generating functions for weak tableaux.  That is, for $(k+1)$-core
$\la$,
\begin{equation}
\label{gendualk}
\SS_\lambda^{(k)} = \sum_{T=\text{weak tab}\atop
\shape(T)=\mfc(\lambda)}
x^{\text{weight}(T)}\,.
\end{equation}

Note that what we call the dual $k$-Schur functions here are also equal to affine
Stanley symmetric functions indexed by affine Grassmannian elements.  This connection is
discussed with more detail in Chapter \ref{chapter.stanley symmetric functions}, Section \ref{sec:noncommkschur}
and Chapter \ref{chapter.k schur primer}, Section \ref{sec:nilcoxeter}.
Note that our notation differs slightly from the notation in Chapter \ref{chapter.stanley symmetric functions}
by adding a superscript indicating $k$.

\begin{example}\label{calcdualkSchurfromposet}
The calculation of the dual $k$-Schur function 
$\SS_{321}^{(3)}$ follows immediately by extracting
the weights of each weak tableau of shape $(5,2,1)$,
listed in Example~\ref{Ex:weaktab}. We conclude that
$$\SS_{321}^{(3)} = m_{321} + 2 m_{3111} + m_{222} + 2 m_{2211} + 3 m_{21111}
+ 4 m_{111111}~.$$
\end{example}

\subsection{Other realizations}
\label{sec:nilcoxeter}

We have now seen how the weak Pieri rule -- given in terms 
of weak horizontal chains in the $(k+1)$-core realization 
of the weak poset -- leads to the family of dual $k$-Schur 
functions in terms of weak tableaux.  Equivalently, we could 
have invoked the definition of weak horizontal chains 
on the level of $k$-bounded partitions or on affine Grassmannian 
elements and this would easily give rise to characterizations for 
dual $k$-Schur functions in these other settings.
Before we move on to draw a weight generating
function characterization for $k$-Schur functions starting
instead from a ``strong'' Pieri rule, some exposition on these
other interpretations is warranted.

Let us start by retracing our steps that led from the
weak Pieri rule to the generating function for dual 
$k$-Schur functions, this time in the setting of $k$-bounded 
partitions.
In these terms, the weak Pieri relation is given for
any partition $\la$ with $\la_1\leq k$ to be
\begin{equation}\label{eq:kweakpieri}
h_r s^{(k)}_\la = \sum_{\mu:\mu_1\leq k} s^{(k)}_\mu\; ,
\end{equation}
where the sum is over partitions $\mu$ 
such that $\mu/\la$ is a horizontal strip and 
$\mu^{\omega_k}/\la^{\omega_k}$ is a vertical strip of size $r$.
Again, we use the iteration of this relation to impose
conditions on a family of weak tableaux in the $k$-bounded setting.

\begin{example}  \label{example.kSchur}
Iteratively, we calculate $h_{431}$ in terms of $k$-Schur functions for $k=6$:
\squaresize=5pt
\begin{align*}
h_{431} &= h_1(s^{(6)}_{\young{\ast&\ast&\ast\cr&&&\cr}} + s^{(6)}_{\young{\ast&\ast\cr&&&&\ast\cr}} + s^{(6)}_{\young{\ast\cr&&&&\ast&\ast\cr}}) \\
&=\big(s^{(6)}_{\young{\ast\cr&&\cr&&&\cr}}+s^{(6)}_{\young{&&&\ast\cr&&&\cr}}+s^{(6)}_{\young{&&\cr&&&&\ast\cr}}\big) 
+ \big(s^{(6)}_{\young{\ast\cr&\cr&&&&\cr}}+s^{(6)}_{\young{&&\ast\cr&&&&\cr}}\big)
+ \big(s^{(6)}_{\young{\ast\cr\cr&&&&&\cr}} + s^{(6)}_{\young{&\ast\cr&&&&&\cr}}\big)~.
\end{align*}
\end{example}

\begin{sagedemo}
\Sage can be used to verify Example~\ref{example.kSchur}:
\begin{sageexample}
    sage: ks6 = Sym.kschur(6,t=1)
    sage: ks6(h[4,3,1])
    ks6[4, 3, 1] + ks6[4, 4] + ks6[5, 2, 1] + 2*ks6[5, 3] 
      + ks6[6, 1, 1] + ks6[6, 2]
\end{sageexample}

\Sage also knows that the $k$-Schur functions live in the subring $\La_{(k)}$ of the ring of 
symmetric functions:
\begin{sageexample}
    sage: Sym = SymmetricFunctions(QQ)
    sage: ks = Sym.kschur(3,t=1)
    sage: ks.realization_of()
    3-bounded Symmetric Functions over Rational Field with t=1
    sage: s = Sym.schur()
    sage: s.realization_of()
    Symmetric Functions over Rational Field
\end{sageexample}
\end{sagedemo}

When $\la$ and $\mu$ are $k$-bounded partitions,
the weak Kostka coefficients $K_{\la\mu}^{(k)}$ are 
interpreted to be the number of sequences of $k$-bounded partitions, 
\begin{equation}
\emptyset = \la^{(1)} \subseteq \la^{(2)} \subseteq \cdots \subseteq\la^{(d)} = \la
\end{equation}
where $\la^{(i)}/\la^{(i-1)}$ is a horizontal strip of size $\mu_i$ and
$(\la^{(i)})^{\omega_k}/(\la^{(i-1)})^{\omega_k}$ is a vertical strip.
We then follow the line of reasoning from earlier to yield
\eqref{gendualk}, where we can instead think of weak tableaux
as these sequences of $k$-bounded shapes.

\begin{example}
The seven tableaux that make up the terms in the expansion of $h_{431}$ in
terms of $k$-Schur functions with $k=6$ are \squaresize=12pt
\begin{equation*}
\young{3\cr2&2&2\cr1&1&1&1\cr}\hskip .3in
\young{2&2&2&3\cr1&1&1&1\cr}\hskip .3in
\young{2&2&2\cr1&1&1&1&3\cr} \hskip .3in
\young{3\cr2&2\cr1&1&1&1&2\cr}
\end{equation*}
\begin{equation*}
\young{2&2&3\cr1&1&1&1&2\cr}\hskip .3in
\young{3\cr2\cr1&1&1&1&2&2\cr}\hskip .3in
\young{2&3\cr1&1&1&1&2&2\cr}~.
\end{equation*}
\end{example}

\begin{sagedemo}
We can reproduce these tableaux in \Sage using the bounded representation of weak $k$-tableaux:
\begin{sageexample}
    sage: k = 6
    sage: weight = Partition([4,3,1])
    sage: for la in Partitions(weight.size(), max_part = k):
    ....:     if la.dominates(weight):
    ....:         print la
    ....:         T = WeakTableaux(k, la, weight, representation = 'bounded')
    ....:         print T.list()
    ....:
    [6, 2]
    [[[1, 1, 1, 1, 2, 2], [2, 3]]]
    [6, 1, 1]
    [[[1, 1, 1, 1, 2, 2], [2], [3]]]
    [5, 3]
    [[[1, 1, 1, 1, 3], [2, 2, 2]], [[1, 1, 1, 1, 2], [2, 2, 3]]]
    [5, 2, 1]
    [[[1, 1, 1, 1, 2], [2, 2], [3]]]
    [4, 4]
    [[[1, 1, 1, 1], [2, 2, 2, 3]]]
    [4, 3, 1]
    [[[1, 1, 1, 1], [2, 2, 2], [3]]]
\end{sageexample}
\end{sagedemo}

Lastly, let us turn to the language of the affine symmetric group and
retrace the steps leading to the dual $k$-Schur functions.
Here, the weak Pieri rule is, for a $k$-bounded partition $\lambda$ the corresponding
affine Grassmannian element (described in Section \ref{section:partitionscores}) 
will be $\mfa(\mfc(\lambda))$ which we will shorten to
$\mfa(\lambda)$.  The weak Pieri rule on affine Grassmannian elements can be stated as
\begin{equation}
h_{r} \, s_\la^{(k)} = \sum_{u=\text{cyclically decreasing}\atop\ell(u)=r, u \mfa(\la) = \mfa(\mu) } s_{\mu}^{(k)}\,,
\end{equation}
where the sum is over cyclically decreasing reduced words such that there is a $k$-bounded
partition $\mu$ such that $u \mfa(\la) = \mfa(\mu)$.

The iteration of this relation produces another interpretation 
of the weak Kostka numbers.  First, for any $k$-bounded
partition $\mu$, define a {\it $\mu$-factorization} of $w$ 
to be a decomposition of the form 
$w=w^{\ell(\mu)}\cdots w^1$ where each $w^i$ is a cyclically decreasing 
element of length $\mu_i$.  Then, for $k$-bounded partitions $\lambda$
and $\mu$, $K^{(k)}_{\la,\mu}$ is the number of 
$\mu$-factorizations of $w = \mfa(\lambda)$.
In particular, if we consider the case that $\mu=(1,1,\dots,1)$,
then the weak Kostka number $K^{(k)}_{\la,(1^n)}$ is precisely the
number of reduced words for $w = \mfa(\la)$.

\begin{example}
Note from the previous example that the coefficient of $m_{111111}$ 
in $\SS^{(k)}_{521}$ is $4$ indicating that there are 4 reduced words
in ${\tilde S}_3$ for the affine Grassmannian element corresponding to the core
$(5,2,1)$; they are $s_2 s_0 s_3 s_2 s_1 s_0 = s_0 s_2 s_3 s_2 s_1 s_0 =
s_0 s_3 s_2 s_3 s_1 s_0 = s_0 s_3 s_2 s_1 s_3 s_0$.
\end{example}

From here, we again use the line of reasoning and duality to arrive at 
the interpretation for dual $k$-Schur functions as
\begin{equation}
\SS_w^{(k)}=\sum_{\mu} 
K^{(k)}_{\mfa^{-1}(w)\mu}\, m_{\mu}
\end{equation}
where we have indexed the dual $k$-Schur function by an affine Grassmannian
permutation to emphasize that the coefficients $K^{(k)}_{\mfa^{-1}(w)\mu}$
represent $\mu$-factorizations of $w$.
From the definition and symmetry of the weak Kostka numbers,
the dual $k$-Schur functions can be suggestively written as the weight generating function:
\begin{equation}
\label{dualingrass}
\SS_w^{(k)}=\sum_{w=w^{1}w^2\dots w^r}
x^{\ell(w^{1})}x^{\ell(w^2)}\cdots x^{\ell(w^r)}
\,,
\end{equation}
over all factorizations of $w$ into products of cyclically decreasing $w^i$.  

This interpretation 
for dual $k$-Schur functions is the starting point in \cite{Lam:2006} 
to a family of symmetric functions called {\it affine Stanley 
symmetric functions} where the condition that $w$ is affine Grassmannian is relaxed
and we allow the function to be indexed by arbitrary affine permutations
(not just an affine Grassmannian element).  For more information see
Chapter~\ref{chapter.stanley symmetric functions}, Definition~\ref{definition.affine Stanley}.

Programs to compute affine Stanley functions are in \Sage and we can 
do computations with dual $k$-Schur functions 
by taking affine Grassmannian elements.

\begin{sagedemo}
We verify Example~\ref{calcdualkSchurfromposet} in \Sage
by first converting the indexing partition $(3,2,1)$
to an affine Grassmannian element:
\begin{sageexample}
    sage: mu = Partition([3,2,1])
    sage: c = mu.to_core(3)
    sage: w = c.to_grassmannian()
    sage: w.stanley_symmetric_function()
    4*m[1, 1, 1, 1, 1, 1] + 3*m[2, 1, 1, 1, 1] + 2*m[2, 2, 1, 1] 
        + m[2, 2, 2] + 2*m[3, 1, 1, 1] + m[3, 2, 1]
    sage: w.reduced_words()
    [[2, 0, 3, 2, 1, 0], [0, 2, 3, 2, 1, 0], [0, 3, 2, 3, 1, 0], 
     [0, 3, 2, 1, 3, 0]]
\end{sageexample}
Alternatively, we can access the dual $k$-Schur functions from the quotient space:
\begin{sageexample}
    sage: Sym = SymmetricFunctions(QQ)
    sage: Q3 = Sym.kBoundedQuotient(3,t=1)
    sage: F3 = Q3.affineSchur()
    sage: m = Q3.kmonomial()
    sage: m(F3([3,2,1]))
    4*m3[1, 1, 1, 1, 1, 1] + 3*m3[2, 1, 1, 1, 1] + 2*m3[2, 2, 1, 1] 
        + m3[2, 2, 2] + 2*m3[3, 1, 1, 1] + m3[3, 2, 1]
\end{sageexample}
\end{sagedemo}

\end{subsection}

\begin{subsection}{Strong marked tableaux and a monomial expansion 
of $k$-Schur functions}
\label{sec:strongsection}
The Pieri rule for the dual $k$-Schur functions $\SS^{(k)}_\la$ is probably less intuitive than the one
for the $k$-Schur functions because it does have coefficients in the expansion which are
not simply $1$ or $0$.  In the last section we gave an explicit definition of the dual $k$-Schur functions
in~\eqref{eq:dualkschurmexp}, so it is possible to experiment with these elements to see how they 
behave under multiplication by an element $h_r \in \La^{(k)}$.  Through the following computations 
we hope to demonstrate how it might be possible to experiment with data for the Pieri rule for the 
dual $k$-Schur functions and then later explain what that Pieri rule is.

\begin{example} \label{example.dual expansion}
In Example~\ref{calcdualkSchurfromposet} we computed the dual $k$-Schur function
indexed by $(3,2,1)$ for $k=3$. Using the tableau definition of the dual $k$-Schur functions
from Equation~\eqref{eq:dualkschurmexp} it is possible to expand $\SS^{(3)}_\la$ for all 3-bounded partitions 
$\la$ of size $7$ and use this to find the expansion of $h_1 \SS_{321}^{(3)}$.  If the $k$-Schur functions 
up to size $7$ are already known, then this also can be computed using the duality.
In particular we obtain
\begin{equation} \label{equation.dual expansion}
	h_1 \SS_{321}^{(3)} = 2 \SS_{331}^{(3)} + \SS_{322}^{(3)} + 3 \SS_{3211}^{(3)} 
	+ \SS_{31111}^{(3)}~.
\end{equation}
This example demonstrates that if a dual $k$-Schur function indexed by a partition
$\la$ is multiplied by $h_1$, the resulting partitions indexing the functions in the expansion
do not necessarily contain $\la$ (notice that $(3,2,1)$ is not contained in $(3,1,1,1,1)$).
\end{example}

\begin{sagedemo}
We now demonstrate how the computations for Example~\ref{example.dual expansion} can be carried out 
in \Sage. The dual $k$-Schur functions can be accessed through the $k$-bounded quotient space:
\begin{sageexample}
    sage: Sym = SymmetricFunctions(QQ)
    sage: Q3 = Sym.kBoundedQuotient(3,t=1)
    sage: F3 = Q3.affineSchur()
    sage: h = Sym.homogeneous()
    sage: f = F3[3,2,1]*h[1]; f
    F3[3, 1, 1, 1, 1] + 3*F3[3, 2, 1, 1] + F3[3, 2, 2] + 2*F3[3, 3, 1]
\end{sageexample}
In terms of 3-bounded partitions, it is not so clear how to interpret the coefficients in the expansion 
in Example~\ref{example.dual expansion}. However, in terms of 4-cores it becomes more obvious.
The 4-core of $(3,2,1)$ is contained in the cores of the partitions in the expansion:
\begin{sageexample}
    sage: c = Partition([3,2,1]).to_core(3)
    sage: for p in f.support():
    ....:   print p, SkewPartition([p.to_core(3).to_partition(),c.to_partition()])
    ....:
    [3, 1, 1, 1, 1] [[5, 2, 1, 1, 1], [5, 2, 1]]
    [3, 2, 1, 1] [[6, 3, 1, 1], [5, 2, 1]]
    [3, 2, 2] [[5, 2, 2], [5, 2, 1]]
    [3, 3, 1] [[7, 4, 1], [5, 2, 1]]
\end{sageexample}
The corresponding skew diagrams are
\begin{equation*}
\young{\gris\cr\gris&\gris&&\cr\gris&\gris&\gris&\gris&\gris&&\cr} \qquad
\young{\gris&\cr\gris&\gris\cr\gris&\gris&\gris&\gris&\gris\cr} \qquad
\young{\cr\gris\cr\gris&\gris&\cr\gris&\gris&\gris&\gris&\gris&\cr} \qquad
\young{\cr\cr\gris\cr\gris&\gris\cr\gris&\gris&\gris&\gris&\gris\cr}\; ,
\end{equation*}
respectively. 
Comparing with~\eqref{equation.dual expansion}, one is led to conjecture that the 
coefficient of $\SS_\mu^{(k)}$ in the expansion of $h_1 \SS^{(k)}_\la$ 
equals the number of 
connected components of $\mfc(\mu)/\mfc(\la)$.
Alternatively, we may mark the lowest rightmost cell of one of these connected
components and the number of ways of marking is equal to this coefficient.  
\end{sagedemo}

Recall that we defined a strong marked horizontal strip of size $r$ in Equation~\eqref{equation.strong h-strip}
as a sequence cores
\begin{equation*}
\kappa^{(0)} \Rightarrow_{k} \kappa^{(1)} \Rightarrow_{k} \kappa^{(2)} \Rightarrow_{k} \cdots
\Rightarrow_{k} \kappa^{(r)}
\end{equation*}
and integers $c_1 < c_2 < \cdots < c_r$ such that $(\kappa^{(i)}, \kappa^{(i-1)},c_i)$
is a marking for each $1 \leq i \leq r$.
It is in these terms that we state the Pieri rule for dual $k$-Schur functions.

\begin{theorem} \cite[Theorem 4.9]{LLMS:2006} \label{thrm:strongPieri}  
For $r \geq 1$,
\begin{equation} \label{strongPieri}
	h_r \SS^{(k)}_\la = \sum_{(\kappa^{(\ast)},c_\ast)} \SS^{(k)}_{\mfp(\kappa^{(r)})} \;,
\end{equation}
where the sum is over all strong marked horizontal strips 
\begin{equation}
\kappa^{(\ast)} = \left( \mfc(\la) = \kappa^{(0)} \Rightarrow_{k}\kappa^{(1)} \Rightarrow_{k} \kappa^{(2)} \Rightarrow_{k} \cdots \Rightarrow_{k} \kappa^{(r)}\right)
\end{equation}
with markings $c_\ast = (c_1 < c_2 < \cdots < c_r)$.
\end{theorem}

Iterating the Pieri rule on the dual $k$-Schur functions defines a different 
notion of `semi-standard tableaux'.
We say that a {\it strong marked tableau} of shape $\la \vdash m$ (or shape $\mfc(\la)$ if
the shape is more properly given as a $(k+1)$-core) and content
$\alpha = ( \alpha_1, \alpha_2, \ldots, \alpha_d)$ with $\alpha_1 + \alpha_2 +
\cdots + \alpha_d =m$ and $\alpha_i\geq 1$ is a sequence of $(k+1)$-cores
\begin{equation}
\kappa^{(0)} = \emptyset \Rightarrow_{k} \kappa^{(1)} \Rightarrow_{k} \kappa^{(2)} \Rightarrow_{k}
\cdots \Rightarrow_{k} \kappa^{(m)} = \mfc(\la)
\end{equation}
and markings $c_\ast = (c_1, c_2, \ldots, c_m)$ such that 
$(\kappa^{(v)},\kappa^{(v+1)}, \ldots,\kappa^{(v+\alpha_{r})})$ and 
$(c_{v+1}, c_{v+2}, \ldots, c_{v+\alpha_r})$ with $v=\alpha_1+\cdots+\alpha_{r-1}$
forms a strong marked horizontal strip for each $1 \leq r \leq d$.

Recall from Section~\ref{subsection.strong order} that a strong marked cover is also an application of a 
transposition $t_{ij}$ in the affine symmetric group to a core (either by the left or right action,
see Proposition \ref{prop:strongcovertransposition}).  Therefore, it is possible to also view a
tableau as an element of the affine Grassmannian written as a sequence of these
transpositions.  The condition that a sequence of transpositions 
$t_{i_{a+b}j_{a+b}}\cdots t_{i_{a+2}j_{a+2}} t_{i_{a+1}j_{a+1}} t_{i_aj_a}$
forms a strong marked horizontal strip (with the left action) implies that $j_a < j_{a+1} < \cdots < j_{a+b}$.

\begin{example} \label{ex:strongtab}
Consider the path 
\begin{multline*}
	\emptyset \Rightarrow_3 (1) \Rightarrow_3 (2) \Rightarrow_3 (2,1) \Rightarrow_3 (3,1,1) 
	\Rightarrow_3 (3,1,1,1) \Rightarrow_3 (3,3,1,1) \\
	\Rightarrow_3 (4,3,2,1) \Rightarrow_3 (4,4,2,2).
\end{multline*}
We choose $\alpha = (2,2,3,1)$ and a sequence
$c = (0,1,-1,2,-3,1,3,-2)$ as a sequence of markings which form strong marked 
horizontal strips of the lengths given by the entries of $\alpha$ to demonstrate 
the concept of a strong marked tableau.  The strong marked tableau of shape $(4,4,2,2)$
records all cells representing strong marked horizontal strips labeled with
the same main label but subscripted by which set of ribbons they 
belong to. The cell which is marked (the head of
one of the ribbons) will be indicated by including an $\ast$ as a superscript.  Hence 
the example above is represented by the diagram:
\squaresize=14pt
$$\young{3_1^\ast&4_1^\ast\cr2_2&3_3\cr2_1^\ast&3_2&3_2^\ast&4_1\cr1_1^\ast&1_2^\ast&2_2^\ast&3_3^\ast\cr}\;.$$

This example is also represented by the following sequence of transpositions with the left action
$$t_{-2-1} t_{34} t_{02} t_{-3-2} t_{23} t_{-10} t_{12} t_{01} 
$$
which act on the empty $4$-core.  Each transposition adds
the strong marking in a horizontal strip according to the tableau.

This same example is also represented using the right action of the transpositions by
$$t_{01} t_{02} t_{-11} t_{05} t_{-21} t_{-12} t_{06} t_{-52}~.
$$
\end{example}

Just as we did for semi-standard and weak tableaux, we set $\sK^{(k)}_{\la\mu}$ to be the number of
strong marked tableaux of shape $\la$ and weight
$\mu$.  By iterating the Pieri rule on dual $k$-Schur
functions and using the notion of strong marked tableaux to record the terms which appear in the expansion
of products of elements of the homogeneous symmetric functions, we have that 
for a partition $\mu \vdash m$ (not necessarily $k$-bounded)
\begin{equation}\label{strongtabexphmu}
	h_\mu = \sum_{\la : \la_1 \leq k} \sK^{(k)}_{\la\mu} \SS^{(k)}_{\la}~.
\end{equation}

Then for a partition $\la \vdash m$ with $\la_1 \leq k$,
\begin{equation}\label{eq:kschurmexp}
	s^{(k)}_{\la} = \sum_{\mu} \langle s^{(k)}_{\la}, h_\mu \rangle m_\mu = \sum_{\mu} \sK^{(k)}_{\la\mu} m_\mu~.
\end{equation}

We have hidden all of the work that is needed to show the monomial expansions of $k$-Schur functions
and their duals are correct by this presentation.  In fact, the proof~\cite{LLMS:2006} follows a nearly reverse path
of reasoning to demonstrate the results we have presented here.
Equation \eqref{eq:kschurmexp} is taken as the definition of the $k$-Schur functions 
and~\eqref{eq:dualkschurmexp}
is the definition of the dual $k$-Schur functions, and then the argument uses combinatorics and algebraic
expressions to show that these elements are dual.  The combinatorial part of that argument is developed in
Section \ref{sec:strongweakduality}.

\begin{example}\squaresize=12pt \label{ex:strongtableaux}
Consider the coefficient of $m_{421}$ in $s_{3211}^{(3)}$. It is calculated by finding all strong marked 
tableaux that begin with a horizontal strip of length $4$ and hence contains the subtableau 
$\small \young{1_4\cr1_1^\ast&1_2^\ast&1_3^\ast&1_4^\ast\cr}$\;. 
This is followed by a horizontal strip of length $2$ and hence contains one of the five subtableaux:
$$\small \young{1_4&2_1^\ast&2_2^\ast\cr1_1^\ast&1_2^\ast&1_3^\ast&1_4^\ast&2_1&2_2\cr},\quad
\young{1_4&2_1^\ast&2_2\cr1_1^\ast&1_2^\ast&1_3^\ast&1_4^\ast&2_1&2_2^\ast\cr},\quad
\young{1_4&2_1&2_2\cr1_1^\ast&1_2^\ast&1_3^\ast&1_4^\ast&2_1^\ast&2_2^\ast\cr},\quad
\young{2_1^\ast\cr1_4&2_2^\ast\cr1_1^\ast&1_2^\ast&1_3^\ast&1_4^\ast&2_2\cr},\quad
\young{2_1^\ast\cr1_4&2_2\cr1_1^\ast&1_2^\ast&1_3^\ast&1_4^\ast&2_2^\ast\cr}~.$$
There are $9$ strong marked tableaux of shape $\mfc(3,2,1,1) = (6,3,1,1)$ of 
weight $(4,2,1)$ which are given by the following
$$\small 
\young{3_1\cr3_1^\ast\cr1_4&2_1^\ast&2_2^\ast\cr1_1^\ast&1_2^\ast&1_3^\ast&1_4^\ast&2_1&2_2\cr},\quad
\young{3_1\cr3_1^\ast\cr1_4&2_1^\ast&2_2\cr1_1^\ast&1_2^\ast&1_3^\ast&1_4^\ast&2_1&2_2^\ast\cr},\quad
\young{3_1\cr3_1^\ast\cr1_4&2_1&2_2\cr1_1^\ast&1_2^\ast&1_3^\ast&1_4^\ast&2_1^\ast&2_2^\ast\cr},\quad
\young{3_1^\ast\cr2_1^\ast\cr1_4&2_2^\ast&3_1\cr1_1^\ast&1_2^\ast&1_3^\ast&1_4^\ast&2_2&3_1\cr},\quad
\young{3_1\cr2_1^\ast\cr1_4&2_2^\ast&3_1^\ast\cr1_1^\ast&1_2^\ast&1_3^\ast&1_4^\ast&2_2&3_1\cr},
$$
$$
\small 
\young{3_1\cr2_1^\ast\cr1_4&2_2^\ast&3_1\cr1_1^\ast&1_2^\ast&1_3^\ast&1_4^\ast&2_2&3_1^\ast\cr},\quad
\young{3_1^\ast\cr2_1^\ast\cr1_4&2_2&3_1\cr1_1^\ast&1_2^\ast&1_3^\ast&1_4^\ast&2_2^\ast&3_1\cr},\quad
\young{3_1\cr2_1^\ast\cr1_4&2_2&3_1^\ast\cr1_1^\ast&1_2^\ast&1_3^\ast&1_4^\ast&2_2^\ast&3_1\cr},\quad
\young{3_1\cr2_1^\ast\cr1_4&2_2&3_1\cr1_1^\ast&1_2^\ast&1_3^\ast&1_4^\ast&2_2^\ast&3_1^\ast\cr}
$$
A similar argument can be used to find any coefficient of $m_\mu$ for any $\mu \vdash 7$, however there are 
$210$ strong marked tableaux of shape $\mfc(3,2,1,1) = (6,3,1,1)$ and 
weight $(1,1,1,1,1,1,1)$ 
and hence we will not show the complete computation of $s_{3211}^{(3)}$ using this method.
\end{example}

In later sections we will roughly outline how this formula is proven
by generalizing the Robinson--Schensted--Knuth algorithm.

\begin{sagedemo}
In this example we show how \Sage can be used to complete calculations from Examples \ref{ex:strongtab} and \ref{ex:strongtableaux}.
The strong tableaux can be entered as a list of entries with markings indicated by a negative number.

\begin{sageexample}
sage: T = StrongTableau([[-1,-1,-2,-3],[-2,3,-3,4],[2,3],[-3,-4]], 3)
sage: T.to_transposition_sequence()
[[-2, -1], [3, 4], [0, 2], [-3, -2], [2, 3], [-1, 0], [1, 2], [0, 1]]
sage: T.intermediate_shapes()
[[], [2], [3, 1, 1], [4, 3, 2, 1], [4, 4, 2, 2]]
sage: [T.content_of_marked_head(v+1) for v in range(8)]
[0, 1, -1, 2, -3, 1, 3, -2]
sage: T.left_action([0,1])
[[-1, -1, -2, -3, 5], [-2, 3, -3, 4], [2, 3, -5], [-3, -4], [5]]
\end{sageexample}

The strong tableaux can be listed and by equation \eqref{eq:kschurmexp}
the number of strong tableaux of shape $\mfc(\lambda)$ and content $\mu$
can be calculated by determining the coefficient
of $m_\mu$ in $s^{(k)}_\lambda$.

\begin{sageexample}
sage: ST = StrongTableaux(3, [6,3,1,1], [4,2,1]); ST
Set of strong 3-tableaux of shape [6, 3, 1, 1] and of weight (4, 2, 1)
sage: ST.list()
[[[-1, -1, -1, -1, 2, 2], [1, -2, -2], [-3], [3]],
 [[-1, -1, -1, -1, 2, -2], [1, -2, 2], [-3], [3]],
 [[-1, -1, -1, -1, -2, -2], [1, 2, 2], [-3], [3]],
 [[-1, -1, -1, -1, 2, 3], [1, -2, 3], [-2], [-3]],
 [[-1, -1, -1, -1, 2, 3], [1, -2, -3], [-2], [3]],
 [[-1, -1, -1, -1, 2, -3], [1, -2, 3], [-2], [3]],
 [[-1, -1, -1, -1, -2, 3], [1, 2, 3], [-2], [-3]],
 [[-1, -1, -1, -1, -2, 3], [1, 2, -3], [-2], [3]],
 [[-1, -1, -1, -1, -2, -3], [1, 2, 3], [-2], [3]]]
sage: ks = SymmetricFunctions(QQ).kschur(3,1)
sage: m = SymmetricFunctions(QQ).m()
sage: m(ks[3,2,1,1]).coefficient([4,2,1])
9
\end{sageexample}

\end{sagedemo}

\begin{remark}
We have been purposely lax in our notation to make some of the concepts slightly easier
to follow, but the duality creates a few issues with the element
which represents $\SS^{(k)}_\la$ in $\La^{(k)}$.  
The equalities in Equations~\eqref{strongPieri} and~\eqref{strongtabexphmu}
represent equality in the realization of the dual algebra 
$\La^{(k)} = \La/\big< m_\la : \la_1>k \big>$,
so the equality really means equivalence in the quotient algebra.
For computational purposes, we would typically take a representative element from the linear
span of $\{ m_\la \}_{\la_1 \leq k}$, but for certain purposes a multiplicative basis might be more desirable.   
In those cases, the basis $\{ p_\la \}_{\la_1 \leq k}$ works well since $p_{r} \in \big< m_\la : \la_1>k \big>$ for
$r > k$.  A combinatorial formula for $\SS^{(k)}_\la$ in terms of the power sum basis is also known by 
reference~\cite{BandlowSchillingZabrocki}.
\end{remark}

\end{subsection}

\begin{subsection}{$k$-Littlewood--Richardson coefficients}
\label{sec:nilcoxeter}

Although the original definition~\cite{LLM:2003} of $k$-Schur functions was
inspired to explain the positivity of the expansion of Macdonald symmetric 
functions in terms of Schur functions, it has since been established
that the theory of $k$-Schur functions and their duals can be naturally
applied to study problems in geometry, physics and representation theory.

The application of $k$-Schur functions to geometric problems
began when Lapointe and Morse discovered that their structure constants 
could be identified with certain geometric invariants
in a way that mimics the identification of 
Littlewood--Richardson coefficients with Schubert
structure constants in the Grassmannian variety
(recall from Section~\ref{sectionschur});
computation in the quantum cohomology of Grassmannians reduces to $k$-Schur 
calculations.  The (small) quantum cohomology ring $QH^*(\Gr_{\ell n})$ is 
a deformation of the classical cohomology ring that is
motivated by ideas in string theory (e.g. \cite{[Ag],[Wi]}).  
As abelian groups, $QH^*(\Gr_{\ell n})=H^*(\Gr_{\ell  n})\otimes\mathbb Z[q]$ 
and the Schubert
classes $\sigma_\lambda$ with $\lambda\in \mathcal P^{\ell n}$ form a 
$\mathbb Z[q]$-linear basis, where recall that $\mathcal P^{\ell n}$ is the set 
of all partitions in an $\ell \times (n-\ell)$ rectangle.  The appeal lies in the multiplicative 
structure which is defined by
$$
\sigma_\lambda * \sigma_\mu = \sum_{\nu\in\mathcal P^{\ell n}\atop
|\nu|=|\lambda|+|\mu|-dn}
q^dC_{\lambda\mu}^{\nu,d}\sigma_\nu
\,,
$$
where $C_{\lambda\mu}^{\nu,d}$ are the 
{\it 3-point Gromov--Witten invariants},
counting the number of rational curves of degree $d$ in $\Gr_{\ell n}$
that meet generic translates of certain Schubert varieties.

As with the usual cohomology, the quantum cohomology ring
can be connected to symmetric functions.  In particular,
$$
QH^*(\Gr_{\ell n}) \cong \left(\Lambda_{(\ell)}\otimes\mathbb Z[q]\right)/
J^{\ell n}_q \, ,
$$
where $J^{\ell n}_q=\langle e_{n-\ell+1},\ldots,e_{n-1},e_n+(-1)^{\ell}q\rangle$.
When $\lambda \in \mathcal P^{\ell n}$, the Schubert class $\sigma_\la$
still maps to the Schur function $s_\la$ implying that
for $\la,\mu\in\mathcal P^{\ell n}$,
\begin{equation}
\label{gwcoef}
\sum_\nu c_{\la\mu}^\nu\, s_\nu 
\mod J^{\ell n}_q
= 
\sum_{\nu\in\mathcal P^{\ell n}\atop
|\nu|=|\lambda|+|\mu|-dn}
q^d C_{\lambda\mu}^{\nu,d}s_\nu 
\, .
\end{equation}
Unfortunately, there are $\nu\not \in \mathcal P^{\ell n}$ where
$s_\nu$ does not lie in $ J^{\ell n}_q$.  Instead,
reduction modulo this ideal requires a complicated
algorithm \cite{BCF,[Cu],Kac,[Wa]} involving negatives.
Therefore, the Schur functions cannot 
be used to directly obtain the quantum structure constants. 

It was proven in \cite{LM:2008} that the $k$-Schur basis
can be used to circumvent this problem; the appropriate
$k$-Schur functions lie in $J^{\ell n}_q$
(as usual, we set $n=k+1$).
To be precise, the {\it $k$-Littlewood--Richardson 
coefficients} are the structure coefficients $c_{\la\mu}^{\nu(k)}$ of the algebra 
of $k$-Schur functions
\begin{equation}
s^{(k)}_\la s^{(k)}_\mu = \sum_\nu c_{\la\mu}^{\nu(k)} s^{(k)}_\nu\,.
\end{equation}
Let $\Pi^{\ell,k+1}$ be the set of partitions with no part larger than $\ell$ and 
no more than $k+1-\ell$ rows of length smaller than $\ell$.
It is proven that $s_\nu^{(k)}$ modulo $J^{\ell n}_q$ 
reduces to a power of $q$ times a positive $s_{\mathfrak r (\nu)}$ 
when $\nu \in \Pi^{\ell n}$ and is otherwise zero
(where $\mathfrak r (\nu)\in \mathcal P^{\ell n}$ is the $n$-core of $\nu$).
Thus, considering
\begin{equation}
\label{klitric}
s_\lambda^{(k)} \, s_\mu^{(k)} = \sum_{\nu\in \Pi^{\ell,k+1}}
a_{\lambda\mu}^{\nu(k)} \, s_{\nu}^{(k)} +
\sum_{\nu\not\in \Pi^{\ell,k+1 }}
c_{\lambda \mu}^{\nu,k} s_{\nu}^{(k)}\,,
\end{equation}
the 3-point Gromov--Witten invariants are none other than a
special case of $k$-Littlewood--Richardson coefficients;
for $\lambda,\mu,\nu\in\mathcal P^{\ell n}$,
\begin{equation}
\label{gw}
C_{\lambda \mu}^{\nu,d} = a_{\lambda \mu}^{\hat \nu(n-1)} \, ,
\end{equation}
where the value of $d$ associates a certain unique 
element $\hat \nu\in\Pi^{\ell n}$ to each $\nu$.
It also follows that the $k$-Littlewood--Richardson 
coefficients include the fusion rules for the Wess--Zumino--Witten 
conformal field theories associated to $\widehat{su}(\ell)$ at 
level $k+1-\ell$ and certain Hecke algebra structure constants
studied by Goodman and Wenzl in \cite{[GW]}.

Note that the quantum structure constants \eqref{gwcoef} are only 
a subset of the complete set of $k$-Littlewood--Richardson 
coefficients \eqref{klitric}.  To understand the bigger picture,
recall that the $k$-Schur functions are a basis for $\Lambda_{(k)}$
and that $\Lambda_{(k)}\cong H_*(\Gr)$ where $\Gr$ is the
{\it affine Grassmannian}
quotient $\Gr=SL_{k+1}(\mathbb C((t)))/SL_{k+1}(\mathbb C[[t]])$
\cite{Bott}.
Bott also showed that $H^*(\Gr)\cong \Lambda^{(k)}$, the space
equipped with the dual $k$-Schur basis.
Morse and Shimozono conjectured
that the $k$-Schur functions and their duals are isomorphic to the Schubert 
classes of the homology and cohomology of the affine Grassmannian.
Chapter~\ref{chapter.stanley symmetric functions} explains
how Lam proved this conjecture.

Many attempts have been made to understand the coefficients $c_{\la\mu}^{\nu(k)}$
but the complete combinatorial picture has yet to be drawn.
Knutson formulated a conjecture for the subset of
quantum Littlewood--Richardson coefficients 
as presented in~\cite{BKT:2003} in terms of puzzles~\cite{KnutsonTao:2003}. 
Coskun~\cite{Coskun:2009} gave a positive geometric rule to compute the structure 
constants of the cohomology ring of two-step flag varieties in terms of Mondrian tableaux. 
As will be discussed in Section~\ref{sec:rectprop},
the case when either $\la$ or $\mu$ is a rectangle with a hook of size $k$ is
special.  Lapointe and Morse~\cite{LM:2007} showed that if $R$ is a rectangular partition
with maximal hook $k$, then
$$s_R^{(k)} s_{\la}^{(k)} = s_{R \cup \la}^{(k)}~.$$
In~\cite{MS:2012,MS:2013}, Morse and Schilling define crystal operators on
$\alpha$-factorizations (or equivalently weak $k$-tableaux) to determine
some structure coefficients of the Schur function times $k$-Schur function
expansion using a sign-reversing involution. This includes the case of fusion
coefficients.

Recently there has been some progress in the understanding
of these coefficients by viewing them in terms of the
nil-Coxeter algebra.
Slightly change the setting from the affine symmetric group to
the affine nil-Coxeter algebra $\mathbb{A}_n$ of type $A_{n-1}$
defined by the generators $\uu_0, \uu_1, \uu_2, \ldots, \uu_{n-1}$ 
satisfying the same relations as the affine symmetric group
Equation~\eqref{eq:affsymgens} except that the 
quadratic relation is altered to be
$$\uu_i^2=0\,,
$$
for any $i\in I:=\{0, 1, 2, \ldots, n-1\}$.

As before, $n=k+1$.  For $1 \leq r \leq k$, we set
$$
	\hh_r = \sum_{w~\text{cyclically~decreasing}\atop \ell(w)=r} \uu_w\,.
$$
Note that the coefficient of $x^\alpha$ in \eqref{dualingrass} is the 
coefficient of $\uu_w$ in $\hh_{\alpha_1} \hh_{\alpha_2}\cdots \hh_{\alpha_\ell}$
and thus, the affine Stanley symmetric function is 
\[
	\FF_w = \sum_{\alpha} (\hbox{coefficient of } \uu_w \in \hh_{\alpha_1} \hh_{\alpha_2} \cdots \hh_{\alpha_\ell}) 
	\; x^\alpha~.
\]

As is stated in Theorem~\ref{thm:ahcommute} of Chapter~\ref{chapter.stanley symmetric functions}, 
the elements $\hh_r$ mutually commute
and there is a subalgebra $\BBaf \subseteq \mathbb{A}_n$
which is generated by the elements $\hh_r$ for $1 \leq r \leq k$. 
Moreover, $\BBaf$ is isomorphic to $\La_{(k)}$ (see Proposition~\ref{prop:aB} of Chapter~\ref{chapter.stanley symmetric functions}) 
and there is a coalgebra structure which encodes the structure of the dual algebra.  
The {\it noncommutative $k$-Schur functions} can be realized as
elements $\ss^{(k)}_\la \in \BBaf$
in this algebra by expanding $s_\lambda^{(k)}$ in the homogeneous generators $h_i$ and replacing
$h_i$ with $\hh_i$.

Let $a_{\lambda w}$ be the coefficients in the expansion of ${\bf s}_\la^{(k)}$ in terms of expressions 
in the generators $\uu_i$ in the affine nil-Coxeter algebra
\begin{equation}\label{eq:nilcoxeterexpansion}
{\bf s}_\la^{(k)} = \sum_{w} a_{\lambda w} \uu_w~.
\end{equation}
The expansion of ${\bf s}_\la^{(k)}$ has a single term indexed by an affine Grassmannian permutation,
it is precisely the affine Grassmannian permutation which corresponds to the
partition $\lambda$ and the coefficient $a_{\la w_\la} = 1$ where $\mfa(\la) = w_\la$.

By~\cite[Proposition 6.7]{Lam:2006},  the  coefficient $a_{\la w}$  is  the
coefficient  of  $\FF_\la$  in  the  element  $\FF_w$.    In~\cite{Lam:2008},  these  coefficients
were  shown  to  be  positive. Furthermore, by the arguments in~\cite[Section 17]{Lam:2006} it follows that
\begin{equation} \label{equation.structure coef}
	c_{\lambda \mu}^{\nu(k)} = a_{\lambda vw^{-1}},
\end{equation}
where $w = \mfa(\mu)$ and $v = \mfa(\nu)$, see Proposition~\ref{prop:affgrass2core} .

Berg, Bergeron, Thomas, and Zabrocki~\cite{BBTZ:2011} gave a combinatorial expansion
of $k$-Schur functions indexed by a rectangle with maximal hook $k$  in the nil-Coxeter algebra
of the form of Equation~\eqref{eq:nilcoxeterexpansion}.
This work was extended by Berg, Saliola, and Serrano~\cite{BSS:2012} to
give expansions of $k$-Schur functions when the indexing partition is a `maximal rectangle' with a 
smaller hook removed. In addition they proved some conjectures of~\cite{LLMS:2006} in~\cite{BSS:2012a},
in particular that ``skew shaped'' strong Schur functions are symmetric.

\begin{sagedemo}
We now show how to compute the noncommutative Schur functions in the affine nil-Coxeter algebra
$\mathbb{A}_n$.
We can construct all cyclically decreasing words from the reduced words of the Pieri factors
for the affine type $A_{n-1}$  where $n=k+1$.
\begin{sageexample}
  sage: W = WeylGroup(["A",3,1])
  sage: [w.reduced_word() for w in W.pieri_factors()]
  [[], [0], [1], [2], [3], [1, 0], [2, 0], [0, 3], [2, 1], [3, 1], [3, 2], 
   [2, 1, 0], [1, 0, 3], [0, 3, 2], [3, 2, 1]]
\end{sageexample}
Then the noncommutative homogeneous symmetric functions are given by summing over all cyclically decreasing
words of specified length:
\begin{sageexample}
  sage: A = NilCoxeterAlgebra(WeylGroup(["A",3,1]), prefix = 'A')
  sage: A.homogeneous_noncommutative_variables([2])
  A[1,0] + A[2,0] + A[0,3] + A[3,2] + A[3,1] + A[2,1]
\end{sageexample}
The noncommutative $k$-Schur functions are obtained by expanding 
the usual $k$-Schur functions in terms of the homogeneous symmetric function:
\begin{sageexample}
  sage: A.k_schur_noncommutative_variables([2,2])
  A[0,3,1,0] + A[3,1,2,0] + A[1,2,0,1] + A[3,2,0,3] + A[2,0,3,1]
  + A[2,3,1,2]
\end{sageexample}
Now let us test that $a_{\lambda w}$ is indeed related to $c_{\la\mu}^{\nu(k)}$. The structure coefficient
$c_{21,21}^{321(5)}=2$ as we can see from the following computation:
\begin{sageexample}
  sage: Sym = SymmetricFunctions(ZZ)
  sage: ks = Sym.kschur(5,t=1)
  sage: ks[2,1]*ks[2,1]
  ks5[2, 2, 1, 1] + ks5[2, 2, 2] + ks5[3, 1, 1, 1] + 2*ks5[3, 2, 1] 
  + ks5[3, 3] + ks5[4, 2]
\end{sageexample}
Let $v$ (resp. $w$) be the affine Grassmannian element corresponding to the $5$-bounded partition
$(3,2,1)$ (resp. $(2,1)$). Then by~\eqref{equation.structure coef} the coefficient of $A_{vw^{-1}}$ in 
${\bf s}^{(5)}_{21}$ should also be 2:
\begin{sageblock}
  sage: mu = Partition([2,1])
  sage: nu = Partition([3,2,1])
  sage: w = mu.from_kbounded_to_grassmannian(5)
  sage: v = nu.from_kbounded_to_grassmannian(5)
  sage: A = NilCoxeterAlgebra(WeylGroup(["A",5,1]), prefix = 'A')
  sage: ks = A.k_schur_noncommutative_variables([2,1])
  sage: ks.coefficient(v*w^(-1))
  2
\end{sageblock}
\end{sagedemo}

\end{subsection}

\begin{subsection}{Notes on references}
Note that in certain references (e.g. \cite{LM:2003, LM2:2003, LM:2007}) 
the notation for $\La_{(k)}$ and $\La^{(k)}$ are switched.
We chose our convention to be consistent with~\cite{Lam:2008}, where the notation is motivated from
the relation between these spaces and the homology/cohomology of the affine Grassmannian.

In reference~\cite[Conjecture 21]{LM:2003} it is stated that the $k$-Schur 
functions defined using an algebraic definition satisfy the Pieri 
rule of Equation~\eqref{eq:kweakpieri}.  Later references \cite{LM:2005,LM:2007} 
do use Equation~\eqref{eq:kweakpieri} as
the definition and prove properties from this point of view.
The definition used in~\cite{LM:2003, LM2:2003}
will be presented in the next section and is 
conjecturally equivalent to Equation~\eqref{eq:kweakpieri} at $t=1$.

The proof of Theorem~\ref{thrm:strongPieri} is cited here as due to the 
affine insertion
algorithm that was studied by Lam, Lapointe, Morse and Shimozono~\cite{LLMS:2006}.
This algorithm will be discussed further in 
Section~\ref{subsection.affine insertion}.
A bijective algorithm is used to show that there is a duality
between strong and weak orders on affine Grassmannians that is manifest in
the Pieri rules for $k$-Schur functions and their duals.

When the details of affine Stanley symmetric functions are
carried out in Chapter~\ref{chapter.stanley symmetric functions},
there is a minor difference of notation; the affine nil-Coxeter algebra 
is denoted $\mathbb{A}_n$ in Section~\ref{sec:nilcoxeter} and as
$\mathbb{A}_{\rm af}$ in Chapter~\ref{chapter.stanley symmetric functions}.
In~\cite{Pos}, Postnikov used the affine nil-Temperley--Lieb algebra, which is a 
quotient of the nil-Coxeter algebra, to provide a model for the quantum cohomology.

\end{subsection}
\end{section}

\begin{section}{Definitions of $k$-Schur functions}\label{sec:threedef}
In the previous section we defined $k$-Schur functions and their dual basis
using the notions of strong and weak tableaux, but these are
parameterless symmetric functions.  In this section we will provide 
the original definition of $k$-Schur functions as well as
several others that are conjecturally equivalent.


\begin{subsection}{Atoms as tableaux} \label{sec:atom}
The origin of 
the tableaux definition \eqref{atom}
comes from identifying the $k$-Schur functions 
on the poset of standard
tableaux ranked by the charge statistic.
It is known that there is one Schur function
in the expansion of the Macdonald symmetric functions for each standard tableau, 
so given that the $k$-Schur functions are these irreducible components 
of the Macdonald symmetric
functions, it is natural to try to connect them with standard tableaux.
These conjecturally symmetric functions were first given the name `atoms' 
because they are the indecomposable pieces that come together  to give
Macdonald symmetric functions.  
They can be determined experimentally and from that
point it is possible to conjecture beautiful properties 
that they possess.

\begin{example}\label{ex:potatos}
From Example \ref{ex:atom2} we `know' that 
$A_{22}^{(2)}[X;t] = s_{22} + t s_{31}+ t^2 s_{4}$,
$A_{211}^{(2)}[X;t] = s_{211} + t s_{31}$, 
$A_{1111}^{(2)}[X;t] = s_{1111} + t s_{211} + t^2 s_{22}$.
The following picture contains the standard tableaux of size $4$ graded by 
charge.  The tableau on the left
has charge 0, the one on the right has charge $6$ and otherwise the left to right positions depends on
the value of the charge.  On this diagram we have circled the groups of these tableaux which represent
the atoms:
\begin{center}
\includegraphics[width=5in]{kschurnotes/potato1.pdf}
\end{center}
Moreover using the same method as in Example~\ref{ex:atom2} to compute the
$3$-Schur functions, we obtain $s_{1111}^{(3)}[X;t] = s_{1111} + t s_{211}$, 
$s_{211}^{(3)} = s_{211} + t s_{31}$,
$s_{22}^{(3)} = s_{22}$, $s_{31}^{(3)} = s_{31} + t s_{4}$.  We then picture these
symmetric functions as groups of tableaux on the tableaux poset ranked by charge by placing
a circle around the `copies' of the $3$-Schur functions:
\begin{center}
\includegraphics[width=5in]{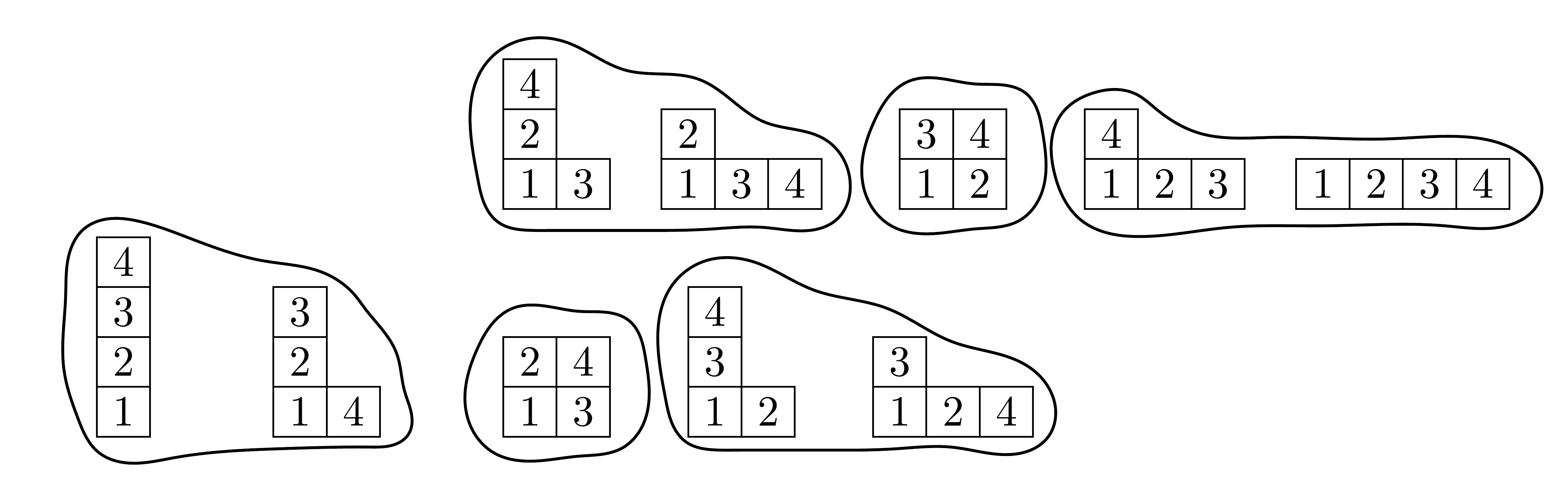}
\end{center}
Notice in this example that the atoms of level $k=3$ are contained in those of level $2$. We will see later 
that it is a property of $k$-Schur functions that the ones of level $k$ always expand positively in terms of 
$k$-Schur functions of level $k+1$.

At $k=4$ the $k$-Schur function $s^{(4)}_\lambda = s_\lambda$ for $\la \vdash 4$. Hence if we were 
to redraw the picture and circle the pieces representing the $k$-Schur functions, then 
each tableau would be in its own circle.  This observation shows that moving from
level $2$ to $3$ to $4$ is that the circles representing the $k$-Schur functions
break into smaller and smaller pieces.
\end{example}

The atoms can be computed by means of a recursive combinatorial procedure.  We will 
describe each of the steps as operations on tableaux.  The operators will act on a single tableau or 
a set of tableaux (depending on what is appropriate) and return a set of tableaux.

First we define $\sigma_i$ as the operator which takes a tableau with $a$ cells labeled with $i$ and $b$
cells labeled with $i+1$ and changes it into a tableau with $b$ cells 
labelled by $i$ and $a$ labelled with $i+1$.
This is done by considering the reading word and placing a closed parenthesis ``)'' 
under each letter in the
word labelled with an $i$ and and an open parenthesis ``('' under each letter in the 
word labelled by an $i+1$.
The word naturally has cells which have matching open and closed parentheses and the 
remainder of the parentheses are `free.'  Change the free open or closed parentheses and their corresponding 
labels so that in the end there are $b$ labels with $i$ and $a$ labels of $i+1$.  The result of this operation is the 
tableau of the same shape, but with this new reading word. This operation is also known as the 
`reflection along an $i$-string' in a crystal graph, see for example~\cite{LLT:1995}.

\begin{example}\squaresize=10pt
Consider the action of $\sigma_3$ on the semi-standard tableau
\begin{equation*}
{\squaresize=12pt\young{4&5&5\cr3&3&4&6&6\cr2&2&3&3&3&5&6\cr1&1&1&1&2&2&3&3&4\cr}}\;.
\end{equation*}

Because there are four more $3$s in the word than there
are $4$s, we should change four of the label $3$s to the label $4$.
Reading the entries in this tableau that are $3$ or $4$ only, the word is $4334333334$ which, replacing $4$ with $($ and $3$ with $)$,
corresponds to the word of parentheses $())()))))($.  If we ignore all the parentheses which match up and change
the rightmost four unmatched closed parentheses to open parentheses, this corresponds to the word $())()((((( \rightarrow
4334344444$.  Now this sequence of labels is placed back in the tableau in the positions of the labels of $3$ and $4$,
and the resulting tableau is
\begin{equation*}
{\squaresize=12pt\young{4&5&5\cr3&3&4&6&6\cr2&2&3&4&4&5&6\cr1&1&1&1&2&2&4&4&4\cr}}~.
\end{equation*}
\end{example}

The next operation is an analogue of the Pieri rule on tableaux.
Define an operation $\BB_r$ for $r \geq 1$ which adds a horizontal strip of size $r$ to the tableau 
in all possible ways and then acts by operators $\sigma_i$ to change 
the weight so that there are $r$ labels of $1$.
That is, if we take a tableau $T$ of shape $\la \vdash m$, 
$$\BB_r(T) =  \{ \sigma_1 \sigma_2 \cdots \sigma_{\ell(\mu)}(T \subseteq \mu) : \mu/\la \hbox{ is a horizontal strip of size }r \}$$
where $T \subseteq \mu$ represents the semi-standard tableau $T$ as a sequence of partitions with an additional partition $\mu$ attached.

\begin{example}\label{ex:B3action}
Consider the action of $\BB_3$ on the tableau $\scriptsize\young{3\cr2\cr1&1\cr}$.  This is equal to
\begin{align*}
\BB_3({\scriptsize\young{3\cr2\cr1&1\cr}}) = 
\sigma_1 \sigma_2 \sigma_3 \left\{ \small {\squaresize=12pt \young{4\cr3\cr2&4\cr1&1&4\cr}, \young{4\cr3\cr2\cr1&1&4&4\cr}, \young{3\cr2&4\cr1&1&4&4\cr},
\young{3\cr2\cr1&1&4&4&4\cr}}\right\}&\\
= \sigma_1 \sigma_2 \left\{ {\squaresize=12pt \young{4\cr3\cr2&3\cr1&1&3\cr}, \young{4\cr3\cr2\cr1&1&3&3\cr}, \young{3\cr2&3\cr1&1&3&4\cr},
\young{3\cr2\cr1&1&3&3&4\cr}}\right\}&\\
= \sigma_1 \left\{ {\squaresize=12pt \young{4\cr3\cr2&2\cr1&1&2\cr}, \young{4\cr3\cr2\cr1&1&2&2\cr}, \young{3\cr2&2\cr1&1&2&4\cr},
\young{3\cr2\cr1&1&2&2&4\cr}}\right\}&\\
= \left\{ {\squaresize=12pt \young{4\cr3\cr2&2\cr1&1&1\cr}, \young{4\cr3\cr2\cr1&1&1&2\cr}, \young{3\cr2&2\cr1&1&1&4\cr},
\young{3\cr2\cr1&1&1&2&4\cr}}\right\}&\;.
\end{align*}
\end{example}

The last notion that we need is that of a {\it katabolizable tableau}.  To introduce this defintion we
need the notion of jeu de taquin, Knuth equivalence or the Robinson--Schensted--Knuth algorithm.  We
introduce a generalization of the Robinson--Schensted--Knuth algorithm in Section~\ref{sec:strongweakduality}.
We assume here that the reader is familiar with this notion (and if not then one can skip ahead a few sections or 
consult~\cite{Sagan} for example).

Let $T$ be a tableau with reading word $w$.  The definition that $T$ is katabolizable with respect to a sequence of
partitions $\la^{(\ast)} = (\la^{(1)}, \la^{(2)}, \ldots, \la^{(r)})$ is recursive.  It is required that $r = \ell(\la^{(1)})$.
Then let $w = uv$, where $u$ is the largest subword of $w$ that does not contain an $r$.  Let $v'$ be 
$v$ with all letters $1$ through $r$ deleted.  We say that $T$ is {\it katabolizable} with respect to
$\la^{(\ast)}$ if $T$ contains as a subtableau the
semi-standard tableau of shape $\la^{(1)}$ and weight
$\la^{(1)}$ and if $\RSK(v'u) = (P,Q)$ where
$P$ is a tableau which is $(\la^{(2)}, \ldots, \la^{(r)})$ katabolizable with the
labels shifted by $r$.

\begin{example}
Consider the tableau
$$T = {\squaresize=12pt \young{3\cr2&2&3\cr1&1&1&4\cr}}$$
which is katabolizable with respect to the sequence of partitions $((3,2),(2,1))$.
This is because $r=\ell(\la^{(1)})=2$. Furthermore, the reading word of $T$ is $w = 32231114 = uv$,
where $u = 3$ and $v = 2231114$.
Then $v' = 34$ and the $P$ tableau of $\RSK(v'u)$ is ${\scriptsize \young{4\cr3&3\cr}}$.

Now $T$ is not katabolizable with respect to the sequence $((3),(2,2),(1))$.  The reason is that
in this case $u = 3223$ and $v = 1114$, so that $v' = 4$. Then the $P$ tableau of $\RSK(v'u)$ is 
${\scriptsize \young{4\cr3\cr2&2&3\cr}}$ and this tableau does not contain ${\scriptsize \young{3&3\cr2&2\cr}}$ and
hence it is not katabolizable with respect to the sequence $((2,2),(1))$.
\end{example}

Recall that we introduced the notion of the $k$-split of a partition in~\eqref{equation.k split}.
Now define an operator on tableaux $\KK^{\rightarrow k}$
that acts on a set of tableaux with partition weight
$\la$ such that $\KK^{\rightarrow k}$ kills
all tableaux that are not katabolizable with respect to the $k$-split of $\la$ and keeps the 
ones that are katabolizable.

\begin{example}
Consider the action of $\KK^{\rightarrow 3}$ on the following set of 
semi-standard tableaux with weight  
$(2,1,1,1,1)$.  Since $(2,1,1,1,1)^{\rightarrow 3} = ((2,1), (1,1,1))$, a tableau will only 
survive if it contains $\scriptsize\young{2\cr1&1\cr}$
as a subtableau.
\begin{equation*}
\KK^{\rightarrow k} \left\{ {\squaresize=12pt \young{3\cr2&4\cr1&1&5\cr}, \young{3\cr2&5\cr1&1&4\cr}, 
\young{4\cr2&3\cr1&1&5\cr},
 \young{4\cr2&5\cr1&1&3\cr}, \young{5\cr2&3\cr1&1&4\cr}, \young{5\cr2&4\cr1&1&3\cr}} \right\} 
= \left\{ {\squaresize=12pt \young{3\cr2&5\cr1&1&4\cr}} \right\}~.
\end{equation*}
\end{example}

Given a $k$-bounded partition $\la$, we define the $k$-atom $\AAA^{(k)}_\la$ as a set of
tableaux which are computed recursively as
\begin{equation}
\AAA^{(k)}_\la = \KK^{\rightarrow k} \BB_{\la_1} \AAA^{(k)}_{(\la_2, \la_3, \ldots, \la_{\ell(\la)})}~.
\end{equation}

\begin{example}\label{ex:atom4}
We can compute the atom $\AAA^{(3)}_{11} = \Big\{ \squaresize=10pt \scriptsize \young{2\cr1\cr} \Big\}$ and 
$\BB_2( \AAA^{(3)}_{11} ) = \Big\{\squaresize=10pt \scriptsize \young{3\cr2\cr1&1\cr}, \young{2\cr1&1&3\cr} \Big\}$.
The atom $\AAA^{(3)}_{211} =  \Big\{\scriptsize\squaresize=10pt \young{3\cr2\cr1&1\cr}, \young{2\cr1&1&3\cr} \Big\}$ 
because each of these tableaux survives the operator $\KK^{\rightarrow 3}$.  As we have seen in 
Example~\ref{ex:B3action}, we
know what the action of $\BB_3$ on the first tableau is, and there are six additional tableaux when
$\BB_3$ acts on $\scriptsize\young{2\cr1&1&3\cr}$ so that
\squaresize=12pt
$$\BB_3( \AAA^{(3)}_{211} )  = \Big\{ \young{4\cr3\cr2&2\cr1&1&1\cr}, \young{4\cr3\cr2\cr1&1&1&2\cr}, \young{3\cr2&2\cr1&1&1&4\cr},
\young{3\cr2\cr1&1&1&2&4\cr},\young{3\cr2&2&4\cr1&1&1\cr},
$$
\begin{equation}\label{eq:3atomexamp}
\young{3\cr2&4\cr1&1&1&2\cr},
\young{4\cr2\cr1&1&1&2&3\cr},
\young{2&2&4\cr1&1&1&3\cr},
\young{2&4\cr1&1&1&2&3\cr},
\young{2\cr1&1&1&2&3&4\cr}
\Big\}
\end{equation}
It is an unusual situation, but all $10$ of these tableaux are katabolizable with respect to
the sequence $(3,2,1,1)^{\rightarrow 3} = ((3),(2,1),(1))$ and hence survive the operator $\KK^{\rightarrow 3}$.

If on the other hand, we do the computation with $k=4$, we find
\begin{equation} \label{equation.A example}
	\AAA^{(4)}_{221} = \Big\{\squaresize=12pt \young{3\cr2\cr1&1\cr} \Big\}\quad \text{and} \quad
	\AAA^{(4)}_{3211}  = \Big\{ \young{4\cr3\cr2&2\cr1&1&1\cr}, \young{3\cr2&2\cr1&1&1&4\cr} \Big\}\;.
\end{equation}
We can check~\eqref{equation.A example} in \Sage via
\begin{sageexample}
    sage: la = Partition([3,2,1,1])
    sage: la.k_atom(4)
    [[[1, 1, 1], [2, 2], [3], [4]], [[1, 1, 1, 4], [2, 2], [3]]]
\end{sageexample}
\end{example}

Now we define the symmetric function $A^{(k)}_\la[X;t]$ in terms of the set of tableaux $\AAA^{(k)}_\la$ as
\begin{equation}
A^{(k)}_\la[X;t] = \sum_{T \in \AAA^{(k)}_\la} t^{\charge(T)} s_{\shape(T)}~.
\end{equation}

\begin{example}\label{ex:atom4sf}
The charge of the tableau of shape $(3,2,1,1)$ in~\eqref{equation.A example} is $0$, 
whereas the charge of the tableau of shape $(4,2,1)$
is 1. Hence we obtain
\begin{equation*}
A_{3211}^{(4)}[X;t] = s_{3211} + t s_{421}~.
\end{equation*}

The element $A^{(3)}_{3211}[X;t]$ is a generating function
for the tableaux in \eqref{eq:3atomexamp} with a term of the form $t$ raised to the charge
times the Schur function indexed by the shape.  By computing the charge of each of
these tableaux we determine
\begin{equation*}
A^{(3)}_{3211}[X;t] = s_{3211} + t s_{4111} + (t +t^2)s_{421} + t s_{331} + (t^2 +t^3)s_{511} +
t^2 s_{43} + t^3 s_{52} + t^4 s_{61}~.
\end{equation*}
\end{example}

The atoms by this definition are conjectured to be a basis of the
space
\begin{align}
\La^t_{(k)} &= \LL_{\QQ(q,t)}\{ H_\la[X;q,t] : \la_1 \leq k \}\nonumber\\
&= \LL_{\QQ(q,t)}\{ s_\la\left[\frac{X}{1-t}\right] : \la_1 \leq k \}\label{Laktspacedef}\\
&= \LL_{\QQ(q,t)}\{ Q'_\la[X;t] : \la_1 \leq k \}~.\nonumber
\end{align}
However, no proof that these functions are even elements of this space currently exists.

\end{subsection}

\begin{subsection}{A symmetric function operator definition} \label{sec:operator}
In this section we present a conjecturally equivalent definition
for $k$-Schur functions that is in similar spirit to the
atom description given in the previous section, but now with an
algebraic flavor.

In Section~\ref{sec:HLsymfunc} we defined the operator $\Bop_m$ which has the property that 
$\Bop_m( Q'_\la[X;t])
= Q'_{(m,\la)}[X;t]$.  We now define a new operator in terms of the $\Bop_m$ as
\begin{equation}\label{def:Bopla}
	\Bop_\la := \prod_{1 \leq i < j \leq \ell(\la)} (1 - t R_{ij}) \Bop_{\la_1} \Bop_{\la_2} \cdots \Bop_{\la_{\ell(\la)}}\;,
\end{equation}
where 
\begin{equation*}
	R_{ij}( \Bop_{\mu_1} \Bop_{\mu_2} \cdots \Bop_{\mu_{\ell(\mu)}}) 
	=  \Bop_{\mu_1} \Bop_{\mu_2} \cdots \Bop_{\mu_i + 1} 
	\cdots \Bop_{\mu_j-1} \cdots \Bop_{\mu_{\ell(\mu)}}~.
\end{equation*}  
Alternatively $\Bop_\la$ is also given by the equation
\begin{equation}\label{eq:opersum}
	\Bop_\la = \sum_{\nu,\mu} c^\nu_{\mu\la} s_\nu[X] s_\mu[X(t-1)]^\perp\;,
\end{equation}
where the sum is over partitions $\nu$ and $\mu$ such that $\ell(\nu) \leq \ell(\la)$ and $\ell(\mu) \leq \ell(\la)$
and $c^\nu_{\mu\la} = \left< s_\mu s_\la, s_\nu \right>$ is the Littlewood--Richardson coefficient
from~\eqref{equation.LR coefficients}.

The `parabolic' Hall--Littlewood symmetric functions that we will need are defined in terms of these
operators.  For a sequence of partitions $\la^{(\ast)} = (\la^{(1)}, \la^{(2)}, \ldots, \la^{(d)})$ define
the symmetric function
\begin{equation*}
H_{\la^{(\ast)}}[X;t] = \Bop_{\la^{(1)}} \Bop_{\la^{(2)}} \cdots \Bop_{\la^{(d)}} (1)\;.
\end{equation*}

It is easy to see from~\eqref{eq:opersum} that when $t=1$, we have 
$H_{\la^{(\ast)}}[X;1] = s_{\la^{(1)}} s_{\la^{(2)}} \cdots s_{\la^{(d)}}$.
The index set of these symmetric functions is much larger than the set of partitions, so that 
these elements are not linearly independent.  In fact, these symmetric function interpolate between 
the Hall--Littlewood symmetric functions and the Schur functions since
$H_{(\la)}[X;t] = s_{\la}$, that is, the element indexed by a list containing exactly one partition,
and $H_{((\la_1),(\la_2), \ldots, (\la_{\ell(\la)}))}[X;t] = Q'_{\la}[X;t]$, that is, the element indexed by 
a list where each partition is a single part of $\la$.

There is a conjectured combinatorial interpretation for the expansion of $H_{\la^{(\ast)}}[X;t]$ 
(see~\cite{ShimozonoWeyman}) if $\la^{(\ast)}$ is a sequence of partitions such that the concatenation 
of the partitions in $\la^{(\ast)}$ is a partition (i.e. if for each of the adjacent partitions in $\la^{(\ast)}$ we 
have $\la^{(i)}_{\ell(\la^{(i)})} \geq \la^{(i+1)}_1$). In this case
\begin{equation} \label{conj:parabolicformula}
	H_{\la^{(\ast)}}[X;t] = \sum_{T} t^{\charge(T)} s_{\shape(T)}\;,
\end{equation}
where the sum is over all tableaux $T$ which are katabolizable with respect to $\la^{(\ast)}$.

The first algebraic definition of $k$-Schur functions requires an
intermediary basis; the $k$-split basis for $\La^t_{(k)}$.
This basis is made up of, for each $k$-bounded partition $\la$, 
$G^{(k)}_\la[X;t] =  H_{\la^{\rightarrow k}}[X;t]$.

\begin{sagedemo} \label{ex:3split}
The $3$-split basis element $G_{3211}^{(3)}$ is a $t$-analogue of the product of $s_{3} s_{21} s_{1}$.  
The operators $\Bop_\lambda$ are programmed in \Sage and we may calculate
$\Bop_{(3)}\Bop_{(2,1)} \Bop_{(1)}(1)$ in the following steps:
\begin{sageexample}
    sage: s = SymmetricFunctions(QQ["t"]).schur()
    sage: G1 = s[1]
    sage: G211 = G1.hl_creation_operator([2,1]); G211
    s[2, 1, 1] + t*s[2, 2] + t*s[3, 1]
    sage: G3211 = G211.hl_creation_operator([3]); G3211
    s[3, 2, 1, 1] + t*s[3, 2, 2] + t*s[3, 3, 1] + t*s[4, 1, 1, 1] 
     + (2*t^2+t)*s[4, 2, 1] + t^2*s[4, 3] + (t^3+t^2)*s[5, 1, 1] 
     + 2*t^3*s[5, 2] + t^4*s[6, 1]
\end{sageexample}
\squaresize=5pt
This calculation shows that
$\Bop_{(2,1)}(s_{1}) = s_{\young{\cr\cr&\cr}} + 
t s_{\young{&\cr&\cr}}+ t s_{\young{\cr&&\cr}}$ and
\begin{align*}\Bop_3( s_{\young{\cr\cr&\cr}} + 
t s_{\young{&\cr&\cr}}+ t s_{\young{\cr&&\cr}} ) &= s_{\young{\cr\cr&\cr&&\cr}} + t s_{\young{&\cr&\cr&&\cr}}
+ t s_{\young{\cr&&\cr&&\cr}} + t s_{\young{\cr\cr\cr&&&\cr}} + (t+2t^2) s_{\young{\cr&\cr&&&\cr}} + t^2 s_{\young{&&\cr&&&\cr}}
 \\&\hskip.4in
+ (t^2 + t^3) s_{\young{\cr\cr&&&&\cr}}
+ 2 t^3 s_{\young{&\cr&&&&\cr}} + t^4 s_{\young{\cr&&&&&\cr}}
\end{align*}
and this expression is equal to $G_{3211}^{(3)}$.
\end{sagedemo}

Then, to arrive at the $k$-Schur functions, a second operator is needed.
Let $T^{(k)}_i$ be an operator on symmetric functions defined so that
\begin{equation}\label{eq:killingoperators}
	T^{(k)}_i( G^{(k)}_\la[X;t] ) = \begin{cases} G^{(k)}_\la[X;t]& \hbox{ if } \la_1 = i,\\
	0&\hbox{ otherwise.}
	\end{cases}
\end{equation}

For $k$-bounded partition $\la$, the elements $\At_\la[X;t]$ 
were defined in \cite{LM:2003} by a recursive algorithm;
if $\ell(\la) = 1$ and $r \leq k$, then $\At_{(r)}[X;t] = s_{(r)}$.
Otherwise we set for $\la_1 \leq m \leq k$,
\begin{equation}\label{def:algatom}
\At^{(k)}_{(m, \la_1, \la_2, \ldots, \la_{\ell(\la)})}[X;t] = T^{(k)}_{m} \Bop_{m} \At_{\la^{(k)}}[X;t]~.
\end{equation}

\begin{example} 
The $3$-split of $(2,1,1)$ is $(2,1,1)^{\rightarrow 3} = ((2,1),(1))$, hence
$$
	\At_{211}^{(3)} = T^{(3)}_2 \Bop_2( s_{11} ) = T^{(3)}_2 ( s_{211} + t s_{31} ) 
	=  T^{(3)}_2 ( G_{211}^{(3)} - G_{22}^{(3)} ) = s_{211} + t s_{31}.
$$ 
Moreover, we may calculate using Example~\ref{ex:3split} that 
\begin{align*}
	\Bop_3( \At_{211}^{(3)} ) &= s_{3211} + t s_{331} + t s_{4111} + (t+t^2) s_{421} 
	+ t^2 s_{43} + (t^2+t^3) s_{51} + t^3 s_{52} + t^4 s_{61}\\
	&= G_{3211}^{(3)} - t s_{322} - t^2 s_{421} - t^3 s_{52}~.
\end{align*}
We can also determine that $G_{322}^{(3)} = s_{322} + t s_{421} + t^2 s_{52}$ (using the same 
techniques as in Example~\ref{ex:3split}), so that
\begin{equation*}
	\At_{3211}^{(3)} = T^{(3)}_3 \Bop_3( \At_{211}^{(3)} ) = G_{3211}^{(3)} - t G_{322}^{(3)}\;.
\end{equation*}
This (not coincidently) is equal to $A_{3211}^{(3)}[X;t]$ as was calculated in Example~\ref{ex:atom4sf}.
\end{example}

The algebraic operations mimic those of the combinatorial definition defined in the previous section
and so it is important to emphasize that the combinatorial and algebraic definitions are equivalent.
\begin{conj} \label{conj:Atomequiv} 
For $k>0$ and a $k$-bounded partition $\la$,
\begin{equation}
A^{(k)}_\la[X;t] = \At^{(k)}_\la[X;t]~.
\end{equation}
\end{conj}
\end{subsection}

\begin{subsection}{Weak tableaux II} \label{sec:weak II}

We have seen that a fruitful characterization for $k$-Schur functions 
(without parameter $t$) is given by inverting \eqref{weakKostka}; for $\mu_1\leq k$,
\begin{equation}
\label{kschurdef}
h_\mu=\sum_{\lambda: \lambda_1\leq k} K_{\lambda\mu}^{(k)}\,s_\lambda^{(k)}\,,
\end{equation}
where the weak Kostka numbers $K_{\lambda\mu}^{(k)}$ count weak $k$-tableaux.
Here we present $k$-Schur functions that reduce to these parameterless
$k$-Schur functions when $t=1$.  The method is to 
introduce {\it weak Kostka-Foulkes polynomials} as 
polynomials in $\mathbb N[t]$ defined by 
refining  the charge statistic to a statistic 
that associates a non-negative integer called the $k$-charge
to each $k$-tableau.  Then setting
$$
K_{\lambda\mu}^{(k)}(t)= \sum_{{\shape(T)=\mfc(\lambda)\atop\text{weight}(T)=\mu}}
t^{\text{kcharge}(T)}\,,
$$
it happens that $K^{(k)}_{\lambda\lambda}(t)=1$ and since there are no  $k$-tableaux of shape $\mfc(\lambda)$
and weight $\mu$ when $\mu>\lambda$, the $k$-charge matrix $K^{(k)}_{\lambda\mu}(t)$ is 
unitriangular.  So, in the spirit of \eqref{kschurdef},
\begin{equation}
\label{tinvdef}
Q'_\mu[X;t] =
\sum_{\lambda} K_{\lambda\mu}^{(k)}(t)\, \tilde s^{(k)}_\lambda[X;t]
\end{equation}
characterizes the functions $\{\tilde s^{(k)}_\lambda[X;t]\}$.


There are several different characterizations for $k$-charge. 
We give here two distinct formulations defined directly on 
$k$-tableaux, discovered by Lapointe-Pinto and Morse
\cite{LapointePinto,DalalMorse:2013}.
There are other formulations including one on 
$\alpha$-factorizations, one on an object called affine Bruhat 
countertableau \cite{DalalMorse:2012,DalalMorse:2013},
and in relation with the energy function on Kirillov--Reshetikhin 
crystals~\cite{MS:2013}.

The $k$-charge statistic on $k$-tableaux is first described
in the standard case since it is in these terms 
that we define it for semi-standard $k$-tableaux.
Important to the definition is a number $\text{diag}(c_1,c_2)$,
associated to cells $c_1$ and $c_2$ in a $(k+1)$-core,
defined to be the number of diagonals of residue 
$x$ that are strictly between $c_1$ anc $c_2$ where $x$ is the residue 
of the lower cell.  When it is well-defined to do so, functions defined
with a cell as input can instead take a letter as input.
For example, in a standard $k$-tableau it is natural to discuss 
the residue of a specific letter (since any cell containing that 
letter has the same residue) instead of the residue of a specific cell.

\begin{definition}
Given a standard $k$-tableau $T$ on $m$ letters, put a bar on  
the topmost occurrence of letter $r$, for each $r=1,\ldots,m$.
Define the {\it index} of $T$, starting from $I_1=0$, by
\begin{equation}
I_r = 
\begin{cases}
\label{meindex}
I_{r-1} +1+\text{diag}({\bar r},{\overline{r-1}}) & \text{if $\bar r$ is east of $\overline{r-1}$ }\\
I_{r-1}-\text{diag}({\bar r},{\overline{r-1}}) & \text{otherwise}\,,
\end{cases}
\end{equation}
for $r=2,\ldots,m$.
The {\it $k$-charge} of $T$ is the sum of entries in $I(T)$, denoted by $\mathrm{kcharge}(T)$.
\end{definition}

\begin{example}
\label{exme}
For $k=3$,
$$
T=\tableau[scY]{4_2\cr 2_3&6_0\cr 1_0&3_1&4_2&5_3&6_0}\implies
I(T) =[0,0,1,1,3,3]\;\;
\implies \mathrm{kcharge}(T)=8
\,.
$$
\end{example}

It is not immediately clear that the $k$-charge is a non-negative integer
and it is sometimes helpful to use a different formulation of $k$-charge.
Let $T_{\leq x}$ denote the subtableau obtained by deleting all
letters larger than $x$ from $T$.
\begin{definition}
Given a $k$-tableau $T$, the $T$-residue order of
$\{0,\ldots,k\}$ is defined by
$$
x>x-1>\cdots>0>k>\cdots > x+1\,,
$$
where $x$ is the residue of the highest addable
corner of $T$.  Note that $x=1-\ell(\lambda)\pmod{k+1}$,
for $\lambda$ the shape of $T$.
\end{definition}
\begin{example}
With $k=3$, consider
$$
T= \tableau[scY]{4_3\cr 1_0&2_1&3_2&4_3}\qquad
T_{\leq 3}= \tableau[scY]{1_0&2_1&3_2}\,.
$$
The $T_{\leq 3}$-residue order is $3>2>1>0$ and
the $T$-residue order is $2>1>0>3$.
\end{example}

Given a standard $k$-tableau $T$ on $m$ letters,
define the {\it index} $J(T)=[J_1,\ldots,J_m]$,
starting from $J_1=0$, by setting for $r=2,\ldots,m$,
\begin{equation}
\label{indexcondition}
J_r = 
\begin{cases}
J_{r-1} +1 & \text{if $\res(r)>\res(r-1)$}\\
J_{r-1} & \text{otherwise}\,,
\end{cases}
\end{equation}
under $T_{\leq r}$-residue order (see Example~\ref{exkindex}).

\begin{prop}
For a standard $k$-tableau $T$ of shape $\lambda$,
$$
\mathrm{kcharge}(T)=\sum_r \left( J_r(T) + \mathrm{diag}(c_r,c^{(r)})\right)\,,
$$
where $c_r$ is the highest cell containing an $r$
and $c^{(r)}=(\ell(\shape(T_{\le r}))+1,1)$.
\end{prop}

\begin{example}
\label{exkindex}
For $k=3$,
$$
T=\tableau[scY]{4_2\cr 2_3&6_0\cr 1_0&3_1&4_2&5_3&6_0}\implies
J(T) =[0,0,1,1,2,3]\,,
\mathrm{diag}(c_5,(4,1))=1
\implies \mathrm{kcharge}(T)=8
\,.
$$
\end{example}

\begin{remark}
Given a standard $k$-tableau of shape $\lambda$
where $k\geq h(\lambda)$,  the index conditions
\eqref{meindex} and \eqref{indexcondition} both
reduce to \eqref{convindex}.
Thus, since a diagonal of residue $x$ occurs at
most once in $\lambda$ for any $x$, 
$k$-charge reduces to charge. 
\end{remark}

As with the charge of Lascoux and Sch\"utzenberger, we extend the definition of
$k$-charge to semi-standard $k$-tableaux by
successively computing on an appropriate choice of
standard sequences.  The trick is to introduce a
method for making this choice in $k$-tableaux.

\begin{definition}
\label{wierdchoice}
From an $x$ (of some residue $i$)
in a semi-standard $k$-tableau $T$,
the appropriate choice of $x+1$ will be
determined by choosing its residue from the set $A$ of
all $(k+1)$-residues labelling $x+1$'s.  Reading counter-clockwise from $i$,
this choice is the closest $j\in A$ on a circle labelled clockwise
with $0,1,\ldots,k$.
\end{definition}

\begin{remark}
Definition~\ref{wierdchoice} reduces to
Definition~\ref{convchoice} when $T$ is a
semi-standard tableau of shape $\lambda$ and
$k\geq h(\lambda)$.  In particular, consider
$x$ of $(k+1)$-residue $i$ in $T$.
Note that $h(\lambda)\leq k$ implies there is
a unique cell $c$ of residue $i$ that contains $x$.
Let $j$ be the first entry on the circle
reading counter-clockwise from $i$ that is a residue
of a cell containing $x+1$.
If there is an $x+1$ above $c$, then
the south-easternmost cell containing
an $x+1$ that is above $c$ has residue $j$
since there are no $x+1$'s of a residue counter-clockwise
between $i$ and $j$.  If there are none above $c$,
then for the same reason, the south-easternmost cell
containing an $x+1$ has residue $j$.
\end{remark}

In the semistandard case, we need to be more specific about the 
residue order used in~\eqref{indexcondition}. The letters $r$ that 
occur in the tableau are ordered with respect to the standard subsequence
they belong to under Definition~\ref{wierdchoice}. So $r$ in the first chosen standard
subword is bigger than the one from the second standard subword etc. Each letter $r$ has
its own distinct residue $i$. The $J$-index of~\eqref{indexcondition} should be
computed with respect to the $T_{\le r_i}$-residue order when dealing with letter $r_i$.

\begin{example}
With $k=4$ and weight $(2,2,2,2,2,2,1)$:
$$
\tiny
\tableau[mbY]
{\tf 7_0\cr \tf 6_1\cr \tf 5_2&6_3\cr\tf 3_3&4_4&\tf 7_0\cr
\tf 2_4&3_0&5_1&\tf 5_2&6_3\cr
1_0&\tf 1_1&2_2&\tf 3_3&4_4&\tf 4_0&5_1&\tf 5_2&6_3}
\qquad
\qquad
\qquad
\tableau[mbY]
{\cr \cr &\tf 6_3\cr&\tf 4_4&\cr
&\tf 3_0&\tf 5_1&&\tf 6_3\cr
\tf 1_0&&\tf 2_2&&\tf 4_4&&\tf 5_1&&\tf 6_3}
\qquad \qquad
\qquad \qquad
\qquad
$$
Note that for example
$$
T_{\le 6_3} = 
\tableau[scY]
{5_2&6_3\cr 3_3&4_4\cr
 2_4&3_0&5_1& 5_2&6_3\cr
1_0& 1_1&2_2& 3_3&4_4& 4_0&5_1& 5_2&6_3}
$$
so that the highest addable cell in $T_{\le 6_3}$ has residue 1.
We have
$$
J=[0,0,0,1,1,1,1]\;\text{and}\; \text{diag}(c_4,c^{(4)})= 1\qquad
\qquad
J=[0,1,1,1,2,2]\;\quad 
\qquad
\qquad
$$
$$
I=[0,0,0,2,1,1,1]\;\text{since $\text{diag}(\bar 3,\bar 4)=1$,
$\text{diag}(\bar 4,\bar 5)=1$}
\qquad
I=[0,1,1,1,2,2]
\qquad
$$
Using either index $I$ or $J$, we find that the $k$-charge
of $T$ is 12.
\end{example}

\begin{sagedemo}
We verify the above example in \Sage:
\begin{sageexample}
  sage: T = WeakTableau([[1,1,2,3,4,4,5,5,6],[2,3,5,5,6],[3,4,7],
  ....:  [5,6],[6],[7]],4)
  sage: T.k_charge()
  12
\end{sageexample}
We can also demonstrate an example of Equation \eqref{tinvdef}
using \Sage.
\begin{sageexample}
  sage: Sym = SymmetricFunctions(QQ["t"].fraction_field())
  sage: Qp = Sym.hall_littlewood().Qp()
  sage: ks = Sym.kBoundedSubspace(3).kschur()
  sage: t = ks.base_ring().gen()
  sage: ks(Qp[3,2,2,1])
  ks3[3, 2, 2, 1] + t*ks3[3, 3, 1, 1] + t^2*ks3[3, 3, 2]
  sage: sum(t^T.k_charge()*ks(la) for la in Partitions(8, max_part=3)
  ....:  for T in WeakTableaux(3,la,[3,2,2,1],representation = 'bounded'))
  ks3[3, 2, 2, 1] + t*ks3[3, 3, 1, 1] + t^2*ks3[3, 3, 2]
\end{sageexample}
\end{sagedemo}

\end{subsection}

\begin{subsection}{Strong tableaux II } \label{sec:strong}
The definition of the $k$-Schur functions in terms of strong marked tableaux as in~\eqref{eq:kschurmexp}
also has a version with a parameter $t$.  For this definition we need to define the spin of a strong marked 
ribbon.  Recall that if $\tau$ and $\kappa$ are $(k+1)$-cores such that $\tau \Rightarrow_k \kappa$, 
then $\kappa/\tau$ is a skew partition which consists of several copies of connected components which 
are ribbons of the same size and shape.  Let $h$ represent the {\it height} of one of these ribbons 
(that is that it occupies $h$ rows).  Now a strong marked cover consists of the skew partition $\kappa/\tau$
and a marking $c$ of one of the connected components.  If there are $r$ connected components
in $\kappa/\tau$, then the {\it spin} of a marked cover is equal to $(h-1) \times r$
plus the number of ribbons which are above the marked one.

The {\it spin of a strong marked tableau},
$\kappa^{(0)} = \emptyset \Rightarrow_{k} \kappa^{(1)} \Rightarrow_{k} \kappa^{(2)} 
\Rightarrow_{k}\cdots \Rightarrow_{k} \kappa^{(m)}$ 
with markings $c_1, c_2, \ldots , c_m$ is the sum of the spins of the strong 
marked ribbons $\kappa^{(i)}/\kappa^{(i-1)}$ with marking $c_i$.

\begin{example}
Recall from Example~\ref{ex:strongtab} the marked 
semi-standard tableau with $k = 3$:
\squaresize=14pt
$$
	\young{3_1^\ast&4_1^\ast\cr2_2&3_3\cr2_1^\ast&3_2&3_2^\ast&4_1\cr1_1^\ast&
	1_2^\ast&2_2^\ast&3_3^\ast\cr}\;.
$$
There is a contribution of $1$ to the spin for the ribbon of cells labelled by $2_2$ and
there is a contribution of $1$ due to the labeling of the lower occurrence of $3_3$.  
Therefore the total spin of this tableau is $2$.
\end{example}

The $k$-Schur function (this time with a $t$) in terms of strong tableaux 
are defined as (see also~\cite[Conjecture 9.11]{LLMS:2006})
\begin{equation} \label{eq:kschurt}
	s^{(k)}_\la[X;t] = \sum_{\mu \vdash |\la|} \sum_{(\kappa^{(\ast)}, c_\ast)} 
	t^{\spin(\kappa^{(\ast)}, c_\ast)} m_\mu\;,
\end{equation}
where the sum is over all strong marked tableaux $(\kappa^{(\ast)}, c_\ast)$ of shape $\la$ and
weight $\mu$.

\begin{example} 
For this definition, there is a reason to choose a smaller example for computation.  We have thus 
far used as our running example the $3$-Schur function indexed by the partition $(3,2,1,1)$.  
For the coefficient of the monomial $m_{1111111}$ there are $210$
strong marked tableaux. To choose a smaller example we take
as an example the $3$-bounded partition $(3,1,1)$ which corresponds to the $4$-core $(4,1,1)$ 
which has $10$ strong marked tableaux in total.  The strong partial order on
$4$-cores contains the following interval (see Figure~\ref{strongkeq3Hasse}):
\begin{center} 
\includegraphics[width=3in]{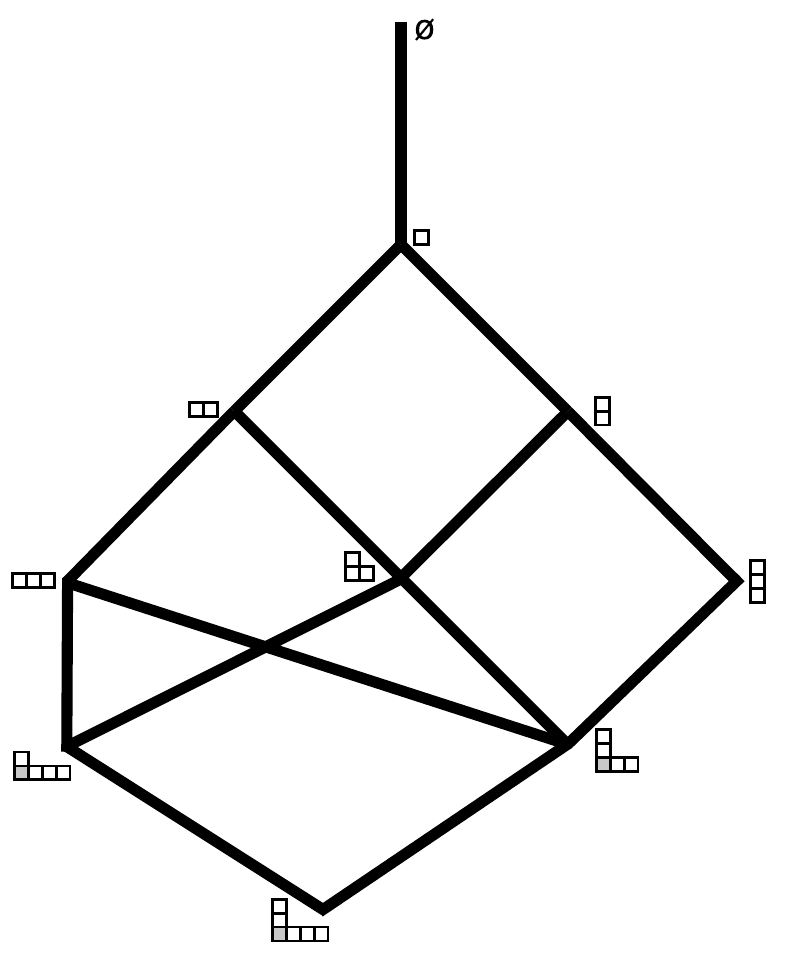}
\end{center}
These correspond to the following 10 marked strong standard tableaux:
\squaresize=12pt
$$\young{5^\ast\cr4\cr1^\ast&2^\ast&3^\ast&4^\ast\cr},\quad
  \young{4\cr4^\ast\cr1^\ast&2^\ast&3^\ast&5^\ast\cr},\quad
  \young{4\cr3^\ast\cr1^\ast&2^\ast&4^\ast&5^\ast\cr},\quad
  \young{4\cr2^\ast\cr1^\ast&3^\ast&4^\ast&5^\ast\cr}, \quad
  \young{5^\ast\cr3^\ast\cr1^\ast&2^\ast&4&4^\ast\cr}, 
  $$
$$\young{4^\ast\cr3^\ast\cr1^\ast&2^\ast&4&5^\ast\cr}, \quad
  \young{4^\ast\cr2^\ast\cr1^\ast&3^\ast&4&5^\ast\cr}, \quad
  \young{5^\ast\cr2^\ast\cr1^\ast&3^\ast&4&4^\ast\cr},\quad
  \young{5^\ast\cr4^\ast\cr1^\ast&2^\ast&3^\ast&4\cr},\quad
  \young{3^\ast\cr2^\ast\cr1^\ast&4&4^\ast&5^\ast\cr}\;.$$
The first four of these strong marked standard tableaux have spin equal to $1$ and the
remaining $6$ have spin equal to $0$.  There is a semi-standard 
tableau of weight $\mu$ which 
corresponds to the standard tableau if the cells labeled $1,2, \ldots, \mu_1$ form a strong marked
horizontal strip, $\mu_1+1, \mu_1+2, \ldots, \mu_1+\mu_2$ form another strong marked horizontal
strip, etc.. From these 10 strong marked standard tableaux it is possible to read off that
$$s_{311}^{(3)}[X;t] = t m_{41} + t m_{32} + (1+2t) m_{311} + (1+2t) m_{221} + (3+3t) m_{2111} 
+ (6+4t) m_{111111}~.$$
\end{example}

\begin{sagedemo}
In this example we will show how \Sage can be used to compute
the monomial expansion of the $k$-Schur function by computing a statistic for each strong $k$-tableau.
\begin{sageexample}
    sage: t = var("t")
    sage: for mu in Partitions(5):
    ....:     print mu, sum(t^T.spin() for T in StrongTableaux(3,[4,1,1],mu))
    [5] 0
    [4, 1] t
    [3, 2] t
    [3, 1, 1] 2*t + 1
    [2, 2, 1] 2*t + 1
    [2, 1, 1, 1] 3*t + 3
    [1, 1, 1, 1, 1] 4*t + 6
    sage: StrongTableaux( 3, [4,1,1], (1,)*5 ).cardinality()
    10
    sage: StrongTableaux( 3, [4,1,1], (1,)*5 ).list()
    [[[-1, -2, -3, 4], [-4], [-5]],
     [[-1, -2, -3, -4], [4], [-5]],
     [[-1, -2, -3, -5], [-4], [4]],
     [[-1, -2, 4, -4], [-3], [-5]],
     [[-1, -2, 4, -5], [-3], [-4]],
     [[-1, -2, -4, -5], [-3], [4]],
     [[-1, -3, 4, -4], [-2], [-5]],
     [[-1, -3, 4, -5], [-2], [-4]],
     [[-1, -3, -4, -5], [-2], [4]],
     [[-1, 4, -4, -5], [-2], [-3]]]
\end{sageexample}
\end{sagedemo}

The definitions $A^{(k)}_\la[X;t], \At^{(k)}_\la[X;t],
s^{(k)}_\la[X;t],$ and $\tilde s^{(k)}_\la[X;t]$ of what 
are generically known as $k$-Schur functions have very different 
characters.   Each connects to a different area of
algebraic combinatorics and has its own benefits and detriments.
The conjectured properties of $k$-Schur functions are sometimes clear 
from one definition, but difficult to prove for another.
In the next section we will examine these properties and
see why it would be very beneficial to resolve the following conjecture.

\begin{conj} \label{conj:allequiv} 
For $k>0$ and a $k$-bounded partition $\la$,
\begin{equation*}
s^{(k)}_\la[X;t] = A^{(k)}_\la[X;t] = \At^{(k)}_\la[X;t]
= \tilde s^{(k)}_\la[X;t]\,.
\end{equation*}
\end{conj}
\end{subsection}

\begin{subsection}{Notes on references}
The atom definition of the $k$-Schur functions was the first description
and they were referred to by Alain Lascoux as potatoes (`les patates')
since when copies of the atoms were identified on photocopies of
the cyclage poset they could be circled to form tuber-like shapes (which
we have endeavored to recreate in Example~\ref{ex:potatos}).
The combinatorial definition uses the concept of 
katabolizable tableaux~\cite{Shimozono:2001, Shimozono2:2001} which
generalizes the notion of cyclage which comes from~\cite{LS:1981}.
The non-recursive definition of the charge statistic
on tableaux also comes from~\cite{LS:1981}.

It turns out that the atom definition is still roughly the easiest and fastest 
definition to implement on a computer.  The implementation
of the $k$-Schur functions in \Sage uses the definition of
$A^{(k)}_\la[X;t]$.  For this reason, it is important to note that while 
Conjecture~\ref{conj:Atomequiv} has been checked up to degree $m=19$,
Conjecture~\ref{conj:allequiv} has only been checked up to bounded partitions
of $m=11$, but for $m>11$ there is currently no proof that these
definitions are equivalent in general, even for $t=1$.

The operators $\Bop_\la$ from Equation~\eqref{def:Bopla} were defined by
Shimozono and Zabrocki in~\cite{ShimozonoZabrocki} as a tool for understanding
the generalized Kostka polynomials (also known as parabolic Kostka coefficients)
that were studied by Shimozono and Weyman \cite{ShimozonoWeyman}.
The combinatorial interpretation for the Schur expansion of a composition of
these operators when the indexing partitions concatenate to a partition
in terms of katabolizable tableaux is still an open problem.  Certain
cases of this are known (e.g.~\cite{SchillingWarnaar:1999}) and its resolution
would be helpful in attacking Conjecture~\ref{conj:Atomequiv}.
In~\cite{BravoLapointe:2009} the Bernstein operators at $t=1$ are used to derive
a recursion relation for the $k$-Schur functions which allow an easy expansion
in the complete homogeneous basis.

In References~\cite{LM:2007,LM:2008} the definition of the $k$-Schur functions
at $t=1$ is taken to be the symmetric functions which satisfy the $k$-Pieri rule
of Equation~\eqref{eq:kweakpieri}.  Later Lam, Lapointe, Morse and 
Shimozono~\cite{LLMS:2006} showed that the $k$-Schur functions which satisfy the
$k$-Pieri rule~\eqref{eq:kweakpieri} are equivalent
to Equation~\eqref{eq:kschurt} at $t=1$, and Equation~\eqref{eq:kschurt} 
is~\cite[Conjecture 9.11]{LLMS:2006}.  Assaf and Billey \cite[Definition 3.2]{AssafBilley} have
a slightly different statement of Equation~\eqref{eq:kschurt} 
as a quasi-symmetric function expansion in the fundamental basis.

Recently, Dalal and Morse~\cite{DalalMorse:2012} have proven a second characterization of
the $k$-Pieri at $t=1$ which leads to a very different sort of tableaux that enumerate $K^{(k)}_{\la\mu}$.
They use these tableaux to provide an alternative combinatorial
interpretation to Equation~\eqref{eq:HLcomb}. 

We note that the line between what is called a `conjectured property' and what is
called a `definition' is sometimes a little blurry because we have provided 
several definitions of $k$-Schur functions which are conjectured to be equivalent.
There are reasons to do this instead of taking one as definition and the
rest as conjectured formulas;  historically the $k$-Schur functions were 
presented in the literature this way.   However, we note that there
may be more definitions of the $k$-Schur functions than presented in this section.
For instance, the $k$-shape poset presented in Section~\ref{sec:kshapeplus} and the 
representation theoretical definition which is briefly discussed in Section~\ref{sec:reptheory} 
are two other properties which may be taken as definitions, but are currently only conjectured 
to be equivalent to the definitions presented here.
\end{subsection}
\end{section}

\begin{section}{Properties of $k$-Schur functions and their duals}
\label{section.properties}
In this section we list many of the properties of the $k$-Schur functions,
both conjectured and proven. This section is mainly meant to
be a statement of the `current state of affairs' and is likely to change.  
There may not be a lot to
say about each of these properties and instead we provide a proper reference
for the statement of the conjecture or property, but we shall endeavor to
add a few words about what needs to be proven in order to say why the property is true.
The properties are each marked in the discussion with one of four headers
$A^{(k)}$, $\At^{(k)}$, $\tilde{s}^{(k)}$, $s^{(k)}$ to indicate comments 
on which of the four definitions
(the $k$-atom definition of Section~\ref{sec:atom}, 
the operator definition of Section~\ref{sec:operator}, 
the weak tableaux definition of Section~\ref{sec:weak II}, and
the strong tableaux definition of Section~\ref{sec:strong}).
Since $s_\la^{(k)} = \tilde{s}_\la^{(k)}$ at $t=1$ by~\cite[Theorem 4.11]{LLMS:2006},
we often group these two definitions together.

\begin{subsection}{$k$-Schur functions are Schur functions when $k\geq|\la|$ and when $t=0$}
\label{subsection.limit}
Note that the stated property for $k\ge |\la|$ (resp. $t=0$) are two different statements.
Nevertheless, it has been proven for all definitions
of the $k$-Schur functions.  It is known that $s^{(k)}_\la[X;t] = s_\la$ plus other terms indexed by
partitions $\mu$ which are 
strictly larger in dominance order, each with a positive power of $t$ as a coefficient.

\TAB (proven, \cite[Property 6]{LLM:2003}) 
When $k$ is larger than the largest hook of $\la$, then the only katabolizable 
tableau is the semi-standard tableau of shape $\la$ and weight $\la$.  The charge of
this tableau is $0$ so it is clear that $A^{(k)}_\la[X;t] = s_\la$. For $t=0$, 
we know that $A^{(k)}_\la[X;0] = s_\la$ because only the tableau
of shape $\la$ and weight $\la$ has charge $0$ in the atom tableaux.

\OPER (proven, \cite[Property 8]{LM:2003}) 
This is slightly more complicated than the tableau definition
but notice that $G^{(k)}_{\la}[X;t] = s_\la$ if $k \geq |\la|$.  The action of the
operator $\Bop_m( s_\la) = s_{(m,\la_1, \ldots, \la_{\ell(\la)})}$ + other terms with
first part which is larger than $m$ (and hence will be killed by the operator $T^{(k)}_m$).
Also at $t=0$, the operator $\Bop_m \coeff_{t=0} = \Sop_m$ and $G^{(k)}_\la[X;0] = s_\la$.

\STRONG (proven, \cite[Property 39]{LM:2007}) 
One would need to trace through the definition of
the strong order, but for $k \geq |\la|$, all marked strong tableaux will be standard tableaux
and the markings of the marked strong tableaux are exactly the condition that these tableaux
should be isomorphic to semi-standard tableaux.  Because the heights of the ribbons are all
$1$, the spins for all tableaux are $0$.  Therefore, 
$s^{(k)}_\la[X;t] = \sum_{\mu} K_{\la\mu} m_\mu = s_\la$.

\WEAK (proven) When $k$ is large, weak $k$-tableaux are usual semi-standard tableaux 
and~\eqref{tinvdef} turns into the definition of Schur functions. Similarly, $Q'_\la[X;0]=s_\la$
which shows the second statement.
\end{subsection}

\begin{subsection}{The $k$-Schur function is Schur positive}
In fact, a more refined conjecture to the statement that $k$-Schur functions are Schur positive,
is that the $k$-Schur functions expand positively in the $(k+1)$-Schur functions.
This property is usually referred to as `$k$-branching' and is discussed in
slightly more detail in Section~\ref{sec:kbranch} and again in Section~\ref{sec:kshapeplus}.
Repeated applications of the $k$-branching property
yields a positive expansion of the $k$-Schur functions in
the Schur functions.  
An explicit rule for the $k$-branching formula with a $t$
is conjectured in \cite[Conjecture 1]{LLMS:2010}.

\TAB (proven, \cite[Property 7]{LLM:2003}) This property follows directly from the definition.  
Since $A^{(k)}_\la[X;t]$ is a sum of Schur functions, one for each tableau in the 
set, this property is true by definition.  

\OPER (conjecture, \cite[Equation (1.6)]{LM:2003}) 
Equation \eqref{conj:parabolicformula}
is an outstanding conjecture due to Shimozono and Weyman \cite{ShimozonoWeyman}
that $H_{\la^{(\ast)}}[X;t]$ is Schur positive
whenever the concatenation of all of the partitions of $\la^{(\ast)}$ form a partition
(which happens for all $k$-splits of partitions).  
Conjecture \ref{conj:Atomequiv} is likely to 
follow from this conjecture and hence so would the property that $\At^{(k)}_\la[X;t]$
is Schur positive.

\STRONG (for arbitrary $t$ it follows from \cite[Conjecture 1]{LLMS:2010}, for $t=1$ follows 
from~\cite[Theorem 2]{LLMS:2010}) 
In Section~\ref{sec:kshapeplus} we will give a precise
conjecture of how $s^{(k)}_\la[X;t]$ expands positively in $s^{(k+1)}_\la[X;t]$.  
If we iteratively apply that rule,
it must be that $s^{(k)}_\la[X;t]$ expands positively in the limit as 
$k$ increases, and when $k\geq \la_1+\ell(\la)-1$
then $s^{(k)}_\la[X;t] = s_\la[X]$.
In addition, some conjectures of~\cite{LLMS:2006} were proven in~\cite{BSS:2012a},
in particular that ``skew shaped'' strong Schur functions are symmetric.

Assaf and Billey~\cite{AssafBilley} convert the 
monomial expansion of $k$-Schur functions
in terms of strong marked tableaux into a 
quasi-symmetric function expansion.   They conjectured that
a dual equivalence graph structure can be placed on strong tableaux and using 
Assaf's earlier work \cite{Assaf:2008,Assaf:2010} this structure
can in theory be used to show that the $s^{(k)}_\la[X;t]$ are Schur positive
as long as a computer check on a finite number of elements is completed.  Billey informs us that
this calculation has a current estimated running time which is quite long and has not yet been
completed.

\WEAK (conjecture, proven for $t=1$) 
Since $s_\la^{(k)} = \tilde{s}_\la^{(k)}$
for $t=1$, this property is true by~\cite[Theorem 2]{LLMS:2010}.
A promising approach for generic $t$ is underway;
by duality, the positivity of $\tilde s_{\lambda}^{(k)}[X;t]$ in 
terms of $\tilde s_{\mu}^{(k+1)}[X;t]$'s follows from the
positivity of dual $k+1$-Schur functions into dual $k$-Schur functions.
This, in turn, would follow by showing that there is a compatibility 
between $k$-charge and the weak bijection introduced in \cite{LLMS:2010}.
A partial solution has been given in \cite{LapointePinto},
where the compatibility is shown  for standard $k$-tableaux.

\end{subsection}

\begin{subsection}{At $t=1$, the $k$-Schur functions satisfy the $k$-Pieri rule}
\label{sec:kpieriprop}
In Sections~\ref{sec:weaksection} and~\ref{sec:strongsection} we discussed the $k$-Pieri
rule at $t=1$. It states that for $1 \leq r \leq k$, $h_r s^{(k)}_\la[X;1] = \sum_\mu s^{(k)}_\mu[X;1]$ where
the sum is over all partitions $\mu$ such that $\mu \vdash |\la|+r$, $\mfc(\la) \subseteq \mfc(\mu)$ and 
$\mfc(\mu)/\mfc(\la)$ is a weak horizontal strip.

\vskip .1in
\noindent
\TAB (conjecture \cite[Conjecture 41]{LLM:2003}) The tableaux operations for the definition of the
$k$-atoms are somewhat unusual and more work needs to be done to understand
many of their properties.  Combinatorially, the $k$-Pieri rule is most 
clearly stated in terms of $(k+1)$-cores, the connection
between $k$-atoms and $(k+1)$-cores is not clear in this definition.

\OPER (conjecture \cite[Conjecture 21]{LM:2003}) In order to prove this property it seems as though
it would be necessary to understand the commutation relationship between
multiplication by a symmetric function $h_r$ and the operators
$T_i$ of Equation \eqref{eq:killingoperators} when $t=1$.

\WEAK (proven \cite[Theorem 29]{LM:2007}) This property follows by showing
that the number of weak $k$-tableaux of weight $\alpha$ equals the
number of $k$-tableaux of weight $\beta$, for $\beta$ any rearrangement
of the parts of $\alpha$.

\STRONG (proven \cite[Theorem 4.11]{LLMS:2006}) 
It was proven in \cite{LLMS:2006} that $s_\lambda^{(k)}=\tilde s_\lambda^{(k)}$
when $t=1$.   The result then follows from the Pieri rule 
on $\tilde s_\lambda^{(k)}$ from \cite{LM:2007}.

\vskip .1in

It is conjectured that the $k$-Schur functions also satisfy a $t$-analogue 
of the $k$-Pieri rule Equation~\eqref{eq:kweakpieri}
where the operator $\Bop_m$ takes the role of 
multiplication by $h_m$ (and reduces as such when $t=1$).
Recall that Zabrocki~\cite{ZabrockiThesis} determined  the action of $\Bop_m$ on 
Schur functions in his thesis and he gave a 
new proof for the charge formulation of
Hall--Littlewood polynomials $Q_\lambda[X;t]$.

Given a $(k+1)$-core $\lambda$ with $\lambda_1=m$, let $\rho$ be 
the unique partition whose first part is $m$ and where
$\rho/\lambda$ is a horizontal strip of
size $m$ (that is, $\rho$ is obtained by adding a cell to the top of
each column of $\lambda$).   For $\mathfrak p(\lambda)_1\leq r \leq k$,
Maria-Elena Pinto conjectured that
\begin{equation}\label{eq:tpieri}
{\mathbf B_r} \, s_{\lambda}^{(k)}[X;t] = 
\sum_{\mu/\lambda~\text{weak~$r$-strip}} t^{a(\mu)}  s_{\mu}^{(k)}[X;t]
\,,
\end{equation}
where ${a(\mu)}$ is the number of cells of $\rho/\lambda$ whose residue 
does not label a cell of $\mu/\lambda$
(see also \cite{DalalMorse:2013} for a different
conjectured formula using strong order chains
in $\tilde A^k$).

\begin{example}
Let $k=5$ and consider $\lambda=(4,2,1,1)$, so that
$\mathfrak p (\lambda)=(3,2,1,1)$.  
$$
\rho =
{\small \tableau[scY]{\tf 2 | | , \tf 5 | , , \tf 1, \tf 2 | ,,,}}\,\raisebox{-0.5cm}{.}$$
Possible horizontal weak 3-strips are
$${\small \tableau[scY]{\tf 2 | | , \tf 5 | , , \tf 1, \tf 2 | ,,,}}
\qquad
{\small \tableau[scY]{\tf 2 | | , \tf 5 | ,  | ,,,, \tf 4 ,\tf 5}}
\qquad
{\small \tableau[scY]{\tf 2 | |  | , , \tf 1, \tf 2 | ,,,, \tf 4 }}
\qquad
{\small \tableau[scY]{ | ,\tf 5 | , , \tf 1 | ,,,, \tf 4, \tf 5 }}
\qquad
{\small \tableau[scY]{ | ,\tf 5 | ,  | ,,,, \tf 4, \tf 5 ,\tf 0}}
$$
and their respective powers are $t^0, t^1$ (the cell of residue 1 in $\rho$), $t^1$ (the cell of residue 5 in $\rho$), 
$t^2$ (the two cells of residue 2 in $\rho$) and $t^3$ (the cell of residue 1 and the 2 cells of residue 2 in $\rho$).
This gives
\begin{multline*}
{\mathbf B_3} \, s_{4211}^{(5)}[X;t] =
s_{44211}^{(5)}[X;t] +t \,
s_{62211}^{(5)}[X;t] + t \,
s_{54111}^{(5)}[X;t] + t^2 \,
s_{6321}^{(5)}[X;t] +  t^3 \, 
s_{7221}^{(5)}[X;t]\;.
\end{multline*}
\end{example}

As with the action of $\Bop_m$ on a Schur functions, if $m < \la_1$, 
the $k$-Schur expansion of $\Bop_m( s^{(k)}_\la[X;t] )$ has
negative terms.  Currently, there is no conjecture describing
these terms that cancel when $t=1$.

\begin{sagedemo}
Here we demonstrate the action of the $\Bop_m$ operator
on the $k$-Schur function basis.  For $k \geq m \geq \la_1-1$ this
is conjectured to expand positively in the $k$-Schur basis.
\begin{sageexample}
    sage: Sym = SymmetricFunctions(FractionField(QQ["t"]))
    sage: ks4 = Sym.kschur(4)
    sage: ks4([3, 1, 1]).hl_creation_operator([1])
    (t-1)*ks4[2, 2, 1, 1] + t^2*ks4[3, 1, 1, 1] + t^3*ks4[3, 2, 1] 
     + (t^3-t^2)*ks4[3, 3] + t^4*ks4[4, 1, 1]
    sage: ks4([3, 1, 1]).hl_creation_operator([2])
    t*ks4[3, 2, 1, 1] + t^2*ks4[3, 3, 1] + t^2*ks4[4, 1, 1, 1]
     + t^3*ks4[4, 2, 1]
    sage: ks4([3, 1, 1]).hl_creation_operator([3])
    ks4[3, 3, 1, 1] + t*ks4[4, 2, 1, 1] + t^2*ks4[4, 3, 1]
    sage: ks4([3, 1, 1]).hl_creation_operator([4])
    ks4[4, 3, 1, 1]
\end{sageexample}
\end{sagedemo}
\end{subsection}

\begin{subsection}{$k$-conjugation}
The $\omega$-involution sending $s_\lambda$ to $s_{\la'}$
acts simply on $k$-Schur functions as well.
It was conjectured in \cite{LLM:2003,LM:2003} that,
for some non-negative power of $t$,
\begin{equation}\label{eq:kconjwt}
\omega( s^{(k)}_\la[X;t]) = t^{d} s^{(k)}_{\la^{\tilde{\omega}_k}}[X;1/t]\,,
\end{equation}
where $\la^{\tilde{\omega}_k}$ is a $k$-bounded partition depending on $\la$ and $k$.
Later, it was shown in \cite{LM:2005} that $\la^{\tilde{\omega}_k}=\la^{\omega_k}$ 
corresponds to usual conjugation in the $k+1$-core framework
(see Definition~\ref{def:kconjugate}).
One consequence of $k$-branching \cite{LLMS:2010} is that
the power of $d$ counts cells of $\mfc(\la)$ 
with hook greater than $k$ (see also \cite{CH:2008}).

\TAB (conjecture~\cite[Conjecture 36]{LLM:2003})
Using the definition of $k$-Schur functions in terms of tableaux atoms, there is definitely 
an orientation in the definition of katabolism that is not compatible with the notion of conjugation
of the tableau.  This is because the first step of the definition involves a split of the reading
word of the tableau which involves reading the rows, and it really is not
clear how this would carry to the columns.

One possible approach to the tableaux definition would be to show that there is a bijection 
between the tableaux in the atoms $\AAA^{(k)}_{\la}$ and $\AAA^{(k)}_{\la^{\omega_k}}$.  
Since the $k$-conjugation is most easily expressed in terms of $(k+1)$-cores, 
it seems that some connection between atoms and $(k+1)$-cores will need to be found.

\OPER (conjecture~\cite[Conjecture 40]{LM:2003}, at $t=1$ this is~\cite[Conjecture 19]{LM:2003}) 
Since the algebraic definition depends on the operation of $k$-split, it is not clear
how the action of $\omega$ interacts with the individual operators.  There is some hope
that some algebraic tools will be developed to resolve this conjecture by looking at the
expansion in terms of $s_\lambda[X/(1-t)]$ since
for an element $f[X;t] \in \La^t_{(n)}$ it is at least known 
that $\omega( f[X;1/t]) \in \La^t_{(n)}$.

\STRONGWEAK (conjecture, for $t=1$ this is~\cite[Theorem 38]{LM:2007})  At $t=1$ this property
follows because of~\cite[Theorem 33]{LM:2007}, where  a formula for the product of $e_r$ and
a $k$-Schur function is given as
\begin{equation}\label{eq:erkpierirule}
e_r s_\lambda^{(k)} = \sum_\mu s_\mu^{(k)}\;,
\end{equation}
where the sum is over all $k$-bounded partitions $\mu$ of size $|\la|+r$ such that
$\mu/\la$ is a vertical strip and $\mu^{\omega_k}/\la^{\omega_k}$ a horizontal strip.

\begin{sagedemo}
Let us check an example of equation~\eqref{eq:kconjwt} by a calculation in \Sage.
In order to invert the parameter $t$, we must first expand
the $k$-Schur function in a basis which is independent of the parameter $t$.
In fact, if we apply the involution $\omega$ alone, the function no longer
lies in the space spanned by the $k$-Schur functions.  However, if we
apply $\omega$ and invert the parameter, then it does belong to the right
space.
\begin{sageexample}
    sage: Sym = SymmetricFunctions(FractionField(QQ["t"]))
    sage: ks3 = Sym.kschur(3)
    sage: ks3([3,2]).omega()
    Traceback (most recent call last):
    ...
    ValueError: t^2*s[1, 1, 1, 1, 1] + t*s[2, 1, 1, 1] + s[2, 2, 1] is not 
    in the image of Generic morphism:
    From: 3-bounded Symmetric Functions over Fraction Field of Univariate 
    Polynomial Ring in t over Rational Field in the 3-Schur basis
    To:   Symmetric Functions over Fraction Field of Univariate Polynomial Ring
    in t over Rational Field in the Schur basis

    sage: s = Sym.schur()
    sage: s(ks3[3,2])
    s[3, 2] + t*s[4, 1] + t^2*s[5]
    sage: t = s.base_ring().gen()
    sage: invert = lambda x: s.base_ring()(x.subs(t=1/t))
    sage: ks3(s(ks3([3,2])).omega().map_coefficients(invert))
    1/t^2*ks3[1, 1, 1, 1, 1]
\end{sageexample}
In fact, there is a short-cut for the last computation in \Sage by simply asking
\begin{sageexample}
    sage: ks3[3,2].omega_t_inverse()
    1/t^2*ks3[1, 1, 1, 1, 1]
\end{sageexample}
\end{sagedemo}
\end{subsection}

\begin{subsection}{The $k$-Schur functions form a basis for $\La^t_{(k)}$} \label{sec:basis}
Recall from Equation~\eqref{Laktspacedef} that the definition of 
$\La^t_{(k)}$ is the linear span over $\QQ(q,t)$ of
the symmetric functions $H_\la[X;q,t]$ (or $s_\la[X/(1-t)]$ or $Q'_\la[X;t]$)
over all partitions $\la$ with $\la_1 \leq k$.  
For each of our definitions, it is not necessarily clear 
that the $k$-Schur functions even lie in $\La^t_{(k)}$.
However, if they do and if they are linearly independent, 
they will form a basis since they are also indexed by $k$-bounded partitions.

\TAB (conjecture \cite[Conjecture 8]{LLM:2003}) 
The atoms are known to be linearly independent~\cite[Property 7]{LLM:2003}
since they are triangular with respect to the Schur functions.
Nevertheless, it remains a conjecture that they are elements 
of $\La^t_{(k)}$; their combinatorial
definition does not give a direct connection with the known bases of $\La^t_{(k)}$. 
It seems as though the most likely means of proving this conjecture is to show
Conjecture~\ref{conj:Atomequiv}, otherwise there is no obvious connection 
with the spanning elements which define $\La^t_{(k)}$.

\OPER (proven \cite[Theorem 33]{LM:2003})  This result is non-trivial because 
it is not easy to demonstrate that the elements $G^{(k)}_\la[X;t]$ form
a basis of $\La^t_{(k)}$.

\STRONG (conjecture, discussion of this definition is in \cite[Section 9.3]{LLMS:2006} but
this particular property is not directly addressed; at $t=1$ this is~\cite[Property 27]{LM:2007})
In reference \cite{LM:2007} the $k$-Schur functions $s^{(k)}_\la$
are defined as the basis which satisfies
the $k$-Pieri rule of \eqref{eq:kweakpieri} and from that definition it is clear that
$s^{(k)}_\la \in \La_{(k)}$. The fact that they form a basis follows from
a unitriangularity relation with the basis $\{ h_\la : \la_1 \leq k \}$ that follows 
from the $k$-Pieri rule.

\WEAK (proven \cite{DalalMorse:2012}) Since $Q'_\la[X;t]$ forms a basis of $\La^t_{(k)}$ and the matrix $K_{\la\mu}^{(k)}(t)$
is invertible, $\tilde{s}_\la^{(k)}[X;t]$ also forms a basis of $\La^t_{(k)}$.

\end{subsection}

\begin{subsection}{The $k$-rectangle property} \label{sec:rectprop}

A remarkable property of the $k$-Schur functions is that 
it is trivial to multiply any $s_\lambda^{(k)}$ by
a $k$-Schur function indexed by a $k$-rectangle -- any partition of the
form $(\ell^{k-\ell+1})$.  Precisely,
for any $k$-bounded partition $\lambda$ and any integer
$1 \leq \ell \leq k$, 
\begin{equation}
s_{\ell^{k+1-\ell}}\,s^{(k)}_\lambda = s^{(k)}_{\lambda\cup\ell^{k+1-\ell}}\,,
\end{equation}
where $\lambda\cup\nu$ depicts the partition obtained by
putting the parts of $\lambda$ and $\nu$ into non-increasing order.
In fact, this property has a generic $t$ analog in which
the Schur function $s_{\ell^{k+1-\ell}}$ is replaced by
the operator $\Bop_{\ell^{k+1-\ell}}$ defined in Equation \eqref{def:Bopla}.  
Then, given $\lambda$ and an integer $1 \leq \ell \leq k$, 
\begin{equation}
\Bop_{(\ell^{k-\ell+1})} s^{(k)}_\la[X;t] = 
t^{|\mu| - \ell(\mu) \ell} s^{(k)}_{(\ell^{k-\ell+1}) \cup \la}[X;t]~,
\end{equation}
where $\la = (\mu, \nu)$ with $\mu_{\ell(\mu)} > \ell \geq \nu_1$.
A by-product of this result is that any $k$-Schur function
can be obtained by $k$-rectangle translation of
elements in a distinguished set of $k!$ $k$-Schur functions.
These $k!$ elements are those indexed by irreducible partitions
--partitions with at most $k-r$ parts of size $r$, for $1\leq r\leq k$.
For any $k$-bounded partition $\nu$, up to a $t$-factor,
\begin{equation}
s_\nu^{(k)}[X;t] = \Bop_{R_1} \Bop_{R_2} \cdots \Bop_{R_d} s^{(k)}_\la[X;t]
\,,
\end{equation}
where $\la$ is the irreducible partition obtained by
removing $k$-rectangles $R_1,\ldots, R_d$ from $\nu$.

\vskip .1in
\noindent
\TAB (conjecture \cite[Conjecture 21]{LLM:2003})  
This is not known and there are no obvious techniques to be tried.
One caveat of the atom definition is that the $t=1$ case does not
simplify things; the result is also unknown when $t=1$.

\OPER (proven \cite[Theorem 26]{LM2:2003}) This is shown by developing properties
of the $\Bop_\lambda$ operators and the commutation relations with
the operator $T_m^{(k)}$.

\STRONGWEAK (conjecture, but proven for $t=1$ in~\cite[Theorem 40]{LM:2007}) 
When $t=1$, the operator $\Bop_R$ reduces to 
multiplication by $s_R$ and it was shown that
the linear operation of adding a $k$-rectangle to the index of a $k$-Schur function 
commutes with the Pieri rule
from Equation \eqref{eq:kweakpieri}.  
\end{subsection}

\begin{subsection}{When $t=1$, the product of $k$-Schur functions is $k$-Schur positive} 
\label{sec:teq1positive}
Note that $\La^t_{(k)}$ is not an algebra and the product of
two arbitrary $k$-Schur functions does not remain in the space.
However, when $t=1$ we have $\La_{(k)}^{t=1} = \La_{(k)}$ 
and as discussed in Section~\ref{sec:weaksection}, the space defined in
Equation~\eqref{kboundedspace} is closed under multiplication.
The structure coefficients $c_{\la\mu}^{\nu(k)}$ defined by
\begin{equation}\label{eq:kLRrule}
s^{(k)}_\la s^{(k)}_\mu = \sum_\nu c_{\la\mu}^{\nu(k)} s^{(k)}_\nu
\end{equation}
are non-negative integer coefficients.
Recall from Section~\ref{sec:nilcoxeter}
that these are
now called $k$-Littlewood-Richardson coefficients and
they are a family of constants that includes Gromov-Witten invariants
for complete flag varieties, WZW-fusion coefficients,
and the structure constants of Schubert polynomials.

\vskip .1in
\noindent
\TAB (conjecture~\cite[Conjecture 39]{LLM:2003}) 
Note that the $k$-Pieri rule is a special case of computing
the product of two $k$-Schur functions.
As discussed in Section~\ref{sec:kpieriprop},
the techniques to work with
atoms have yet to be developed and even this `simple' case
remains unproven.  However, the atom definition is likely to be
useful in gaining insight into the 
combinatorial nature of the structure coefficients.

\OPER (conjecture \cite[Conjecture 20]{LM:2003})
In this case we have again that the $k$-Pieri rule at $t=1$ is a conjecture.
It would be sufficient to show that $s^{(k)}_\lambda[X;1] =
\At^{(k)}_\la[X;1]$ since the $k$-Pieri rule characterizes $s^{(k)}_\lambda[X;1]$ and
hence the result is known in this case.
It seems that the algebraic definition using operators might be
helpful finding a $t$-analogue of the coefficients $c_{\la\mu}^{\nu(k)}$.

\STRONGWEAK (proven \cite[Corollary 8.2]{Lam:2008}) 
Lam~\cite{Lam:2008} proved that the $s^{(k)}_\la[X;1]$ 
are isomorphic to the Schubert basis for the
homology of the affine Grassmannian.  From geometric considerations, 
it follows that the structure coefficients $c_{\la\mu}^{\nu(k)}$ 
enumerate certain curves in a finite flag variety.
\end{subsection}

\begin{subsection}{Positively closed under coproduct}\label{sec:positivecoprod}
The space $\La^t_{(k)}$ is not an algebra, but since it is linearly spanned by the elements
$s_\lambda[X/(1-t)]$ for $\lambda_1 \leq k$, it is a coalgebra under the coproduct defined by
\begin{equation}\label{coproductdef}
\Delta( s_\la[X/(1-t)] ) = \sum_{\mu, \nu}  c^{\la}_{\mu\nu} s_\mu[X/(1-t)] s_{\nu}[Y/(1-t)]
\end{equation}
(where the $c^\la_{\mu\nu}$ are the Littlewood--Richardson coefficients). Compare this also
with the coproduct~\eqref{equation.coproduct schur} on $\La$. Since if $\la$ is $k$-bounded 
then all of the terms $\mu,\nu$ which appear in this expansion will also be $k$-bounded,
it is conjectured that if the coefficients $C^{\la}_{\mu\nu}(t)$ are defined as the coefficients in the expansion
\begin{equation}\label{coproductexp}
\Delta( s^{(k)}_\la[X;t] ) = \sum_{\mu,\nu} C^{\la(k)}_{\mu\nu}(t) s^{(k)}_\mu[X;t] s^{(k)}_\nu[Y;t]\;,
\end{equation}
then the $C^{\la}_{\mu\nu}(t)$ are polynomials in $t$ with non-negative integer coefficients.

\TAB (conjecture \cite[Conjecture 17]{LLM:2003})

\OPER (conjecture \cite[Conjecture 41]{LM:2003})

\STRONG (conjecture, proven for $t=1$ in~\cite[Corollary 8.1]{Lam:2008}) 
Because of the duality of the elements
$\SS^{(k)}_\mu$ to the elements $s^{(k)}_\la$, it follows
that at $t=1$, the basis $\SS^{(k)}_\la$ of the space $\La^{(k)}$ multiplies as
\begin{equation*}
\SS^{(k)}_\mu \SS^{(k)}_\nu = \sum_{\la} C_{\mu\nu}^{\la(k)} \SS^{(k)}_\la \;,
\end{equation*}
where $C^{\la(k)}_{\mu\nu} = C^{\la(k)}_{\mu\nu}(1)$.

Although we can say little about the $A^{(k)}$ and $\At^{(k)}$ cases, 
we can determine from the definition of $s^{(k)}_\la[X;t]$ that
\begin{equation*}
	s^{(k)}_\la[X+z;t] = \sum_{r \geq 0} z^r h_r^\perp s^{(k)}_\la[X;t] 
	= \sum_{r\geq 0} z^r C^{\la(k)}_{\mu(r)}(t) s^{(k)}_\mu[X;t]\;.
\end{equation*}
In this special case the coefficients are
\begin{equation*}
	C^{\la(k)}_{\mu(r)}(t) = \sum_{\kappa^{(\ast)},c_\ast} t^{\spin(\kappa^{(\ast)},c_\ast)}
\end{equation*}
with the sum is over all strong marked horizontal strips from $\mfc(\mu)$ to
$\mfc(\la)$.  That is, the sum runs over $\kappa^{(\ast)},c_\ast$ which are $(k+1)$-core tableaux of the form
\begin{equation}
\kappa^{(0)} = \mfc(\la) \Rightarrow_k \kappa^{(1)} \Rightarrow_k \kappa^{(2)} \Rightarrow_k \cdots \Rightarrow_k \kappa^{(r)} = \mfc(\mu)
\end{equation}
and markings $c_1 < c_2 < \cdots < c_r$ where $c_i$ is the 
content of the lower right
hand cell of one of the ribbons of $\kappa^{(i)}/\kappa^{(i-1)}$ and 
where $\spin(\kappa^{(\ast)},c_\ast)$ is defined, as before, as the sum of the spins of the
strong marked ribbons $(\kappa^{(i)}/\kappa^{(i-1)}, c_i)$.

\WEAK (conjecture)
For $t=1$, $s_\lambda^{(k)} = {\tilde s}_\lambda^{(k)}$ and hence also follows from \cite[Corollary 8.1]{Lam:2008}.

\begin{sagedemo}
Here is a calculation in \Sage where we observe that
the coefficients that appear in an example of~\eqref{coproductexp} are
polynomials in $\NN[t]$.
\begin{sageexample}
    sage: Sym = SymmetricFunctions(FractionField(QQ["t"]))
    sage: ks3 = Sym.kschur(3)
    sage: ks3[3,1].coproduct()
    ks3[] # ks3[3, 1] + ks3[1] # ks3[2, 1] + (t+1)*ks3[1] # ks3[3] 
    + ks3[1, 1] # ks3[2] + ks3[2] # ks3[1, 1] + (t+1)*ks3[2] # ks3[2] 
    + ks3[2, 1] # ks3[1] + (t+1)*ks3[3] # ks3[1]  + ks3[3, 1] # ks3[]
\end{sageexample}
\end{sagedemo}

\end{subsection}

\begin{subsection}{The product of a $k$-Schur and $\ell$-Schur function is $(k+\ell)$-Schur positive}

Recall that one characterization of $\La^t_{(k)}$ is that it is the linear span
of $\left\{ s_\lambda\left[ X/(1-t) \right] \right\}_{\lambda_1\leq k}$.  Now if we know that $f \in \La^t_{(k)}$
and $g \in \La^t_{(\ell)}$ then we know by the Littlewood--Richardson rule that
$f g$ will be in the linear span of $\left\{ s_\lambda\left[X/(1-t) \right] \right\}_{\lambda_1 \leq k+\ell}$.  
By the discussion in Section~\ref{sec:basis}, it is not even clear that the products 
$s_\lambda^{(k)}[X;t] s_\mu^{(\ell)}[X;t]$ or $A_\lambda^{(k)}[X;t] A_\mu^{(\ell)}[X;t]$ will
be in the space $\La^t_{(k+\ell)}$, but it has been proven that $\At_\lambda^{(k)}[X;t]
\At_\mu^{(\ell)}[X;t]$ and $\tilde{s}_\lambda^{(k)}[X;t] \tilde{s}_\mu^{(\ell)}[X;t]$ is an element of $\La^t_{(k+\ell)}$.

Given that the product of a $k$-Schur function and an $\ell$-Schur function
is in the linear span of $\La^t_{(k+\ell)}$, it is natural to conjecture that
the resulting product will be $(k+\ell)$-Schur positive.  We do not know of an attribution
for this conjecture but it seems to have been passed around in discussions and talks on
the subject.
		
\begin{sagedemo}
We demonstrate an example of this conjecture in \Sage by showing that
the product of a $3$-Schur function and a $2$-Schur function expands
positively in terms of $5$-Schur functions.
\begin{sageexample}
    sage: Sym = SymmetricFunctions(FractionField(QQ["t"]))
    sage: ks2 = Sym.kschur(2)
    sage: ks3 = Sym.kschur(3)
    sage: ks5 = Sym.kschur(5)
    sage: ks5(ks3[2])*ks5(ks2[1])
    ks5[2, 1] + ks5[3]
    sage: ks5(ks3[2])*ks5(ks2[2,1])
    ks5[2, 2, 1] + ks5[3, 1, 1] + (t+1)*ks5[3, 2] + (t+1)*ks5[4, 1] 
      + t*ks5[5]
\end{sageexample}
\end{sagedemo}

\end{subsection}

\begin{subsection}{Branching property from $k$ to $k+1$} \label{sec:kbranch}
One of the properties that is easy to observe when conjecturing the
existence of atoms is that the atoms seem to split into smaller pieces
as $k$ increases.  In the limit (when $k \geq |\la|$), we know that $s^{(k)}_\la[X;t] = s_\la$
(see Section~\ref{subsection.limit}).
Although it is clear that $\La^t_{(k)} \subseteq \La^t_{(k+1)}$, it is not easy
to prove this branching property.

One reason in particular that this is a difficult property to understand is that both the
definition of $A^{(k)}_\la$ and $\At^{(k)}_\la$ involve the operation of the $k$-split, one
in the katabolism procedure, and the other in the $k$-split basis $G^{(k)}_\la[X;t]$.  In theory
the $k$-split of a partition $\la$ can be very different than the $(k+1)$-split of the same
partition (e.g. consider the $4$ and $5$ split of $(4,4,4,3,3,3,2,2,1,1)$ which are 
$((4),(4),(4),(3,3),(3,2),(2,1,1))$ and $((4,4),(4,3),(3,3,2),(2,1,1))$ respectively).  A priori we would 
not expect to see that $A^{(k)}_\la[X;t]$ expands positively in $A^{(k+1)}[X;t]$ or $\At^{(k)}[X;t]$ expands 
positively in $\At^{(k+1)}_\la[X;t]$. However this property was one that was used to conjecture/compute 
the $k$-atoms before there was a first formal definition.

\STRONG (proven for $t=1$ in \cite{LLMS:2010}, \cite[Conjecture 3]{LLMS:2010} 
is combinatorial formula) 
At $t=1$, the proof of the $k \rightarrow k+1$ branching 
for $s^{(k)}_\la[X;t]$ 
follows from the study \cite{LLMS:2010} of a poset on particular partitions
called $k$-shapes.  In Section~\ref{sec:kshapeplus} we will give some 
details on the combinatorics behind these results and state the
explicit combinatorial formula for this rule in Theorem~\ref{thm:brancht1}.
In short, \cite{LLMS:2010} proves that
$s^{(k)}_\la$ expands positively in the elements $s^{(k+1)}_\la$ 
and gives a conjecture for the expansion 
of $s^{(k)}_\la[X;t]$ in terms of elements
of the form $s^{(k+1)}_\la[X;t]$.
For generic $t$, the same paper gives a conjecture formula
Conjecture~\ref{conj:ktbranch} and discusses the additional properties 
needed for the result to hold in general.
Some progress has been made in this direction in \cite{LapointePinto}.

\WEAK (conjecture) At $t=1$, $s^{(k)}_\la = {\tilde s}^{(k)}_\la$, hence the result in this case also follows from
refrence \cite{LLMS:2010}.

\begin{sagedemo}
Here are some examples confirming the branching conjecture (for the implementation
in \Sage):
\begin{sageexample}
    sage: Sym = SymmetricFunctions(FractionField(QQ["t"]))
    sage: ks3 = Sym.kschur(3)
    sage: ks4 = Sym.kschur(4)
    sage: ks5 = Sym.kschur(5)
    sage: ks4(ks3[3,2,1,1])
    ks4[3, 2, 1, 1] + t*ks4[3, 3, 1] + t*ks4[4, 1, 1, 1] + t^2*ks4[4, 2, 1]
    sage: ks5(ks3[3,2,1,1])
    ks5[3, 2, 1, 1] + t*ks5[3, 3, 1] + t*ks5[4, 1, 1, 1] + t^2*ks5[4, 2, 1] 
     + t^2*ks5[4, 3] + t^3*ks5[5, 1, 1]

    sage: ks5(ks4[3,2,1,1])
    ks5[3, 2, 1, 1]
    sage: ks5(ks4[4,3,3,2,1,1])
    ks5[4, 3, 3, 2, 1, 1] + t*ks5[4, 4, 3, 1, 1, 1] 
     + t^2*ks5[5, 3, 3, 1, 1, 1]
    sage: ks5(ks4[4,3,3,2,1,1,1])
    ks5[4, 3, 3, 2, 1, 1, 1] + t*ks5[4, 3, 3, 3, 1, 1] 
     + t*ks5[4, 4, 3, 1, 1, 1, 1] + t^2*ks5[4, 4, 3, 2, 1, 1] 
     + t^2*ks5[5, 3, 3, 1, 1, 1, 1] + t^3*ks5[5, 3, 3, 2, 1, 1] 
     + t^4*ks5[5, 4, 3, 1, 1, 1]
\end{sageexample}
\end{sagedemo}
\end{subsection}

\begin{subsection}{$k$-Schur positivity of Macdonald symmetric functions}\label{sec:macpos}
Even equipped with all these definitions, it has yet to be understood why
the Macdonald polynomials expand positively in terms of 
$k$-Schur functions.  Recall from \eqref{mackkostka} that,
for any $k$-bounded partition $\mu$, the coefficients  in
$$H_\mu[X;q,t] = 
\sum_{\la: \la_1 \leq k} K^{(k)}_{\la\mu}(q,t) s^{(k)}_\la[X;t]$$
are conjectured to be polynomials in $q$ and $t$ with non-negative
integer coefficients.  An ideal solution to this problem would be 
to find statistics
$a_\mu^{(k)}$ and $b_\mu^{(k)}$ on weak tableaux such that
$$K^{(k)}_{\la\mu}(q,t) = \sum_{T} q^{a_\mu^{(k)}(T)} t^{b_\mu^{(k)}(T)}$$
where the sum is over all standard weak tableaux of shape $\lambda$.
Section \ref{sec:weak II} discusses partial progress in this direction
where such a solution is given for the cases
$K^{(k)}_{\la\mu}(1,1) = K_{\la{1^{|\mu|}}}^{(k)}$ 
and $K^{(k)}_{\la\mu}(0,1) = K_{\la\mu}^{(k)}$.

\TAB (conjecture \cite[Conjecture 8]{LLM:2003})  
This conjecture was the original motivation for
studying $k$-Schur functions;  it was a promising attack on
a combinatorial interpretation of the Macdonald--Kostka coefficients
especially when coupled with the
conjecture \cite[Eq. (1.15)]{LLM:2003} that 
$K_{\la\mu}(q,t) - K^{(k)}_{\lambda\mu}(q,t)$ is in $\NN[q,t]$.
However, a clear combinatorial interpretation of $K_{\lambda\mu}(q,t)$ 
remains elusive as does even the positivity of the polynomials 
$K^{(k)}_{\lambda\mu}(q,t)$.

\OPER (conjecture \cite[Eq. (1.7)]{LM:2003}) 
A preliminary attack of the $k=2$ case of this conjecture was considered in
\cite{LM:1998} and \cite{Zabrocki:1998} although the complete
formulation of the conjecture had not been yet made.  
Lapointe and Morse together with Lascoux
developed the ideas further into $k$-atoms.  Even without knowledge
of this conjecture, the latter reference also refers
to collections of tableaux as `atoms'
and some of the symmetric functions defined there
were actually the $2$-atoms.

\STRONG (conjecture \cite[Eq. (11.6)]{LM:2005})  It was because of the characterization
of the $k$-Schur functions as the basis that satisfies the $k$-Pieri rule of 
Equation~\eqref{eq:kweakpieri} that there is a
combinatorial interpretation for $K^{(k)}_{\lambda\mu}(1,1)$ in terms of weak tableaux.

\WEAK (conjecture) 
This is a conjecture, but for $q=0$, $H_\mu[X;0,t] = Q'_\mu[X;t]$ and 
the definition \cite{DalalMorse:2013} of $\tilde s_\lambda^{(k)}$ yields
that $K_{\lambda\mu}^{(k)}(0,t) = K_{\lambda\mu}^{(k)}(t)$.

\begin{sagedemo}
Here are some of the $k$-analogues of the $(q,t)$-Macdonald--Kostka coefficients computed
in \Sage:
\begin{sageexample}
    sage: Sym = SymmetricFunctions(FractionField(QQ["q,t"]))
    sage: H = Sym.macdonald().H()
    sage: ks = Sym.kschur(3)
    sage: ks(H[3])
    q^3*ks3[1, 1, 1] + (q^2+q)*ks3[2, 1] + ks3[3]
    sage: ks(H[3,2])
    q^4*ks3[1, 1, 1, 1, 1] + (q^3*t+q^3+q^2)*ks3[2, 1, 1, 1] 
     + (q^3*t+q^2*t+q^2+q)*ks3[2, 2, 1] 
     + (q^2*t+q*t+q)*ks3[3, 1, 1] + ks3[3, 2]
    sage: ks(H[3,1,1])
    q^3*ks3[1, 1, 1, 1, 1] + (q^3*t^2+q^2+q)*ks3[2, 1, 1, 1] 
     + (q^2*t^2+q^2*t+q*t+q)*ks3[2, 2, 1] 
     + (q^2*t^2+q*t^2+1)*ks3[3, 1, 1] + t*ks3[3, 2]
\end{sageexample}
\end{sagedemo}

\end{subsection}
\end{section}

\begin{section}{Directions of research and open problems}
\label{section.directions}
In this section we consider further directions of  $k$-Schur
research, some in their early stages.

\begin{subsection}{A $k$--Murnaghan-Nakayama rule}
\label{subsection.MN}
The Murnaghan--Nakayama rule~\cite{LR:1934, Mur:1937, Nak:1941} is a combinatorial formula for
the characters $\chi_\la(\mu)$ of the symmetric group in terms of ribbon tableaux. Under the 
Frobenius characteristic map, there exists an analogous statement on the level of symmetric functions,
which follows directly from the formula
\begin{equation} \label{origMNrule}
	p_r s_\la = \sum_\mu (-1)^{\height(\mu/\la)} s_\mu.
\end{equation}
Here $p_r$ is the $r$-th power sum symmetric function, $s_\la$ is the Schur function
labeled by partition $\la$, and the sum is over all partitions $\la \subseteq \mu$ for which 
$\mu/\la$ is a border strip of size $r$. Recall that a border strip is a connected skew shape
without any $2\times 2$ squares. The height $\height(\mu/\la)$ of a border strip
$\mu/\la$ is one less than the number of rows.

In~\cite{BandlowSchillingZabrocki}, an analogue of the Murnaghan--Nakayama rule for
the product of $p_r$ times $s^{(k)}_\la$ is given. This is derived using the $k$-Pieri rule,
and is expressed in terms of the action of the affine symmetric group (resp. nil-Coxeter group)
on cores. To give the precise result we need to make a couple of definitions.
We define a {\it vertical domino} in a skew-partition
to be a pair of cells in the diagram, with one sitting directly above the
other.  For the skew of two $k$-bounded partitions $\la \subseteq \mu$ we
define the height as
\begin{equation}
\label{e:height}
	\height(\mu/\la) = \text{number of vertical dominos in $\mu/\la$} \;.
\end{equation}
For ribbons, that is skew shapes without any $2\times 2$ squares,
the definition of height can be restated as the number of occupied rows minus the number
of connected components.  Notice that this is compatible with the usual definition of
the height of a border strip.

\begin{definition} \label{def:kribbon}
  The skew of two $k$-bounded partitions, $\mu / \lambda$, is called a
  \emph{$k$-ribbon of size $r$} if $\mu$ and $\lambda$ satisfy the following
  properties: 
  \begin{enumerate}
    \item[(0)] (containment condition) $\la \subseteq \mu$ and $\la^{\omega_k}
      \subseteq \mu^{\omega_k}$; 
    \item[(1)] (size condition) $|\mu/\la|=r$;
    \item[(2)] (ribbon condition) $\mfc(\mu)/\mfc(\la)$ is a ribbon;
    \item[(3)] (connectedness condition) $\mfc(\mu)/\mfc(\la)$
      is $k$-connected, that is, the contents of  $\mfc(\mu)/\mfc(\la)$ form an interval of $[0,k]$
      (where 0 is $k$ are adjacent);
    \item[(4)] (height statistics condition) $\height(\mu / \la) + \height(
      \mu^{\omega_k} / \lambda^{\omega_k}) = r-1$.  
  \end{enumerate}
\end{definition}

Then the $k$-Murnaghan--Nakayama rule states:
\begin{theorem} \label{thm:MN rule}
For $1\le r\le k$ and $\la$ a $k$-bounded partition, we have
\begin{equation*}
  p_r s_\lambda^{(k)} = \sum_\mu (-1)^{\height(\mu/\la)} s_\mu^{(k)},
\end{equation*}
where the sum is over all $k$-bounded partitions $\mu$ such that $\mu /
\lambda$ is a $k$-ribbon of size $r$.
\end{theorem}

Computer evidence suggests that the ribbon condition (2) of Definition~\ref{def:kribbon}
might be superfluous because it is implied by the other conditions of the definition. 
This was checked for $k,r\le $11 and for all $|\la|=n\le 12$ and
$|\mu| = n+r$. Also, the $k$-Murnaghan--Nakayama rule of Theorem~\ref{thm:MN rule}
was only proven for the definition of $k$-Schur functions $s^{(k)}_\la[X;1]$ and not in terms of
$A^{(k)}_\la[X;1]$ or $\At^{(k)}_\la[X;1]$.

Note that a Murnaghan--Nakayama rule potentially provides us 
with a fourth, independent definition of the $k$-Schur functions
in a manner similar to Equation~\eqref{eq:kweakpieri}.
The fault with this approach is that it is not immediately obvious that this
system of equations is invertible and consequently defines the elements $s^{(k)}_\la$.

An analog of the Murnaghan--Nakayama rule for the elements $\SS^{(k)}_\lambda[X]$ would give
a combinatorial interpretation of the $k$-Schur functions in the power sum basis at $t=1$.

\begin{sagedemo}
Let us show how to compute the Murnaghan--Nakayama rule for $\SS^{(k)}_\lambda[X]$. 
The quotient space $\Lambda^{(k)}$ is implemented in \Sage, but there are several
means of computing the coefficients of $p_k \SS^{(k)}_\lambda[X]$ by duality. 
As an example, let us compute $p_2 \SS^{(3)}_{21}[X]$:
\begin{sageexample}
    sage: Sym = SymmetricFunctions(QQ)
    sage: Q3 = Sym.kBoundedQuotient(3,t=1)
    sage: F = Q3.affineSchur()
    sage: p = Sym.power()
    sage: F[2,1]*p[2]
    -F3[1, 1, 1, 1, 1] - F3[2, 1, 1, 1] + F3[3, 1, 1] + F3[3, 2]
\end{sageexample}
Hence this computation shows that 
\[
	p_2 \SS^{(3)}_{21} = \SS^{(3)}_{32} + \SS^{(3)}_{311} - \SS^{(3)}_{2111} - \SS^{(3)}_{11111}.
\]
\end{sagedemo}
\end{subsection}

\begin{subsection}{A rectangle generalization at $t$ a root of unity}
Let $\zeta_m$ be an $m^{th}$ root of unity (take $\zeta_m = e^{2\pi i/m}$).  A result due to 
Lascoux, Leclerc and Thibon~\cite{LLT:1993} states that if $\lambda = (1^{m_1}, 2^{m_2}, \ldots, d^{m_d})$ 
and $m_i = q_i m + r_i$ for  $0 \leq r_i < m$, then
\begin{equation*}
Q'_\lambda[X; \zeta_m ] =
(Q'_{(1^m)}[X; \zeta_m ])^{q_1} (Q'_{(2^m)}[X; \zeta_m ])^{q_2}
\cdots (Q'_{(d^m)}[X; \zeta_m ])^{q_d} Q'_{\nu}[X; \zeta_m]
\end{equation*}
where $\nu = (1^{r_1},2^{r_2}, \ldots, d^{r_d})$.
A similar property is shown by Descouens and Morita~\cite{DM:2007}
for the Macdonald symmetric functions.  Namely they show
\begin{equation*}
H_\lambda[X; q,\zeta_m ] =
(H_{(1^m)}[X; q,\zeta_m ])^{q_1} (H_{(2^m)}[X; q,\zeta_m ])^{q_2}
\cdots (H_{(d^m)}[X; q,\zeta_m ])^{q_d} H_{\nu}[X; q,\zeta_m]~.
\end{equation*}
Moreover it is shown in these references that
$$ Q'_{(r^m)}[X;\zeta_m] = p_m \circ h_r$$
and
$$ H_{(r^m)}[X;q, \zeta_m] = p_m \circ h_r[X/(1-q)] \left( \prod_{i=1}^r (1-q^{i m}) \right)$$
where $\circ$ is the operation of plethysm.

Since at arbitrary $t$, both of these functions expand positively in $k$-Schur
functions, it is natural to ask if this property is shared by the $k$-Schur functions
themselves.  At $t=1$ the $k$-Schur functions satisfy (see Section \ref{sec:rectprop})
\begin{equation}
s^{(k)}_{(\ell^{k-\ell+1})}[X;1] s^{(k)}_{\lambda}[X;1]
= s^{(k)}_{(\ell^{k-\ell+1}) \cup \lambda}[X;1]~.
\end{equation}
At an $m^{th}$ root of unity this property seems to generalize and we conjecture

\begin{conj}  For $\ell \leq k$ and $\zeta_m = e^{2\pi i/m}$
\begin{equation*}
s^{(k)}_{(\ell^{m(k-\ell+1)})}[X;\zeta_m] s^{(k)}_{\lambda}[X;\zeta_m]
= s^{(k)}_{(\ell^{m(k-\ell+1)}) \cup \lambda}[X;\zeta_m]
\end{equation*}
and moreover
$$s^{(k)}_{(\ell^{m(k-\ell+1)})}[X;\zeta_m] = p_m \circ s^{(k)}_{(\ell^{k-\ell+1})}[X;1]~.$$
\end{conj}

\begin{sagedemo} 
We demonstrate an example of this conjecture by building two copies of
the $k$-Schur functions in \Sage, one where the parameter $t$ is specialized
to a fourth root of unity, and the other
where $t=1$.  Expanding these $k$-Schur functions in the power sum basis makes it possible
to see the relationship between these elements, checking the second relation:
\begin{sageexample}
    sage: R = QQ[I]; z4 = R.zeta(4)
    sage: Sym = SymmetricFunctions(R)
    sage: ks3z = Sym.kschur(3,t=z4)
    sage: ks3 = Sym.kschur(3,t=1)
    sage: p = Sym.p()
    sage: p(ks3z[2, 2, 2, 2, 2, 2, 2, 2])
    1/12*p[4, 4, 4, 4] + 1/4*p[8, 8] - 1/3*p[12, 4]
    sage: p(ks3[2,2])
    1/12*p[1, 1, 1, 1] + 1/4*p[2, 2] - 1/3*p[3, 1]
    sage: p(ks3[2,2]).plethysm(p[4])
    1/12*p[4, 4, 4, 4] + 1/4*p[8, 8] - 1/3*p[12, 4]
\end{sageexample}
The first relation can be checked as follows:
\begin{sageexample}
    sage: ks3z[3, 3, 3, 3]*ks3z[2, 1]
    ks3[3, 3, 3, 3, 2, 1]
\end{sageexample}
\end{sagedemo}

\end{subsection}

\begin{subsection}{A dual-basis to $s_\lambda^{(k)}[X;t]$}

Recall from Section~\ref{sec:HLsymfunc} that $P_\lambda[X;t]$ is the dual 
basis to $Q'_\lambda[X;t]$ with respect to the $\left<~.~,~.~\right>$ scalar product.
Moreover, we have the expansion \eqref{tinvdef}
\begin{equation}
Q'_\mu[X;t] = \sum_{\lambda \vdash |\mu|, \lambda_1\leq k} 
K^{(k)}_{\lambda\mu}(t) \tilde s^{(k)}_\lambda[X;t]
\end{equation}
and
\begin{equation}
\tilde s^{(k)}_\lambda[X;t] = \sum_{\mu \vdash |\la|, \mu_1 \leq k} 
K^{(k)}_{\mu\lambda}(t)^{-1} Q'_\mu[X;t]~,
\end{equation}
where $K^{(k)}_{\lambda\mu}(t)$ $t$-enumerate weak tableaux
of shape $\mfc(\la)$ and weight $\mu$.
A $t$-generalization for dual $k$-Schur functions 
of Equation \eqref{eq:dualkschurmexp}
then comes out of this \cite{DalalMorse:2013} by duality, 
\begin{equation}\label{eq:dualkschurwitht}
\SS^{(k)}_\lambda[X;t] := \sum_{\mu\vdash |\la|, \mu_1\leq k} K^{(k)}_{\lambda\mu}(t) P_\mu[X;t]
\end{equation}
These elements clearly live in a space spanned by $\{ P_\la[X;t] \}_{\la_1\leq k}$.  In fact,
by triangularity considerations of the symmetric functions $P_\la[X;t]$, we have that
$$\La_t^{(k)} := 
\LL\left\{ \SS^{(k)}_\lambda[X;t] \right\}_{\la_1\leq k} =
\LL \left\{ P_\lambda[X;t] \right\}_{\la_1\leq k} =
\LL \left\{ s_\lambda[X] \right\}_{\la_1\leq k} =
\LL \left\{ m_\lambda[X] \right\}_{\la_1\leq k}~.$$
While this space is not closed under the usual product, it is closed under coproduct.  

Now recall from Section \ref{sec:HLsymfunc}, that 
$\left< Q'_\la[X;t], P_\mu[X;t] \right> = \delta_{\la\mu}$. Hence
\begin{equation*}
\begin{split}
	\left< s^{(k)}_\mu[X;t], \SS^{(k)}_\lambda[X;t] \right> 
	&= \sum_{\substack{\gamma \vdash |\mu|\\\gamma_1 \leq k}} K^{(k)}_{\gamma\mu}(t)^{-1} 
	\left< Q'_\gamma[X;t], \SS^{(k)}_\lambda[X;t] \right>\\
	&= \sum_{\substack{\gamma \vdash |\mu|\\\gamma_1 \leq k}} K^{(k)}_{\gamma\mu}(t)^{-1} 
	K^{(k)}_{\lambda\gamma}(t)
	= \delta_{\lambda\mu}~.
\end{split}
\end{equation*}
Therefore we can see that the elements 
$\left\{\SS_\la^{(k)}[X;t]\right\}_{\la_1 \leq k}$ are another $t$-analogue 
of the Schur functions which,
by triangularity considerations, live in the linear span of $\LL\{ m_\la : \la_1 \leq k \}$
and are dual to $\left\{ s_\la^{(k)}[X;t] \right\}_{\la_1 \leq k}$ with respect
to the usual scalar product.

Just as with the space $\La^t_{(k)}$, the linear span of
the dual elements for the $k$-Schur functions
is a subspace and not an algebra with respect to the usual product.
We can however make it an algebra, by introducing a product 
\begin{equation}\label{eq:newtproduct}
\SS^{(k)}_\nu[X;t] \cdot^t \SS^{(k)}_\mu[X;t] := 
\sum_{\lambda} C_{\nu\mu}^{\lambda(k)}(t) \SS^{(k)}_\la[X;t]\;,
\end{equation}
where the coefficients $C_{\nu\mu}^{\lambda(k)}(t)$ are precisely those defined by
Equation \eqref{coproductexp}.  The coefficients $C^{\la(k)}_{\mu\nu}(t)$ are 
discussed in Section \ref{sec:positivecoprod}
and they are conjectured to be polynomials in $t$ with non-negative integer coefficients.
To be clear, we take as definition that the space is
closed under a product where the structure coefficients are
\begin{equation}\label{eq:prodcoeffs}
C_{\nu\mu}^{\lambda(k)}(t) = 
\left< \Delta(s^{(k)}_\la[X;t]), \SS^{(k)}_\nu[X;t] \SS^{(k)}_\mu[Y;t]\right>~.
\end{equation}

There are several ways of computing a more explicit formula for these coefficients,
but let us consider one that can be found by expanding the elements
$s^{(k)}_\la[X;t]$ and $\SS^{(k)}_\la[X;t]$ in the Schur basis.
Since
\begin{equation}\label{eq:ksexpansion}
s^{(k)}_\la[X;t] = \sum_{\mu \vdash |\la|} 
\left< s^{(k)}_\la[X;t], s_\mu[X] \right> s_\mu[X]
\end{equation}
and
\begin{equation}\label{eq:dksexpansion}
\SS^{(k)}_\la[X;t] = \sum_{\mu \vdash |\la|} 
\left< \SS^{(k)}_\la[X;t], s_\mu[X] \right> s_\mu[X]\;,
\end{equation}
a formula for these coefficients is found by combining~\eqref{eq:prodcoeffs},
\eqref{eq:ksexpansion}, and~\eqref{eq:dksexpansion} to obtain
\begin{equation}
\label{equation.C}
C_{\nu\mu}^{\lambda(k)}(t) = \sum_{\substack{\theta\vdash|\la|\\
\tau\vdash|\nu|,\gamma\vdash|\mu|}}
c^{\theta}_{\tau\gamma}
\left< s^{(k)}_\la[X;t], s_\theta[X] \right>
\left< \SS^{(k)}_\nu[X;t], s_\tau[X] \right>
\left< \SS^{(k)}_\mu[X;t], s_\gamma[X] \right>
\end{equation}
where the $c^\la_{\mu\nu}$ are the Littlewood--Richardson coefficients previously discussed.

Little is known about this product and it would be useful to understand it in terms
of another basis.  However, it is possible to compute these coefficients as elements
of a quotient algebra or, as we do here, define a projection operator and notice that
the product $\cdot^t$ is simply the usual product followed by a projection into the
space $\La_t^{(k)}$.

\begin{prop} \label{prop:prodproj}
Let $\Theta^{(k)}$ be a projection from $\La$ to $\La_t^{(k)}$, the space spanned by
functions dual to the $k$-Schur functions, defined by
$\Theta^{(k)}( P_\la[X;t] ) = P_\la[X;t]$ if $\la_1 \leq k$, 
and $\Theta^{(k)}( P_\la[X;t] ) = 0$
if $\la_1 > k$, then 
\begin{equation}
\SS^{(k)}_\nu[X;t] \cdot^t \SS^{(k)}_\mu[X;t]
= \Theta^{(k)}( \SS^{(k)}_\nu[X;t] \SS^{(k)}_\mu[X;t] )\;.
\end{equation}
\end{prop}
\begin{proof}
In order to prove this we need to show that the coefficient of $\SS^{(k)}_\la[X;t]$ in
the expression $\Theta^{(k)}( \SS^{(k)}_\nu[X;t] \SS^{(k)}_\mu[X;t] )$ is equal to
$C^{\la(k)}_{\nu\mu}(t)$.  Since $s^{(k)}_\la[X;t]$ is in the linear span
of elements $\{ Q'_\la[X;t] \}_{\la_1 \leq k}$, we can conclude that
the coefficient of $\SS^{(k)}_\la[X;t]$ in $\Theta^{(k)}( \SS^{(k)}_\nu[X;t] \SS^{(k)}_\mu[X;t] )$
is equal to
\begin{equation}\label{eq:dualstructurecoeffs}
\left< \Theta^{(k)}( \SS^{(k)}_\nu[X;t] \SS^{(k)}_\mu[X;t] ), s^{(k)}_\la[X;t] \right> =
\left< \SS^{(k)}_\nu[X;t] \SS^{(k)}_\mu[X;t], s^{(k)}_\la[X;t] \right>~.
\end{equation}
We can use Equations~\eqref{eq:ksexpansion} and~\eqref{eq:dksexpansion} to expand the
right hand side so that it is equal to
\begin{align*}
&\sum_{\substack{\theta\vdash|\la|\\
\tau\vdash|\nu|,\gamma\vdash|\mu|}} \left< s_\tau s_\gamma, s_\theta \right>
\left< \SS^{(k)}_\nu[X;t], s_\tau[X] \right> \left< \SS^{(k)}_\mu[X;t], s_\gamma[X] \right>
\left< s^{(k)}_\la[X;t], s_\theta[X] \right>\\
=&\sum_{\substack{\theta\vdash|\la|\\
\tau\vdash|\nu|,\gamma\vdash|\mu|}} c^\theta_{\tau\gamma}
\left< \SS^{(k)}_\nu[X;t], s_\tau[X] \right> \left< \SS^{(k)}_\mu[X;t], s_\gamma[X] \right>
\left< s^{(k)}_\la[X;t], s_\theta[X] \right> = C^{\la(k)}_{\nu\mu}(t)
\end{align*}
by~\eqref{equation.C}.
This shows that the coefficients that appear in the $t$-product in Equation~\eqref{eq:newtproduct} 
are the same that appear by usual multiplication followed by a projection by $\Theta^{(k)}$.
\end{proof}

\begin{sagedemo}
We can compute the coefficients $C_{\la\mu}^{\nu(k)}(t)$ as structure coefficients of
the dual basis elements $\SS^{(k)}_\la[X;t]$.  These elements are implemented in \Sage
in a space representing the quotient of the ring of symmetric functions $\Lambda$ 
by the ideal generated by the Hall-Littlewood symmetric functions 
$P_\lambda[X;t]$ with $\lambda_1>k$.
\begin{sageexample}
    sage: Sym = SymmetricFunctions(QQ["t"].fraction_field())
    sage: Q3 = Sym.kBoundedQuotient(3)
    sage: dks = Q3.dual_k_Schur()
    sage: dks[2, 1, 1]*dks[3, 2, 1]
    (t^7+t^6)*dks3[2, 1, 1, 1, 1, 1, 1, 1, 1] 
    + (t^4+t^3+t^2)*dks3[2, 2, 2, 1, 1, 1, 1] 
    + (t^3+t^2)*dks3[2, 2, 2, 2, 1, 1] 
    + (t^5+2*t^4+2*t^3+t^2)*dks3[2, 2, 2, 2, 2] 
    + (t^5+2*t^4+t^3)*dks3[3, 1, 1, 1, 1, 1, 1, 1] 
    + (2*t^5+3*t^4+4*t^3+3*t^2+t)*dks3[3, 2, 1, 1, 1, 1, 1] 
    + (2*t^2+t+1)*dks3[3, 2, 2, 1, 1, 1]
    + (t^4+3*t^3+4*t^2+3*t+1)*dks3[3, 2, 2, 2, 1]
    + (t^5+t^4+4*t^3+4*t^2+3*t+1)*dks3[3, 3, 1, 1, 1, 1]
    + (2*t^5+3*t^4+5*t^3+6*t^2+4*t+2)*dks3[3, 3, 2, 1, 1]
    + (t^4+t^3+3*t^2+2*t+1)*dks3[3, 3, 2, 2]
    + (t^5+3*t^4+3*t^3+4*t^2+2*t+1)*dks3[3, 3, 3, 1]
\end{sageexample}
\end{sagedemo}
\end{subsection}

\begin{subsection}{A product on $\La_{(k)}^t$}

What is interesting about the $\SS^{(k)}_\la[X;t]$ 
elements is that they are clearly closed under
the usual coproduct operation of the symmetric functions.  It is therefore natural
to consider the coefficients that appear in the coproduct of the
elements which are dual to the $k$-Schur functions as the structure constants
of the product of $k$-Schur functions.  

Define the coefficients $c^{\la(k)}_{\nu\mu}(t)$ by the coproduct formula
$$\Delta(\SS^{(k)}_\la[X;t]) = \sum_{\nu,\mu} c^{\la(k)}_{\nu\mu}(t) \SS^{(k)}_{\nu}[X;t]
\SS^{(k)}_{\mu}[Y;t]~.$$
We will give a more precise calculation of these coefficients below, but assuming
that they exist, we then define
\begin{equation}\label{eq:defkschurtprod}
s^{(k)}_\nu[X;t] \cdot_t s^{(k)}_\mu[X;t] := \sum_{\la \vdash |\nu|+|\mu|} c^{\la(k)}_{\nu\mu}(t)
s^{(k)}_{\la}[X;t]
\end{equation}
where
$$c^{\la(k)}_{\nu\mu}(t) := \left< \Delta(\SS^{(k)}_\la[X;t]), s^{(k)}_{\nu}[X;t]
s^{(k)}_{\mu}[Y;t]\right>~.$$

We can then derive from Equations \eqref{eq:ksexpansion} and \eqref{eq:dksexpansion},
that
\begin{equation}
c^{\la(k)}_{\nu\mu}(t) = \sum_{\substack{\theta\vdash|\la|\\
\tau\vdash|\nu|,\gamma\vdash|\mu|}} c^{\theta}_{\tau\gamma}
\left< \SS^{(k)}_\la[X;t], s_\theta[X] \right>
\left< s^{(k)}_{\nu}[X;t], s_\tau[X]\right>
\left< s^{(k)}_{\mu}[X;t], s_\gamma[X] \right>~.
\end{equation}

At this point it is possible to see that while it is not obvious what the
product structure on the $k$-Schur functions should be, if we take this to be the definition
then, as in the case with the product $\cdot^t$, the product on $k$-Schur functions
can also be realized as a projection of the usual product.

\begin{prop}  Let $\Theta_{(k)}$ be a projection from the space $\La$ to the space
$\La^t_{(k)}$ spanned by the $k$-Schur functions defined by $\Theta_{(k)}( Q'_\la[X;t] ) = Q'_\la[X;t]$
if $\la_1 \leq k$ and $\Theta_{(k)}( Q'_\la[X;t] ) = 0$, then
$$s^{(k)}_\nu[X;t] \cdot_t s^{(k)}_\mu[X;t] = \Theta_{(k)}( s^{(k)}_\nu[X;t] s^{(k)}_\mu[X;t] )~.$$
\end{prop}

\begin{proof}
The proof proceeds exactly as it did in Proposition~\ref{prop:prodproj}.  We compute that
the coefficient of $s^{(k)}_\la[X;t]$ in $\Theta_{(k)}( s^{(k)}_\nu[X;t] s^{(k)}_\mu[X;t] )$ is
\begin{align*}
&\left< \Theta_{(k)}( s^{(k)}_\nu[X;t] s^{(k)}_\mu[X;t] ), \SS^{(k)}_\la[X;t] \right>
= \left<s^{(k)}_\nu[X;t] s^{(k)}_\mu[X;t], \SS^{(k)}_\la[X;t] \right>\\
=& \sum_{\substack{\theta\vdash|\la|\\
\tau\vdash|\nu|,\gamma\vdash|\mu|}}
c^{\theta}_{\tau\gamma} \left< s^{(k)}_\nu[X;t], s_\tau[X] \right>
\left< s^{(k)}_\mu[X;t], s_\gamma[X] \right>
\left< \SS^{(k)}_\la[X;t], s_\theta[X] \right>
=c^{\la(k)}_{\nu\mu}(t)~.
\end{align*}
So we see again that the structure coefficients in the definition of the $t$-product
in Equation~\eqref{eq:defkschurtprod} are exactly those that occur by taking the usual product and then
projecting using the map $\Theta_{(k)}$.
\end{proof}

It was discussed in Section~\ref{sec:teq1positive} that the coefficients $c^{\la(k)}_{\mu\nu}(1)$
are non-negative integers (conjecturally for certain definitions), but in the following 
example we will see that $c^{\la(k)}_{\mu\nu}(t)$
are not generally elements of $\NN[t]$.  This makes us believe that this product is not `the'
product to define on the space $\La^t_{(k)}$ (if there is such a product).  At least in certain
cases, the operators $\Bop_\lambda$ from Equation~\eqref{eq:opersum} 
also provide a means of defining a $t$-analogue of multiplication
and in light of Equation~\eqref{eq:tpieri}, it seems possible that there is some other
$t$-product on $\Lambda_{(k)}^t$ which is the right one to consider on this space.

\begin{sagedemo} We show how a product of $k$-Schur functions can be computed under the
projection $\Theta_{(k)}$.  The product is first computed in the $Q'_\mu[X;t]$ basis and
then the projection is computed by restricting the support to those partitions whose
parts are less than or equal to $k$, which is 2 in this example.
\begin{sageexample}
    sage: ks2 = SymmetricFunctions(QQ["t"]).kschur(2)
    sage: HLQp = SymmetricFunctions(QQ["t"]).hall_littlewood().Qp()
    sage: ks2( (HLQp(ks2[1,1])*HLQp(ks2[1])).restrict_parts(2) )                               
    ks2[1, 1, 1] + (-t+1)*ks2[2, 1]
\end{sageexample}
The coefficient $t-1$ should appear as the coefficient of $\SS^{(2)}_{1} \otimes \SS^{(2)}_{11}$
in the expansion of the coproduct $\Delta(\SS^{(2)}_{21})$.  We can calculate this using
\Sage using the following commands.
\begin{sageexample}
    sage: dks = SymmetricFunctions(QQ["t"]).kBoundedQuotient(2).dks()
    sage: dks[2,1].coproduct()
    dks2[] # dks2[2, 1] + (-t+1)*dks2[1] # dks2[1, 1] + 
     dks2[1] # dks2[2] + (-t+1)*dks2[1, 1] # dks2[1] +
     dks2[2] # dks2[1] + dks2[2, 1] # dks2[]
\end{sageexample}
\end{sagedemo}

\end{subsection}

\begin{subsection}{A representation theoretic model of $k$-Schur functions}
\label{sec:reptheory}

The {\it Frobenius map} is a map from $S_m$-modules to symmetric functions of degree $m$
which takes an irreducible module indexed by the partition $\lambda$ to the Schur function also 
indexed by the partition $\lambda \vdash m$.  Here we denote this map by $\F$ with 
$\F( V_\lambda ) = s_\lambda$, where $V_\lambda$ is an irreducible $S_m$-module.

The parameter $t$ in the ring of symmetric functions represents a grading and we can define
for a graded module $V = \bigoplus_{d \geq 0} V_d$, 
$$\F_t(V) = \sum_{d \geq 0} t^d \F( V_d)~.$$

For example, if we consider the ring of polynomials $\CC[a_1, a_2, \ldots, a_m]$ as a module graded 
by the degree in the variables $a_i$, then $\F_t( \CC[a_1, a_2, \ldots, a_m] ) = h_m\left[ \frac{X}{1-t} \right]$.
In particular, we also have for any irreducible $S_m$-module $V_\lambda$ that the tensor product  
$V_\lambda \otimes  \CC[a_1, a_2, \ldots, a_m]$ is an $S_m$-module, where $S_m$ acts diagonally 
on the tensors and 
$$\F_t( V_\lambda \otimes  \CC[a_1, a_2, \ldots, a_m] ) = s_\lambda\left[ \frac{X}{1-t} \right]~.$$

Mark Haiman and Li-Chung Chen~\cite{CH:2008} defined $\M^{(k)}$ to be the category of graded 
finitely generated $\CC[a_1, a_2, \ldots, a_m] \ast S_m$-modules $V$
such that $V$ has an $S_m$-equivariant $\CC[a_1, a_2, \ldots, a_m]$ free resolution using only
the $V_\lambda \otimes \CC[a_1, a_2, \ldots, a_m]$ with $\la_1 \leq k$.  Then they conjecture the 
following.
\begin{conj}  The modules which are irreducible in $\M^{(k)}$
have images under the map $\F_t$ which are equal to
$\omega s_\lambda^{(k)}[X;t]$ for $\lambda \vdash m$ and $\la_1 \leq k$.
\end{conj}

We say that two partitions $\la$ and $\mu$ of the same size are called {\em skew-linked} if there 
exists a skew partition $\gamma\slash\tau$ such that $\la_i$ is the number of cells in the $i^{th}$ row 
of $\gamma\slash\tau$ and $\mu_i'$ is the number of cells in the $i^{th}$ column of $\gamma\slash\tau$.  
The definition of skew-linked is not associated with a particular value of $k$, but we have seen 
examples of partitions which are skew-linked through the $(k+1)$-cores.  We always have that $\la$ 
and $(\la^{\omega_k})'$ are skew-linked since the skew partition representing the cells of $\mfc_k(\la)$ with 
hook less than $k+1$ has $\la_i$ cells in row $i$ and the transpose of $\mfc_k(\la)$ has $\la^{\omega_k}_i$ 
cells with hook less than $k+1$ in row $i$.  But there are other examples of pairs of partitions which 
are skew-linked. Of course, if $\la$ and $\mu$ are skew-linked, then $\mu'$ and $\la'$ are skew-linked.

\begin{example}
For a small example, consider that the partition $(2,2,1)$ is skew-linked to $(2,2,1)$, $(3,2)$, $(4,1)$ 
and $(5)$ through the skew partitions $(2,2,1)$, $(3,2,1)/(1)$, $(4,2,1)/(2)$, and $(5,3,1)/(3,1)$ respectively.

For a larger example, the partition $(4,4,2,2,2,1)$ is skew linked to $(6,4,2,2,1)$ because the skew partition $(6,4,2,2,2,1)/(2)$ has
rows given by $(4,4,2,2,2,1)$ and columns given by the partition $(5,4,2,2,1,1)~.$
\begin{center}
\begin{tabular}{ccc}\squaresize=9pt
$\young{\cr&\cr&\cr&\cr&&&\cr&&&\cr}$&$\longleftrightarrow$&$\young{\cr&\cr&\cr&\cr&&&\cr\blk&\blk&&&&\cr}$\cr
&&$\updownarrow$\cr
&&$\young{\cr&\cr&\cr&&&\cr&&&&&\cr}$
\end{tabular}
\end{center}
Notice that $(6,4,2,2,1)$ is not the conjugate of the $k$-conjugate of $(4,4,2,2,2,1)$ for any $k$.
\end{example}

\begin{theorem} (L.-C. Chen) 
If $\la$ is skew-linked to $\mu$ by a skew partition $\gamma/\tau$ with $\tau \vdash d$, 
then $V_\la$ occurs with multiplicity $1$ at degree 
$d$  in $V_\mu \otimes \CC[a_1, a_2, \ldots, a_m]$ and does not appear at a lower degree.
\end{theorem}

As a corollary they construct a module by moding out  the module $V_\mu \otimes \CC[a_1, a_2, \ldots, a_m]$ 
by the space generated by all modules $V_\gamma$  which are at degree $d$ which are not $V_\lambda$.
In the particular case when $\la$ is a $k$-bounded partition, $\la'$ is skew-linked to $\mu = \la^{\omega_k}$ 
and the module $V_\la$ is conjectured to have Frobenius image equal to $\omega s_\la^{(k)}[X;t]$.

\begin{conj}
For a $k$-bounded partition $\lambda$, let $d$ be the number of cells in $\mfc_k(\la)$ which have a 
hook length greater than $k$. Let $W$ be the module which is generated by $V_{\lambda^{\omega_k}}$ 
as an $S_m \ast \CC[a_1, a_2, \ldots, a_m]$-module and cogenerated by the copy of $V_{\lambda'}$ 
of degree $d$. Then $\F_t( W ) = \omega s_\la^{(k)}[X;t]$.
\end{conj}

Haiman and Li-Chung Chen~\cite{CH:2008} note that this module is only conjectured to lie within 
the category of $\M^{(k)}$.

\end{subsection}

\begin{subsection}{From Pieri to $K$-theoretic $k$-Schur functions}
\label{sec:Kkpieri}

As mentioned in the introduction, a trend in Schubert calculus 
is to generalize the classical setup.  The replacement of cohomology 
by $K$-theory is a particularly fruitful variation.
Lascoux and Sch\"utzenberger introduced the Grothendieck
polynomials in \cite{LS:Schub} as representatives for the
$K$-theory classes determined by structure sheaves of Schubert
varieties.  Grothendieck polynomials have since been connected to 
representation theory and algebraic geometry and combinatorics is
again at the forefront 
(e.g. \cite{Dem,KK:K,Lascoux:1990,FK:K}).
For example, the stable Grothendieck polynomials $G_\lambda$
are inhomogeneous
symmetric polynomials whose lowest homogeneous degree component
is a Schur function.  Buch proved in~\cite{Buc} that they are the
weight generating functions
\begin{equation}
\label{defgroth}
G_\lambda = \sum_{T~\text{set-valued}\atop \shape(T)=\lambda}
(-1)^{|\lambda|-|\weight(T)|}\,x^{\weight(T)} \,,
\end{equation}
of tableaux called set-valued tableaux.  Such
a tableau $T$ is a filling of each cell in a shape 
with a set of integers, where a set $X$ below (west of) $Y$ 
satisfies $\max X<(\leq)\min Y$.  The weight of  $T$ is 
$\alpha$ where $\alpha_i$ is the number of cells in $T$ containing an $i$.
Pieri rules are given in \cite{Lenart:2000}  
in terms of binomial numbers and a generalization for Yamanouchi tableaux
gives \cite{Buc} a combinatorial rule for the structure
constants.

Ideas in $k$-Schur theory extend to the inhomogeneous setting
providing combinatorial tools that apply to torus-equivariant $K$-theory 
of the affine Grassmannian of $SL_{k+1}$.
Similar to the development described in Section~\ref{sec:weaksection},
a close study of a Pieri rule and its iteration is carried out in \cite{Morse},
leading to a family of affine set-valued tableaux that are
in bijection with elements of the affine nil-Hecke algebra.
These simultaneously generalize set-valued tableaux {\it and}
weak $k$-tableaux.

These tableaux are defined along the lines described in
Remark~\ref{def:standardizedtab}; the semi-standard case is
given by putting conditions on the reading words of the
standard case.  Recall from Section~\ref{sec:weak II} that $T_{\leq x}$ 
is the subtableau obtained 
by deleting all letters larger than $x$ from $T$ and note that this is well-defined
for a set-valued tableau $T$.
A standard {\it affine set-valued tableau} $T$ of degree $n$ is then
defined as a set-valued filling such that, for each $1\leq x\leq n$,
$\shape(T_{\leq x})$ is a core and the cells containing an
$x$ form the set of all removable corners
of $T_{\leq x}$ with the same residue.

\begin{example}
\label{exstdsvktab}
With $k=2$, the standard affine set-valued tableaux of degree 5 with shape
$\mfc(2,1,1)=(3,1,1)$ are
\begin{equation}
{\tiny
{{\tableau*[lcY]{\{3,4\}_1\cr\{2\}_2\cr
\{1\}_0&\{3,4\}_1&\{5\}_2 }}} \quad
{{\tableau*[lcY]{\{3,5\}_1\cr\{2\}_2\cr
\{1\}_0&\{3\}_1&\{4\}_2 }}} \quad
{{\tableau*[lcY]{\{5\}_1\cr\{4\}_2\cr
\{1,2\}_0&\{3\}_1&\{4\}_2 }}} \quad
{{\tableau*[lcY]{\{4\}_1\cr\{3\}_2\cr
\{1,2\}_0&\{4\}_1&\{5\}_2 }}} \quad
{{\tableau*[lcY]{\{4\}_1\cr\{2,3\}_2\cr
\{1\}_0&\{4\}_1&\{5\}_2 }}}}
\end{equation}
\begin{equation}
\label{ssasvtab}
{\tiny
{{\tableau*[lcY]{\{5\}_1\cr\{4\}_2\cr
\{1\}_0&\{2,3\}_1&\{4\}_2 }}} \quad
{{\tableau*[lcY]{\{5\}_1\cr\{3,4\}_2\cr
\{1\}_0&\{2\}_1&\{3,4\}_2 }}} \quad
{{\tableau*[lcY]{\{4\}_1\cr\{3\}_2\cr
\{1\}_0&\{2\}_1&\{3,5\}_2 }}} }
\,\raisebox{-1cm}{.}
\end{equation}
\end{example}

For the semi-standard case, first note that a set-valued tableau $T$ 
of weight $\alpha$ is a standard set-valued tableau
with increasing reading words in the alphabets
$\mathcal A_{\alpha,x}$ of \eqref{readingwordcond}, where the reading 
word is obtained by reading
letters from a cell in decreasing order (and as usual, cells are
taken from top to bottom and left to right).
Since letters in standard affine set-valued tableaux
can occur with multiplicity, the {\it lowest reading word} in $\mathcal A$
-- reading the lowest occurrence of the letters in $\mathcal A$
from top to bottom and left to right -- is used.  Again,
letters in the same cell are read in decreasing order.
In Example~\ref{exstdsvktab}, the lowest reading words
in $\{1,\ldots,5\}$ are $21435,52134,52134,32145,32145,51324,51243,41253$.

The affine $K$-theoretic generalization of a tableau with weight $\alpha$
is then given, for any $k$-bounded composition $\alpha$, by 
a standard affine set-valued tableau of degree $|\alpha|$ where, for each 
$1\leq x\leq \ell(\alpha)$,
\begin{enumerate}
\item the lowest reading word in
$\mathcal A_{\alpha,x}$ is increasing
\item the letters of $\mathcal A_{\alpha,x}$ occupy $\alpha_x$ distinct
residues
\item the letters of $\mathcal A_{\alpha,x}$ form a horizontal strip.
\end{enumerate}

\begin{example}
The affine set-valued tableaux in \eqref{ssasvtab} of Example~\ref{exstdsvktab}
all have weight $(2,1,1,1)$ and shape $(3,1,1)=\mfc(2,1,1)$, for $k=2$.
\end{example}

It is proven in~\cite{Morse} that the weight generating functions
of affine set-valued tableaux
\begin{equation}
G_\lambda^{(k)} = \sum_{T~\text{affine set-valued tableau}\atop
\shape(T)=\mfc(\lambda)}
\!  \!
\!  \!
(-1)^{|\lambda|+|\weight(T)|}\,x^{\weight(T)}
\end{equation}
 are Schubert representatives for $K$-theory
of the affine Grassmannian of $SL_{k+1}$ called affine 
stable Grothendieck polynomials \cite{Lam:2006,LSS}.
These reduce to Grothendieck polynomials for large $k$ and 
the term of lowest degree in $G_\lambda^{(k)}$ is
the dual $k$-Schur function $\SS_\lambda^{(k)}$.
As in the $k$-Schur set-up, these do not form a self-dual basis.
Instead, the dual to $\{G_\lambda^{(k)}\}_{\lambda_1\leq k}$ is an inhomogeneous 
basis $\{g_\lambda^{(k)}\}_{\lambda_1\leq k}$ for $\Lambda_{(k)}$.

The set-up discussed in Sections~\ref{sec:colstrict} and \ref{sec:weaksection} is extended 
in~\cite{Morse} and gives affine $K$-theoretic properties for this basis  
such as Pieri rules and an
analog to property \eqref{eq:kconjwt} for an inhomogeneous
involution $\Omega$;
\begin{equation}
\label{invoO}
\Omega g_\lambda^{(k)} = g_{\lambda^{\omega_k}}^{(k)}\,.
\end{equation}
Among some of the open combinatorial problems, it remains to develop 
an affine $K$-theoretic set-up in the dual world along the lines discussed 
in Section~\ref{sec:strongsection};
finding the Pieri rules for $\{G_\lambda^{(k)}\}_{\lambda_1\leq k}$ and 
giving a weight-generating formulation for $\{g_\lambda^{(k)}\}_{\lambda_1\leq k}$.
In addition, there are properties of Grothendieck polynomials that
conjecturally extend to these new bases.  For example, there is
a combinatorial expansion of Grothendieck polynomials into Schur functions 
described by a family of skew tableaux \cite{Lenart:2000}.
It is conjectured that $G_\lambda^{(k)}$ and $g_\lambda^{(k)}$
have combinatorial expansions in terms of dual $k$-Schur functions
and $k$-Schur functions, respectively.  
Noncommutative versions of the affine stable Grothendieck polynomials
and $\{g_\lambda^{(k)}\}_{\lambda_1\leq k}$ are given in~\cite{LSS}, where geometric
aspects of these bases are also explored.  Further conjectures relating 
to these bases can be found in \cite{LSS,Morse}.

\begin{sagedemo} The basis $\{g_\lambda^{(k)}\}_{\lambda_1\leq k}$ has been implemented in
\Sage and so it is possible to begin experimenting with these combinatorial
problems.
\begin{sageexample}
sage: Sym = SymmetricFunctions(QQ)
sage: Sym3 = Sym.kBoundedSubspace(3,t=1)
sage: Kks3 = Sym3.K_kschur()      
sage: s = Sym.s()
sage: m = Sym.m()
sage: s(Kks3[3,1])
s[3] + s[3, 1] + s[4]
sage: m(Kks3[3,1])
m[1, 1, 1] + 4*m[1, 1, 1, 1] + m[2, 1] + 3*m[2, 1, 1] + 2*m[2, 2] 
+ m[3] + 2*m[3, 1] + m[4]
sage: ks3 = Sym3.kschur()
sage: ks3(Kks3[3,1])
ks3[3] + ks3[3, 1]
sage: Kks3[3,1]*Kks3[2]
-Kks3[3, 1, 1] - Kks3[3, 2] + Kks3[3, 2, 1] + Kks3[3, 3]
sage: Kks3[3,1].coproduct()
Kks3[] # Kks3[3, 1] - Kks3[1] # Kks3[2] + Kks3[1] # Kks3[2, 1]
+ 2*Kks3[1] # Kks3[3] + Kks3[1, 1] # Kks3[2] - Kks3[2] # Kks3[1]
+ Kks3[2] # Kks3[1, 1] + 2*Kks3[2] # Kks3[2] + Kks3[2, 1] # Kks3[1]
+ 2*Kks3[3] # Kks3[1] + Kks3[3, 1] # Kks3[]
\end{sageexample}
Since the elements $\{G_\lambda^{(k)}\}_{\lambda_1\leq k}$ are an infinite sum of
elements of degree greater than or equal to $|\lambda|$ and we
do not know their algebra structure, we cannot currently 
represent them as we do other bases in \Sage.  However, a
preliminary implementation of this basis does exist that allows
one to compute the elements up to a a given degree.  
\begin{sageexample}
sage: SymQ3 = Sym.kBoundedQuotient(3,t=1)
sage: G1 = SymQ3.AffineGrothendieckPolynomial([1],6) 
sage: G2 = SymQ3.AffineGrothendieckPolynomial([2],6)                                                  
sage: (G1*G2).lift().scalar(Kks3[3,1])  
-1
\end{sageexample}
Notice how the coproduct applied to $g_{(3,1)}^{(3)}$ agrees with the product structure
in this one calculation
since we see that coefficient of $G_{(3,1)}^{(3)}$ in the product of $G_{(2)}^{(3)}$ and
$G_{(1)}^{(3)}$ is equal to the coefficient of $g_{(2)}^{(3)} \otimes g_{(1)}^{(3)}$ in
$\Delta(g_{(3,1)}^{(3)})$.

We may also list all the terms which appear in a product of elements
of $\{G_\lambda^{(k)}\}_{\lambda_1\leq k}$.  This is sufficient
for generating data for a Pieri rule on $\{G_\lambda^{(k)}\}_{\lambda_1\leq k}$ since
we can compute products and use the duality with the $\{g_\lambda^{(k)}\}_{\lambda_1\leq k}$
basis to determine the structure coefficients up to a degree higher 
than what appears in our product.
\begin{sageexample}
sage: G31 = SymQ3.AffineGrothendieckPolynomial([3,1], 8)
sage: for d in range(5,10):              
....:     for la in Partitions(d,max_part=3):
....:         c = (G1*G31).lift().scalar(Kks3(la))
....:         if c!=0:
....:             print la, c
....:
[3, 2] 2
[3, 1, 1] 1
[3, 3] -1
[3, 2, 1] -2
[3, 3, 1] 1
[3, 2, 1, 1] 1
[3, 3, 1, 1] -1
\end{sageexample}
This \Sage calculation shows that no terms indexed by partitions
of size 9 appear in the product
of $G_{(3,1)}^{(3)}$ and $G_{(1)}^{(3)}$, and since we believe that this indicates 
highest degree of terms which will appear in our product will be $8$, then
$$G_{(3,1)}^{(3)} G_{(1)}^{(3)} = 2 G_{(3,2)}^{(3)} + G_{(3,1,1)}^{(3)} - G_{(3,3)}^{(3)} - 2 G_{(3,2,1)}^{(3)}
+ G_{(3,3,1)}^{(3)} + G_{(3,2,1,1)}^{(3)} - G_{(3,3,1,1)}^{(3)}~. $$
\end{sagedemo}
\end{subsection}

\end{section}

\begin{section}{Duality between the weak and strong orders} \label{sec:strongweakduality}
In this section we consider a $k$-analogue of the Cauchy identity and the Robinson--Schensted--Knuth
(RSK) algorithm (or insertion algorithm). The RSK algorithm provides a bijection between
permutations and pairs of tableaux satisfying certain conditions. As shown in~\cite{LLMS:2006}
this can be generalized to the affine setting.

\begin{subsection}{$k$-analogue of the Cauchy identity}

In the algebra of symmetric functions (or rather polynomials) an important identity is
the {\it Cauchy identity}, stating that
\begin{equation}\label{eq:cauchykernel}
\prod_{i,j=1}^m \frac{1}{1-x_i y_j} = \prod_{j=1}^m \sum_{r \geq 0} h_r[X_m] y_j^r
= \sum_{\la} h_\la[X_m] m_\la[Y_m] = \sum_{\la} s_\la[X_m] s_\la[Y_m]\;,
\end{equation}
where the last two sums run over all partitions $\la$.

Although there is an algebraic proof of this identity that follows from calculations in Section~\ref{sec:symfunc},
there is also a direct combinatorial proof of this result.  Recall from Equation~\eqref{eq:schurtableauxexp}
that the Schur function is equal to
\begin{equation} \label{eq:tableauxsum}
	s_\la[X_m] = \sum_T {\bf x}^T \;,
\end{equation}
where the sum is over all semi-standard tableaux of shape $\la$ and ${\bf x}^T$ is a monomial 
which represents the product over $i$ of $x_i$ raised to the number of $i$ in the tableaux.  

The Schur functions are in fact characterized as the unique basis which satisfies 
Equation~\eqref{eq:cauchykernel} and which is triangularly related to the monomial symmetric functions.
Notice that
\begin{equation} \label{equation.kernel}
\prod_{i,j=1}^m \frac{1}{1-x_i y_j}  = \sum_{M} \prod_{i,j=1}^m
(x_i y_j)^{m_{ij}}= \sum_{M} ({\bf xy})^M \;,
\end{equation}
where the sum is over all $m \times m$ matrices $M = ( m_{ij} )_{1 \leq i,j \leq m}$ with 
non-negative integer entries $m_{ij}$. There is a famous bijection due to Robinson--Schensted--Knuth 
\cite{Knuth, Robinson:1938, Schensted:1961} (see for instance \cite{Sagan} for a clear exposition of the 
bijection) that identifies such a matrix $M$ to a biword in the alphabet of letters $\pchoose{r}{s}$ with 
$1 \leq r,s \leq m$. The bijection maps these biwords to pairs $(P, Q)$, where $P$ and $Q$ are 
semi-standard tableaux of the same shape and ${\bf x}^P {\bf y}^Q = \prod_{i,j=1}^m (x_i y_j)^{m_{ij}}$.
As a consequence we conclude that
\begin{equation}\label{eq:secondCauchy}
	\prod_{i,j=1}^m \frac{1}{1-x_i y_j} = \sum_{M} ({\bf xy})^M = 
	\sum_{(P,Q)} {\bf x}^P {\bf y}^Q 
	= \sum_{\la} \left(\sum_{\shape(P) = \la} {\bf x}^P \right)\left(\sum_{\shape(Q) = \la} {\bf y}^Q\right)~.
\end{equation}
Since $\sum_{\shape(P) = \la} {\bf x}^P$ is triangularly related to the monomial basis,
Equation \eqref{eq:tableauxsum} must hold by comparing Equations~\eqref{eq:cauchykernel}
and~\eqref{eq:secondCauchy}. In this section we consider the generalization of the Cauchy
identity and the RSK algorithm to $k$-Schur functions and their dual basis at $t=1$.

By taking the coefficient of $x_1 x_2 \cdots x_m y_1 y_2 \cdots y_m$
in Equation~\eqref{eq:cauchykernel}, we find the identity
$$m! = \sum_{\la \vdash m} f_\la^2\;,$$
where $f_\la$ is the number of standard tableaux of shape $\la$.  This may be seen as an algebraic
formulation of the more standard presentation of the RSK algorithm on permutations, 
namely, there is a bijection between permutations $\pi$ and pairs of tableaux $(P,Q)$,
where $P$ and $Q$ are standard tableaux of the same shape.

\vskip .2in

In Equation~\eqref{eq:dualkschurmexp}, we stated that $\SS^{(k)}_\la = \sum_\mu K^{(k)}_{\la\mu} m_\mu$,
where $K^{(k)}_{\la\mu}$ is equal to the number of weak tableaux of shape $\mfc(\la)$ and weight $\mu$.  
The collective results in~\cite{Lam:2006,LM:2008} show that 
\begin{equation}
\SS^{(k)}_\la[X_m] = \sum_T {\bf x}^T\;,
\end{equation}
where the sum is over all weak tableaux $T$ of shape $\mfc(\la)$ in the weight
$\{1, 2, \ldots, m\}$.  

Now consider the following $k$-bounded analogue of the kernel~\eqref{equation.kernel} given by
\begin{equation}\label{eq:klevelkernel}
	\prod_{j=1}^m (1 + h_1[X_m] y_j + h_2[X_m] y_j^2 + \cdots + h_k[X_m] y_j^k) 
	= \sum_{\la : \la_1 \leq k} h_\la[X_m] m_\la[Y_m] 
	= \sum_{M} ({\bf xy})^M,
\end{equation}
where the sum on the right hand side of the equation is over all $k$-bounded matrices 
$M = (m_{ij})_{1\leq i,j \leq n}$, whose entries are non-negative integers and satisfy
$\sum_{i=1}^n m_{ij} \leq k$, and $({\bf xy})^M = \prod_{i,j} (x_i y_j)^{m_{ij}}$.

In~\cite{LLMS:2006} the authors provide a bijection between the set of $k$-bounded matrices and
pairs of tableaux $(P, Q)$ such that $P$ is a strong tableau and $Q$ is a weak tableau that are both 
of the same $(k+1)$-core shape.  This is done by introducing an insertion algorithm which generalizes 
that of the RSK bijection (and reduces to RSK case when $k \geq \sum_{i,j} m_{ij}$).

Here we provide an exposition of a special case of this bijection, namely between
permutation matrices (with entries in $\{0,1\}$ with a single $1$ in each row and column) and
pairs of tableaux $(P,Q )$, where $P$ is a strong standard tableau and $Q$ is a weak standard tableau.

The bijection in~\cite{LLMS:2006} (which is called `affine insertion' in analogy with
RSK-insertion) shows that the duality of the weak and strong functions can be expressed 
through the duality of the kernel in~\eqref{eq:klevelkernel} and hence
\begin{equation}\label{eq:strongweakkernel}
\prod_{j=1}^m (1 + h_1[X_m] y_j + h_2[X_m] y_j^2 + \cdots + h_k[X_m] y_j^k) = 
\sum_{\la : \la_1 \leq k} \Strong_{\mfc(\la)}[X_m] \Weak_{\mfc(\la)}[Y_m],
\end{equation}
where the functions are defined as
\begin{equation}
\Weak_{\kappa}[X_m] = \sum_{T} {\bf x}^T
\end{equation}
with the sum over all weak tableaux of shape $\kappa$ (a $(k+1)$-core) and 
\begin{equation}
\Strong_{\kappa}[X_m] = \sum_{T} {\bf x}^T
\end{equation}
with the sum is over all strong tableaux of shape $\kappa$. In both cases ${\bf x}^T$ represents a
monomial associated to the weight.
The equality $\SS^{(k)}_\la[X_m] = \Weak_{\mfc(\la)}[X_m]$ with dual
$k$-Schur functions relies on a permutation action on the weight
which can be found in~\cite{LM:2007}.
The equality $s^{(k)}_\la[X_m] = \Strong_{\mfc(\la)}[X_m]$ follows by a duality argument after
showing that the functions $\Strong_{\kappa}[X_m]$ form a basis~\cite{LLMS:2006}.

Let $f^{\mathrm{weak}}_\kappa$ be the number of standard weak tableaux of shape $\kappa$ 
and $f^{\mathrm{strong}}_\kappa$ be the number of standard strong tableaux of shape $\kappa$,
 where $\kappa$ is a $(k+1)$-core.
If we take the coefficient of $ x_1 x_2 \cdots x_m y_1 y_2 \cdots y_m$
in Equation \eqref{eq:strongweakkernel}, 
we find the following combinatorial result
\begin{equation}\label{eq:kRSKenumeration}
m! = \sum_{\la : \la_1 \leq k} f^{\mathrm{strong}}_{\mfc(\la)} f^{\mathrm{weak}}_{\mfc(\la)}\;,
\end{equation}
where the sum is over $k$-bounded partitions of $\la$ of $m$.
This formula can be seen as a manifestation of a bijection between
the set of permutations $\sigma$ of $S_m$ and pairs of tableaux $(P^{(k)},Q^{(k)})$, where
$P^{(k)}$ is a strong standard tableau and $Q^{(k)}$ is a weak standard tableau.

\end{subsection}

\begin{subsection}{A brief introduction to Fomin's growth diagrams}
Affine insertion is proved using Fomin's growth diagrams~\cite{Fomin:1995}
which is a tool of presenting insertion algorithms on graded posets via certain local rules.
To put the affine insertion algorithm in context, we give a brief
presentation of the usual RSK algorithm between permutations and
pairs of standard tableaux in terms of growth diagrams in order to
show how the algorithms compare.
In Section~\ref{subsection.affine insertion}  we will demonstrate how the algorithm can be 
generalized to the bijection
which explains Equation \eqref{eq:kRSKenumeration}.  The treatment
we present in this section follows roughly the way of viewing
Fomin's growth diagrams that is presented in \cite[Section 5.2]{Sagan}
with a few modifications in orientation.

We begin with an $n \times n$ permutation matrix corresponding to a permutation
$\pi$ (we use the convention that row $i$ has a $1$ in column $\pi_i$) and
convert it into a pair of standard tableaux of the same shape.
To do this we draw an $n \times n$ lattice of squares and label the vertices of 
this lattice with partitions and the centers of these squares with the entries of the permutation
matrix.  At the start of the procedure we begin by labeling only
the first row and first column of vertices with empty
partitions and fill in the rest of the diagram by a recursive procedure using a set of {\it local rules}.

To describe the RSK algorithm we describe a `local rule'
which is a bijection between two types of arrays.

\noindent
{\bf Case 1}:
\begin{center}
\begin{tabular}{ccc}
$\lambda$&$\rightarrow$&$\mu$\cr
$\downarrow$&0&\cr
$\nu$&&\cr
\end{tabular}
\hskip .5in
$\longleftrightarrow$
\hskip .5in
\begin{tabular}{ccc}
&&$\mu$\cr
&&$\downarrow$\cr
$\nu$&$\rightarrow$&$\gamma$\cr
\end{tabular}
\end{center}
\begin{enumerate}
\item[(a)] If $\lambda = \mu = \nu$, then $\gamma = \nu$.
\item[(b)] If $\mu \neq \nu$, then $\gamma = \mu \cup \nu$.
\item[(c)] If $\lambda$ is strictly contained in $\mu = \nu$, then if $\mu$
is obtained from $\lambda$ by adding a cell in row $i$, then $\gamma$ is
obtained from $\mu$ by adding a cell in row $i+1$.
\end{enumerate}

\noindent
{\bf Case 2}:
\begin{center}
\begin{tabular}{ccc}
$\lambda$&$\rightarrow$&$\mu$\cr
$\downarrow$&1&\cr
$\nu$&&\cr
\end{tabular}
\hskip .5in
$\longleftrightarrow$
\hskip .5in
\begin{tabular}{ccc}
&&$\mu$\cr
&&$\downarrow$\cr
$\nu$&$\rightarrow$&$\gamma$\cr
\end{tabular}
\end{center}
This case can only occur when $\lambda = \mu = \nu$, and then $\gamma$ is obtained from $\mu$ 
by adding a cell in the first row.

By successively applying these local rules, the growth diagram
is filled in until it is an $(n+1) \times (n+1)$ array of partitions.  
Because each of the rules we apply is
a bijection, we need only remember the last row and last column of the array and the rest
of the table can be recovered by applying the local rule in reverse.
The last row of this table is a sequence of partitions each of which differ by a single cell
and so can be interpreted as a standard tableau which agrees with the insertion tableau of $\pi$.  
The last column of table is also a sequence of partitions each of which differ by a single cell;
the corresponding standard tableau agrees with the recording tableau corresponding to the 
permutation $\pi$.

The important thing to notice is that the local rules can be reversed.  For this
reason if we are given just the pair of tableaux that represent the last row and the last
column of the table, it is possible to reconstruct the entire table and
hence the permutation matrix.

A beautiful feature of the growth diagram perspective on the RSK insertion algorithm is also
that it makes it completely manifest that interchanging the insertion tableau $P$ and recording tableau
$Q$ inverts the permutation $\pi$. This can be seen by interchanging the rows and columns of the array,
which inverts the permutation matrix and interchanges $P$ and $Q$.

\begin{example} \label{example:rskgrowth} 
Let us consider the permutation $4132$ as a running example.  
We begin with a row and a column of 5 empty partitions and the
entries of the permutation matrix in an array pictured below.
\squaresize=7pt
\begin{center}
\begin{tabular}{ccccccccc}
$\emptyset$&$\rightarrow$&$\emptyset$&$\rightarrow$&$\emptyset$&$\rightarrow$&$\emptyset$&$\rightarrow$&$\emptyset$\cr
$\downarrow$&$0$&&$0$&&$0$&&$1$&\cr
$\emptyset$&&&&&&&&\cr
$\downarrow$&$1$&&$0$&&$0$&&$0$&\cr
$\emptyset$&&&&&&&&\cr
$\downarrow$&$0$&&$0$&&$1$&&$0$&\cr
$\emptyset$&&&&&&&&\cr
$\downarrow$&$0$&&$1$&&$0$&&$0$&\cr
$\emptyset$&&&&&&&&\cr
\end{tabular}
\end{center}
The local rules may be applied at first only in one place:
\begin{center}
\begin{tabular}{ccccccccc}
$\emptyset$&$\rightarrow$&$\emptyset$&$\rightarrow$&$\emptyset$&$\rightarrow$&$\emptyset$&$\rightarrow$&$\emptyset$\cr
$\downarrow$&$0$&$\downarrow$&$0$&&$0$&&$1$&\cr
$\emptyset$&$\rightarrow$&$\emptyset$&&&&&&\cr
$\downarrow$&$1$&&$0$&&$0$&&$0$&\cr
$\emptyset$&&&&&&&&\cr
$\downarrow$&$0$&&$0$&&$1$&&$0$&\cr
$\emptyset$&&&&&&&&\cr
$\downarrow$&$0$&&$1$&&$0$&&$0$&\cr
$\emptyset$&&&&&&&&\cr
\end{tabular}
\end{center}
In successive steps, the local rules may be applied in each corner.
\begin{center}
\begin{tabular}{ccccccccc}
$\emptyset$&$\rightarrow$&$\emptyset$&$\rightarrow$&$\emptyset$&$\rightarrow$&$\emptyset$&$\rightarrow$&$\emptyset$\cr
$\downarrow$&$0$&$\downarrow$&$0$&$\downarrow$&$0$&&$1$&\cr
$\emptyset$&$\rightarrow$&$\emptyset$&$\rightarrow$&$\emptyset$&&&&\cr
$\downarrow$&$1$&$\downarrow$&$0$&&$0$&&$0$&\cr
$\emptyset$&$\rightarrow$&\young{\cr}&&&&&&\cr
$\downarrow$&$0$&&$0$&&$1$&&$0$&\cr
$\emptyset$&&&&&&&&\cr
$\downarrow$&$0$&&$1$&&$0$&&$0$&\cr
$\emptyset$&&&&&&&&\cr
\end{tabular}
\end{center}

\begin{center}
\begin{tabular}{ccccccccc}
$\emptyset$&$\rightarrow$&$\emptyset$&$\rightarrow$&$\emptyset$&$\rightarrow$&$\emptyset$&$\rightarrow$&$\emptyset$\cr
$\downarrow$&$0$&$\downarrow$&$0$&$\downarrow$&$0$&$\downarrow$&$1$&\cr
$\emptyset$&$\rightarrow$&$\emptyset$&$\rightarrow$&$\emptyset$&$\rightarrow$&$\emptyset$&&\cr
$\downarrow$&$1$&$\downarrow$&$0$&$\downarrow$&$0$&&$0$&\cr
$\emptyset$&$\rightarrow$&\young{\cr}&$\rightarrow$&\young{\cr}&&&&\cr
$\downarrow$&$0$&$\downarrow$&$0$&&$1$&&$0$&\cr
$\emptyset$&$\rightarrow$&\young{\cr}&&&&&&\cr
$\downarrow$&$0$&&$1$&&$0$&&$0$&\cr
$\emptyset$&&&&&&&&\cr
\end{tabular}
\end{center}
By continuing to apply the local rules, we arrive at the following completed table:
\begin{center}
\begin{tabular}{cccccccccc}
$\emptyset$&$\rightarrow$&$\emptyset$&$\rightarrow$&$\emptyset$&$\rightarrow$&$\emptyset$&$\rightarrow$&$\emptyset$&\cr
$\downarrow$&$0$&$\downarrow$&$0$&$\downarrow$&$0$&$\downarrow$&$1$&$\downarrow$&\cr
$\emptyset$&$\rightarrow$&$\emptyset$&$\rightarrow$&$\emptyset$&$\rightarrow$&$\emptyset$&$\rightarrow$&\young{\cr}&\cr
$\downarrow$&$1$&$\downarrow$&$0$&$\downarrow$&$0$&$\downarrow$&$0$&$\downarrow$&\cr
$\emptyset$&$\rightarrow$&\young{\cr}&$\rightarrow$&\young{\cr}&$\rightarrow$&\young{\cr}&$\rightarrow$&\young{\cr\cr}&$Q$\cr
$\downarrow$&$0$&$\downarrow$&$0$&$\downarrow$&$1$&$\downarrow$&$0$&$\downarrow$&\cr
$\emptyset$&$\rightarrow$&\young{\cr}&$\rightarrow$&\young{\cr}&$\rightarrow$&\young{&\cr}&$\rightarrow$&\young{\cr&\cr}&\cr
$\downarrow$&$0$&$\downarrow$&$1$&$\downarrow$&$0$&$\downarrow$&$0$&$\downarrow$&\cr
$\emptyset$&$\rightarrow$&$\young{\cr}$&$\rightarrow$&\young{&\cr}&$\rightarrow$&\young{\cr&\cr}&$\rightarrow$&\young{\cr\cr&\cr}&\cr
&&&&$P$&&&&\cr
\end{tabular}
\end{center}

\squaresize=7pt
Now in order to reconstruct this entire table, we need only remember the two standard tableaux 
$(P,Q)={\Big(} \scriptsize \young{4\cr3\cr1&2\cr},\young{4\cr2\cr1&3\cr} {\Big)}$ which
correspond to the last row and the last column of the table. This indeed agrees with the usual
insertion tableau $P$ and recording tableau $Q$ when row-inserting $4132$ (see for 
example~\cite{Sagan}).
\end{example}

\end{subsection}

\begin{subsection}{Affine insertion}
\label{subsection.affine insertion}

Now that we have presented Fomin's growth diagrams as a tool for understanding RSK, 
we will give the local rules necessary to understand $k$-affine insertion for permutations.
This is the algorithm presented in~\cite{LLMS:2006}.
Parts of these local rules use operations which are described in Sections~\ref{section:partitionscores}, 
\ref{subsection.weak order} and \ref{subsection.strong order}
(in particular the action of $s_i$, the notion of strong and weak cover, and the vocabulary of content and residue).

We have stripped down the algorithm presented in~\cite{LLMS:2006} in hopes of making it clearer by
having fewer details to follow.  The algorithm presented in~\cite{LLMS:2006} is slightly more general 
because it includes an additional rule that generalizes the bijection from $k$-bounded non-negative 
integer matrices to pairs tableaux of the same shape, the first is
a strong semi-standard tableau, the second is a weak 
semi-standard tableau.  This more general bijection
is sufficient to show Equation~\eqref{eq:strongweakkernel}. Here we are pairing down the presentation
rule to only demonstrate how Equation~\eqref{eq:kRSKenumeration} works.

To do this we construct again a table that will be a growth diagram, where
there is an $(n+1) \times (n+1)$ array of $(k+1)$-cores and between these entries we
put the $n \times n$ permutation matrix.  

There is an additional piece of information which is recorded in this matrix besides the shapes of the 
$k+1$-cores. The horizontal connectors in our growth diagrams keep track of (possibly empty) strong 
covers and the vertical connectors keep track of (possibly empty) weak covers.  If there are two 
$(k+1)$-cores that are adjacent in the same row, $\tau \rightarrow \kappa$, then either
$\tau = \kappa$ or $\kappa$ covers $\tau$ in the strong order. In the second case, we also need to
mark one of the connected components of $\kappa \slash \tau$ in the diagram of $\kappa$ or keep track of
this marking on the arrow $\xrightarrow{c}$, where $c$ represents the content of the diagonal of
the marked cell.  When working with the diagram in the examples below we only record this information 
by marking a cell of $\kappa\slash\tau$ within core $\kappa$ for compactness of notation.

We begin with the first row and column of this array consisting of empty cores only (and hence there are 
no markings necessary). Given a corner of the table that is partially filled in, we complete the rest with 
the following local rules.

\noindent
{\bf Case 1}:
\begin{center}
\begin{tabular}{ccc}
$\tau$&$\xrightarrow{c}$&$\kappa$\cr
$\downarrow$&0&\cr
$\theta$&&\cr
\end{tabular}
\hskip .5in
$\longleftrightarrow$
\hskip .5in
\begin{tabular}{ccc}
&&$\kappa$\cr
&&$\downarrow$\cr
$\theta$&$\xrightarrow{c'}$&$\zeta$\cr
\end{tabular}
\end{center}
Try to apply (a)-(c) in this order. If a case does not apply, proceed to the next case.
\begin{enumerate}
\item[(a)] If $\tau = \kappa$, then $\zeta = \theta$ and neither $\kappa$ nor $\zeta$ will be marked; 
if $\tau = \theta$, then $\zeta = \kappa$ and $c'=c$.
\item[(b)] If $\kappa \slash \tau$ is not contained in $\theta \slash \tau$, let $r$ be the residue of the cells 
$\theta\slash\tau$ $\pmod{k+1}$ (there is exactly one). In this case $\zeta$ is $s_r$ applied to $\kappa$.  
One cell of $\kappa$ on the diagonal with content $c$ is marked. In $\zeta$ mark the component that
has an overlap with the marked ribbon in $\kappa$ (alternatively, 
if $c$ does not have residue $r$, then $c'=c$; otherwise $c'$ is on one diagonal higher than $c$).
\item[(c)] If $\tau$ is strictly contained in $\kappa = \theta$, let $c'$ be the content of the first
diagonal which is weakly to the left of the marked ribbon of $\kappa/\tau$ and is an addable cell of $\kappa$. 
Then $\zeta = s_{c'} \kappa$ and the marked cell of $\zeta$ is on the diagonal with content $c'$.
\end{enumerate}

\noindent
{\bf Case 2}:
\begin{center}
\begin{tabular}{ccc}
$\tau$&$\rightarrow$&$\kappa$\cr
$\downarrow$&1&\cr
$\theta$&&\cr
\end{tabular}
\hskip .5in
$\longleftrightarrow$
\hskip .5in
\begin{tabular}{ccc}
&&$\kappa$\cr
&&$\downarrow$\cr
$\theta$&$\xrightarrow{c'}$&$\zeta$\cr
\end{tabular}
\end{center}
This case can only occur when $\tau = \kappa = \theta$, and then $\zeta$ is $s_{\tau_1}$ applied to 
$\tau$ (the effect of adding a cell in the first row of $\tau$, but as a $(k+1)$-core).
A marking $c'$ is added in last cell of the first row of $\zeta$.

\begin{example}  Let us compute the growth diagram for the matrix corresponding
to the permutation 4132.
As in Example~\ref{example:rskgrowth}, we begin our growth diagram with the first row
and column consisting of empty cores.

\squaresize=7pt
\begin{center}
\begin{tabular}{ccccccccc}
$\emptyset$&$\rightarrow$&$\emptyset$&$\rightarrow$&$\emptyset$&$\rightarrow$&$\emptyset$&$\rightarrow$&$\emptyset$\cr
$\downarrow$&$0$&&$0$&&$0$&&$1$&\cr
$\emptyset$&&&&&&&&\cr
$\downarrow$&$1$&&$0$&&$0$&&$0$&\cr
$\emptyset$&&&&&&&&\cr
$\downarrow$&$0$&&$0$&&$1$&&$0$&\cr
$\emptyset$&&&&&&&&\cr
$\downarrow$&$0$&&$1$&&$0$&&$0$&\cr
$\emptyset$&&&&&&&&\cr
\end{tabular}
\end{center}

Below is the growth diagram for $k=1$:
\begin{center}
\begin{tabular}{ccccccccc}
$\emptyset$&$\rightarrow$&$\emptyset$&$\rightarrow$&$\emptyset$&$\rightarrow$&$\emptyset$&$\rightarrow$&$\emptyset$\cr
$\downarrow$&$0$&$\downarrow$&$0$&$\downarrow$&$0$&$\downarrow$&$1$&$\downarrow$\cr
$\emptyset$&$\rightarrow$&$\emptyset$&$\rightarrow$&$\emptyset$&$\rightarrow$&$\emptyset$&$\rightarrow$&$\young{\ast\cr}$\cr
$\downarrow$&$1$&$\downarrow$&$0$&$\downarrow$&$0$&$\downarrow$&$0$&$\downarrow$\cr
$\emptyset$&$\rightarrow$&$\young{\ast\cr}$&$\rightarrow$&$\young{\cr}$&$\rightarrow$&$\young{\cr}$&$\rightarrow$&$\young{\ast\cr&\cr}$\cr
$\downarrow$&$0$&$\downarrow$&$0$&$\downarrow$&$1$&$\downarrow$&$0$&$\downarrow$\cr
$\emptyset$&$\rightarrow$&$\young{\ast\cr}$&$\rightarrow$&$\young{\cr}$&$\rightarrow$&$\young{\cr&\ast\cr}$&$\rightarrow$&$\young{\ast\cr&\cr&&\cr}$\cr
$\downarrow$&$0$&$\downarrow$&$1$&$\downarrow$&$0$&$\downarrow$&$0$&$\downarrow$\cr
$\emptyset$&$\rightarrow$&$\young{\ast\cr}$&$\rightarrow$&$\young{\cr&\ast\cr}$&$\rightarrow$&$\young{\cr&\ast\cr&&\cr}$&$\rightarrow$&$\young{\ast\cr&\cr&&\cr&&&\cr}$\cr
\end{tabular}
\end{center}
Note that since the value of $k$ is too small, Case 1(b) is not used.
\squaresize=9pt
By reading the last row of this table we can encode it as a single strong tableau.  When $k=1$ there is only one weak tableau of shape $(m, m-1, \ldots, 2, 1)$.  The last row and column of this table can then be encoded as
\squaresize=12pt
$$\Big( \young{4^\ast\cr3&4\cr2&3^\ast&4\cr1^\ast&2^\ast&3&4\cr},
\young{4\cr3&4\cr2&3&4\cr1&2&3&4\cr} \Big)~.$$

\squaresize=7pt
For $k=2$ and starting with the same permutation the situation is a little more complicated:
\begin{center}
\begin{tabular}{ccccccccc}
$\emptyset$&$\rightarrow$&$\emptyset$&$\rightarrow$&$\emptyset$&$\rightarrow$&$\emptyset$&$\rightarrow$&$\emptyset$\cr
$\downarrow$&$0$&$\downarrow$&$0$&$\downarrow$&$0$&$\downarrow$&$1$&$\downarrow$\cr
$\emptyset$&$\rightarrow$&$\emptyset$&$\rightarrow$&$\emptyset$&$\rightarrow$&$\emptyset$&$\rightarrow$&$\young{\ast\cr}$\cr
$\downarrow$&$1$&$\downarrow$&$0$&$\downarrow$&$0$&$\downarrow$&$0$&$\downarrow$\cr
$\emptyset$&$\rightarrow$&$\young{\ast\cr}$&$\rightarrow$&$\young{\cr}$&$\rightarrow$&$\young{\cr}$&$\rightarrow$&\young{\ast\cr\cr}\cr
$\downarrow$&$0$&$\downarrow$&$0$&$\downarrow$&$1$&$\downarrow$&$0$&$\downarrow$\cr
$\emptyset$&$\rightarrow$&$\young{\ast\cr}$&$\rightarrow$&$\young{\cr}$&$\rightarrow$&$\young{&\ast\cr}$&$\rightarrow$&$\young{\cr\ast\cr&\cr}$\cr
$\downarrow$&$0$&$\downarrow$&$1$&$\downarrow$&$0$&$\downarrow$&$0$&$\downarrow$\cr
$\emptyset$&$\rightarrow$&$\young{\ast\cr}$&$\rightarrow$&$\young{&\ast\cr}$&$\rightarrow$&$\young{\ast\cr&&\cr}$&$\rightarrow$&$\young{\ast\cr\cr&&\cr}$\cr
\end{tabular}
\end{center}

Now it is necessary to apply all four rules to fill in the growth diagram.
The first time that Case 1 (b) occurs is constructing the last entry in the fourth row of cores
(the last entry of the third row of the permutation matrix).  

\begin{center}
\begin{tabular}{ccc}
$\young{\cr}$&$\rightarrow$&$\young{\ast\cr\cr}$\cr
$\downarrow$&0&\cr
$\young{&\ast\cr}$&&\cr
\end{tabular}
\hskip .5in
$\longleftrightarrow$
\hskip .5in
\begin{tabular}{ccc}
&&$\young{\ast\cr\cr}$\cr
&&$\downarrow$\cr
$\young{&\ast\cr}$&$\rightarrow$&$\zeta = s_1 \young{\ast\cr\cr} = 
\young{\cr\ast\cr&\cr}$\cr
\end{tabular}
\end{center}

If we just record the last row as a strong tableau and the last column as a weak tableau, we have the follow pair:
\squaresize=12pt
$$\Big( \young{4^*\cr3^*\cr1^*&2^*&3\cr}, \young{3\cr2\cr1&3&4\cr}\Big)$$
and this pair of tableaux is sufficient to reconstruct the entire table.

\squaresize=7pt
For $k=3$, the case is similar to the $k=2$ case in that there are examples where all four rules are applied
in order to construct the table:
\begin{center}
\begin{tabular}{ccccccccc}
$\emptyset$&$\rightarrow$&$\emptyset$&$\rightarrow$&$\emptyset$&$\rightarrow$&$\emptyset$&$\rightarrow$&$\emptyset$\cr
$\downarrow$&$0$&$\downarrow$&$0$&$\downarrow$&$0$&$\downarrow$&$1$&$\downarrow$\cr
$\emptyset$&$\rightarrow$&$\emptyset$&$\rightarrow$&$\emptyset$&$\rightarrow$&$\emptyset$&$\rightarrow$&$\young{\ast\cr}$\cr
$\downarrow$&$1$&$\downarrow$&$0$&$\downarrow$&$0$&$\downarrow$&$0$&$\downarrow$\cr
$\emptyset$&$\rightarrow$&$\young{\ast\cr}$&$\rightarrow$&$\young{\cr}$&$\rightarrow$&$\young{\cr}$&$\rightarrow$&$\young{\ast\cr\cr}$\cr
$\downarrow$&$0$&$\downarrow$&$0$&$\downarrow$&$1$&$\downarrow$&$0$&$\downarrow$\cr
$\emptyset$&$\rightarrow$&$\young{\ast\cr}$&$\rightarrow$&$\young{\cr}$&$\rightarrow$&$\young{&\ast\cr}$&$\rightarrow$&$\young{\ast\cr&\cr}$\cr
$\downarrow$&$0$&$\downarrow$&$1$&$\downarrow$&$0$&$\downarrow$&$0$&$\downarrow$\cr
$\emptyset$&$\rightarrow$&$\young{\ast\cr}$&$\rightarrow$&$\young{&\ast\cr}$&$\rightarrow$&$\young{\ast\cr&\cr}$&$\rightarrow$&$\young{\ast\cr\cr&&\cr}$\cr
\end{tabular}
\end{center}
Because we can reconstruct the table from the last row and column of this table,
it is necessary to keep track only of the strong tableau representing the last row and
the weak tableau representing the last column.  This is represented by the pair,
\squaresize=12pt
$$\Big( \young{4^*\cr3^*\cr1^*&2^*&4\cr}, \young{4\cr2\cr1&3&4\cr}\Big).$$
\end{example}

\begin{example} \label{ex:kRSKformeq3}
In the following table we have presented permutations of $3$ (corresponding to 
permutation matrices) which
are in bijection with pairs of weak and strong tableaux for $k= 1,2$ and $3$.  When $k=3$, the
tableaux are in bijection with pairs of standard tableaux of the same
shape by dropping the markings in the strong tableaux.

{\squaresize=9pt
\begin{center}
\begin{tabular}{|c|c|c|c|}
\hline
$\sigma$&$(P^{(1)}, Q^{(1)})$&$(P^{(2)}, Q^{(2)})$&$(P^{(\infty)}, Q^{(\infty)})$\cr
\hline
$123$ &
$\scriptsize\Big( \young{3\cr2&3\cr1^\ast&2^\ast&3^\ast\cr}, \young{3\cr2&3\cr1&2&3\cr}\Big)$&
$\scriptsize\Big( \young{3\cr1^\ast&2^\ast&3^\ast\cr}, \young{3\cr1&2&3\cr }\Big)$&
$\scriptsize\Big( \young{1^*&2^*&3^*\cr}, \young{1&2&3\cr} \Big)$\cr
\hline
$132$&
$\scriptsize\Big ( \young{3\cr2&3^\ast\cr1^\ast&2^\ast&3\cr}, \young{3\cr2&3\cr1&2&3\cr} \Big)$&
$\scriptsize\Big( \young{3^\ast\cr1^\ast&2^\ast&3\cr}, \young{3\cr1&2&3\cr } \Big)$&
$\scriptsize\Big( \young{3^*\cr1^*&2^*\cr}, \young{3\cr1&2\cr} \Big)$\cr
\hline
$213$&
$\scriptsize\Big ( \young{3\cr2^\ast&3\cr1^\ast&2&3^\ast\cr}, \young{3\cr2&3\cr1&2&3\cr} \Big)$&
$\scriptsize\Big( \young{3\cr2^\ast\cr1^\ast&3^\ast\cr}, \young{3\cr2\cr1&3\cr } \Big)$&
$\scriptsize\Big( \young{2^*\cr1^*&3^*\cr}, \young{2\cr1&3\cr} \Big)$\cr
\hline
$231$&
$\scriptsize\Big ( \young{3\cr2^\ast&3^\ast\cr1^\ast&2&3\cr}, \young{3\cr2&3\cr1&2&3\cr} \Big)$&
$\scriptsize\Big( \young{2^\ast\cr1^\ast&3&3^\ast\cr}, \young{3\cr1&2&3\cr } \Big)$&
$\scriptsize\Big( \young{2^*\cr1^*&3^*\cr}, \young{3\cr1&2\cr} \Big)$\cr
\hline
$312$&
$\scriptsize\Big ( \young{3^\ast\cr2&3\cr1^\ast&2^\ast&3\cr}, \young{3\cr2&3\cr1&2&3\cr} \Big)$&
$\scriptsize\Big( \young{3\cr3^\ast\cr1^\ast&2^\ast\cr}, \young{3\cr2\cr1&3\cr } \Big)$&
$\scriptsize\Big( \young{3^*\cr1^*&2^*\cr}, \young{2\cr1&3\cr} \Big)$\cr
\hline
$321$&
$\scriptsize\Big ( \young{3^\ast\cr2^\ast&3\cr1^\ast&2&3\cr}, \young{3\cr2&3\cr1&2&3\cr} \Big)$&
$\scriptsize\Big( \young{3^\ast\cr2^\ast\cr1^\ast&3\cr}, \young{3\cr2\cr1&3\cr } \Big)$&
$\scriptsize\Big( \young{3^*\cr2^*\cr1^*\cr}, \young{3\cr2\cr1\cr} \Big)$\cr
\hline
\end{tabular}
\end{center}
}
\end{example}
\begin{example}
In the following table we have permutations of $4$ (corresponding to 
permutation matrices) which
are in bijection with pairs of weak and strong tableaux for $k= 1,2$ and $3,4$.  When $k=4$, the
tableaux are in bijection with pairs of standard tableaux of the same
shape by dropping the markings in the strong tableaux.

{\squaresize=8pt
\begin{center}
\begin{tabular}{|c|c|c|c|c|}
\hline
$\sigma$&$(P^{(1)}, Q^{(1)})$&$(P^{(2)}, Q^{(2)})$&$(P^{(3)}, Q^{(3)})$&$(P^{(\infty)}, Q^{(\infty)})$\cr
\hline
$1234$ & 
$\scriptsize\Big( \young{4\cr3&4\cr2&3&4\cr1^*&2^*&3^*&4^*\cr}, \young{4\cr3&4\cr2&3&4\cr1&2&3&4\cr}\Big)$&
$\scriptsize\Big( \young{3&4\cr1^*&2^*&3^*&4^*\cr}, \young{3&4\cr1&2&3&4\cr}\Big)$&
$\scriptsize\Big( \young{4\cr1^*&2^*&3^*&4^*\cr}, \young{4\cr1&2&3&4\cr}\Big)$&
$\scriptsize\Big( \young{1^*&2^*&3^*&4^*\cr}, \young{1&2&3&4\cr}\Big)$\cr\hline
$1243$ & 
$\scriptsize\Big( \young{4\cr3&4\cr2&3&4^*\cr1^*&2^*&3^*&4\cr}, \young{4\cr3&4\cr2&3&4\cr1&2&3&4\cr}\Big)$&
$\scriptsize\Big( \young{3&4^*\cr1^*&2^*&3^*&4\cr}, \young{3&4\cr1&2&3&4\cr}\Big)$&
$\scriptsize\Big( \young{4^*\cr1^*&2^*&3^*&4\cr}, \young{4\cr1&2&3&4\cr}\Big)$&
$\scriptsize\Big( \young{4^*\cr1^*&2^*&3^*\cr}, \young{4\cr1&2&3\cr}\Big)$\cr\hline
$1324$ & 
$\scriptsize\Big( \young{4\cr3&4\cr2&3^*&4\cr1^*&2^*&3&4^*\cr}, \young{4\cr3&4\cr2&3&4\cr1&2&3&4\cr}\Big)$&
$\scriptsize\Big( \young{3^*&4\cr1^*&2^*&3&4^*\cr}, \young{3&4\cr1&2&3&4\cr}\Big)$&
$\scriptsize\Big( \young{4\cr3^*\cr1^*&2^*&4^*\cr}, \young{4\cr3\cr1&2&4\cr}\Big)$&
$\scriptsize\Big( \young{3^*\cr1^*&2^*&4^*\cr}, \young{3\cr1&2&4\cr}\Big)$\cr\hline
$1342$ & 
$\scriptsize\Big( \young{4\cr3&4\cr2&3^*&4^*\cr1^*&2^*&3&4\cr}, \young{4\cr3&4\cr2&3&4\cr1&2&3&4\cr}\Big)$&
$\scriptsize\Big( \young{3^*&4^*\cr1^*&2^*&3&4\cr}, \young{3&4\cr1&2&3&4\cr}\Big)$&
$\scriptsize\Big( \young{3^*\cr1^*&2^*&4&4^*\cr}, \young{4\cr1&2&3&4\cr}\Big)$&
$\scriptsize\Big( \young{3^*\cr1^*&2^*&4^*\cr}, \young{4\cr1&2&3\cr}\Big)$\cr\hline
$1423$ & 
$\scriptsize\Big( \young{4\cr3&4^*\cr2&3&4\cr1^*&2^*&3^*&4\cr}, \young{4\cr3&4\cr2&3&4\cr1&2&3&4\cr}\Big)$&
$\scriptsize\Big( \young{4^*\cr3\cr1^*&2^*&3^*\cr}, \young{4\cr3\cr1&2&3\cr}\Big)$&
$\scriptsize\Big( \young{4\cr4^*\cr1^*&2^*&3^*\cr}, \young{4\cr3\cr1&2&4\cr}\Big)$&
$\scriptsize\Big( \young{4^*\cr1^*&2^*&3^*\cr}, \young{3\cr1&2&4\cr}\Big)$\cr\hline
$1432$ & 
$\scriptsize\Big( \young{4\cr3&4^*\cr2&3^*&4\cr1^*&2^*&3&4\cr}, \young{4\cr3&4\cr2&3&4\cr1&2&3&4\cr}\Big)$&
$\scriptsize\Big( \young{4^*\cr3^*\cr1^*&2^*&3\cr}, \young{4\cr3\cr1&2&3\cr}\Big)$&
$\scriptsize\Big( \young{4^*\cr3^*\cr1^*&2^*&4\cr}, \young{4\cr3\cr1&2&4\cr}\Big)$&
$\scriptsize\Big( \young{4^*\cr3^*\cr1^*&2^*\cr}, \young{4\cr3\cr1&2\cr}\Big)$\cr\hline
$2134$ & 
$\scriptsize\Big( \young{4\cr3&4\cr2^*&3&4\cr1^*&2&3^*&4^*\cr}, \young{4\cr3&4\cr2&3&4\cr1&2&3&4\cr}\Big)$&
$\scriptsize\Big( \young{3\cr2^*\cr1^*&3^*&4^*\cr}, \young{3\cr2\cr1&3&4\cr}\Big)$&
$\scriptsize\Big( \young{4\cr2^*\cr1^*&3^*&4^*\cr}, \young{4\cr2\cr1&3&4\cr}\Big)$&
$\scriptsize\Big( \young{2^*\cr1^*&3^*&4^*\cr}, \young{2\cr1&3&4\cr}\Big)$\cr\hline
$2143$ & 
$\scriptsize\Big( \young{4\cr3&4\cr2^*&3&4^*\cr1^*&2&3^*&4\cr}, \young{4\cr3&4\cr2&3&4\cr1&2&3&4\cr}\Big)$&
$\scriptsize\Big( \young{4\cr3\cr2^*&4^*\cr1^*&3^*\cr}, \young{4\cr3\cr2&4\cr1&3\cr}\Big)$&
$\scriptsize\Big( \young{2^*&4^*\cr1^*&3^*\cr}, \young{2&4\cr1&3\cr}\Big)$&
$\scriptsize\Big( \young{2^*&4^*\cr1^*&3^*\cr}, \young{2&4\cr1&3\cr}\Big)$\cr\hline
$2314$ & 
$\scriptsize\Big( \young{4\cr3&4\cr2^*&3^*&4\cr1^*&2&3&4^*\cr}, \young{4\cr3&4\cr2&3&4\cr1&2&3&4\cr}\Big)$&
$\scriptsize\Big( \young{2^*&4\cr1^*&3&3^*&4^*\cr}, \young{3&4\cr1&2&3&4\cr}\Big)$&
$\scriptsize\Big( \young{4\cr2^*\cr1^*&3^*&4^*\cr}, \young{4\cr3\cr1&2&4\cr}\Big)$&
$\scriptsize\Big( \young{2^*\cr1^*&3^*&4^*\cr}, \young{3\cr1&2&4\cr}\Big)$\cr\hline
$2341$ & 
$\scriptsize\Big( \young{4\cr3&4\cr2^*&3^*&4^*\cr1^*&2&3&4\cr}, \young{4\cr3&4\cr2&3&4\cr1&2&3&4\cr}\Big)$&
$\scriptsize\Big( \young{2^*&4^*\cr1^*&3&3^*&4\cr}, \young{3&4\cr1&2&3&4\cr}\Big)$&
$\scriptsize\Big( \young{2^*\cr1^*&3^*&4&4^*\cr}, \young{4\cr1&2&3&4\cr}\Big)$&
$\scriptsize\Big( \young{2^*\cr1^*&3^*&4^*\cr}, \young{4\cr1&2&3\cr}\Big)$\cr\hline
$2413$ & 
$\scriptsize\Big( \young{4\cr3&4^*\cr2^*&3&4\cr1^*&2&3^*&4\cr}, \young{4\cr3&4\cr2&3&4\cr1&2&3&4\cr}\Big)$&
$\scriptsize\Big( \young{3\cr2^*\cr1^*&3^*&4^*\cr}, \young{4\cr3\cr1&2&3\cr}\Big)$&
$\scriptsize\Big( \young{2^*&4^*\cr1^*&3^*\cr}, \young{3&4\cr1&2\cr}\Big)$&
$\scriptsize\Big( \young{2^*&4^*\cr1^*&3^*\cr}, \young{3&4\cr1&2\cr}\Big)$\cr\hline
$2431$ & 
$\scriptsize\Big( \young{4\cr3&4^*\cr2^*&3^*&4\cr1^*&2&3&4\cr}, \young{4\cr3&4\cr2&3&4\cr1&2&3&4\cr}\Big)$&
$\scriptsize\Big( \young{4^*\cr2^*\cr1^*&3&3^*\cr}, \young{4\cr3\cr1&2&3\cr}\Big)$&
$\scriptsize\Big( \young{4^*\cr2^*\cr1^*&3^*&4\cr}, \young{4\cr3\cr1&2&4\cr}\Big)$&
$\scriptsize\Big( \young{4^*\cr2^*\cr1^*&3^*\cr}, \young{4\cr3\cr1&2\cr}\Big)$\cr\hline
\end{tabular}
\end{center}
\begin{center}
\begin{tabular}{|c|c|c|c|c|}
\hline
$\sigma$&$(P^{(1)}, Q^{(1)})$&$(P^{(2)}, Q^{(2)})$&$(P^{(3)}, Q^{(3)})$&$(P^{(\infty)}, Q^{(\infty)})$\cr
\hline
$3124$ & 
$\scriptsize\Big( \young{4\cr3^*&4\cr2&3&4\cr1^*&2^*&3&4^*\cr}, \young{4\cr3&4\cr2&3&4\cr1&2&3&4\cr}\Big)$&
$\scriptsize\Big( \young{3\cr3^*\cr1^*&2^*&4^*\cr}, \young{3\cr2\cr1&3&4\cr}\Big)$&
$\scriptsize\Big( \young{4\cr3^*\cr1^*&2^*&4^*\cr}, \young{4\cr2\cr1&3&4\cr}\Big)$&
$\scriptsize\Big( \young{3^*\cr1^*&2^*&4^*\cr}, \young{2\cr1&3&4\cr}\Big)$\cr\hline
$3142$ & 
$\scriptsize\Big( \young{4\cr3^*&4\cr2&3&4^*\cr1^*&2^*&3&4\cr}, \young{4\cr3&4\cr2&3&4\cr1&2&3&4\cr}\Big)$&
$\scriptsize\Big( \young{4\cr3\cr3^*&4^*\cr1^*&2^*\cr}, \young{4\cr3\cr2&4\cr1&3\cr}\Big)$&
$\scriptsize\Big( \young{3^*&4^*\cr1^*&2^*\cr}, \young{2&4\cr1&3\cr}\Big)$&
$\scriptsize\Big( \young{3^*&4^*\cr1^*&2^*\cr}, \young{2&4\cr1&3\cr}\Big)$\cr\hline
$3214$ & 
$\scriptsize\Big( \young{4\cr3^*&4\cr2^*&3&4\cr1^*&2&3&4^*\cr}, \young{4\cr3&4\cr2&3&4\cr1&2&3&4\cr}\Big)$&
$\scriptsize\Big( \young{3^*\cr2^*\cr1^*&3&4^*\cr}, \young{3\cr2\cr1&3&4\cr}\Big)$&
$\scriptsize\Big( \young{4\cr3^*\cr2^*\cr1^*&4^*\cr}, \young{4\cr3\cr2\cr1&4\cr}\Big)$&
$\scriptsize\Big( \young{3^*\cr2^*\cr1^*&4^*\cr}, \young{3\cr2\cr1&4\cr}\Big)$\cr\hline
$3241$ & 
$\scriptsize\Big( \young{4\cr3^*&4\cr2^*&3&4^*\cr1^*&2&3&4\cr}, \young{4\cr3&4\cr2&3&4\cr1&2&3&4\cr}\Big)$&
$\scriptsize\Big( \young{4\cr3^*\cr2^*&4^*\cr1^*&3\cr}, \young{4\cr3\cr2&4\cr1&3\cr}\Big)$&
$\scriptsize\Big( \young{3^*\cr2^*\cr1^*&4&4^*\cr}, \young{4\cr2\cr1&3&4\cr}\Big)$&
$\scriptsize\Big( \young{3^*\cr2^*\cr1^*&4^*\cr}, \young{4\cr2\cr1&3\cr}\Big)$\cr\hline
$3412$ & 
$\scriptsize\Big( \young{4\cr3^*&4^*\cr2&3&4\cr1^*&2^*&3&4\cr}, \young{4\cr3&4\cr2&3&4\cr1&2&3&4\cr}\Big)$&
$\scriptsize\Big( \young{3\cr3^*\cr1^*&2^*&4^*\cr}, \young{4\cr3\cr1&2&3\cr}\Big)$&
$\scriptsize\Big( \young{3^*&4^*\cr1^*&2^*\cr}, \young{3&4\cr1&2\cr}\Big)$&
$\scriptsize\Big( \young{3^*&4^*\cr1^*&2^*\cr}, \young{3&4\cr1&2\cr}\Big)$\cr\hline
$3421$ & 
$\scriptsize\Big( \young{4\cr3^*&4^*\cr2^*&3&4\cr1^*&2&3&4\cr}, \young{4\cr3&4\cr2&3&4\cr1&2&3&4\cr}\Big)$&
$\scriptsize\Big( \young{3^*\cr2^*\cr1^*&3&4^*\cr}, \young{4\cr3\cr1&2&3\cr}\Big)$&
$\scriptsize\Big( \young{3^*\cr2^*\cr1^*&4&4^*\cr}, \young{4\cr3\cr1&2&4\cr}\Big)$&
$\scriptsize\Big( \young{3^*\cr2^*\cr1^*&4^*\cr}, \young{4\cr3\cr1&2\cr}\Big)$\cr\hline
$4123$ & 
$\scriptsize\Big( \young{4^*\cr3&4\cr2&3&4\cr1^*&2^*&3^*&4\cr}, \young{4\cr3&4\cr2&3&4\cr1&2&3&4\cr}\Big)$&
$\scriptsize\Big( \young{4^*\cr3\cr1^*&2^*&3^*\cr}, \young{3\cr2\cr1&3&4\cr}\Big)$&
$\scriptsize\Big( \young{4\cr4^*\cr1^*&2^*&3^*\cr}, \young{4\cr2\cr1&3&4\cr}\Big)$&
$\scriptsize\Big( \young{4^*\cr1^*&2^*&3^*\cr}, \young{2\cr1&3&4\cr}\Big)$\cr\hline
$4132$ & 
$\scriptsize\Big( \young{4^*\cr3&4\cr2&3^*&4\cr1^*&2^*&3&4\cr}, \young{4\cr3&4\cr2&3&4\cr1&2&3&4\cr}\Big)$&
$\scriptsize\Big( \young{4^*\cr3^*\cr1^*&2^*&3\cr}, \young{3\cr2\cr1&3&4\cr}\Big)$&
$\scriptsize\Big( \young{4^*\cr3^*\cr1^*&2^*&4\cr}, \young{4\cr2\cr1&3&4\cr}\Big)$&
$\scriptsize\Big( \young{4^*\cr3^*\cr1^*&2^*\cr}, \young{4\cr2\cr1&3\cr}\Big)$\cr\hline
$4213$ & 
$\scriptsize\Big( \young{4^*\cr3&4\cr2^*&3&4\cr1^*&2&3^*&4\cr}, \young{4\cr3&4\cr2&3&4\cr1&2&3&4\cr}\Big)$&
$\scriptsize\Big( \young{4^*\cr3\cr2^*&4\cr1^*&3^*\cr}, \young{4\cr3\cr2&4\cr1&3\cr}\Big)$&
$\scriptsize\Big( \young{4\cr4^*\cr2^*\cr1^*&3^*\cr}, \young{4\cr3\cr2\cr1&4\cr}\Big)$&
$\scriptsize\Big( \young{4^*\cr2^*\cr1^*&3^*\cr}, \young{3\cr2\cr1&4\cr}\Big)$\cr\hline
$4231$ & 
$\scriptsize\Big( \young{4^*\cr3&4\cr2^*&3^*&4\cr1^*&2&3&4\cr}, \young{4\cr3&4\cr2&3&4\cr1&2&3&4\cr}\Big)$&
$\scriptsize\Big( \young{4^*\cr2^*\cr1^*&3&3^*\cr}, \young{3\cr2\cr1&3&4\cr}\Big)$&
$\scriptsize\Big( \young{4^*\cr2^*\cr1^*&3^*&4\cr}, \young{4\cr2\cr1&3&4\cr}\Big)$&
$\scriptsize\Big( \young{4^*\cr2^*\cr1^*&3^*\cr}, \young{4\cr2\cr1&3\cr}\Big)$\cr\hline
$4312$ & 
$\scriptsize\Big( \young{4^*\cr3^*&4\cr2&3&4\cr1^*&2^*&3&4\cr}, \young{4\cr3&4\cr2&3&4\cr1&2&3&4\cr}\Big)$&
$\scriptsize\Big( \young{4^*\cr3\cr3^*&4\cr1^*&2^*\cr}, \young{4\cr3\cr2&4\cr1&3\cr}\Big)$&
$\scriptsize\Big( \young{4\cr4^*\cr3^*\cr1^*&2^*\cr}, \young{4\cr3\cr2\cr1&4\cr}\Big)$&
$\scriptsize\Big( \young{4^*\cr3^*\cr1^*&2^*\cr}, \young{3\cr2\cr1&4\cr}\Big)$\cr\hline
$4321$ & 
$\scriptsize\Big( \young{4^*\cr3^*&4\cr2^*&3&4\cr1^*&2&3&4\cr}, \young{4\cr3&4\cr2&3&4\cr1&2&3&4\cr}\Big)$&
$\scriptsize\Big( \young{4^*\cr3^*\cr2^*&4\cr1^*&3\cr}, \young{4\cr3\cr2&4\cr1&3\cr}\Big)$&
$\scriptsize\Big( \young{4^*\cr3^*\cr2^*\cr1^*&4\cr}, \young{4\cr3\cr2\cr1&4\cr}\Big)$&
$\scriptsize\Big( \young{4^*\cr3^*\cr2^*\cr1^*\cr}, \young{4\cr3\cr2\cr1\cr}\Big)$\cr\hline
\end{tabular}
\end{center}
}
\end{example}

\end{subsection}

\begin{subsection}{The $t$-compatible affine insertion algorithm}
Since the expression on the left of Equation \eqref{eq:kRSKenumeration} is independent of $k$, we
realize that since $w \leftrightarrow (P^{(k)}, Q^{(k)}) \leftrightarrow (P^{(k+1)}, Q^{(k+1)})$,
where $P^{(k)}$ and $P^{(k+1)}$ are strong $k$ and $(k+1)$-tableaux
may help us to see how $s^{(k)}_\la$ can be expressed as a positive sum of elements
$s^{(k+1)}_\mu$. It may be possible to understand the expansion of $s^{(k)}_\la$ in
terms of Schur functions or in terms of $s^{(k+1)}_\la$ using this bijection, but this would impose certain 
conditions on the shapes of tableaux $P^{(k)}$ and $P^{(k+1)}$ that do not seem to hold.

In this section we present some evidence suggesting that one might hope for
an affine insertion bijection which has additional properties
that are not shared with the affine insertion algorithm of Section~\ref{subsection.affine insertion}.

There is also a $t$-analogue of Equation~\eqref{eq:strongweakkernel} which can be used
as a stronger guide for the combinatorics of an analogue of the Robinson--Schensted--Knuth
bijection. The coefficient of $m_{(1^m)}[X]$ in $Q'_{(1^m)}[X;t]$ is equal to the $t$-analogue of $m!$, 
$[m]_t! = (1-t)^{-m} \prod_{i=1}^m (1-t^i)$,
and hence we have
\begin{align}
[m]_t! &= \left< Q'_{(1^m)}[X;t], h_{(1^m)}[X] \right> = \sum_{\la \vdash m} K_{\la(1^m)}(t) \left< s_\la[X], h_{(1^m)}[X] \right>\nonumber\\
       &= \sum_{\la \vdash m} K_{\la(1^m)}(t) f_\la 
       = \sum_{\la : \la_1 \leq k} K_{\la(1^m)}^{(k)}(t) \left< s_{\la}^{(k)}[X;t], h_{(1^m)}[X] \right>~. \label{eq:tkernel}
\end{align}
The polynomial $\left< s_{\la}^{(k)}[X;t], h_{(1^m)}[X] \right>$ is equal to 
$\sum_{P} t^{\spin(P)}$, where the sum is over all strong 
standard $k$-tableaux of shape $\mfc(\la)$ by Equation~\eqref{eq:kschurt}.  The coefficient
$K_{\la(1^m)}^{(k)}(t)$ is a $t$-analogue of the number of standard weak $k$-tableaux of shape $\mfc(\la)$.
This equation is just one possible refinement of Equation~\eqref{eq:kRSKenumeration}.
Similarly, the right hand side of this equation depends on $k$ while the left hand side
does not and this indicates that we might hope to see some relationship between the
bijection at level $k$ and level $k+1$ that relates the $t$ statistic.

If we are looking to explain this algebraic expression with a bijection,
we would like to find a statistic $charge$ on standard weak $k$-tableaux
and a bijection between permutations and pairs of strong and weak tableaux
of the same shape $w^{-1} \leftrightarrow (P^{(k)}, Q^{(k)})$ such that
\begin{equation}\label{eq:chargespinrelation}
\charge(w) = \spin(P^{(k)}) + \charge(Q^{(k)})~.
\end{equation}
Note that we taking the association with
$w^{-1}$ to ensure that everything agrees since as $k\rightarrow\infty$
the statistic charge was defined so that the charge of a permutation
is the charge of its insertion tableau.
The statistic $\spin$ on $k$-strong tableaux is different in nature than the
$\charge$ statistic and in general
$\spin(P^{(\infty)}) = 0$.

The reason the previous affine insertion algorithm is not quite `real' is that we are unable to
use it to explain this $t$-analogue.  A `real' affine insertion algorithm would
allow us to take a definition of the charge statistic so that if 
$w^{-1} \leftrightarrow (P^{(k)}, Q^{(k)})$, then
$\charge(Q^{(k)}) = \charge(w) - \spin(P^{(k)})$.  It would need to be the case
that if $u$ and $v$ are two different permutations such that
$u^{-1} \leftrightarrow (P^{1(k)}, Q^{(k)})$ and $v^{-1} \leftrightarrow (P^{2(k)},Q^{(k)})$,
then $\charge(u) - \spin(P^{1(k)}) = \charge(v) - \spin(P^{2(k)})$.
The following example, based on the calculations from Example~\ref{ex:kRSKformeq3},
shows that this affine insertion algorithm is not compatible with the
spin statistic in this sense.

\begin{example}  It is probably easiest to see that Equation \eqref{eq:chargespinrelation}
cannot hold in our example unless $\charge(w) = \spin( P^{(1)} )$ because when $k=1$ all of
the $Q^{(1)}$ tableaux are the same.
Consider the case $k=1$ and $u = u^{-1} = 132$ with $\charge(u)=2$.
Then $P_1^{(1)} = \scriptsize\young{3\cr2&3^\ast\cr1^\ast&2^\ast&3\cr}$
and $\spin( P_1^{(1)} ) = 1$.

Also, $v = v^{-1} = 213$ with $\charge(v)=1$. 
Then $P_2^{(1)} = \scriptsize\young{3\cr2^\ast&3\cr1^\ast&2&3^\ast\cr}$
and $\spin( P_2^{(1)} ) = 1$, but 
$Q^{(1)} = \scriptsize\young{3\cr2&3\cr1&2&3\cr}$ is the same in both cases
and if there is a charge statistic it should be the same.
\end{example}
\end{subsection}
\end{section}

\begin{section}[The $k$-shape poset]{The $k$-shape poset and a branching rule for expressing 
$k$-Schur in $(k+1)$-Schur functions}\label{sec:kshapeplus}
One of the more recent developments with the definition of $k$-Schur functions defined
as the sum over strong tableaux is an explanation of why they expand positively in the $(k+1)$-Schur
functions.  That is, there are nonnegative integer coefficients $b^{(k)}_{\mu\la}$ such that
$$s^{(k-1)}_\mu = \sum_{\la} b^{(k)}_{\mu\la} s^{(k)}_\la$$
and in this section we will describe a combinatorial interpretation 
for the coefficients $b^{(k)}_{\mu\la}$.

If we consider a partition $\la$ as a collection of cells, then $\Int^k(\la) = \{ b \in \diag(\la) : \hook_\la(b)>k \}$
and $\partial^k(\la) = \la/\Int^k(\la)$.  We define the row shape (resp. column shape) of $\la$ 
to be the composition $rs^k(\la)$
(resp. $cs^k(\la)$) consisting of the number of cells in each of the rows of $\partial^k(\la)$.  
The partition $\la$ is said to be a $k$-shape if both $rs^k(\la)$ and $cs^k(\la)$ are partitions.
Let $\Pi^k$ denote the set 
of $k$-shapes and $\Pi^k_N$ represent the set of $k$-shapes $\la$ such that $|\partial^k(\la)| = N$.
Notice that both the $k$-cores and the $(k+1)$-cores of size $N$ are a subset of $\Pi^k_N$.

\begin{example}
If $k=3$, then
$$\Pi^3_3 = \{ (1,1,1), (2,1), (3), (3,1), (2,1,1) \}$$
since they correspond to the $3$-boundaries
$$\young{\cr\cr\cr} \hskip .2in
\young{\cr&\cr} \hskip .2in
\young{&&\cr} \hskip .2in
\young{\cr\blk&&\cr} \hskip .2in
\young{\cr\cr\blk&\cr}$$
$$\Pi^3_4 = \{ (2,2), (4,1), (3,1,1), (2,1,1,1), (4,2), (2,2,1,1) \}$$
$$\young{&\cr&\cr} \hskip .2in
\young{\cr\blk&&&\cr} \hskip .2in
\young{\cr\cr\blk&&\cr} \hskip .2in
\young{\cr\cr\cr\blk&\cr} \hskip .2in
\young{&\cr\blk&\blk&&\cr} \hskip .2in
\young{\cr\cr\blk&\cr\blk&\cr}$$
\end{example}

\begin{example}
The partition $\la = (6,2,1)$ is a $4$-shape but it is not a $3$-shape.
We calculate that
$$\partial^4(\la) = \young{\cr&\cr\blk&\blk&&&&&\blk\cr}$$
and hence $rs^4(\la) = (4,2,1)$ and $cs^4(\la) = (2,1,1,1,1,1)$, but
$$\partial^3(\la) = \young{\cr&\cr\blk&\blk&\blk&&&&\blk\cr}$$
so that $rs^3(\la) = (3,2,1)$ and $cs^3(\la) = (2,1,0,1,1,1)$ and hence it is not a $3$-shape.
\end{example}

\begin{example}
We include a table of the number of $k$-shapes for 
$1 \leq k \leq 9$ and $1 \leq N \leq 13$.  In the limit
(for $N<k$) it is the case that $|\Pi_N^k|$ is equal to the
number of partitions of $N$.
\begin{center}
\begin{tabular}{|c||c|c|c|c|c|c|c|c|c|c|c|c|c|c|}
\hline
$k\backslash N$&0&$1$&$2$&$3$&$4$&$5$&$6$&$7$&$8$&$9$&$10$&$11$&$12$&$13$\cr
\hline
$k=1$&1&1&1&1&1&1 &1 & 1&1&1&1&1&1&1\cr
\hline
$k=2$&1&1&3&3&6&6 &10&10&15&15&21&21&28&28\cr
\hline
$k=3$&1&1&2&5&6&10&15&21&27&40&48&65&81&103\cr
\hline
$k=4$&1&1&2&3&8&9 &15&23&35&42&69&86&116&155\cr
\hline
$k=5$&1&1&2&3&5&11&14&21&30&49&67&90&120&177\cr
\hline
$k=6$&1&1&2&3&5&7 &16&19&30&41&60&89&127&163\cr
\hline
$k=7$&1&1&2&3&5&7 &11&21&27&40&56&79&107&163\cr
\hline
$k=8$&1&1&2&3&5&7 &11&15&29&36&54&73&105&138\cr
\hline
$k=9$&1&1&2&3&5&7 &11&15&22&38&49&70&97&134\cr
\hline
\end{tabular}
\end{center}
\end{example}

We will define a poset structure on the set $\Pi^k_N$ by describing
how the elements are related by a set of row  and column moves.  
In order to define row and column moves we need to define a notion
of row-type and column-type strings which describe the movement of cells
to get from one $k$-shape to another.

For a cell $b = (x,y)$ we say that the diagonal index of $b$ is
$d(b) = y-x$.  Two cells $b,b'$ are called contiguous if $|d(b) - d(b')| \in
\{ k, k+1 \}$.  A string of length $\ell$ is a skew shape $\mu/\la$ consisting
of cells $\{a_1, a_2, \ldots, a_\ell \}$ such that $a_{i+1}$ and $a_i$ are
contiguous for each $1 \leq i < \ell$ and $a_{i+1}$ lies strictly below $a_i$.

For a skew diagram $D$, define $\lleft_a(D)$ to be the leftmost cell in the same row as the cell $a$
and $\bot_a(D)$ to be the bottommost cell in same column as $a$.

A string $\mu/\la = \{a_1, a_2, \ldots, a_\ell \}$ is called a row-string if $\hook_\la(\lleft_{a_1}(\partial^k(\la))) = k$
and $\hook_\la(\bot_{a_\ell}(\partial^k(\la))) < k$.  It is called a column-string if
the transpose picture is a row-string, or, if $\hook_\la(\lleft_{a_1}(\partial^k(\la))) < k$ and
$\hook_\la(\bot_{a_\ell}(\partial^k(\la))) = k$.

A row move (resp. column move) of rank $r$ and length $\ell$ is a chain of parititions
$\la = \la^0 \subset \la^1 \subset \cdots \subset \la^r = \mu$ that satisfies
\begin{itemize}
\item $\la,\mu \in \Pi^k$
\item $s_i = \la^{i}/\la^{i-1}$ is a row-type (resp. column-type) string consisting of $\ell$ cells for all $1 \leq i \leq r$.
\item the strings $s_i$ are all translates of each other
\item the top cells (resp. rightmost cells) of $s_1, s_2, \ldots, s_r$ occur in consecutive columns (resp. rows) from left to right
(resp. bottom to top).
\end{itemize}

To be clear, a column move is a sequence of partitions whose conjugate partitions are a row move.

\begin{example}  If $k=5$ and $N=18$, then $\la = (9,5,4,4,2,1,1,1)$ and $\mu = (9,7,5,4,2,1,1,1)$
are both $5$-shapes and the sequence of partitions 
$\la^0 = \la  \subset \la^1 = (9,5,5,4,2,1,1,1) \subset \la^2 = (9,6,5,4,2,1,1,1) \subset \la^3 = \mu$
have the following boundaries:
\squaresize=9pt
\begin{equation*}
\young{\cr\cr\cr&\cr\blk&&&\cr\blk&&&\cr\blk&\blk&&&\cr\blk&\blk&\blk&\blk&\blk&&&&\cr} \rightarrow
\young{\cr\cr\cr&\cr\blk&&&\cr\blk&\blk&&&\cr\blk&\blk&&&\cr\blk&\blk&\blk&\blk&\blk&&&&\cr} \rightarrow
\young{\cr\cr\cr&\cr\blk&&&\cr\blk&\blk&&&\cr\blk&\blk&\blk&&&\cr\blk&\blk&\blk&\blk&\blk&&&&\cr} \rightarrow
\young{\cr\cr\cr&\cr\blk&&&\cr\blk&\blk&&&\cr\blk&\blk&\blk&\blk&&&\cr\blk&\blk&\blk&\blk&\blk&&&&\cr} 
\end{equation*}
\end{example}

The set $\Pi^k_N$ is endowed with a poset structure with an edge  
from $\la$ to $\mu$ (or a cover relation $\la \cover \mu$) if there is a row
or a column move from $\la$ to $\mu$.  With some analysis one can show that the minimal elements of
the poset structure on $\Pi^k_N$ are the $k$-cores and the maximal elements are the $(k+1)$-cores.

\begin{example} \label{ex:keq3Neq5poset}
The set of elements $\Pi^3_5$ is endowed with the
the following poset structure.  The edges representing a row move are
labeled with an $r$ and those with a column move with a $c$.
\begin{center}
\includegraphics[width=3in]{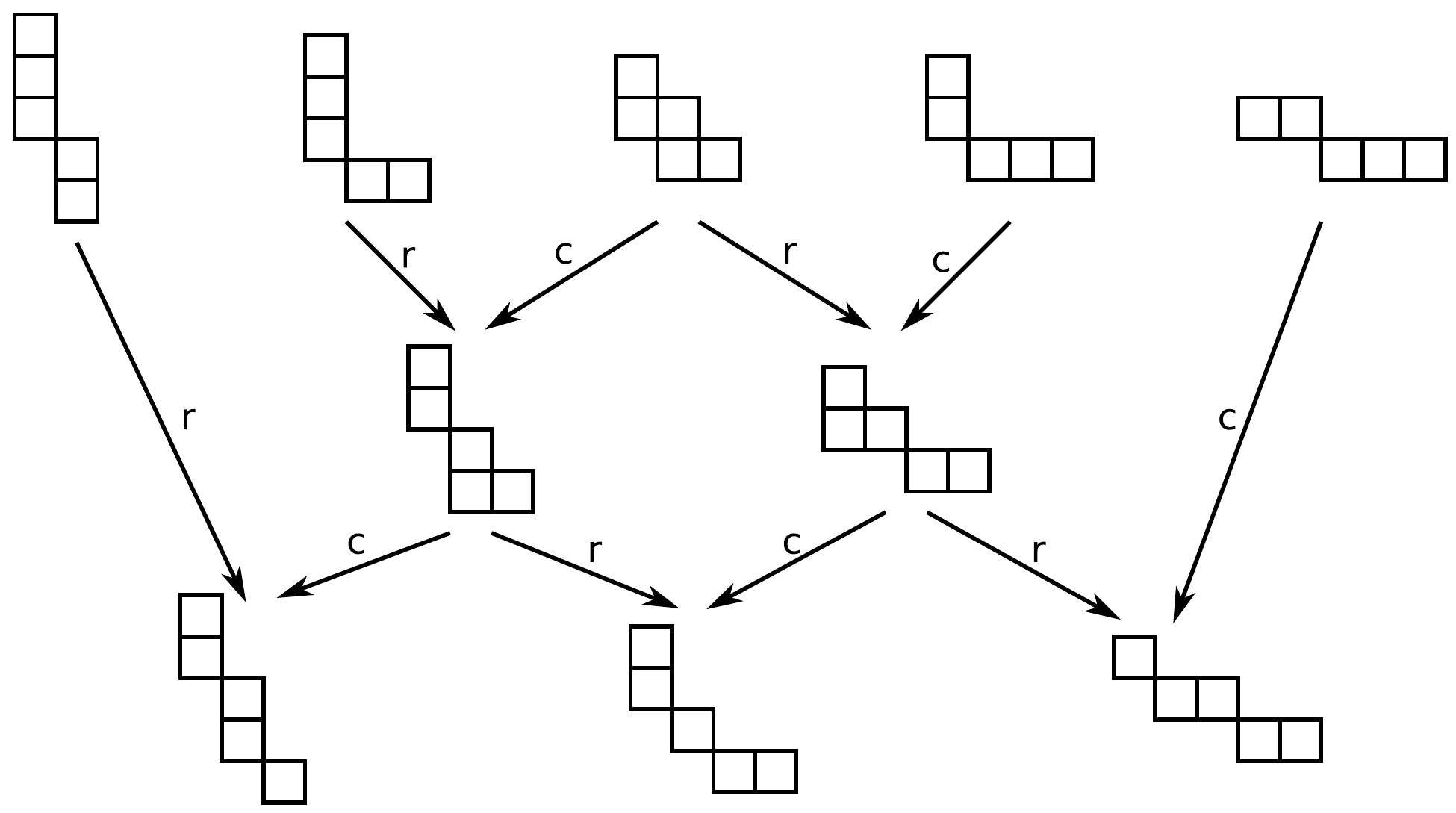}
\end{center}
\end{example}

\begin{example} \label{ex:keq2Neq5poset}
In the set of elements of $\Pi^2_5$, we see that the highest elements in
this partial order are those that are the lowest elements of $\Pi^3_5$.  The Hasse diagram resembles
the following.
\begin{center}
\includegraphics[width=2.5in]{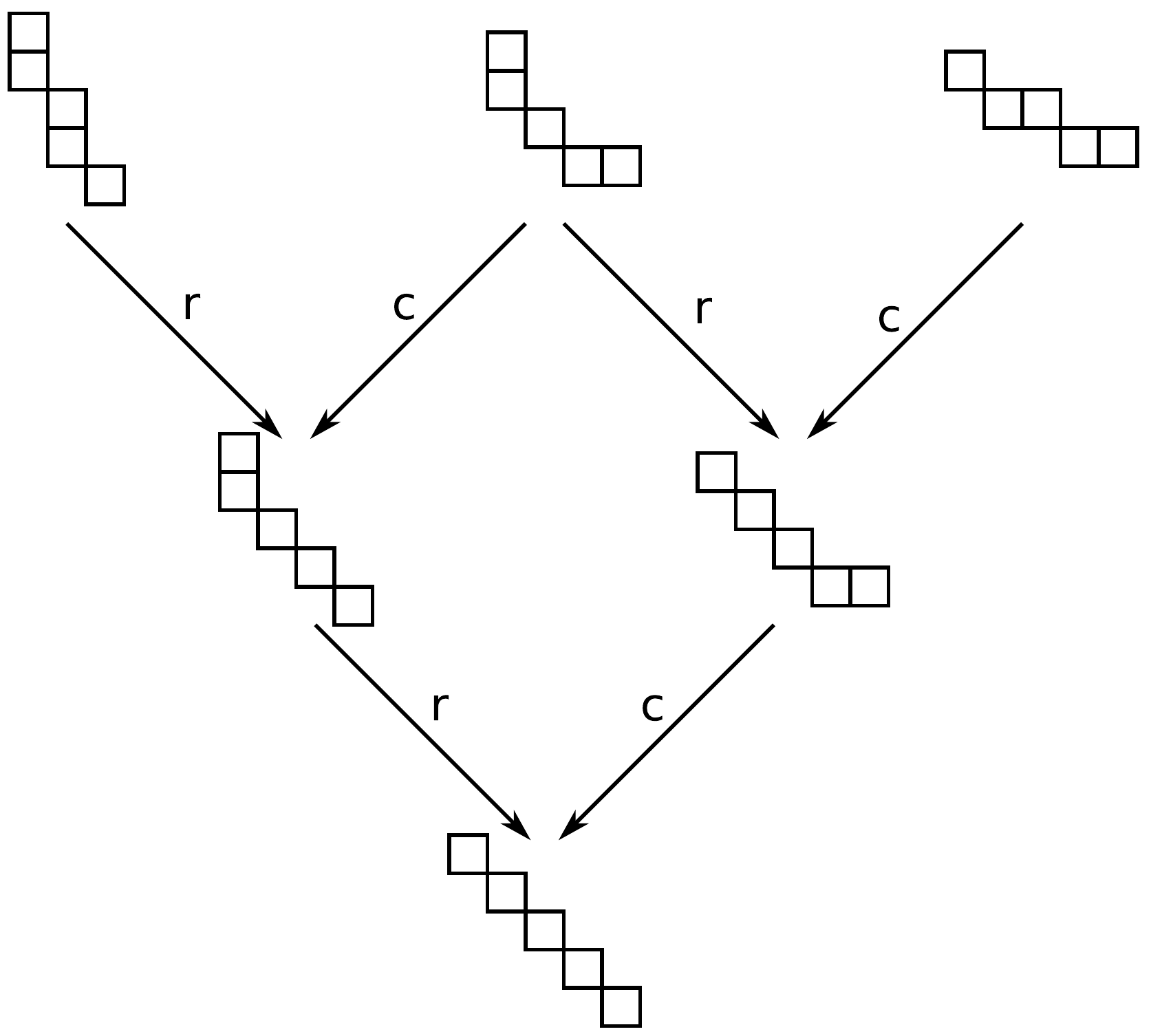}
\end{center}
\end{example}

Now the combinatorial interpretation for the branching coefficients of the $k$-Schur
functions is given in terms of paths within this poset (up to an equivalence on diamonds).

Define a charge of a move to be $0$ if it is a row move and $r\ell$ for a move of
length $\ell$ and rank $r$.

If $m, M, {\tilde m}, {\tilde M}$ are moves relating the $k$-shapes $\la$ and $\gamma$ through the
following diagram, 
 \dgARROWLENGTH=2.5em
\begin{equation} \label{E:commudiagram}
\begin{diagram}
\node[2]{\la} \arrow{sw,t}{m} \arrow{se,t}{M} \\ \node{\mu}
\arrow{se,b}{\tilde M} \node[2]{\nu} \arrow{sw,b}{\tilde m} \\
\node[2]{\gamma}
\end{diagram}
\end{equation}
such that
$$\charge(m) + \charge({\tilde M}) = \charge({\tilde m}) + \charge(M)$$
then the two paths from $\la$ to $\gamma$ are equivalent.  Now consider two
$k$-shapes $\kappa, \tau \in \Pi^k_N$ where $\kappa$ is a $(k+1)$-core and
$\tau$ is a $k$-core.  Let $\PP(\kappa, \tau)$ be the set of paths from $\kappa$
to $\tau$ with respect to this equivalence relation.

The reason the $k$-shapes are related to $k$-Schur functions is that we have the following
theorem.
\begin{theorem}{\cite[Theorem 2]{LLMS:2010}} \label{thm:brancht1} 
Let $\la$ be $k-1$ bounded partition, 
$$s^{(k-1)}_\la[X] = \sum_{ \mu } |\PP(\mfc_{k}(\mu), \mfc_{k-1}(\la))| s^{(k)}_\mu[X]$$
where the sum is over all $k$-bounded partitions $\mu$.
\end{theorem}

While there is not a complete proof, the charge is defined so that the following conjecture
should also hold.
\begin{conj}{\cite[Conjecture 3]{LLMS:2010}} \label{conj:ktbranch} 
Let $\la$ be $k-1$ bounded partition, 
$$s^{(k-1)}_\la[X;t] 
= \sum_{ \mu } \sum_{{\bf p} \in \PP(\mfc_{k+1}(\mu), \mfc_k(\la))} t^{\charge({\bf p})} s^{(k)}_\mu[X;t]$$
where the sum is over all $k$-bounded partitions $\mu$.
\end{conj}

\begin{example} Example \ref{ex:keq3Neq5poset} and Theorem \ref{thm:brancht1} can be used to calculate the following expansions of
$2$-Schur functions in $3$-Schur functions.
$$s_{11111}^{(2)}[X] = s_{11111}^{(3)}[X] + s_{2111}^{(3)}[X] + s_{221}^{(3)}[X]$$
$$s_{2111}^{(2)}[X] = s_{2111}^{(3)}[X] + s_{221}^{(3)}[X] + s_{311}^{(3)}[X]$$
$$s_{221}^{(2)}[X] = s_{221}^{(3)}[X] + s_{311}^{(3)}[X]+ s_{32}^{(3)}[X]~.$$
This is because the two paths from the shape $(3,2,1)$ to $(4,2,1,1)$ are equivalent under
the diamond relation.

The relation that appears in Conjecture \ref{conj:ktbranch} says that
$$s_{11111}^{(2)}[X;t] = s_{11111}^{(3)}[X;t] + t^2 s_{2111}^{(3)}[X;t] + t^3 s_{221}^{(3)}[X;t]$$
$$s_{2111}^{(2)}[X;t] = s_{2111}^{(3)}[X;t] + t s_{221}^{(3)}[X;t] + t^2 s_{311}^{(3)}[X;t]$$
$$s_{221}^{(2)}[X;t] = s_{221}^{(3)}[X;t] + t s_{311}^{(3)}[X;t]+ t^2 s_{32}^{(3)}[X;t]~.$$
This is because the column moves with charge $2$ are those from $(3,2,1,1)$ to $(3,2,2,1,1)$ 
and from $(5,2)$ to $(5,3,1)$.  The others all have charge $1$.
\end{example}

\begin{example} The other poset we have drawn shows that 
$s_{11111}^{(1)}[X] = s_{11111}^{(2)}[X]+2s_{2111}^{(2)}[X]+s_{221}^{(2)}[X].$  In this example the two 
paths from $(2,1,1,1)$ to $(1,1,1,1,1)$ are not equivalent.
We can check that the charge of the move from $(4,2,1,1)$ to $(4,3,2,1,1)$ is 3, while the move from 
$(5,3,2,1)$ to $(5,4,3,2,1)$ is 4.  Since the charge from $(5,3,1)$ to $(5,3,2,1)$ is 2, 
we conclude that Conjecture~\ref{conj:ktbranch} can be used to compute
$$s_{11111}^{(1)}[X;t] = s_{11111}^{(2)}[X;t]+(t^3+t^4)s_{2111}^{(2)}[X;t]+t^6 s_{221}^{(2)}[X;t].$$
\end{example}

\end{section}
%
%

\chapter{Stanley symmetric functions and Peterson algebras}
\label{chapter.stanley symmetric functions}

\setcounter{section}{0}

\begin{center}
{\sc Thomas Lam} \footnote{The author was supported by NSF grants DMS-0652641, DMS-0901111, and DMS-1160726, and by a Sloan Fellowship.} \\
{\tt tfylam@umich.edu}
\end{center}

\renewcommand\A{{\mathbb A}}
\def\B{{\mathbb B}}
\def\C{{\mathcal C}}

\def\a{{\mathbf{a}}}
\def\af{{\mathrm{af}}}
\def\afA{{\mathbb A}_{\mathrm{af}}}
\def\ah{\tilde \h}
\def\aNC{(\A_\af)_0}
\def\Pet{{\mathbb P}}
\def\aB{{\mathbb B}_{\af}}
\def\aI{I_\af}
\def\al{\alpha}
\def\aPsi{\Xi_\af}
\def\aQ{Q_\af}
\def\aW{W_\af}
\def\Bo{{\mathcal B}}
\def\Des{\mathrm{Des}}
\def\fW{W}
\def\fI{I}
\def\Fl{{\mathrm{Fl}}}
\def\Frac{{\mathrm{Frac}}}
\def\Fun{{\mathrm{Fun}}}
\def\Gr{{\mathrm{Gr}}}
\def\h{{\mathbf{h}}}
\def\i{{\mathbf{i}}}
\def\id{\mathrm{id}}
\def\j{{\mathbf{j}}}
\def\p{{\mathcal P}}
\def\Q{{\mathbb Q}}
\def\om{\omega}
\def\s{{\mathbf{s}}}
\def\Sym{{\mathrm{Sym}}}
\def\tA{\tilde{\mathbb A}}
\renewcommand\tF{\tilde F}
\def\tS{\tilde S}
\def\wt{{\mathrm{wt}}}
\def\Z{{\mathbb Z}}

\newcommand\os[1]{\overline{#1}\;}
\newcommand\ip[1]{\langle #1 \rangle}
\newcommand\fixit[1]{{\bf ** #1 **}}
\newcommand\remind[1]{{\bf ** #1 **}}

This purpose of this chapter is to introduce Stanley symmetric functions and affine Stanley symmetric functions from the combinatorial and algebraic point of view.  The presentation roughly follows 3 lectures I gave at a conference titled
``Affine Schubert Calculus'' held in July of 2010 at the Fields Institute in 
Toronto~\footnote{see {\tt http://www.fields.utoronto.ca/programs/scientific/10-11/schubert/}}.

 The goal is to develop the theory (with the exception of positivity) without appealing to geometric reasoning.  The material is aimed at an audience with some familiarity with symmetric functions, Young tableaux and Coxeter groups/root systems. 

Stanley symmetric functions are a family $\{F_w \mid w \in S_n\}$ of symmetric functions indexed by permutations.  They were invented by Stanley \cite{Sta} to enumerate the reduced words of elements of the symmetric group.  The most important properties of the Stanley symmetric functions are their symmetry, established by Stanley, and their Schur positivity, first proven by Edelman and Greene \cite{EG}, and by Lascoux and Sch\"{u}tzenberger \cite{LS:Schub}.

Recently, a generalization of Stanley symmetric functions to affine permutations was developed in \cite{Lam:2006}.  These affine Stanley symmetric functions turned out to have a natural geometric interpretation \cite{Lam:2008}: they are pullbacks of the cohomology Schubert classes of the affine flag variety $LSU(n)/T$ to the affine Grassmannian (or based loop space) $\Omega SU(n)$ under the natural map $\Omega SU(n) \to LSU(n)/T$.  The combinatorics of reduced words and the geometry of the affine homogeneous spaces are connected via the nilHecke ring of Kostant and Kumar \cite{KK}, together with a remarkable commutative subalgebra due to Peterson \cite{Pet}.    The symmetry of affine Stanley symmetric functions follows from the commutativity of Peterson's subalgebra, and the positivity in terms of affine Schur functions is established via the relationship between affine Schubert calculus and quantum Schubert calculus \cite{LS:QH,LL}.  The affine-quantum connection was also discovered by Peterson.  

The affine generalization also connects Stanley symmetric functions with the theory of Macdonald polynomials \cite{Mac:1995} and $k$-Schur functions \cite{LLM:2003} -- my own involvement in this subject began when I heard a conjecture of Mark Shimozono relating the Lapointe-Lascoux-Morse $k$-Schur functions to the affine Grassmannian.

While the definition of (affine) Stanley symmetric functions does not easily generalize to other (affine) Weyl groups (see \cite{BH,BL,FK:C,LSS:C,Pon}), the algebraic and geometric constructions mentioned above do.

The first third (Sections \ref{sec:combin} - \ref{sec:affine}) of the chapter centers on the combinatorics of reduced words.  We discuss reduced words in the (affine) symmetric group, the definition of the (affine) Stanley symmetric functions, and introduce the Edelman-Greene correspondence.  Section \ref{sec:Weyl} reviews the basic notation of Weyl groups and affine Weyl groups.  In Sections \ref{sec:algebra}-\ref{sec:finiteFS} we introduce and study four algebras: the nilCoxeter algebra, the Kostant-Kumar nilHecke ring, the Peterson centralizer subalgebra of the nilHecke ring, and the Fomin-Stanley subalgebra of the nilCoxeter algebra.  The discussion in Section \ref{sec:finiteFS} is new, and is largely motivated by a conjecture (Conjecture \ref{conj:LP}) of the author and Postnikov.  In Section \ref{sec:geom}, we give a list of geometric interpretations and references for the objects studied in the earlier sections.

We have not intended to be comprehensive, especially with regards to generalizations and variations.  There are four such which we must mention: 
\begin{enumerate} 
\item
There is an important and well-developed connection between Stanley symmetric functions and Schubert polynomials, see \cite{BJS, LS:Schub}.  
\item
There is a theory of (affine) Stanley symmetric functions in classical types; see  \cite{BH,BL,FK:C,LSS:C,Pon}.
\item
Nearly all the constructions here have $K$-theoretic analogues.  For full details see \cite{Buc, BKSTY,FK:K,LSS}.
\item
There is a $t$-graded version of the theory.  See \cite{LLM:2003,LM:2005,LM:2007}.
\end{enumerate}

We have included exercises and problems throughout which occasionally assume more prerequisites.  The exercises vary vastly in terms of difficulty.  Some exercises essentially follow from the definitions, but other problems are questions for which the answers might not yet be known.

\section{Stanley symmetric functions and reduced words}
\label{sec:combin}
For an integer $m \geq 1$, let $[m] = \{1,2,\ldots,m\}$.  For a partition (or composition) $\la = (\la_1,\la_2,\ldots,\la_\ell)$, we write $|\la| = \la_1 + \cdots + \la_\ell$.  The {\it dominance order} $\preceq$ on partitions is given by $\la \prec \mu$ if for some $k > 0 $ we have $\la_1 + \la_2 + \cdots + \la_j = \mu_1 + \mu_2 + \cdots + \mu_j$ for $1 \leq j <k$ and $\la_1 + \la_2 + \cdots + \la_k < \mu_1 + \mu_2 + \cdots + \mu_k$.  The {\it descent set} $\Des(\a)$ of a word $a_1a_2\cdots a_n$ is given by $\Des(\a) = \{i \in [n-1] \mid a_i > a_{i+1}\}$.

\subsection{Young tableaux and Schur functions}
We shall assume the reader has some familiarity with symmetric functions and Young tableaux see Chapter \ref{chapter.k schur primer}, Section \ref{section.background} or \cite{Mac:1995} \cite[Ch. 7]{EC2}.  We write $\La$ for the ring of symmetric functions.  We let $m_\la$, where $\la$ is a partition, denote the monomial symmetric function, and let $h_k$ and $e_k$,  for integers $k \geq 1$, denote the homogeneous and elementary symmetric functions respectively.  For a partition $\la = (\la_1,\la_2,\ldots,\la_\ell)$, we define $h_\la:= h_{\la_1}h_{\la_2}\cdots h_{\la_\ell}$, and similarly for $e_\la$.  We let $\ip{.,.}$ denote the Hall inner product of symmetric functions.  Thus $\ip{h_\la,m_\mu} = \ip{m_\la,h_\mu}=\ip{s_\la,s_\mu} = \delta_{\la\mu}$.

We shall draw Young diagrams in French notation.  A tableau of shape $\la$ is a filling of the Young diagram of $\la$ with integers.  A tableau is {\it column-strict} (resp. {\it row-strict}) if it is increasing along columns (resp. rows).  A tableau is {\it standard} if it is column-strict and row-strict, and uses each number $1,2,\ldots,|\la|$ exactly once.  A tableau is {\it semi-standard} if it is column-strict, and weakly increasing along rows.  Thus the tableaux
$$
\tableau[sY]{7\\3&6\\1&2&4&5} \qquad \tableau[sY]{6\\4&4\\1&1&2&3}
$$
are standard and semistandard respectively.  The weight $\wt(T)$ of a tableau $T$ is the composition $(\alpha_1,\alpha_2,\ldots)$ where $\alpha_i$ is equal to the number of $i$'s in $T$.  The Schur function $s_\la$ is given by
$$
s_\la(x_1,x_2,\ldots) = \sum_T x^{\wt(T)}
$$
where the summation is over semistandard tableaux of shape $\la$, and for a composition $\alpha$, we define $x^\alpha:=x_1^{\alpha_1}x_2^{\alpha_2} \cdots$.  For a standard Young tableau $T$ we define $\Des(T) = \{i \mid \text{$i+1$ is in a lower row than $i$}\}$.  We also write $f^\la$ for the number of standard Young tableaux of shape $\la$.  Similar definitions hold for skew shapes $\la/\mu$.

We shall often use the Jacobi-Trudi formula for Schur functions (see \cite{Mac:1995,EC2}).
\begin{theorem}\label{thm:JT} 
Let $\la = (\la_1\geq \la_2 \geq \cdots \geq \la_\ell >0)$ be a partition.  Then
$$
s_\la = {\rm det}(h_{\la_i+j-i})_{i,j=1}^{\ell}.
$$
\end{theorem}

\subsection{Permutations and reduced words}\label{sec:perm}
Let $S_n$ denote the symmetric group of permutations on the letters $[n]$.  We think of permutations $w,v \in S_n$ as bijections $[n] \to [n]$, so that the product $w \, v \in S_n$ is the composition $w \circ v$ as functions.  The simple transposition $s_i \in S_n$, $i \in \{1,2,\ldots,n-1\}$ swaps the letters $i$ and $i+1$, keeping the other letters fixed.  The symmetric group is generated by the $s_i$, with the relations
\begin{align*}
s_i^2 &= 1 & \mbox{for $1 \leq i \leq n-1$}\\
s_i s_{i+1} s_i &=s_{i+1} s_i s_{i+1} &\mbox{for $1 \leq i \leq n-2$}\\
s_i s_j &= s_j s_i &\mbox{for $|i-j|>1$}
\end{align*}
The {\it length} $\ell(w)$ of a permutation $w \in S_n$ is the length of the shortest expression $w = s_{i_1} \cdots s_{i_\ell}$ for $w$ as a product of simple generators.  Such a shortest expression is called a {\it reduced expression} for $w$, and the word $i_1 i_2 \cdots i_\ell$ is a {\it reduced word} for $w$.  Let $R(w)$ denote the set of reduced words of $w \in S_n$.  We usually write permutations in one-line notation, or alternatively give reduced words.  For example $3421 \in S_4$ has reduced word $23123$.

There is a natural embedding $S_n \hookrightarrow S_{n+1}$ and we will sometimes not distinguish between $w \in S_n$ and its image in $S_{n+1}$ under this embedding.  

\subsection{Reduced words for the longest permutation}
The longest permutation $w_0 \in S_n$ is $w_0 = n\, (n-1)\, \cdots \,2\,1$ in one-line notation.  Stanley \cite{Sta} conjectured the following formula for the number of reduced words of $w_0$, which he then later proved using the theory of Stanley symmetric functions.  Let $\delta_n = (n,n-1,\ldots,1)$ denote the staircase of size $n$.
\begin{theorem}[\cite{Sta}]\label{thm:longestword}
The number $R(w_0)$ of reduced words for $w_0$ is equal to the number $f^{\delta_{n-1}}$ of staircase shaped standard Young tableaux.
\end{theorem}

\subsection{The Stanley symmetric function}
\begin{definition}[Original definition] \label{def:storig}
Let $w \in S_n$.  Define the Stanley symmetric function\footnote{Our conventions differ from Stanley's original definitions by $w \leftrightarrow w^{-1}$.} $F_w$ by
$$
F_w(x_1,x_2,\ldots) = \sum_{\stackrel{a_1 a_2 \ldots a_\ell \in  R(w) \ \  1 \leq b_1 \leq b_2 \leq \cdots \leq b_\ell}{a_i < a_{i+1} \implies b_i < b_{i+1}}} x_{b_1} x_{b_2} \cdots x_{b_\ell}.
$$
\end{definition}

We shall establish the following fundamental result \cite{Sta} in two different ways in Sections \ref{sec:EG} and \ref{sec:algebra}, but shall assume it for the remainder of this section.
\begin{theorem}[\cite{Sta}]\label{thm:stsymm}
The generating function $F_w$ is a symmetric function.
\end{theorem}

A word $a_1 a_2 \cdots a_\ell$ is {\it decreasing} if $a_1 > a_2 > \cdots > a_\ell$.  A permutation $w \in S_n$ is decreasing if it has a (necessarily unique) decreasing reduced word.  The identity $\id \in S_n$ is considered decreasing.  A {\it decreasing factorization} of $w \in S_n$ is an expression $w = v_1 v_2 \cdots v_r$ such that $v_i \in S_n$ are decreasing, and $\ell(w) = \sum_{i=1}^r \ell(v_i)$.

\begin{definition}[Decreasing factorizations] \label{def:stdec}
Let $w \in S_n$.  Then
$$
F_w(x_1,x_2,\ldots) = \sum_{w = v_1 v_2 \cdots v_r} x_1^{\ell(v_1)} \cdots x_r^{\ell(v_r)}.
$$
\end{definition}

\begin{example}\label{ex:1323}
Consider $w = s_1s_3s_2s_3 \in S_4$.  Then $R(w) = \{1323,3123,1232\}$.  Thus
$$
F_w = m_{211}+3m_{1111} = s_{211}.
$$
The decreasing factorizations which give $m_{211}$ are $\overline{31}\;\overline{2}\;\overline{3}, \overline{1}\;\overline{32}\;\overline{3}, \overline{1}\;\overline{2}\;\overline{32}$.
\end{example}

\subsection{The code of a permutation}
Let $w \in S_n$.  The {\it code} $c(w)$ is the sequence $c(w) = (c_1, c_2, \ldots )$ of nonnegative integers given by $c_i = \#\{j\in [n]\mid j > i \text{ and } w(j) < w(i)\}$ for $i \in [n]$, and $c_i = 0$ for $i > n$.  Note that the code of $w$ is the same regardless of which symmetric group it is considered an element of.

Let $\la(w)$ be the partition conjugate to the partition obtained from rearranging the parts of $c(w^{-1})$ in decreasing order.

\begin{example}
Let $w = 216534 \in S_6$.  Then $c(w) = (1,0,3,2,0,0,\ldots)$, and $c(w^{-1}) = (1,0,2,2,1,0,\ldots)$.  Thus $\la(w) = (4,2)$.
\end{example}

For a symmetric function $f \in \La$, let $[m_\la]f$ denote the coefficient of $m_\la$ in $f$.

\begin{prop}[\cite{Sta}] \label{prop:dominant} Let $w \in S_n$.
\begin{enumerate}
\item
Suppose $[m_\la]F_w \neq 0$.  Then $\la \prec \la(w)$.
\item
$[m_{\la(w)}]F_w = 1$.
\end{enumerate}
\end{prop}
\begin{proof}
Left multiplication of $w$ by $s_i$ acts on $c(w^{-1})$ by
$$
(c_1,\ldots , c_i,c_{i+1},\ldots)\longmapsto
(c_1,\ldots,c_{i+1},c_i-1,\ldots)$$
whenever $\ell(s_iw) < \ell(w)$.  Thus factorizing a decreasing permutation $v$ out of $w$ from the left will decrease $\ell(v)$ different entries of $c(w^{-1})$ each by $1$.  (1) follows easily from this observation.

To obtain (2), one notes that there is a unique decreasing permutation $v$ of length $\mu_1(w)$ such that $\ell(w) = \ell(v^{-1}w) + \ell(v)$.
\end{proof}

\begin{example}
Continuing Example \ref{ex:1323}, one has $w = 2431$ in one-line notation.  Thus $\la(w) =  (2,1,1)$, agreeing with Proposition \ref{prop:dominant} and our previous calculation.
\end{example}

\subsection{Fundamental Quasi-symmetric functions}
Let $D \subset [n-1]$.  Define the (Gessel) {\it fundamental quasi-symmetric function} $L_D$ by 
$$
L_D(x_1,x_2,\ldots) = \sum_{\stackrel{1 \leq b_1 \leq b_2 \cdots \leq b_n}{i \in D \implies b_{i+1} > b_i}} x_{b_1}x_{b_2} \cdots x_{b_n}.
$$
Note that $L_D$ depends not just on the set $D$ but also on $n$.

The {\it descent set} $\Des(T)$ of a standard Young tableau $T$ is the set of labels $i$, such that $i+1$ occurs in a higher row than $i$.  A basic fact relating Schur functions and fundamental quasi-symmetric functions is:
\begin{prop}\label{prop:Schurquasi}
Let $\la$ be  a partition.  Then
$$
s_\lambda = \sum_T L_{\Des(T)}
$$
where the summation is over all standard Young tableaux of shape $\lambda$.
\end{prop}

\begin{definition}[Using quasi symmetric functions] \label{def:stquasi}
Let $w \in S_n$.  Then
$$
F_w(x_1,x_2,\ldots) = \sum_{\a \in R(w^{-1})} L_{\Des(\a)}.
$$
\end{definition}

\begin{example}
Continuing Example \ref{ex:1323}, we have $F_w = L_{2} + L_{1}+L_{3}$, where all subsets are considered subsets of $[3]$.  Note that these are exactly the descent sets of the
tableaux
$$
\tableau[sY]{4\\1&2&3} \qquad \tableau[sY]{3\\1&2&4} \qquad \tableau[sY]{2\\1&3&4}
$$
\end{example}


\subsection{Exercises}
\label{ssec:combinprob}
\begin{enumerate}
\item 
Prove that $|c(w)|:=\sum_i c_i(w)$ is equal to $\ell(w)$.
\item
Let $S_\infty = \cup_{n \geq 1} S_n$, where permutations are identified under the embeddings $S_1 \hookrightarrow S_2 \hookrightarrow S_3 \cdots$.  
Prove that $w \longmapsto c(w)$ is a bijection between $S_\infty$ and nonnegative integer sequences with finitely many non-zero entries.
\item
Prove the equivalence of Definitions \ref{def:storig}, \ref{def:stdec}, and \ref{def:stquasi}.  
\item
What happens if we replace decreasing factorizations by increasing factorizations in Definition \ref{def:stdec}?
\item
What is the relationship between $F_w$ and $F_{w^{-1}}$?
\item 
(Grassmannian permutations)
A permutation $w \in S_n$ is {\it Grassmannian} if it has at most one descent.
\begin{enumerate}
\item  Characterize the codes of Grassmannian permutations.
\item
Show that if $w$ is Grassmannian then $F_w$ is a Schur function. 
\item
Which Schur functions are equal to $F_w$ for some Grassmannian permutation $w \in S_n$?
\end{enumerate}
\item 
(321-avoiding permutations \cite{BJS})
A permutation $w \in S_n$ is {\it 321-avoiding} if there does not exist $a < b <c$ such that $w(a) > w(b) > w(c)$.  Show that $w$ is 321-avoiding if and only if no reduced word $\i \in R(w)$ contains a consecutive subsequence of the form $j (j+1) j$.  If $w$ is 321-avoiding, show directly from the definition that $F_w$ is a skew Schur function.
\end{enumerate}

\section{Edelman-Greene insertion}\label{sec:EG}
\subsection{Insertion for reduced words}
We now describe an insertion algorithm for reduced words, due to Edelman and Greene \cite{EG}, which establishes Theorem \ref{thm:stsymm}, and in addition stronger positivity properties.  Related bijections were studied by Lascoux-Sch\"utzenberger \cite{LS} and by Haiman \cite{Hai}.

Let $T$ be a column and row strict Young tableau.  The {\it reading word} $r(T)$ is the word obtained by reading the rows of $T$ from left to right, starting with the top row.

Let $w \in S_n$.  We say that a tableau $T$ is a {\it EG-tableau} for $w$ if $r(T)$ is a reduced word for $w$.  For example,
$$
T = \tableau[sY]{2&3\\1&2&3}
$$
has reading word $r(T) = 23123$, and is an EG-tableau for $3421 \in S_4$.

\begin{theorem}[\cite{EG}]\label{thm:EG}
Let $w \in S_n$.  There is a bijection between $R(w)$ and the set of pairs $(P,Q)$, where $P$ is an EG-tableau for $w$, and $Q$ is a standard Young tableau with the same shape as $P$.  Furthermore, under the bijection $\i \leftrightarrow (P(\i),Q(\i))$ we have $\Des(\i) = \Des(Q)$.
\end{theorem}

Combining Theorem \ref{thm:EG} with Proposition \ref{prop:Schurquasi} and Definition \ref{def:stquasi}, we obtain:
\begin{corollary}\label{cor:stpos}
Let $w \in S_n$.  Then $F_w = \sum_\lambda \al_{w \la} s_\la$, where 
$\al_{w\la}$ is equal to the number of EG-tableau for $w^{-1}$.  In particular, $F_w$ is Schur positive.
\end{corollary}
As a consequence we obtain Theorem \ref{thm:stsymm}.

\begin{lemma}\label{lem:EGshape}
Suppose $T$ is an EG-tableau for $S_n$.  Then the shape of $T$ is contained in the staircase $\delta_{n-1}$.
\end{lemma}
\begin{proof}
Since $T$ is row-strict and column-strict, the entry in the $i$-th row and $j$-th column is greater than or equal to $i+j-1$.  But EG-tableaux can only be filled with the numbers $1,2,\ldots,n-1$, so the shape of $T$ is contained inside $\delta_{n-1}$.
\end{proof}

\begin{proof}[Proof of Theorem \ref{thm:longestword}]
The longest word $w_0$ has length $\binom{n}{2}$.  Suppose $T$ is an EG-tableau for $w_0$.  Since the staircase $\delta_{n-1}$ has exactly $\binom{n}{2}$ boxes, Lemma \ref{lem:EGshape} shows that $T$ must have shape $\delta_{n-1}$.  But then it is easy to see that the only possibility for $T$ is the tableau
$$
\tableau[mbY]{n\!-\!1\\ \vdots &\vdots \\3&\cdots&n\!-\!1 \\2 & 3 & \cdots & n\!-\!1\\ 1 & 2 & 3& \cdots & n\!-\!1}
$$
Thus it follows from Theorem \ref{thm:EG} that $R(w_0) = f^{\delta_{n-1}}$.
\end{proof}

The proof of Theorem \ref{thm:EG} is via an explicit insertion algorithm.  Suppose $T$ is an EG-tableau.  We describe the insertion of a letter $a$ into $T$.  If the largest letter in the first (that is, bottom) row of $T$ is less than $a$, then we add $a$ to the end of the first row, and the insertion is complete.  Otherwise, we find the smallest letter $a'$ in $T$ greater than $a$, and bump $a'$ to the second row, where the insertion algorithm is recursively performed.  The first row $R$ of $T$ changes as follows: if both $a$ and $a+1$ were present in $R$ (and thus $a' = a+1$) then the row remains unchanged; otherwise, we replace $a'$ by $a$ in $R$.

For a reduced word $\i = i_1\,i_2\, \cdots \, i_\ell$, we obtain $P(\i)$ by inserting $i_1$, then $i_2$, and so on, into the empty tableau.  The tableau $Q(\i)$ is the standard Young tableau which records the changes in shape of the EG-tableau as this insertion is performed.  

\begin{example}
Let $\i = 21232$.  Then the successive EG-tableau are
$$
\tableau[sY]{2} \qquad \tableau[sY]{2\\1} \qquad \tableau[sY]{2\\1&2} \qquad \tableau[sY]{2\\1&2&3} \qquad \tableau[sY]{2&3\\1&2&3}
$$
so that 
$$
Q(\i) = \tableau[sY]{2&5\\1&3&4}.
$$
\end{example}

\subsection{Coxeter-Knuth relations}
\label{sec:CK}
Let $\i$ be a reduced word.  A {\it Coxeter-Knuth relation} on $\i$ is one of the following transformations on three consecutive letters of $\i$:
\begin{enumerate}
\item
$a\, (a+1) \,a \sim (a+1)\, a \, (a+1)$
\item
$a \, b \, c \sim a \, c \, b$ when $b < a < c$
\item
$a \, b \, c \sim b \, a \, c$ when $b < c < a$
\end{enumerate}
Since Coxeter-Knuth relations are in particular Coxeter relations for the symmetric group, it follows that if two words are related by Coxeter-Knuth relations then they represent the same permutation in $S_n$.  The following result of Edelman and Greene states that Coxeter-Knuth equivalence is an analogue of Knuth-equivalence for reduced words.

\begin{theorem}[\cite{EG}]\label{thm:CK}
Suppose $\i,\i' \in R(w)$.  Then $P(\i)=P(\i')$ if and only if $\i$ and $\i'$ are Coxeter-Knuth equivalent.
\end{theorem}

\subsection{Exercises and Problems}\label{ssec:EGprob}
\begin{enumerate}
\item
For $w \in S_n$ let $1 \times w \in S_{n+1}$ denote the permutation obtained from $w$ by adding $1$ to every letter in the one-line notation, and putting a $1$ in front.  Thus if $w = 24135$, we have $1 \times w = 135246$.  Show that $F_w = F_{1 \times w}$.
\item
Suppose $w \in S_n$ is 321-avoiding (see Section \ref{ssec:combinprob}).  Show that Edelman-Greene insertion of $\i \in R(w)$ is the usual Robinson-Schensted insertion of $\i$.
\item (Vexillary permutations \cite{BJS})
A permutation $w \in S_n$ is {\it vexillary} if it avoids the pattern 2143.  That is, there do not exist $a < b < c < d$ such that $w(b) < w(a) < w(d) < w(c)$.  In particular, $w_0$ is vexillary.  

The Stanley symmetric function $\tF_w$ is equal to a Schur function $s_\la$ if and only if $w$ is vexillary \cite[p.367]{BJS}.  Is there a direct proof using Edelman-Greene insertion? 

\item (Shape of a reduced word)
The shape $\la(\i)$ of a reduced word $\i \in R(w)$ is the shape of the tableau $P(\i)$ or $Q(\i)$ under Edelman-Greene insertion.  Is there a direct way to read off the shape of a reduced word?  (See \cite{TY1} for a description of $\la_1(\i)$.)

For example, Greene's invariants (see for example \cite[Ch. 7]{EC2}) describe the shape of a word under Robinson-Schensted insertion.

\item (Coxeter-Knuth relations and dual equivalence (graphs))
\label{rem:CK}
Show that Coxeter-Knuth relations on reduced words correspond exactly to elementary dual equivalences on the recording tableau (see \cite{Hai}).  They thus give a structure of a dual equivalence graph \cite{Assaf:2010} on $R(w)$. 

An independent proof  of this (in particular not using EG-insertion), together with the technology of \cite{Assaf:2010}, would give a new proof of the Schur positivity of Stanley symmetric functions.

\item (Lascoux-Sch\"{u}tzenberger transition)
Let $(i,j) \in S_n$ denote the transposition which swaps $i$ and $j$.  Fix $r \in [n]$ and $w \in S_n$.  The Stanley symmetric functions satisfy \cite{LS} the equality
\begin{equation}\label{eq:transition}
\sum_{u=w\,(r,s): \,\,
\ell(u)=\ell(w)+1 \,\, r < s} F_u = 
\left(\sum_{v=w\,(s',r): \,\,
\ell(v)=\ell(w)+1 \,\, s' < r} F_v\right)\ \ (+F_x)
\end{equation}
where the last term with $x = (1 \times w)(1,r)$ is only present if $\ell(x) = \ell(w)+1$.

One obtains another proof of the Schur positivity of $F_u$ as follows.  Let $r$ be the last descent of $u$, and let $k$ be the largest index such that $u(r) > u(k)$.  Set $w = u(r,k)$.  Then the left hand side of \eqref{eq:transition} has only one term $F_u$.  Recursively repeating this procedure for the terms $F_v$ on the right hand side one obtains a positive expression for $F_u$ in terms of Schur functions.

\item (Little's bijection)
Little \cite{Lit} described an algorithm to  establish \eqref{eq:transition}, which we formulate in the manner of \cite{LS:affineLittle}.  A {\it $v$-marked nearly reduced word} is a pair $(\i,a)$ where $\i = i_1 i_2 \cdots i_\ell$ is a word with letters in $\Z_{>0}$ and $a$ is an index such that $\j = i_1 i_2 \cdots \hat i_a \cdots i_\ell$ is a reduced word for $v$, where $\hat i_a$ denotes omission.  We say that $(\i,a)$ is a marked nearly reduced word if it is a $v$-marked nearly reduced word for some $v$.  A marked nearly reduced word is a marked reduced word if $\i$ is reduced.

Define the directed {\it Little graph} on marked nearly reduced words, where each vertex has a unique outgoing edge $(\i,a) \to (\i',a')$ as follows: $\i'$ is obtained from $\i$ by changing $i_a$ to $i_a-1$.  If $i_a-1 = 0$, then we also increase every letter in $\i$ by $1$.  If $\i$ is reduced then $a'=a$.  If $\i$ is not reduced then $a'$ is the unique index not equal to $a$ such that $i_1 i_2 \cdots \hat i_{a'} \cdots i_\ell$ is reduced.  (Check that this is well-defined.)

For a marked reduced word $(\i,a)$ such that $\i$ is reduced, the {\it forward Little move} sends $(\i,a)$ to $(\j,b)$ where $(\j,b)$ is the first marked reduced word encountered by traversing the Little graph.  
\begin{example} \label{ex:Little} Beginning with $\i = 2134323$ and $a = 5$ one has
$$
2134{\bf 3}21 \to 21342{\bf 2} 1 \to 213421{\bf 1} \to 324532{\bf 1}.
$$
Note that $\i$ is a reduced word for $u = 53142$ which covers $w = 43152$.  The word $3245321$ is a reduced word for $514263 = (1 \times w)(1,2)$.  \end{example}
Check that if you apply the forward Little move to a $w$-marked reduced word $(\i,a)$ where $\i \in R(u)$ for some $u$ on the left hand side of \eqref{eq:transition}, you will get a ($w$ or $1\times w$)-marked reduced word $(\j,b)$ where $\j \in R(v)$ for some $v$ on the right hand side of \eqref{eq:transition}.  This can then be used to prove \eqref{eq:transition}.

\item (Dual Edelman-Greene equivalence)
Let $R(\infty)$ denote the set of all reduced words of permutations.  We say that $\i,\i' \in R(\infty)$ are dual EG-equivalent if the recording tableaux under EG-insertion are the same: $Q(\i) = Q(\i')$.

\begin{conjecture}\label{conj:EG}
Two reduced words are dual $EG$-equivalent if and only if they are connected by forward and backwards Little moves.
\end{conjecture}

For example, both $2134321$ and $3245321$ of Example \ref{ex:Little} Edelman-Greene insert to give recording tableau $$\tableau[sbY]{7\\6\\2&5\\1&3&4}$$

Hamaker and Young \cite{HY} have announced a proof of Conjecture \ref{conj:EG}.

\item
\label{prob:enumerateEG}
Fix a symmetric group $S_n$.  Is there a formula for the number of EG-tableau of a fixed shape $\la$?  (See also Section \ref{ssec:algebraprob} and compare with formulae for the number of (semi)standard tableaux \cite{EC2}.)

\item There are two common bijections which demonstrate the symmetry of Schur functions: the Bender-Knuth involution \cite{BK}, and the Lascoux-Sch\"{u}tzenberger/crystal operators (see for example \cite{LLT}).

Combine this with Edelman-Greene insertion to obtain an explicit weight-changing bijection on the monomials of a Stanley symmetric function, which exhibits the symmetry of a Stanley symmetric function.  Compare with Stanley's original bijection \cite{Sta}.

\item (Jeu-de-taquin for reduced words)
There are Jeu-de-taquin theories for skew EG-tableaux.  This was developed for two-row tableaux by Bergeron and Sottile \cite{BS}, and by Thomas and Yong \cite{TY,TY1} in the context of Hecke insertion.  For example one possible slide is
$$
\tableau[mbY]{ i & i\!+\!1\\\bl & i } \qquad \raisebox{25pt}{$\longleftrightarrow$} \qquad \tableau[mbY]{ i\!+\!1\\i& i\!+\!1}
$$
\end{enumerate}

\section{Affine Stanley symmetric functions}
\label{sec:affine}
\subsection{Affine symmetric group}
For basic facts concerning the affine symmetric group, we refer the reader to Chapter \ref{chapter.k schur primer}, Section \ref{section.background} and to \cite{BB}.

Let $n > 2$ be a positive integer.  Let $\tS_n$ denote the affine symmetric group with simple
generators $s_0,s_1,\ldots,s_{n-1}$ satisfying the relations
\begin{align*}
s_i^2 &= 1 &\mbox{for all $i$} \\
s_i s_{i+1} s_i &= s_{i+1} s_i s_{i+1} &  \mbox{for all $i$} \\
  s_is_j &= s_j s_i &\mbox{for $|i-j| \geq 2$}.
\end{align*}
Here and elsewhere, the indices will be taken modulo $n$ without
further mention.  The length $\ell(w)$ and reduced words $R(w)$ for affine permutations $w \in \tS_n$ are defined in an analogous manner to Section \ref{sec:perm}.  The symmetric group $S_n$ embeds in $\tS_n$ as the subgroup
generated by $s_1,s_2, \ldots, s_{n-1}$.  

One may realize $\tS_n$ as the set of all
bijections $w:\Z\rightarrow\Z$ such that $w(i+n)=w(i)+n$ for all
$i$ and $\sum_{i=1}^n w(i) = \sum_{i=1}^n i$. In this realization,
to specify an element $w \in \tS_n$ it suffices to give the
``window" $[w(1),w(2),\dotsc,w(n)]$.  The product $w\cdot v$ of
two affine permutations is then the composed bijection $w \circ v:
\Z \rightarrow \Z$.  Thus $ws_i$ is obtained from $w$ by swapping
the values of $w(i+kn)$ and $w(i+kn+1)$ for every $k \in \Z$.  An affine permutation $w \in \tS_n$ is Grassmannian if $w(1) < w(2) < \cdots < w(n)$.  For example, the affine Grassmannian permutation $[-2,2,6] \in \tS_3$ has reduced words $2120$ and $1210$.  We denote by $\tS_n^0 \subset \tS_n$ the subset of affine Grassmannian permutations.

\subsection{Definition}
A word $a_1a_2 \cdots a_\ell$ with letters in $\Z/n\Z$ is called {\it cyclically decreasing} if (1) each letter occurs at most once, and (2) whenever $i$ and $i+1$ both occur in the word, $i+1$ precedes $i$.

An affine permutation $w \in \tS_n$ is called cyclically decreasing if it has a cyclically decreasing reduced word.  Note that such a reduced word may not be unique.  

\begin{lemma}
There is a bijection between strict subsets of $\Z/n\Z$ and cyclically decreasing affine permutations $w \in \tS_n$, sending a subset $S$ to the unique cyclically decreasing affine permutation which has reduced word using exactly the simple generators $\{s_i \mid i \in S\}$.
\end{lemma}

We define cyclically decreasing factorizations of $w \in \tS_n$ in the same way as decreasing factorizations in $S_n$.  See also Chapter \ref{chapter.k schur primer}, Section \ref{subsection.weak order} where cyclically decreasing factorizations are discussed from the point of view of $k$-bounded partitions and $(k+1)$-cores.

\begin{definition}
\label{definition.affine Stanley}
Let $w \in \tS_n$.  The affine Stanley symmetric function $\tF_w$ is given by
$$
\tF_w = \sum_{w = v_1v_2\cdots v_r} x_1^{\ell(v_1)} x_2^{\ell(v_2)} \cdots x_r^{\ell(v_r)}
$$
where the summation is over cyclically decreasing factorizations of $w$.
\end{definition}

\begin{theorem}[\cite{Lam:2006}]\label{thm:afstsymm}
The generating function $\tF_w$ is a symmetric function.
\end{theorem}

Theorem \ref{thm:afstsymm} can be proved directly, as was done in \cite{Lam:2006}.  We shall establish Theorem \ref{thm:afstsymm} using the technology of the affine nilHecke algebra in Sections \ref{sec:affinenilHecke}-\ref{sec:affineFS}.  Some immediate observations:
\begin{enumerate}
\item
$\tF_w$ is a homogeneous of degree $\ell(w)$.
\item
If $w \in S_n$, then a cyclically decreasing factorization of $w$ is just a decreasing factorization of $w$, so $\tF_w = F_w$.
\item
The coefficient of $x_1 x_2 \cdots x_{\ell(w)}$ in $\tF_w$ is equal to $|R(w)|$.
\end{enumerate}

\begin{example}\label{ex:21202}
Consider the affine permutation $w = s_2 s_1 s_2 s_0 s_2$.  The reduced words are $R(w) = \{21202, 12102, 21020\}$.  The other cyclically decreasing factorizations are $$\os{21}\os{2}\os{0}\overline{2}, \os{2}\os{1}\os{2}\overline{02},
\os{1}\os{21}\os{0}\overline{2}, \os{1}\os{2}\os{1}\overline{02},\os{1}\os{2}\os{10}\overline{2},\os{21}\os{0}\os{2}\overline{0},\os{2}\os{1}\os{02}\overline{0},\os{2}\os{10}\os{2}\overline{0}$$
$$\os{21}\os{2}\overline{02}, \os{1}\os{21}\overline{02},\os{21}\os{02}\overline{0}$$
Thus
$$
\tF_w = m_{221}+2m_{2111}+3m_{11111}.
$$
\end{example}
\subsection{Codes}
Let $w \in \tS_n$.  The {\it code} $c(w)$ is a vector $c(w) =
(c_1,c_2,\ldots,c_{n}) \in \Z_{\geq 0}^n - \Z_{>0}^n$ of non-negative entries
with at least one 0.  The entries are given by $c_i = \#\{j \in
\Z \, \mid \, j > i \,\, \text{and} \,\, w(j) < w(i)\}$.

It is shown in \cite{BB} that there is a bijection between codes and
affine permutations and that $\ell(w) = |c(w)| := \sum_{i=1}^n c_i$.  We define $\la(w)$ as for usual permutations (see Section \ref{sec:combin}).  For example, for $w = s_2s_0s_1s_2s_1s_0 = [ -4,3,7] \in \tS_3$, one has $c(w^{-1})=(5,1,0)$ and $\la=(2,1,1,1,1)$.

Let $\Bo^n$ denote the set of partitions $\la$ satisfying $\la_1 < n$, called the set of $(n-1)$-bounded partitions.

\begin{lemma}[{\cite{BB}}]
The map $w \mapsto \la(w)$ is a bijection between $\tS_n^0$ and $\Bo^n$.
\end{lemma}

The analogue of Proposition \ref{prop:dominant} has a similar proof.
\begin{prop}[{\cite{Lam:2006}}] \label{prop:affdominant} Let $w \in \tS_n$.
\begin{enumerate}
\item
Suppose $[m_\la]F_w \neq 0$.  Then $\la \prec \la(w)$.
\item
$[m_{\la(w)}]F_w = 1$.
\end{enumerate}
\end{prop}

\subsection{$\La_{(n)}$ and $\La^{(n)}$}
Let $\La_{(n)} \subset \La$ be the subalgebra generated by $h_1,h_2,\ldots,h_{n-1}$, and let $\La^{(n)}:= \La/I_{(n)}$ where $I_{(n)}$ is the ideal generated by $m_\mu$ for $\mu \notin \Bo^n$.  A basis for $\La_{(n)}$ is given by $\{h_\la \mid \la \in \Bo^n\}$.  A basis for $\La^{(n)}$ is given by $\{m_\la \mid \la \in \Bo^n\}$.  Note that the rings $\La_{(n)}$ and $\La^{(n)}$ are denoted $\La_{(k)}$ and $\La^{(k)}$ in Chapter \ref{chapter.k schur primer}, where $k = n-1$.

The ring of symmetric functions $\La$ is a Hopf algebra, with coproduct 
given by $\Delta(h_k) = \sum_{j=0}^k h_j \otimes h_{k-j}$.  Equivalently, the coproduct of $f(x_1,x_2,\ldots) \in \La$ can be obtained by writing $f(x_1,x_2,\ldots,y_1,y_2,\ldots)$ in the form $\sum_i f_i(x_1,x_2,\ldots) \otimes g_i(y_1,y_2,\ldots)$ where $f_i$ and $g_i$ are symmetric in $x$'s and $y$'s respectively.  Then $\Delta(f)= \sum_i f_i \otimes g_i$.

The ring $\La$ is self Hopf-dual under the Hall inner product.  That is, one has $\ip{\Delta f, g\otimes h}  = \ip{f,gh}$ for $f,g,h\in \La$.  Here the Hall inner product is extended to $\La \otimes \La$ in the obvious way.  The rings $\La_{(n)}$ and $\La^{(n)}$ are in fact Hopf algebras, which are dual to each other under the same inner product.  We refer the reader to \cite{Mac:1995} for further details.

\subsection{Affine Schur functions}
Stanley symmetric functions expand positively in terms of the basis of Schur functions (Corollary \ref{cor:stpos}).  We now describe the analogue of Schur functions for the affine setting.

For $\la \in \Bo^n$, we let $\tF_\la:=\tF_w$ where $w \in \tS_n^0$ is the unique affine Grassmannian permutation with $\la(w) = \la$.  These functions $\tF_\la$ are called {\it affine Schur functions} (or dual $k$-Schur functions, or weak Schur functions).

\begin{theorem}[\cite{LM:2007,Lam:2006}]\label{thm:affineSchurbasis}
The affine Schur functions $\{\tF_\la\mid \la \in \Bo^n\}$ form a basis of $\La^{(n)}$.
\end{theorem}
\begin{proof}
By Proposition \ref{prop:affdominant}, the leading monomial term of $\tF_\la$ is $m_\la$.  Thus $\{\tF_\la\mid \la \in \Bo^n\}$ is triangular with respect to the basis $\{m_\la\mid \la \in \Bo^n\}$, so that it is also a basis.
\end{proof}

We let $\{s^{(k)}_{\la}\} \subset \La_{(n)}$ denote the dual basis to $\tF_\la$.  These are the (ungraded) {\it $k$-Schur functions}, where $k = n-1$.  It turns out that the $k$-Schur functions are Schur positive.  However, affine Stanley symmetric functions are not.  Instead, one has:

\begin{theorem}[\cite{Lam:2008}]\label{thm:affstpos}
The affine Stanley symmetric functions $\tF_w$ expand positively in terms of the affine Schur functions $\tF_\la$.
\end{theorem}

Theorem \ref{thm:affstpos} was established using geometric methods.  See Section \ref{sec:geom} and \cite{Lam:2008}.  It is an open problem to give a combinatorial interpretation of the affine Stanley coefficients, expressing affine Stanley symmetric functions in terms of affine Schur functions.

\subsection{Example: The case of $\tS_3$}
To illustrate Theorem \ref{thm:affineSchurbasis}, we completely describe the affine Schur functions for $\tS_3$.  

\begin{prop}\label{prop:redts3}
Let $w \in \tS_n$ be the affine Grassmannian permutation corresponding to the partition $(2^a1^b)$.  Then $|R(w)| = \binom{\lfloor b/2+a\rfloor}{a}$.
\end{prop}
\begin{prop}\label{prop:ts3}
The affine Schur function $\tF_{2^a,1^b}$ is given by
$$
\tF_{2^a,1^b} = \sum_{j = 0}^a \binom{\lfloor b/2+a-j\rfloor}{a-j} m_{2^j1^{b+2a-2j}}.
$$
The $k$-Schur function $s^{(2)}_{2^a,1^b}$ is given by
$$
s^{(2)}_{2^a,1^b}=h_2^a e_2^{\lfloor b/2 \rfloor} h_1^{b- 2\lfloor b/2 \rfloor}.
$$
\end{prop}

\begin{example}
For $w = 1210$, we have $a = 1$ and $b = 2$.  Thus $R(w) = \{1210,2120\}$ has cardinality $\binom{2}{1} = 2$, and $\tF_{2,1^2} = m_{211}+2m_{1111}$.
\end{example}
\begin{example}
The affine Stanley symmetric function of Example \ref{ex:21202} expands as
$$
\tF_w = \tF_{2^2,1}+\tF_{2,1^3}+\tF_{1^5}
$$
agreeing with Theorem \ref{thm:affstpos}.
\end{example}

\subsection{Exercises and problems}
\begin{enumerate}
\item (Coproduct formula \cite{Lam:2006})
Show that $\Delta \tF_w = \sum_{uv = w: \ell(w) = \ell(u)+\ell(v)} \tF_u \otimes \tF_v$.
\item (321-avoiding affine permutations \cite{Lam:2006})
Extend the results in Section \ref{ssec:combinprob} on 321-avoiding permutations to the affine case.  
\item (Affine vexillary permutations)
For which $w \in \tS_n$ is $\tF_w$ equal to an affine Schur function $\tF_\la$?  See the discussion of vexillary permutations in Section \ref{ssec:EGprob} and also \cite[Problem 1]{Lam:2006}.
\item \label{prob:limit} ($n \to \infty$ limit)
Show that for a fixed partition $\la$, we have $\lim_{n \to \infty} \tF^{(n)}_\la = s_\la$, where $\tF^{(n)}$ denotes the affine Schur function for $\tS_n$.
\item
Extend Proposition \ref{prop:ts3} to all affine Stanley symmetric functions in $\tS_3$, and thus give a formula for the affine Stanley coefficients.
\item
Is there an affine analogue of the fundamental quasi-symmetric functions?  For example, one might ask that affine Stanley symmetric functions expand positively in terms of such a family of quasi-symmetric functions.  Affine Stanley symmetric functions do not in general expand positively in terms of fundamental quasi-symmetric functions (see \cite[Theorem 5.7]{McN}).
\item
Find closed formulae for numbers of reduced words in the affine symmetric groups $\tS_n$, $n > 3$, extending Proposition \ref{prop:redts3}.  Are there formulae similar to the determinantal formula, or hook-length formula for the number of standard Young tableaux?
\item ($n$-cores)
A skew shape $\la/\mu$ is a $n$-ribbon if it is connected, contains $n$ squares, and does not contain a $2\times 2$ square.
An $n$-core $\la$ is a partition such that there does not exist $\mu$ so that $\la/\mu$ is a $n$-ribbon.
There is a bijection between the set of $n$-cores and the affine Grassmannian permutaitons $\tS_n^0$.  Affine Schur functions can be described in terms of tableaux on $n$-cores, called $k$-tableau \cite{LM:2005} (or weak tableau in \cite{LLMS:2006}).  See Chapter \ref{chapter.k schur primer} for further details.

\item (Cylindric Schur functions \cite{Pos,McN})
Let $C(k,n)$ denote the set of lattice paths $p$ in $\Z^2$ where every step either goes up or to the right, and which is invariant under the translation $(x,y) \mapsto (x+n-k,y+k)$.  Such lattice paths can be thought of as the boundary of an infinite periodic Young diagram, or equivalently of a Young diagram on a cylinder.  We write $p \subset q$ if $p$ lies completely to the left of $q$.  A cylindric skew semistandard tableau is a sequence $p_0 \subset p_1 \subset \cdots \subset p_k$ of $p_i \in C(k,n)$ where the region between $p_i$ and $p_{i+1}$ does not contain two squares in the same column.  One obtains \cite{Pos} a natural notion of a cylindric (skew) Schur function.  Show that every cylindric Schur function is an affine Stanley symmetric function, and every affine Stanley symmetric function of a 321-avoiding permutation is a cylindric Schur function (\cite{Lam:2006}).

\item (Kashiwara-Shimozono affine Grothendieck polynomials)
The usual Stanley symmetric functions can be expressed as stable limits of Schubert polynomials \cite{BJS}.  What is the relationship between affine Stanley symmetric functions and the affine Grothendieck polynomials of Kashiwara and Shimozono \cite{KS}?
\item
Is there a good notion of Coxeter-Knuth equivalence for reduced words of affine permutations?  This may have an application to the affine Schur positivity of affine Stanley symmetric functions (Theorem \ref{thm:affstpos}).  See also Section \ref{ssec:EGprob} (\ref{rem:CK}).
\item (Affine Little bijection \cite{LS:affineLittle})
There is an affine analogue of Little's bijection (Section \ref{ssec:EGprob}) developed in \cite{LS:affineLittle}.  It gives a combinatorial proof of the affine analogue of the transition formula \eqref{eq:transition}.  Can the affine Little bijection, or the affine transition formula lead to a proof of Theorem \ref{thm:affstpos}?  Can one define a notion of dual EG-equivalence using the affine Little bijection?
\item (Branching positivity \cite{LLMS:2010,Lam:ASP}) Let $\tF_\la^{(n)}$ denote the affine Schur functions for $\tS_n$.  Then $\tF_\la^{(n+1)}$ expands positively in terms of $\tF_\mu^{(n)}$ modulo the ideal in symmetric functions generated by $m_\nu$ with $\nu_1 \geq n$.  Deduce using (\ref{prob:limit}) that $k$-Schur functions are Schur positive.
\end{enumerate}

\section{Root systems and Weyl groups}
\label{sec:Weyl}
In this section, we let $\fW$ be a finite Weyl group and $\aW$ denote the corresponding affine Weyl group.
We shall assume basic familiarity with Weyl groups, root systems, and weights \cite{Hum,Kac}.

\subsection{Notation for root systems and Weyl groups}
Let $A = (a_{ij})_{i,j \in \aI}$ denote an affine Cartan matrix, where $\aI = \fI \cup\{0\}$, so that $(a_{ij})_{i,j \in \fI}$ is the corresponding finite Cartan matrix.  For example, for type $\tilde A_{n-1}$ (corresponding to $\tS_n$) and $n > 2$ we have $\aI = \Z/n\Z$ and 
$$
a_{ij} = \begin{cases} 2 &\mbox{if $i = j$} \\
-1 &\mbox{if $j = i\pm 1$} \\
0 & \mbox{otherwise.}
\end{cases}
$$

The affine Weyl group $\aW$ is generated by  involutions $\{s_i \mid i \in \aI\}$ satisfying the relations $(s_is_j)^{m_{ij}} = \id$, where for $i \neq j$, one defines $m_{ij}$ to be $2,3,4,6,\infty$ according as $a_{ij}a_{ji}$ is $0, 1, 2, 3, \ge4$.  The finite Weyl group $\fW$ is generated by $\{s_i \mid i \in \fI\}$.  For the symmetric group $W = S_n$, we have $I = [n-1]$, $m_{i,i+1} = 3$, and $m_{ij} = 2$ for $|i-j| \geq 2$.

Let $R$ be the root system
for $W$.  Let $R^+$ denote the positive roots, $R^-$ denote the
negative roots and $\{ \alpha_i \mid i \in \fI\}$ denote the simple
roots.   Let $\theta$ denote the highest root of $R^+$.  Let $\rho = \frac{1}{2}\sum_{\alpha \in R^+} \alpha$ denote the half sum of positive roots.  Also let $\{\alpha_i^\vee \mid i \in \fI\}$ denote the simple coroots.  

We write $R_\af$ and $R^+_\af$ for the affine root system, and positive affine roots.  The positive simple affine roots (resp. coroots) are $\{\alpha_i \mid i \in \aI\}$ (resp. $\{\alpha^\vee_i \mid i \in \aI\}$).  The null root $\delta$ is given by $\delta = \alpha_0 + \theta$.  A root $\alpha$ is {\it real} if it is in the $\aW$-orbit of the simple affine roots, and {\it imaginary} otherwise.  The imaginary roots are exactly $\{k\delta \mid k \in \Z \setminus \{0\}\}$.  Every real affine root is of the form $\alpha + k\delta$, where $\alpha \in R$.   The root $\alpha+k\delta$ is positive if $k > 0$, or if $k = 0$ and $\alpha \in R^+$.  

Let $Q = \oplus_{i \in \fI} \Z \cdot \alpha_i$ denote the root lattice and let $Q^\vee = \oplus_{i \in I} \Z \cdot \alpha_i^\vee$ denote the co-root
lattice.  Let $P$ and $P^\vee$ denote the weight lattice and
co-weight lattice respectively.  Thus $Q \subset P$ and $Q^\vee \subset P^\vee$.  We also have a map $Q_\af = \oplus_{i \in \aI} \Z \cdot \alpha_i \to P$ given by sending $\alpha_0$ to $-\theta$ (or equivalently, by sending $\delta$ to $0$).  Let $\ip{.,.}$ denote
the pairing between $P$ and $P^\vee$.  In particular, one requires that $\ip{\al_i^\vee,\al_j}=a_{ij}$.  

The Weyl group acts on weights via $s_i \cdot \la = \la - \ip{\alpha_i^\vee,\la}\alpha_i$ (and via the same formula on $Q$ or $Q_\af$), and on coweights via $s_i \cdot \mu = \mu - \ip{\mu,\alpha_i}\alpha^\vee_i$ (and via the same formula on $Q^\vee$).  For a real root $\alpha$ (resp. coroot $\alpha^\vee$), we let $s_\alpha$ 
(resp. $s_{\alpha^\vee}$) denote the corresponding reflection, defined by $s_\alpha = wr_i w^{-1}$ if $\alpha = w \cdot \alpha_i$.  The reflection $s_\alpha$ acts on weights by $s_\alpha \cdot \la = \la - \ip{\alpha^\vee,\la}\alpha$.

\begin{example}
Suppose $W = S_n$ and $\aW = \tS_n$.  We have positive simple roots $\alpha_1,\alpha_2,\ldots,\alpha_{n-1}$ and an affine simple root $\alpha_0$.
The finite positive roots are $R^+=\{ \alpha_{i,j} := \alpha_i + \alpha_{i+1} \cdots + \alpha_{j-1} \mid 1 \leq i < j \leq n\}$.  The reflection $s_{\alpha_{i,j}}$ is the transposition $(i,j)$.  The highest root is $\theta = \alpha_1 + \cdots + \alpha_{n-1}$.  The affine positive roots are $R^+_\af = \{\alpha_{i,j} \mid i < j \}$, where for simple roots the index is taken modulo $n$.  Note that one has 
$\alpha_{i,j} = \alpha_{i+n,j+n}$.  The imaginary roots are of the form $\alpha_{i,i+kn}$.  For a real root $\alpha_{i,j}$, the reflection $s_{\alpha_{i,j}}$ is the affine transposition $(i,j)$.

The weight lattice can be taken to be $P = \Z^n/(1,1,\ldots,1)$, and the coweight lattice to be $P^\vee = \{(x_1,x_2,\ldots,x_n) \in \Z^n \mid \sum_i x_i = 0\}$.  The roots and coroots are then $\alpha_{i,j} = e_i-e_j = \alpha^\vee_{i,j}$ (though the former is in $P$ and the latter is in $P^\vee$).  The inner product $P^\vee \times P \to \Z$ is induced by the obvious one on $\Z^n$.
\end{example}

\subsection{Affine Weyl group and translations}
The affine Weyl group can be expressed as the semi-direct product $\aW= W \ltimes Q^\vee$, as follows.  For each $\la \in Q^\vee$, one has a translation element $t_\la \in \aW$.  Translations are multiplicative, so that $t_\lambda \cdot t_\mu$ =
$t_{\lambda+\mu}$.  We also have the conjugation formula $w\, t_\lambda
w^{-1} = t_{w \cdot \lambda}$ for $w \in W$ and $\la \in Q^\vee$.
Let
$s_0$ denote the additional simple generator of $\aW$.  Then translation elements are related to the simple generators via the formula$$s_0 = s_{\theta^\vee} t_{-\theta^\vee}.$$ 

\begin{example}
For $\aW = \tS_n$, and $\la = (\la_1,\la_2,\ldots,\la_n) \in Q^\vee$, we have
$$
t_\la= [1+n\la_1,2+n\la_2,\ldots,n+n\la_n].
$$
Thus $t_{-\theta^\vee} = [1-n,2,\ldots,n-1,2n]$ and $s_0=s_{\theta^\vee}t_{-\theta^\vee}$ is the equality
$$
[0,2,\ldots,n-1,n+1]=[n,2,\ldots,n-1,1]\cdot[1-n,2,\ldots,n-1,2n].
$$
\end{example}

The element $wt_\la$ acts on $\mu \in P$ via
\begin{equation}\label{eq:affinewaction}
wt_\la \cdot \mu = w \cdot \mu.
\end{equation}
In other words, the translations act trivially on the {\it finite} weight lattice.  This action is called the {\it level zero action}.  

Let $\ell: \aW \rightarrow \Z_{\geq 0}$ denote the length function of
$\aW$.  Thus $\ell(w)$ is the length of the shortest reduced factorization of $w$.  

\begin{exercise}
For $w\, t_\la \in \aW$, we have
\begin{equation}
\label{eq:length} \ell(w\,t_\lambda) = \sum_{\alpha \in R^+} |
\ip{\lambda,\alpha} + \chi(w\cdot\alpha)|,
\end{equation}
where $\chi(\alpha) = 0$ if $\alpha \in R^+$ and $\chi(\alpha) = 1$
otherwise.
\end{exercise}

A coweight $\la$ is {\it dominant} (resp. {\it anti-dominant}) if $\ip{\la,\alpha} \geq 0$ (resp. $\leq 0$) for every root $\alpha \in R^+$.

\begin{exercise}
Suppose $\la \in Q^\vee$ is dominant.  Then $\ell(t_{w\la}) = 2\ip{\la,\rho}$.
\end{exercise}

Let $\aW^0$ denote the minimal length coset representatives of
$\aW/W$, which we call {\it Grassmannian} elements. There is a
natural bijection between $\aW^0$ and $Q^\vee$: each coset $\aW/W$
contains one element from each set.  

\begin{exercise}
We have $\aW^0 \cap Q^\vee = Q^-$, the elements of the coroot lattice which are anti-dominant.
\end{exercise}

In fact an element $w t_\lambda$ lies in $\aW^0$ if and only if $t_\lambda \in Q^-$ and $w \in W^\lambda$ where
$W^\lambda$ is the set of minimal length representatives of $W/W_\lambda$ and $W_\lambda$ is the stabilizer subgroup of
$\lambda$.

\section{NilCoxeter algebra and Fomin-Stanley construction}
\label{sec:algebra}
Let $W$ be a Weyl group and $\aW$ be the corresponding affine Weyl group.
\subsection{The nilCoxeter algebra}
The {\it nilCoxeter algebra} $\A_0$ is the algebra over $\Z$ generated by $\{A_i \mid i \in I\}$ with relations
\begin{align*}
A_i^2 &= 0 \\ 
(A_iA_j)^b &= (A_jA_i)^b & \text{if } (s_is_j)^b = (s_js_i)^b \\ A_j(A_iA_j)^b &= (A_iA_j)^bA_i &\text{if } s_j(s_is_j)^b = (s_is_j)^bs_i
\end{align*}
The algebra $\A_0$ is graded, where $A_i$ is given degree $1$.

The corresponding algebra for the affine Weyl group will be denoted $\aNC$.  This algebra is also discussed in Chapter \ref{chapter.k schur primer}, Section \ref{sec:nilcoxeter}.

\begin{proposition}
The nilCoxeter algebra $\A_0$ has basis $\{A_w \mid w \in W\}$, where $A_w = A_{i_1}A_{i_2} \cdots A_{i_\ell}$ for any reduced word $i_1 i_2 \cdots i_\ell$ of $w$.  The mulitplication is given by
$$
A_w \,A_v = \begin{cases} A_{wv} & \mbox{if $\ell(wv) = \ell(w) + \ell(v)$} \\
0 & \mbox{otherwise.}
\end{cases}
$$
\end{proposition}

\subsection{Fomin and Stanley's construction}
Suppose $W = S_n$.  We describe the construction of Stanley symmetric functions of Fomin and Stanley \cite{FS}.  Define
$$
\h_k = \sum_{w \text{ decreasing: } \ell(w) = k} A_w.
$$
For example, when $n = 4$, we have
\begin{align*}
\h_0 &= \id \\
\h_1&=A_1 + A_2 + A_3 \\
\h_2&=A_{21} + A_{31} + A_{32} \\
\h_3&=A_{321}
\end{align*}
where $A_{i_1i_2\cdots i_\ell}:= A_{i_1}A_{i_2} \cdots A_{i_\ell}$.  

\begin{lemma}[\cite{FS}]\label{lem:hproduct}
The generating function $\h(t) = \sum_k \h_k t^k$ has the product expansion
$$
\h(t) = (1+t\,A_{n-1})(1+t\,A_{n-2}) \cdots (1+t\,A_1).
$$
\end{lemma}

\begin{definition}[NilCoxeter algebra] \label{def:stalg}
The Stanley symmetric function $F_w$ is the coefficient of $A_w$ in the product
$$
\h(x_1)\h(x_2)\cdots
$$
\end{definition}

\begin{lemma}[\cite{FS}]\label{lem:finstcommute}
We have 
$$
\h(x) \h(y) = \h(y) \h(x).
$$
Thus for every $k, l$ we have
$$
\h_k \h_l = \h_l \h_k.
$$
\end{lemma}
\begin{proof}
We observe that $(1+xA_i)$ and $(1+yA_j)$ commute whenever $|i-j| \geq 2$ and that
$$
(1+ xA_{i+1})(1+ xA_{i})(1+yA_{i+1})
=(1+yA_{i+1})(1+ yA_{i})(1+ xA_{i+1})(1-yA_i)(1+xA_i).
$$
Assuming by induction that the result is true for $S_{n-1}$ we calculate
\begin{align*}
&(1+xA_{n-1})\cdots (1+xA_1)(1+yA_{n-1})\cdots (1+yA_1) \\
&= \left[(1+yA_{n-1})(1+yA_{n-2})(1+xA_{n-1})(1-yA_{n-2})(1+xA_{n-2})\right]\\&(1+xA_{n-3})\cdots (1+xA_1)(1+yA_{n-2})\cdots (1+yA_1) \\
&= (1+yA_{n-1})(1+yA_{n-2})(1+xA_{n-1})(1-yA_{n-2})\left[(1+yA_{n-2})\cdots (1+yA_1)\right]\\& \left[ (1+xA_{n-2})\cdots (1+xA_1) \right]\\
&=(1+yA_{n-1})\cdots (1+yA_1)(1+xA_{n-1})\cdots (1+xA_1).
\end{align*}
\end{proof}

\begin{proof}[Proof of Theorem \ref{thm:stsymm}]
Follows immediately from Definition \ref{def:stalg} and Lemma \ref{lem:finstcommute}.
\end{proof}

The following corollary of Lemma \ref{lem:finstcommute} suggests a way to generalize these constructions to other finite and affine Weyl groups.

\begin{corollary}\label{cor:finFS}
The elements $\h_k$ generate a commutative subalgebra of $\A_0$.
\end{corollary}

We call the subalgebra of Corollary \ref{cor:finFS} the {\it Fomin-Stanley} sublagebra of $\A_0$, and denote it by $\B$.  As we shall explain, the combinatorics of Stanley symmetric functions    is captured by the algebra $\B$, and the information can be extracted by ``picking a basis''.

\subsection{A conjecture}
We take $W$ to be an arbitrary Weyl group.  For basic facts concerning the exponents of $W$, we refer the reader to \cite{Hum}.
The following conjecture was made by the author and Alex Postnikov.  Let $(R^+, \prec)$ denote the partial order on the positive roots of $W$ given by $\alpha \prec \beta$ if $\beta - \alpha$ is a positive sum of simple roots, and let $J(R^+,\prec)$ denote the set of upper order ideals of $(R^+,\prec)$.

\begin{conjecture}\label{conj:LP}
The (finite) nilCoxeter algebra $\A_0$ contains a graded commutative subalgebra $\B'$ satisfying:
\begin{enumerate}
\item
Over the rationals, the algebra $\B' \otimes_\Z \Q$ is generated by homogeneous elements $\h_{i_1}, \h_{i_2}, \ldots, \h_{i_r} \in \A_0$ with degrees $\deg(\h_{i_j}) = i_j$ given by the exponents of $W$.
\item
The Hilbert series $P(t)$ of $\B'$ is given by 
$$
P(t) = \sum_{I \in J(R^+,\prec)} t^{|I|}
$$
In particular, the dimension of $\B'$ is a the generalized Catalan number for $W$ (see for example \cite{FR}).
\item 
The set $\B'$ has a homogeneous basis $\{b_I \mid I \in J(R^+,\prec)\}$ consisting of elements which are nonnegative linear combinations of the $A_w$.
\item
The structure constants of the basis $\{b_I\}$ are positive.
\end{enumerate}
\end{conjecture}

In the sequel we shall give an explicit construction of a commutative subalgebra $\B$ and provide evidence that it satisfies Conjecture \ref{conj:LP}.  

Suppose $W = S_n$.  We show that $\B'=\B$ satisfies Conjecture \ref{conj:LP}.  The upper order ideals $I$ of $(R^+,\prec)$ are naturally in bijection with Young diagrams fitting inside the staircase $\delta_{n-1}$.  For each partition $\la$ we define the noncommutative Schur functions $\s_\la \in \B$, following Fomin and Greene \cite{FG:1998}, by writing $s_\la$ as a polynomial in the $h_i$, and then replacing $h_i$ by $\h_i$.  (We set $\h_k = 0$ if $k \geq n$, and $\h_0 = 1$.)  Fomin and Greene show that $\s_\la$ is a nonnegative linear combination of the $A_w$'s (but it will also follow from our use of the Edelman-Greene correspondence below).

\begin{prop}[Lam - Postnikov]\label{prop:LP}
The set $\{\s_\la \mid \la \subseteq \delta_{n-1}\} \subset \B$ is a basis for $\B$.
\end{prop}
\begin{proof}
Let $\ip{.,.}:\A_0 \otimes \A_0 \to \Z$ be the inner product defined by extending bilinearly $\ip{A_w,A_v} = \delta_{wv}$.
Rewriting the definition of Stanley symmetric functions and using the Cauchy identity, one has
\begin{align*}
F_w &= \sum_\la \ip{\h_{\la_1}\h_{\la_2}\cdots, A_w}m_\la\\
&=\sum_\la \ip{\s_\la, A_w}s_\la.
\end{align*}
It follows that the coefficient of $A_w$ in $\s_\la$ is $\alpha_{w\la}$, the coefficient of $s_\la$ in $F_w$.  By Lemma \ref{lem:EGshape}, we have $\s_\la = 0$ unless $\la \subseteq \delta_{n-1}$.  It remains to show that this set of $\s_\la$ are linearly independent.  To demonstrate this, we shall find, for each $\la \in \delta_{n-1}$, some $w \in S_n$ such that $F_w = s_\la + \sum_{\mu \prec \la} \alpha_{w\mu} s_\mu$.

Since $s_\la = m_\la + \sum_{\mu \prec \la} K_{\la\mu} m_\mu$ (where the $K_{\la\mu}$ are the Kostka numbers), by Proposition \ref{prop:dominant} it suffices to show that the permutation $w$ with code $c(w) = \la$ lies in $S_n$.  Let $\la = (\la_1\geq\la_2\geq \ldots \geq\la_{\ell} > 0)$.  Then define $w$ recursively by $w(i) = \min\{j > \la_i \mid j \notin \{w(1),w(2),\ldots,w(i-1)\}$.  Since $\la_i \leq n-i$, we have $w(i) \leq n$.  By construction $w \in S_n$ and has code $c(w) = \la$.
\end{proof}

\begin{example}\label{ex:S4}
Let $W = S_4$.  Write $A_{i_1i_2\cdots i_k}$ for $A_{s_{i_1}s_{i_2}\cdots s_{i_k}}$.

Then we have
\begin{align*}
\begin{array}{ll}
\s_\emptyset = 1 &
\s_1 = A_1 + A_2 + A_3 \\
\s_{11} = A_{12} + A_{23} + A_{13} & \s_2 = A_{21}+A_{32}+A_{31} \\
\s_{111} = A_{123}  & \s_3 = A_{321} \\
\s_{21} = A_{213}+A_{212}+A_{323}+A_{312}&
\s_{211}= A_{1323} + A_{1213} \\ 
\s_{22}= A_{2132} 
&\s_{31}= A_{3231}+A_{3121} \\
\s_{221}=A_{23123} & \s_{311}=A_{32123} \\
\s_{32}=A_{32132}&
\s_{321}=A_{321323}
\end{array}
\end{align*}
\end{example}

\subsection{Exercises and Problems}
\label{ssec:algebraprob}
\begin{enumerate}
\item 
In Example \ref{ex:S4} every $A_w$ occurs exactly once, except for $w = s_1s_3$.  Explain this using Theorem \ref{thm:CK}.
\item (Divided difference operators)
Let $\partial_i$, $1 \leq i \leq n-1$ denote the divided difference operator acting on polynomials in $x_1,x_2,\ldots,x_n$ by
$$
\partial_i \cdot f(x_1,\ldots,x_n) = \frac{f(x_1,x_2,\ldots,x_n) - f(x_1,\ldots,x_{i-1},x_{i+1},x_i,x_{i+2},\ldots,x_n)}{x_i - x_{i+1}}.
$$
Show that $A_i \mapsto \partial_i$ generates an action of the nilCoxeter algebra $\A_0$ for $S_n$ on polynomials.  Is this action faithful?
\item (Center)
What is the center of $\A_0$ and of $\aNC$?  See \cite{Bri}.
\item
Does an analogue of Conjecture \ref{conj:LP} also hold for finite Coxeter groups which are not Weyl groups?
\item 
How many terms (counted with multiplicity) are there in the elements $\s_\la$ of Proposition \ref{prop:LP}?  This is essentially the same problem as (\ref{prob:enumerateEG}) in Section \ref{ssec:EGprob} (why?).
\end{enumerate}

\section{The affine nilHecke ring}\label{sec:affinenilHecke}
Kostant and Kumar \cite{KK} introduced a nilHecke ring to study the topology of Kac-Moody flag varieties.  This ring is studied in detail in Chapter \ref{chapter.affineSchubert}, Section \ref{sec:MarknilHecke}.  We give a brief development here.  Let $W$ be a Weyl group and $\aW$ be the corresponding affine Weyl group.

\subsection{Definition of affine nilHecke ring}
\label{chapter3.section.affine nilHecke}
In this section we study the affine nilHecke ring $\afA$ of Kostant and Kumar \cite{KK}.  Kostant and Kumar define the nilHecke ring in the Kac-Moody setting.  The ring $\afA$ below is a ``small torus'' variant of their construction for the affine Kac-Moody case.  

The affine nilHecke ring $\afA$ is the ring with a $1$
given by generators $\{A_i \mid i \in \aI \} \cup \{\lambda
\mid \lambda \in P \}$ and the relations
\begin{align}\label{eq:aila}
A_i \,\lambda &= (s_i \cdot \lambda)\, A_i +
\ip{\alpha_i^\vee,\la}\cdot1 & \mbox{for $\la \in P$,} \\ 
A_i\, A_i &= 0, \\ 
\label{eq:aiaj}
(A_iA_j)^m &= (A_jA_i)^m & \mbox{if $(s_is_j)^m = (s_js_i)^m$.}
\end{align}
where the ``scalars'' $\la \in P$ commute with other
scalars.  Thus $\afA$ is obtained from the affine nilCoxeter algebra $\aNC$ by adding the scalars $P$.  The finite nilHecke ring is the subring $\A$ of $\afA$
generated by $\{A_i \mid i \in \fI \} \cup \{\lambda \mid \lambda \in
P\}$. 

Let $w \in \aW$ and let $w = s_{i_1} \cdots s_{i_l}$ be a reduced
decomposition of $w$.  Then $A_w := A_{i_1} \cdots A_{i_l}$ is a
well defined element of $\afA$, where $A_\id = 1$.

Let $S = \Sym(P)$ denote the symmetric algebra of $P$.  The following 
basic result follows from \cite[Theorem 4.6]{KK}, and can be proved directly from the definitions.

\begin{lemma}[See Chapter \ref{chapter.affineSchubert}, Lemma \ref{L:Abasis}]\label{lem:Abasis}
The set $\{A_w \mid w \in \aW\}$ is an
$S$-basis of $\afA$.
\end{lemma}

\begin{lemma}\label{lem:whomo}
The map $\aW \mapsto \afA$ given by 
$s_i \mapsto 1 - \alpha_i A_i \in \afA$
is a homomorphism.
\end{lemma}
\begin{proof}
We calculate that
\begin{align*}
s_i^2 &= 1 -2\al_i A_i + \al_i A_i \al_i A_i \\
&= 1 - 2\al_i A_i + \al_i (-\al_iA_i + 2)A_i &\mbox{using \eqref{eq:aila}}\\
&= 1 & \mbox{using $A_i^2 = 0$.}
\end{align*}
If $A_i A_j = A_i A_j$ then $\al_i A_j = A_j \al_i$ so it is easy to see that $s_is_j = s_js_i$.  Suppose $A_i A_j A_i = A_j A_i A_j$.  Then $a_{ij} = a_{ji} = -1$ and we calculate that
\begin{align*}
s_i s_j s_i & = (1-\al_iA_i)(1-\al_j A_j)(1-\al_i A_i)  \\
&=  1 - (2\al_iA_i+\al_jA_j) + (\al_iA_i \al_j A_j +\al_i A_i\al_i A_i + \al_jA_j \al_iA_i)
- \al_iA_i \al_j A_j \al_i A_i \\
&= 1 - (2\al_iA_i+\al_jA_j )+ (2\al_iA_j+2\al_iA_i + 2\al_j A_i)  + \al_i(\al_i+\al_j)A_i A_j + \\ &\;\;\;\;\;\al_j(\al_i+\al_j)A_j A_i - (\al_i A_i + \al_i (\al_i + \al_j) \al_jA_i A_j A_i) \\
&= 1 + (2 \al_i A_j + 2 \al_j A_i -\al_iA_i -\al_j A_j ) + \al_i(\al_i+\al_j)A_i A_j +\\ &\;\;\;\;\;\al_j(\al_i+\al_j)A_j A_i -  \al_i (\al_i + \al_j) \al_j A_i A_j A_i.
\end{align*}
Since the above expression is symmetric in $i$ and $j$, we conclude that $s_is_js_i = s_js_is_j$.
\end{proof}

\begin{exercise}
Complete the proof of Lemma \ref{lem:whomo} for $(A_iA_j)^2= (A_jA_i)^2$ and $(A_iA_j)^3= (A_jA_i)^3$.
\end{exercise}

It follows from Lemma \ref{lem:Abasis} that the map of Lemma \ref{lem:whomo} is an isomorphism onto its image.  Abusing notation, we write $w \in \afA$ for the element in the
nilHecke ring corresponding to $w \in \aW$ under the map of Lemma \ref{lem:whomo}.  Then
$\aW$ is a basis for $\afA \otimes_S \Frac(S)$ over $\Frac(S)$ (not over $S$
since $A_i = \frac{1}{\alpha_i}(1-s_i)$).

\begin{lemma}\label{lem:waction}
Suppose $w \in \aW$ and $s \in S$.  Then $w s = (w\cdot s) w$.
\end{lemma}
\begin{proof}
It suffices to establish this for $w = s_i$, and $s = \la \in P$.  We calculate
\begin{align*}
s_i \la &= (1-\al_i A_i)\la \\
&= \la - \al_i (s_i \cdot \la) A_i - \al_i \ip{\al_i^\vee,\la}\\
&= (s_i \cdot \la) - (s_i\cdot \la) \al_i A_i \\
& = (s_i \cdot \la) s_i.
\end{align*}
\end{proof}


\subsection{Coproduct}
\label{sec:nilHeckecoproduct}
The following Proposition is established in Chapter \ref{chapter.affineSchubert}, Section \ref{sec:coproduct}.   Note that the Weyl group $W$ there should be taken to be the affine Weyl group $\aW$ here.

\begin{prop} \label{P:coprodaction} Let $M$ and $N$ be left $\afA$-modules.
Define
\begin{equation*}
  M \otimes_{S} N = (M\otimes_\Z N)/\langle sm\otimes n - m \otimes sn \mid
  \text{$s\in S$, $m\in M$, $n\in N$} \rangle.
\end{equation*}
Then $\afA$ acts on $M \otimes_{S} N$ by
\begin{align*}
s \cdot (m\otimes n) &= sm \otimes n \\
 A_i \cdot (m\otimes n) &= A_i \cdot m \otimes n + m \otimes A_i\cdot
  n - \al_i A_i\cdot m\otimes A_i\cdot n.
\end{align*}
Under this action we have
\begin{equation} \label{E:tensorWeyl}
  w \cdot (m\otimes n) = wm \otimes wn
\end{equation}
for any $w \in \aW$.
\end{prop}

Consider the case $M=\afA = N$. By Proposition \ref{P:coprodaction}
there is a left $S$-module homomorphism $\Delta:\afA \to \afA
\otimes_{S} \afA$ defined by $\Delta(a) = a \cdot (1 \otimes 1)$. It satisfies
\begin{align}
  \Delta(q) &= s \otimes 1 &\qquad&\text{for $s\in S$} \\
\label{E:DeltaT}
  \Delta(A_i) &= 1 \otimes A_i + A_i \otimes 1 - \al_i A_i \otimes A_i&\qquad&\text{for $i\in I$.}
\end{align}

Let $a\in\afA$ and $\Delta(a)=\sum_{v,w} a_{v,w} A_v \otimes A_w$ with
$a_{v,w}\in S$. In particular if
$b\in\afA$ and $\Delta(b)=\sum_{v',w'} b_{v',w'} A_{v'} \otimes A_{w'}$ then
\begin{equation}\label{E:Deltamult}
\Delta(ab) =  \Delta(a)\cdot \Delta(b) := \sum_{v,w,v',w'} a_{v,w}
b_{v',w'} A_v A_{v'} \otimes A_wA_{w'}.
\end{equation}

\begin{remark}
We caution that $\afA \otimes _S \afA$ is not a well-defined ring with the obvious multiplication.
\end{remark}

\subsection{Exercises and Problems}
The theory of nilHecke rings in the Kac-Moody setting is well-developed \cite{KK,Kum}.  See Chapter \ref{chapter.affineSchubert}, Section \ref{sec:MarknilHecke}.
\begin{enumerate}
\item
The following result is \cite[Proposition 4.30]{KK}.  Let $w \in \aW$ and $\la \in P$. Then
\[
A_w \lambda = (w \cdot \lambda)A_w + \sum_{w\,s_\alpha: \; \ell(ws_\alpha) = \ell(w) - 1}\ip{\al^\vee,\lambda} A_{w\,s_\alpha},
\]
where $\alpha$ is always taken to be a positive root of $\aW$.  The coefficients $\ip{\al^\vee,\la}$ are known as {\it
Chevalley coefficients}.
\item
(Center) What is the center of $\afA$?  (See \cite[Section 9]{Lam:2008} for related discussion.)
\end{enumerate}

\section{Peterson's centralizer algebras}\label{sec:Pet}
Peterson studied a subalgebra of the affine nilHecke ring in his work \cite{Pet} on the homology of the affine Grassmannian.

\subsection{Peterson algebra and $j$-basis}
The Peterson centralizer subalgebra $\Pet$ is the centralizer $Z_{\afA}(S)$
of the scalars $S$ in the affine nilHecke ring $\afA$.  In this section we establish some basic
properties of this subalgebra.  The results here are unpublished works of Peterson.

\begin{lemma}\label{lem:Ptrans}
Suppose $a \in \afA$.  Write $a = \sum_{w \in \aW} a_w w$, where $a_w \in \Frac(S)$.  Then
$a \in \Pet$ if and only if $a_w = 0$ for all non-translation elements $w \neq t_\la$.
\end{lemma}
\begin{proof}
By Lemma \ref{lem:waction}, we have for $s \in S$
$$
(\sum_w a_w w)s = \sum_w a_w (w\cdot s) w
$$
and so $a \in \Pet$ implies $a_w(w \cdot s) = a_w s$ for each $s$.  But using \eqref{eq:affinewaction}, one sees that every $w \in \aW$ acts non-trivially on $S$ except for the translation elements $t_\la$.  Since $S$ is an integral domain, this implies that $a_w = 0$ for all non-translation elements.
\end{proof}

\begin{lemma}\label{lem:Petcomm}
The subalgebra $\Pet$ is commutative.
\end{lemma}
\begin{proof}
Follows from Lemma \ref{lem:Ptrans} and the fact that the elements $t_\la$ commute, and commute with $S$.
\end{proof}

The following important result is the basis of Peterson's approach to affine Schubert calculus via the affine nilHecke ring.

\begin{theorem}[\cite{Pet,Lam:2008}]\label{thm:jbasis}
The subalgebra $\Pet$ has a basis $\{j_w \mid w \in \aW^0\}$ where
$$
j_w = A_w + \sum_{x \in \aW - \aW^0} j_w^x A_x
$$
for $j_w^x \in S$.
\end{theorem}
Peterson constructs the basis of Theorem \ref{thm:jbasis} using the geometry of
based loop spaces (see \cite{Pet,Lam:2008}).  We sketch a purely algebraic proof of this theorem, following the ideas of Lam, Schilling, and Shimozono \cite{LSS}.  A full proof of this theorem is given in Chapter \ref{chapter.affineSchubert}.

\subsection{Sketch proof of Theorem \ref{thm:jbasis}}
\label{chapter3.section.sketch}
Let $\Fun(\aW, S)$ denote the set of functions $\xi: \aW \to S$.  We may think of functions $\xi \in \Fun(\aW,S)$ as functions on $\afA$, by the formula $\xi(\sum_{w \in \aW} a_w \, w) = \sum_{w \in \aW} \xi(w) a_w$.  Note that if $a = \sum_{w \in \aW} a_w \, w \in \afA$, the $a_w$ may lie in $\Frac(S)$ rather than $S$, so that in general $\xi(a) \in \Frac(S)$.  Define 
$$
\aPsi:= \{\xi \in \Fun(\aW,S) \mid \xi(a) \in S \text{ for all $a \in \afA$}\}.
$$
(The ring $\aPsi$ is denoted $\Lambda_{\af}$ in Chapter \ref{chapter.affineSchubert}.)
Let 
$$
\aPsi^0 = \{\xi \in \aPsi \mid \xi(w) = \xi(v)  \text{ whenever $wW = vW$}\}.
$$
It follows easily that $\aPsi$ has a basis over $S$ given by $\{\xi^w \mid w \in \aW\}$ where $\xi^w(A_v) = \delta_{vw}$ for every $ v\in \aW$.  Similarly, $\aPsi^0$ has a basis given by  $\{\xi^w_0 \mid w \in \aW^0\}$ where $\xi^w_0(A_v) = \delta_{vw}$ for every $v \in \aW^0$.

The most difficult step is the following statement:
\begin{lem}\label{lem:wrongway}
There is a map $\om: \aPsi \to \aPsi^0$ defined by $\om(\xi)(t_\la) = \xi(t_\la)$.
\end{lem}

In other words, $\om(\xi)$ remembers only the values of $\xi$ on translation elements, and one notes that $\om(\xi^w) = \xi^w_0$ for $w \in \aW^0$.  By Lemma \ref{lem:Ptrans}, $\xi(a) = \om(\xi)(a)$ for $\xi \in \aPsi$ and $a \in \Pet$.  We define $\{j_w \in \Pet \otimes \Frac(S) \mid w \in \aW^0\}$ by the equation
$$
\xi_0^v(j_w) = \delta_{wv}
$$
for $w,v \in \aW^0$.  Such elements $j_w$ exist (and span $\Pet \otimes \Frac(S)$ over $\Frac(S)$) since each $\xi \in \aPsi^0$ is determined by its values on the translations $t_\la$.  The coefficient $j_w^x$ of $A_x$ in $j_w$ is equal to the coefficient of $\xi^w_0$ in $\om(\xi^x)$, which by Lemma \ref{lem:wrongway}, must lie in $S$.  This proves Theorem \ref{thm:jbasis}.

Finally, Lemma \ref{lem:wrongway} follows from the following description of $\aPsi$ and $\aPsi^0$, the latter due to Goresky-Kottwitz-Macpherson \cite[Theorem 9.2]{GKM}.  See \cite{LSS} for an algebraic proof in a slightly more general situation.
\begin{proposition}
Let $\xi \in \Fun(\aW,S)$.  Then $\xi \in \aPsi$ if and only if for each $\alpha \in R$, $w \in \aW$, and each integer $d > 0$ we have
\begin{equation}\label{eq:smallGKM}
\xi(w(1-t_{\alpha^\vee})^{d-1}) \qquad \mbox{is divisible by $\alpha^{d-1}$}
\end{equation}
and
\begin{equation}
\xi(w(1-t_{\alpha^\vee})^{d-1}(1-s_\alpha)) \qquad \mbox{is divisible by $\alpha^d$.}
\end{equation}
Let $\xi \in \Fun(\aW,S)$ satisfy $\xi(w) = \xi(v)$ whenever $wW = vW$.  Then $\xi \in \aPsi^0$ if it satisfies \eqref{eq:smallGKM}.
\end{proposition}
\begin{remark}
The ring $\aPsi$ is studied in detail by Kostant and Kumar \cite{KK} in the Kac-Moody setting.  See Chapter \ref{chapter.affineSchubert}.
\end{remark}

\subsection{Exercises and Problems}
\begin{enumerate}
\item
Show that $\Delta$ sends $\Pet$ to $\Pet \otimes \Pet$.  Show that the 
coproduct structure constants of $\Pet$ in the $j$-basis are special cases of the coproduct structure constants of $\afA$ in the
$A_w$-basis.

\item ($j$-basis for translations \cite{Pet, Lam:2008})
Prove using Theorem \ref{thm:jbasis} that $$j_{t_\la} = \sum_{\mu \in W \cdot \la} A_{t_{\mu}}.$$

\item ($j$-basis is self-describing)
Show that the coefficients $j_w^x$ directly determine the structure constants of the $\{j_w\}$ basis.

\item
Find a formula for $j_{s_i t_\la}$ (see \cite[Proposition 8.5]{LS:QH} for a special case).

\item
Extend the construction of $\Pet$, and Theorem \ref{thm:jbasis} to extended affine Weyl groups.  See \cite{CMP} and \cite[Theorem 12]{LS:EQToda}.


\item (Generators)
Find generators and relations for $\Pet$.  This does not appear to be known even in type $A$.

\item 
Find general formulae for $j_w$ in terms of $A_x$.  See \cite{LS:QH} for a formula in terms of quantum Schubert polynomials, 
which however is not very explicit.
\end{enumerate}

\section{(Affine) Fomin-Stanley algebras}\label{sec:affineFS}
Let $\phi_0: S \to \Z$ denote the map which sends a polynomial to its constant term.  For example, $\phi_0(3\alpha^2_1\alpha_2 + \alpha_2 +5) = 5$.

\subsection{Commutation definition of affine Fomin-Stanley algebra}
\label{chapter3.section.commutation}
We write $\aNC$ for the affine nilCoxeter algebra.  There is an {\it evaluation at 0} map $\phi_0: \afA \to \aNC$ given by $\phi_0(\sum_w a_w\, A_w) = \sum_w \phi_0(a_w)\,A_w$.  We define the {\it affine Fomin-Stanley subalgebra} to be $\aB = \phi_0(\Pet) \subset \aNC$.  The following results follow from Lemma \ref{lem:Petcomm} and Theorem \ref{thm:jbasis}.

\begin{lemma}
The set $\aB \subset \aNC$ is a commutative subalgebra of $\aNC$.
\end{lemma}

\begin{theorem}[\cite{Lam:2008}]\label{thm:j0basis}
The algebra $\aB$ has a basis $\{j_w^0 \mid w \in \aW^0\}$ satisfying
$$
j_w^0 = A_w + \sum_{\stackrel{x \notin \aW^0}{\ell(x) = \ell(w)}} j_w^x  A_x
$$
and $j_w^0$ is the unique element in $\aB$ with unique Grassmanian term $A_w$.
\end{theorem}

\begin{prop}[\cite{Lam:2008}]\label{prop:affStan}
The subalgebra $\aB \subset \aNC$ is given by
$$
\aB = \{a\in \aNC | \phi_0(as) = \phi_0(s)a \text{ for all } s \in S\}.
$$
\end{prop}
Proposition \ref{prop:affStan} is proved in the following exercises (also see \cite[Propositions 5.1,5.3,5.4]{Lam:2008}).
\begin{exercise} \
\begin{enumerate}
\item
Check that $a \in \aB$ satisfies $\phi_0(as) = \phi_0(s)a$ for all $s \in S$, thus obtaining one inclusion.
\item
Show that if $a = \sum_{w \in W} a_w A_w \in \A_0$ lies in $\aB$, then $a$ is a multiple of $A_{\id}$.  (Hint: the action of $\A_0$ on $S$ via divided difference operators is faithful.  See Section \ref{ssec:algebraprob}.)
\item
Suppose that $a \in \aNC$ satisfies the condition $\phi_0(as) = \phi_0(s)a$ for all $s \in S$.  Use (2) to show that $a$ must contain some affine Grassmannian term $A_w$, $w \in \aW^0$.  Conclude by Theorem \ref{thm:j0basis} that $a \in \aB$.
\end{enumerate}
\end{exercise}

A basic problem is to describe $\aB$ explicitly.  We shall do so in type $A$, following \cite{Lam:2008} and connecting to affine Stanley symmetric functions.  For other types, see \cite{LSS:C, Pon}.

\subsection{Noncommutative $k$-Schur functions} \label{sec:noncommkschur}
In the remainder of this section, we take $W = S_n$ and $\aW = \tS_n$.  Part of the material of this section is also presented in Chapter \ref{chapter.k schur primer}, Section \ref{sec:nilcoxeter}.  We define
$$
\ah_k = \sum_{w \text{ cyclically decreasing: } \ell(w) = k} A_w
$$
for $k = 1,2,\ldots,n-1$.  Introduce an inner product $\ip{.,.}:\aNC \times \aNC \to \Z$ given by extending linearly $\ip{A_w,A_v}=\delta_{wv}$.

\begin{definition}\label{def:afstalgebra}
The affine Stanley symmetric function $\tF_w$ is given by 
$$
\tF_w = \sum_\alpha \ip{h_{\alpha_{\ell}} \cdots h_{\alpha_1}, A_w} x^\alpha
$$
where the sum is over compositions $\alpha = (\alpha_1,\alpha_2,\ldots,\alpha_\ell)$.
\end{definition}

Below we shall show that
\begin{theorem}[\cite{Lam:2008}]\label{thm:ahcommute}
The elements $\ah_1,\ah_2,\ldots,\ah_{n-1} \in \aNC$ commute.
\end{theorem}
Assuming Theorem \ref{thm:ahcommute}, Theorem \ref{thm:afstsymm} follows.

Define the noncommutative $k$-Schur functions $\s^{(k)}_\la \in \aNC$ by writing the $k$-Schur functions $s^{(k)}_\la$ (see Section \ref{sec:affine}) as a polynomial in $h_i$, and replacing $h_i$ by $\ah_i$.

\begin{proposition}[\cite{Lam:2008}]\label{prop:noncommCauchy}
Inside an appropriate completion of $\aNC \otimes \La^{(n)}$, we have
$$
\sum_{\alpha} \ah_\alpha x^\alpha = \sum_{\la \in \Bo^n} \s^{(k)}_\la \tF_\la
$$
where the sum on the left hand side is over all compositions.
\end{proposition}
\begin{proof}
Since $\{s^{(k)}_\la\}$ and $\{\tF_\la\}$ are dual bases, we have by standard results in symmetric functions \cite{EC2,Mac:1995} that
$$
\sum_{\alpha} h_\alpha(y) x^\alpha = \sum_{\la \in \Bo^n} s^{(k)}_\la(y) \tF_\la(x)
$$
inside $\La_{(n)} \otimes \La^{(n)}$.  Now take the image of this equation under the map $\La_{(n)} \to \aNC$, given by $h_i \mapsto \ah_i$.
\end{proof}

It follows from Definition \ref{def:afstalgebra} and Proposition \ref{prop:noncommCauchy} that
$$
\tF_w = \sum_{\la} \ip{\s^{(k)}_\la,A_w} \tF_\la.
$$
Thus the coefficient of $A_w$ in $\s^{(k)}_\la$ is equal to the coefficient of $\tF_\la$ in $\tF_w$.  By Theorem \ref{thm:affineSchurbasis}, it follows that
\begin{equation}\label{eq:noncommkSchur}
\s^{(k)}_\la = A_v + \sum_{w \notin \aW^0} \alpha_{w\la} A_w.
\end{equation}
In particular, 
\begin{proposition}[\cite{Lam:2008}]\label{prop:aB}
The subalgebra of $\aNC$ generated by $\ah_1,\ldots,\ah_{n-1}$ is isomorphic to $\La_{(n)}$, with basis given by $\s^{(k)}_\la$.
\end{proposition}

\subsection{Cyclically decreasing elements}
For convenience, for $I \subsetneq \Z/n\Z$ we define $A_I:=A_w$, where $w$ is the unique cyclically decreasing affine permutation which uses exactly the simple generators in $I$.

\begin{theorem}[\cite{Lam:2008}]\label{thm:aB}
The affine Fomin-Stanley subalgebra $\aB$ is generated by the elements $\ah_k$, and we have $j_w^0 = \s^{(k)}_\la$ where $w \in \aW^0$ satisfies $\la(w) = \la$.
\end{theorem}

\begin{example}\label{ex:js3}
Let $W = S_3$.  A part of the $j$-basis for $\aB$ is 
\begin{align*}
j^0_{\id}&=1\\
j^0_{s_0} &= A_0+A_1 + A_2 \\
j^0_{s_1s_0}&= A_{10}+A_{21}+A_{02} \\
j^0_{s_2s_0}&=A_{01}+A_{12}+A_{20} \\
j^0_{s_2s_1s_0}&=A_{101}+A_{102}+A_{210}+A_{212}+A_{020}+A_{021} \\
j^0_{s_1s_2s_0}&=A_{101}+A_{201}+A_{012}+A_{212}+A_{020}+A_{120}
\end{align*}
\end{example}

We give a slightly different proof to the one in \cite{Lam:2008}.
\begin{proof}
We begin by showing that $\ah_k \in \aB$.  We will view $S$ as sitting inside the polynomial ring $\Z[x_1,x_2,\ldots,x_n]$; the commutation relations of $\afA$ can easily be extended to include all such polynomials.  To show that $\ah_k \in \aB$ it suffices to check that $\phi_0(\ah_k \,x_i) = 0$ for each $i$.  But by the $\Z/n\Z$-symmetry of the definition of cyclically decreasing, we may assume $i = 1$.  

We note that 
$$
A_j\, x_i = \begin{cases}
x_{i+1}A_i  + 1& \mbox{$j = i$}\\
x_{i-1}A_i - 1&\mbox{$j = i-1$} \\
x_i A_j & \mbox{otherwise.} 
\end{cases}
$$
Now let $I \subsetneq \Z/n\Z$ be a subset of size $k$.  Then 
$$
\phi_0(A_I x_1) = \begin{cases}
A_{I \setminus\{1\}}& \mbox{$r,r+1,\ldots,n-1,0,1 \in I$ but $r-1 \notin I$} \\
-\sum_{i =r}^{i = 0} A_{I \setminus \{i\}}& \mbox{$r,r+1,\ldots,n-1,0 \in I$ but $r-1,1 \notin I$}\\
0 &\mbox{otherwise.} 
\end{cases}
$$
Given a size $k-1$ subset $J \subset \Z/n\Z$ not containing $1$, we see that the term $A_J$ comes up in two ways: from $\phi_0(A_{J \cup \{1\}}x_1)$ with a positive sign, and from $\phi_0(A_{J \cup \{r'\}}x_1)$ with a negative sign, where $r+1,\ldots,n-1,0$ all lie in $J$.

Thus $\ah_k \in \aB$ for each $k$.  It follows that the $\ah_k$ commute, and by \eqref{eq:noncommkSchur}, it follows that $j_w^0 = \s^{(k)}_\la$ and in particular $\ah_i$ generate $\aB$.
\end{proof}

\subsection{Coproduct}
The map $\Delta: \afA \to \afA \otimes_S \afA$ equips $\Pet$ with the structure of a Hopf-algebra over $S$: to see that $\Delta$ sends $\Pet$ to $\Pet \otimes_S \Pet$, one uses $\Delta(t_\la) = t_\la \otimes t_\la$ and Lemma \ref{lem:Ptrans}.  Applying $\phi_0$, the affine Fomin-Stanley algebra $\aB$ obtains a structure of a Hopf algebra over $\Z$.

\begin{theorem}[\cite{Lam:2008}]\label{thm:aBHopf}
The map $\La_{(n)} \to \aB$ given by $h_i \mapsto \ah_i$ is an isomorphism of Hopf algebras.
\end{theorem}
By Proposition \ref{prop:aB}, to establish Theorem \ref{thm:aBHopf} it suffices to show that $\Delta(\ah_k) = \sum_{j=0}^k \ah_j \otimes \ah_{k-j}$.  This can be done bijectively, using the definition in Section \ref{sec:nilHeckecoproduct}.

\subsection{Exercises and Problems}
Let $W$ be of arbitrary type.
\begin{enumerate}
\item 
Find a formula for $j_{s_0}^0$.  See \cite[Proposition 2.17]{LS:DGG}.
\item
For $W = S_3$, use Proposition \ref{prop:ts3} to give an explicit formula for the noncommutative $k$-Schur functions.
\item (Dynkin automorphisms)
Let $\omega$ be an automorphism of the corresponding Dynkin diagram.  Then $\omega$ acts on $\aNC$, and it is easy to see that $\omega(\aB) = \aB$.  Is $\aB$ invariant under $\omega$?  (For $W = S_n$ this follows from Theorem \ref{thm:aB}, since the $\ah_k$ are invariant under cyclic symmetry.) \item (Generators)
Find generators and relations for $\aB$, preferably using a subset of the $j$-basis as generators.  See \cite{LSS:C, Pon} for the classical types.  See also the proof of Proposition \ref{prop:gen} below.
\item (Power sums \cite[Corollary 3.7]{BandlowSchillingZabrocki})
Define the noncommutative power sums $\mathbf{p}_k \in \aB$ as the image in $\aB$ of $p_k \in \La_{(n)}$ under the isomorphism of Proposition \ref{prop:aB}.  Find an explicit combinatorial formula for $\mathbf{p}_k$.  (See also \cite{FG:1998}.)
\item
Is there a nice formula for the number of terms in the expression of $j_w^0$ in terms of $\{A_x\}$?  This is an ``affine nilCoxeter'' analogue of asking for the number of terms $s_\la(1,1,\ldots,1)$ in a Schur polynomial.
\item 
Find a combinatorial formula for the coproduct structure constants in the $j$-basis.  These coefficients are known to be positive \cite{Kum}.
\item 
Find a combinatorial formula for the $j$-basis.  It follows from work of Peterson (see also \cite{Lam:2008, LS:QH}) that the coefficients $j_w^x$ are positive.
\end{enumerate}

\section{Finite Fomin-Stanley subalgebra}
\label{sec:finiteFS}
In this section we return to general type.

There is a linear map $\kappa:\A \to \aNC$ given by 
$$
\kappa(A_w) = \begin{cases} A_w & \mbox{$w \in W$} \\
0 & \mbox{otherwise.}
\end{cases}
$$
The {\it finite Fomin-Stanley algebra} $\B$ is the image of $\Pet$ under $\kappa$.  Since $\kappa(\ah_k) = \h_k$, this agrees with the definitions in Section \ref{sec:algebra}.

\begin{conjecture}\label{conj:finite}\
\begin{enumerate}
\item
The finite Fomin-Stanley algebra $\B$ satisfies Conjecture \ref{conj:LP}.
\item
The image $\kappa(j_w^0) \in \B$ of the $j$-basis element $j_w^0 \in \aB$ is a nonnegative integral linear combination of the $b_I$-basis.
\item
If $w \in \aW^0$ is such that there is a Dynkin diagram automorphism $\omega$ so that $\omega(w) \in W$, then $\kappa(j_w^0)$ belongs to the $b_I$-basis.
\end{enumerate}
\end{conjecture} 

\begin{example}
Let $W = S_4$ and $\aW = \tS_4$.  Then $\B$ is described in Example \ref{ex:S4}.  One calculates that $s^{(k)}_{221} = s_{221}+s_{32}$ so that
$$
\kappa(\s^{(k)}_{221}) = \s_{221}+\s_{32} = A_{32132}+A_{23123} \in \B.
$$
This supports Conjecture \ref{conj:finite}(2), and shows that $\kappa(j_w^0)$ does not have to be equal to 0 or to some $b_I$.
\end{example}

In the following we allow ourselves some geometric arguments and explicit calculations in other types to provide evidence for Conjecture \ref{conj:finite}.
\begin{example}
Let $W = G_2$ with long root $\alpha_1$ and short root $\alpha_2$.  Thus $\alpha_0$ is also a long root for $\tilde G_2$.  As usual we shall write $A_{i_1 i_2 \cdots i_\ell}$ for $A_{s_{i_1}s_{i_2} \cdots s_{i_\ell}}$ and similarly for the $j$-basis.  The affine Grassmannian elements of length less than or equal to 5 are 
$$
\id, s_0, s_1s_0, s_2s_1s_0, s_1s_2s_1s_0, s_2s_1s_2s_1s_0,s_0s_1s_2s_1s_0.
$$
And the $j$-basis is given by
\begin{align*}
j_\id^0 &=1 \\
j_0^0 &= A_0 + 2A_1 + A_2 \\
j_{10}^0 &= \frac{1}{2}(j_0^0)^2 \\
j_{210}^0 &= \frac{1}{2}(j_0^0)^3 \\
j_{1210}^0 &= \frac{1}{4}(j_0^0)^4 \\
j_{21210}^0&= A_{21210} + A_{21201}+2A_{21012}+A_{12102}+3A_{12101}+2A_{12121}\\
&+A_{12012}+A_{02121}+3A_{01201}+A_{01212} \\
j_{01210}^0 &= A_{01210}+A_{01201}+A_{02121}+A_{12012}+A_{12101}+A_{12102}+A_{21012}\\ &+A_{21201}+A_{21212} 
\end{align*}
which can be verified by using Theorem \ref{thm:j0basis}.  Note that $j_{21210}^0+j_{01210}^0 = \frac{1}{4}(j_0^0)^5$.  Thus $\B$ has basis
$$
\id, \ 2A_1 + A_2,  \ A_{12}+A_{21}, \ 2A_{121}+A_{212}, \ A_{1212}+A_{2121}, \ 2A_{12121}, \ A_{21212}, \ ?A_{121212}
$$
where the coefficient of $A_{121212}$ depends on the $j$-basis in degree 6.
The root poset of $G_2$ is $\alpha_1,\alpha_2 \prec \alpha_1+\alpha_2 \prec \alpha_1 + 2\alpha_2 \prec \alpha_1 + 3\alpha_2 \prec 2\alpha_1+3\alpha_2$.  Both Conjectures \ref{conj:LP} and \ref{conj:finite} hold with this choice of basis.
\end{example}

\begin{prop}\label{prop:gen}
Over the rationals, the finite Fomin-Stanley subalgebra $\B \otimes_\Z \Q$ has a set of generators in degrees equal to the exponents of $W$.
\end{prop}
\begin{proof}
Let $\Gr_G$ denote the affine Grassmannian of the simple simply-connected complex algebraic group with Weyl group $W$.
It is known \cite[Theorem 5.5]{Lam:2008} that $\aB \simeq H_*(\Gr_G,\Z)$.  But over the rationals the homology $H_*(\Gr_G,\Q)$ of the affine Grassmannian is known to be generated by elements in degrees equal to the exponents of $W$, see \cite{Gin}.  (In cohomology these generators are obtained by {\it transgressing} generators of $H^*(K,\Q)$ which are known to correspond to the degrees of $W$.)  Since these elements generate $\aB$, their image under $\kappa$ generate $\B$.
\end{proof}

\subsection{Problems}
\begin{enumerate}
\item
Find a geometric interpretation for the finite Fomin-Stanley algebras $\B$ and the conjectural basis $b_I$.
\item
Find an equivariant analogue of the finite Fomin Stanley algebra $\B$, with the same relationship to $\B$ as $\Pet$ has to $\aB$.  Extend Conjectures \ref{conj:LP} and \ref{conj:finite} to the equivariant setting.
\end{enumerate}

\section{Geometric interpretations}
\label{sec:geom}
In this section we list some geometric interpretations of the material we have discussed.  Let $G$ be the simple simply-connected complex algebraic group associated to $W$.  (Co)homologies are with $\Z$-coefficients.

\begin{enumerate}
\item 
The affine Fomin-Stanley subalgebra $\aB$ is Hopf-isomorphic to the homology $H_*(\Gr_G)$ of the affine Grassmannian associated to $G$.  The $j$-basis $\{j_w\}$ is identified with the Schubert basis.  \cite{Pet} \cite[Theorem 5.5]{Lam:2008}
\item
The affine Schur functions $\tF_\la$ represent Schubert classes in $H^*(\Gr_{SL(n)})$.  \cite[Theorem 7.1]{Lam:2008}
\item
The affine Stanley symmetric functions are the pullbacks of the cohomology Schubert classes from the affine flag variety to the affine Grassmannian. \cite[Remark 8.6]{Lam:2008}
\item
The positivity of the affine Stanley to affine Schur coefficients is established via the connection between the homology of the $\Gr_G$ and the quantum cohomology of $G/B$.  \cite{Pet} \cite{LS:QH} \cite{LL}
\item
The (positive) expansion of the affine Schur symmetric functions $\{\tF_\la\}$ for $\aW = \tS_n$ in terms of $\{\tF_\mu\}$ for $\aW = \tS_{n-1}$ has an interpretation in terms of the map $H^*(\Omega SU(n+1)) \to H^*(\Omega SU(n))$ induced by the inclusion $\Omega SU(n) \hookrightarrow \Omega SU(n+1)$ of based loop spaces.  \cite{Lam:ASP}
\item
Certain affine Stanley symmetric functions represent the classes of {\it positroid varieties} in the cohomology of the Grassmannian. \cite{KLS}
\item
The expansion coefficients of Stanley symmetric functions in terms of Schur functions are certain quiver coefficients. \cite{Buc2}
\end{enumerate}




\def\A{{\mathbb{A}}}
\renewcommand{\af}{\mathrm{af}}
\def\alv{\alpha^\vee}
\def\al{\alpha}
\def\AQ{\Fr_W}
\def\Aut{\mathrm{Aut}}
\def\C{\mathbb{C}}
\def\colev{\mathrm{colevel}}
\def\fin{f}
\renewcommand{\Fl}{\mathrm{Fl}}
\def\Frac{\mathrm{Frac}}
\def\ftimes{\,\bar{\otimes}\,}
\def\Fr{\mathrm{F}}
\def\Fun{\mathrm{Fun}}
\renewcommand{\Gr}{\mathrm{Gr}}
\def\hLa{\hat{\Lambda}}
\def\Hom{\mathrm{Hom}}
\def\hPhi{\hat{\Phi}}
\def\id{\mathrm{id}}
\def\II{\mathcal{I}}
\def\Image{\mathrm{Im}}
\def\Inv{\mathrm{Inv}}
\renewcommand{\ip}[2]{\langle #1\,,\,#2\rangle}
\def\La{\Lambda}
\def\la{\lambda}
\def\lev{\mathrm{level}}
\def\LL{\mathcal{L}}
\def\nullity{\mathrm{nullity}}
\def\omv{\omega^\vee}
\def\om{\omega}
\def\OO{\mathcal{O}}
\def\pb{\overline{\pi}}
\def\PD{\mathcal{P}}
\def\pit{\tilde{\pi}}
\def\pnt{\mathrm{pt}}
\def\Q{\mathbb{Q}}
\def\rank{\mathrm{rank}}
\def\res{\mathrm{res}}
\def\re{\mathrm{re}}
\def\R{\mathbb{R}}
\def\Schub{\mathfrak{S}}
\def\slh{\hat{\mathrm{sl}}}
\def\Supp{\mathrm{Supp}}
\def\Sym{\mathrm{Sym}}
\def\tFl{\tilde{\Fl}}
\def\tGr{\tilde{\mathrm{Gr}}}
\def\tLa{\tilde{\La}}
\def\xib{\bar{\xi}}
\def\Z{\mathbb{Z}}






\chapter{Affine Schubert calculus}
\label{chapter.affineSchubert}

\setcounter{section}{0}

\begin{center}
\textsc{Mark Shimozono}
\footnote{The author was supported by NSF grants DMS--0652641, DMS--0652648, and DMS--1200804.} \\
\texttt{mshimo@math.vt.edu}
\end{center}

\section{Introduction}
This chapter discusses how $k$-Schur and dual $k$-Schur functions can be defined for all types.
This is done via some combinatorial problems that come from 
the geometry of a very large family of generalized flag varieties. They apply to the expansion of products
of Schur functions, $k$-Schur functions and their dual basis, and Schubert polynomials.
Despite the geometric origin of these problems, concrete algebraic models will be given for
the relevant cohomology rings and their Schubert bases. 

Let $G\supset P\supset T$ be a Kac-Moody group over $\C$ \cite{Kum}, 
a parabolic subgroup, and a maximal algebraic torus. For example, we may take
$G=GL_n(\C)$, $P$ the block upper triangular matrices with diagonal blocks of sizes
$k$ and $n-k$, and $T\subset G$ the diagonal matrices. The group $G$
may be specified by a graph called a Dynkin diagram, plus a little more discrete data
called a root datum; see \S \ref{S:rootdata}. The subgroup $P$ is obtained by choosing a subset $J$
of the vertex set of the Dynkin diagram. Examples of $G/P$ are the Grassmannian $\Gr(k,n)$ of $k$-dimensional 
subspaces of $\C^n$, the variety $\Fl_n$ of full flags in $\C^n$, and the affine Grassmannian $\Gr_{SL_n}=SL_n(\C((t)))/SL_n(\C[[t]])$
where $\C[[t]]$ is the ring of formal power series and $\C((t))=\C[[t]][t^{-1}]$ is the field of formal
Laurent series. Given this modest amount of discrete data (root datum and Dynkin node subset)
one may build the $T$-equivariant cohomology ring $H^T(G/P)$ of $G/P$. Following Kostant and Kumar
\cite{KK} we shall give an explicit algebraic construction $\La^J$ of $H^T(G/P)$; see \S \ref{SS:parabolic}.
The ring $H^T(G/P)$ is a free module over the polynomial ring $S=H^T(pt)$,
and comes equipped with a distinguished basis $\{[X^w]\}$ called the Schubert basis,
which is indexed by a certain subset $W^J$ of the Weyl group $W$:
\begin{align}\label{E:basis}
  H^T(G/P) = \bigoplus_{w\in W^J} S [X^w].
\end{align}
For $u,v,w\in W^J$ define the ``structure constants" $p_{uv}^w\in S$ by
\begin{align}\label{E:struct}
  [X^u] [X^v] = \sum_{w\in W^J} p_{uv}^w [X^w].
\end{align}
For nontrivial geometric reasons it is known that the polynomial $p_{uv}^w$
is homogeneous of degree $\ell(u)+\ell(v)-\ell(w)$ and satisfies
a property called \textit{Graham positivity} \cite{Gr}; see Section \S \ref{SS:mult}.
\begin{prob} Find explicit combinatorial
formulas for $p_{uv}^w$ which are obviously Graham-positive.
\end{prob}
Using the algebraic model $\La^J$, there is an effective algorithm for expanding
any element of $H^T(G/P)$ into the Schubert basis, and particular
for computing the polynomials $p_{uv}^w$.

Ordinary cohomology $H^*(G/P)$ can be recovered by forgetting equivariance,
which means setting all the variables in the polynomial ring $S$ to zero.
Consequently the Schubert structure constants for $H^*(G/P)$
are merely the polynomials $p_{uv}^w$ which are of degree zero, in which
Graham-positivity reduces to the nonnegativity of an integer.
Special cases of these nonnegative integers are the structure constants for
the products of Schur functions (in the case of $\Gr(k,n)$),
dual $k$-Schur functions (for $\Gr_{SL_n}$), and Schubert polynomials (for $\Fl_n$).

To understand $H^T(G/P)$ and its Schubert calculus,
it suffices to consider the case that $P=B$ is the Borel subgroup.
Kostant and Kumar defined a ring of operators $\A$ on $H^T(G/B)$
called the nilHecke ring, which is a noncommutative ring which is linearly dual (over $S$)
with $H^T(G/B)$. Following Kostant and Kumar, we develop the algebraic construction $\La$
of $H^T(G/B)$ via the systematic use of $\A$.

The above machinery is then specialized to the case of the
affine Grassmannian $\Gr$. In this setting the equivariant cohomology ring
$H^T(\Gr)$ has the rich additional structure of a Hopf algebra,
with Hopf dual given by the equivariant \textit{homology}
ring $H_T(\Gr)$ under the Pontryagin product. In this context,
we explain Peterson's theory, which asserts the equality of the Schubert structure constants
$d_{uv}^w$ for equivariant homology $H_T(\Gr)$ with the Gromov-Witten invariants (quantum Schubert structure
constants) for finite-dimensional flag varieties $G/B$. As a special case we obtain
Lam's noncommutative $k$-Schur functions \cite{Lam:2008} and the relationship between
$k$-Schur functions and the quantum Schubert polynomials of Fomin, Gelfand, and Postnikov \cite{FGP}.
We give a general explicit formula (Proposition \ref{P:jcoefs}) for 
the coefficients $d_{uv}^w$ which involves only some computations in $H^T(\Gr)$ 
(localizing Schubert classes) that we already
know how to do explicitly from the general machinery for $H^T(G/P)$.

\section{Root Data}
\label{S:rootdata}

We specify the data with which one may construct a
Kac-Moody Lie algebra $\mathfrak{g}$ \cite{Kac}, Kac-Moody group $G$
\cite{Kum}, flag ind-variety $\Fl$ \cite{Kum}, thick flag scheme $\tFl$ \cite{Kas},
and nilHecke algebra \cite{KK}. Affine and finite root systems
are treated briefly in Chapter 3 Section~\ref{sec:Weyl}; see also Chapter 4 Section~\ref{chapter4.section.affine root}.

\subsection{Cartan Data and the Weyl group}
A \textit{Cartan datum} is a pair $(I,A)$ where $I$ is a
set and $A=(a_{ij})_{i,j\in I}$ is a \textit{generalized
Cartan matrix}, that is, an $I\times I$ matrix of integers such that
$a_{ii}=2$ for all $i\in I$ and if $i\ne j$ then $a_{ij}\le 0$ with
$a_{ij}<0$ if and only if $a_{ji}<0$. A Cartan datum may be specified
by a graph called a \textit{Dynkin diagram} which has node set $I$ and an edge
joining $i\in I$ and $j\in I$ if $a_{ij}<0$, in which case the edge
is labeled by the pair $(-a_{ij},-a_{ji})$. There are some
abbreviations. Suppose $a_{ij}=-1$.
If $a_{ji}=-1$ then the edge joining $i$ and $j$ is a plain single
edge. If $a_{ji}=-2$ then there is a double edge
pointing from $i$ to $j$. If $a_{ji}=-3$ then
there is a triple edge pointing from $i$ to $j$.

\begin{equation*}
\begin{array}{|c||c|c|c|}\hline
(a_{ij},a_{ji}) & (-1,-1) & (-1,-2) & (-1,-3) \\ \hline &&&
\\
\text{arrow} &
\xymatrix{
*+[F]{i} \ar@{-}[r] & *+[F]{j}
}
&
\xymatrix{
*+[F]{i} \ar@{=>}[r] & *+[F]{j}
}
&
\xymatrix{
*+[F]{i} \ar@3{->}[r] & *+[F]{j}
}
\\ &&& \\ \hline
\end{array}
\end{equation*}

A Cartan datum is \textit{finite} if $A$ is positive definite;
\textit{irreducible} if the Dynkin diagram is connected;
\textit{affine} if it is irreducible, $A$ has corank $1$,
and the restriction of $A$ to $I'\times I'$ for every proper subset $I'\subset I$,
yields a finite Cartan datum.

Given a Cartan datum $(I,A)$, the \textit{Weyl group} $W=W(I,A)$ is the group
by elements $\{s_i\mid i\in I\}$ called \textit{simple reflections} with
relations
\begin{align}
\label{E:sinvolution}
  s_i^2 &= 1 \\
\label{E:braid}
  \underbrace{s_is_js_i\dotsm}_{\text{$m_{ij}$ times}} &=
  \underbrace{s_js_is_j\dotsm}_{\text{$m_{ij}$ times}}.
\end{align}
for $i\ne j$ where $m_{ij}$ and $a_{ij}a_{ji}$ are related by:
\begin{align}\label{E:braidpowers}
\begin{array}{|c||c|c|c|c|c|} \hline
a_{ij}a_{ji} & 0 & 1 & 2 & 3 & >3 \\ \hline
m_{ij} & 2 & 3 & 4 & 6 & \infty \\  \hline
\end{array}
\end{align}
The \textit{length} $\ell(w)$ of $w\in W$ is the minimum $\ell$
such that $w=s_{i_1}s_{i_2}\dotsm s_{i_\ell}$ for
some indices $i_1,i_2,\dotsc,i_\ell\in I$. Such a factorization is
called a \textit{reduced decomposition} for $w$. A \textit{reflection} in $W$ is an
element of the form $w s_i w^{-1}$ for some $w\in W$ and $i\in I$.
The \textit{Bruhat order} $\le$ on $W$ is the partial order generated by
relations of the form $t w < w$ where $t$ is a reflection and $w\in
W$ are such that $\ell(t w)<\ell(w)$. It is also generated by
relations $wt<w$ where $\ell(wt)<\ell(w)$. Equivalently $v\le w$ if
every (or equivalently some) reduced decomposition of $w$ contains a subexpression that is
a reduced decomposition for $v$. We denote by $v\lessdot w$ the covering relation for $W$.
It is equivalent to saying that $v<w$ and $\ell(w)=\ell(v)+1$.

\begin{example} \label{X:ACartandatum} The Cartan datum of finite type $A_{n-1}$ is defined by
the Dynkin diagram
\begin{equation*}
\xymatrix{
*+[F]{1} \ar@{-}[r] &
*+[F]{2} \ar@{-}[r] &
\dotsm \ar@{-}[r] &
*+[F]{n-2} \ar@{-}[r] &
*+[F]{n-1}
}
\end{equation*}
That is, $I=\{1,2,\dotsc,n-1\}$, $a_{ii}=2$ for $1\le i\le n-1$,
$a_{i,i+1}=a_{i+1,i}=-1$ for $1\le i\le n-2$, with $a_{ij}=0$ for
other $(i,j)\in I\times I$. The Weyl group is the symmetric group $S_n$ with $s_i=(i,i+1)$
the adjacent transposition. The reflections are the transpositions $(i,j)$ for $1\le i<j\le n$.

\end{example}

%

\begin{example} \label{X:CCartandatum} A Cartan datum of finite type $C_n$ is given by the Dynkin diagram
\begin{equation*}
\xymatrix{
*+[F]{1}\ar@{-}[r] & *+[F]{2}\ar@{-}[r] & \dotsm \ar@{-}[r] & *+[F]{n-1} \ar@{<=}[r]& *+[F]{n}
}
\end{equation*}
In particular $a_{n-1,n}=-2$ and $a_{n,n-1}=-1$. $W$ is the hyperoctahedral group,
which is isomorphic to the subgroup of $S_{2n}$
generated by $s_i = (i,i+1)(2n-i,2n-i+1)$ for $1\le i\le n-1$ and $s_n=(n,n+1)$.
\end{example}

\begin{example}\label{X:afACartandatum} A Cartan datum of affine type $A_{n-1}^{(1)}$ is given by the Dynkin diagram
\begin{equation*}
\xymatrix{
&& *+[F]{0} \ar@{-}[dll] \ar@{-}[drr] && \\
*+[F]{1} \ar@{-}[r] &
*+[F]{2} \ar@{-}[r] &
\dotsm \ar@{-}[r] &
*+[F]{n-2} \ar@{-}[r] &
*+[F]{n-1}
}\end{equation*}
with Dynkin node set $I_\af = \Z/n\Z$, $a_{ii}=2$ for all $i\in I_\af$, $a_{i,i+1}=a_{i+1,i}=-1$ for
all $i\in I_\af$, including $a_{n-1,0}=a_{0,n-1}=-1$. $W=\tilde{S}_n$ is the affine symmetric group.
There is an injective homomorphism from $\tilde{S}_n$ into the permutations of $\Z$
such that $s_i$ is sent to the permutation of $\Z$ that exchanges $(i+kn,i+kn+1)$ for all $k\in \Z$.
\end{example}

\begin{example}\label{X:afCCartandatum}
 A Cartan datum of affine type $C_n^{(1)}$ is given by the Dynkin diagram
\begin{equation*}
\xymatrix{
*+[F]{0}\ar@{=>}[r] &
*+[F]{1}\ar@{-}[r] & *+[F]{2}\ar@{-}[r] & \dotsm \ar@{-}[r] & *+[F]{n-1} \ar@{<=}[r]& *+[F]{n}
}
\end{equation*}
\end{example}

\subsection{Root data}
A \textit{root datum}
$(I,A,X,X^*,\{\al_i\},\{\alv_i\})$ consists of a Cartan datum
$(I,A)$ together with a free $\Z$-module $X$, the
dual lattice $X^*=\Hom_\Z(X,\Z)$, linearly independent elements
$\{\alpha_i\mid i\in I\} \subset X$ called \textit{simple roots}, linearly
independent elements $\{\alv_i\mid i\in I\} \subset X^*$ called \textit{simple
coroots}, such that
\begin{align}\label{E:Cartanmatrix}
\ip{\alv_i}{\al_j}=a_{ij}\qquad\text{for $i,j\in I$}
\end{align}
where $\ip{\cdot}{\cdot}:X^*\times X\to \Z$ is the evaluation pairing.

A root datum must satisfy
\begin{align}\label{E:rankXbound}
  \rank(X) \ge |I| + \nullity(A).
\end{align}
A root datum is \textit{centerless} if equality holds in \eqref{E:rankXbound}.

$W$ acts on $X$ and $X^*$ via
\begin{align}
\label{E:WonX}
s_i \cdot \la &= \la - \ip{\alv_i}{\la}\al_i &\qquad&\text{for $i\in I$, $\la\in X$} \\
\label{E:WonX*}
s_i \cdot \mu &= \mu - \ip{\mu}{\al_i} \alv_i &&\text{for $i\in I$, $\mu\in X^*$.}
\end{align}
These restrict to actions on the root lattice $Q$ and the coroot lattice $Q^\vee$:
\begin{align}
  Q &= \bigoplus_{i\in I} \Z \al_i \subset X \\
  Q^\vee &= \bigoplus_{i\in I} \Z \alv_i \subset X^*
\end{align}
One may show that
\begin{align}
\label{E:Winvariantpairing}
\ip{w\mu}{w\la}&=\ip{\mu}{\la}\qquad\text{for $w\in W$, $\la\in
X$, and $\mu\in X^*$.}
\end{align}
The set of \textit{real roots} is
\begin{align}
\label{E:Realroots}
\Phi_\re = W \cdot \{\al_i\mid i\in
I\}\subset X.
\end{align}
A real root
\begin{align}\label{E:arealroot}
\alpha =
w \cdot \al_i\in \Phi_\re
\end{align}
has an \textit{associated coroot}
\begin{align}\label{E:assoccoroot}
\alv = w\cdot
\alv_i\in X^*
\end{align}
and a \textit{reflection}
\begin{align}\label{E:rootreflection}
s_\alpha = w s_i w^{-1}\in W.
\end{align}
Both $\alv$ and $s_\al$ are independent of the choices of $i$ and $w$. The reflection
$s_\al$ acts on $X$ and $X^*$ by
\begin{align}\label{E:reflectionX}
s_\alpha\cdot\la &= \la - \ip{\alv}{\la}\al&\qquad&\text{for $\la\in X$} \\
s_\al \cdot \mu&=\mu-\ip{\mu}{\al}\alv&&\text{for $\mu\in X^*$.}
\end{align}
For any $\al\in \Phi_\re$ and $w\in W$, we have that
\begin{align}
\label{E:Weylrealroot}
  \beta &:= w\cdot \al \in \Phi_\re \\
\label{E:Weylassoccoroot}
  \beta^\vee &= w \cdot \alv \\
\label{E:Weylreflection}
  s_\beta &= w s_\al w^{-1}.
\end{align}
The set of \textit{positive real roots}
is $\Phi_\re^+ = \Phi_\re \cap \bigoplus_{i\in I} \Z_{\ge0} \al_i$.
One may show that
\begin{align}\label{E:rootsposneg}
\Phi_\re = \Phi_\re^+ \cup -
\Phi_\re^+.
\end{align}

The \textit{inversion set} of $w$ is
\begin{align}\label{E:Inv}
\Inv(w) = \Phi_\re^+ \cap w^{-1}(-\Phi_\re^+)\qquad\text{for $w\in W$.}
\end{align}
It is the set of positive roots which are sent to negative roots by $w$.

\begin{lem} \label{L:inversionset}
Let $w\in W$. Then
\begin{align}
\label{E:invset}
  \Inv(w) &= \{ \al\in \Phi_\re^+ \mid w s_\al < w\} \\
\label{E:invreducedword}
  &=  \{ \al^{(j)} = s_{i_1}s_{i_2}\dotsm s_{i_{j-1}}\cdot \al_{i_j} \mid 1\le j\le \ell\} \qquad\text{where} \\
\notag
  w &= s_{i_\ell}\dotsm s_{i_1}\qquad\text{is a reduced decomposition.}
\end{align}
Moreover none of the elements of $\Inv(w)$ is a scalar multiple of another.
\end{lem}

Now for some details on root data. Let
\begin{align*}
  L(Q) &= \{ \la\in \Q\otimes_\Z X\mid \ip{Q^\vee}{\la}\subset \Z \}\\
  L(Q^\vee) &= \{ \mu\in \Q \otimes_\Z X^* \mid \ip{\mu}{Q} \subset \Z \}
\end{align*}
be the weight and coweight lattices. We have
\begin{align}
  Q &\subset X\subset L(Q) \\
  Q^\vee &\subset X^* \subset L(Q^\vee)
\end{align}
where $L(Q)/X$ and $L(Q^\vee)/X^*$ are finite.
A root datum is \textit{simply connected} if it is centerless and $L(Q)=X$.

Suppose the root datum is simply connected. Then $X$ has
a basis $\{\om_i\mid i\in I\} \cup \{\delta_i\mid 1\le i\le \nullity(A)\}$
and $X^*$ has a dual basis $\{\alv_i\mid i\in I\} \cup \{d_i \mid 1\le i\le\nullity(A)\}$.
The $\om_i$ are called fundamental weights and are uniquely defined. We have
\begin{align}
  \bigoplus_{i=1}^{\nullity(A)} \Z \delta_i &= \{ \la\in X\mid \ip{Q^\vee}{\la}=0\} = X^W \\
  \bigoplus_{i=1}^{\nullity(A)} \Z d_i &= \{ \mu\in X^* \mid \ip{\mu}{Q}=0\} = X^{*W}.
\end{align}
Since $\al_k\in X$ we have
\begin{align}
  \al_k = \sum_{j\in I} a_{jk} \,\om_j \mod X^W.
\end{align}

\begin{example} \label{X:Arootdatum}
We give several root data for the type $A_{n-1}$ Cartan datum of Example \ref{X:ACartandatum}.
Recall that $I=\{1,2,\dotsc,n-1\}$.

Let $\Z^n=\bigoplus_{i=1}^n \Z e_i$ be the standard basis
and $x_i\in \Hom(\Z^n,\Z)$ be defined by $x_i(e_j)=\delta_{ij}$ for $1\le i,j\le n$.
Let $\al_i=e_i-e_{i+1}$ and $\alv_i = x_i-x_{i+1}$ for $i\in I$.
These will be the simple roots and coroots, possibly up to taking suitable cosets.
We have $\Phi_\re^+ = \{\alpha_{ij}=e_i-e_j\mid
1\le i<j\le n\}$, and the associated reflection for $\alpha = \alpha_{ij}$
is the transposition $s_\alpha=(i,j)$.
\begin{enumerate}
\item The $GL_n$ root datum is given by taking $X=\Z^n$ and $X^*=\Hom(\Z^n,\Z)$.
This root datum is not centerless: $\nullity(A)=0$ and $X^W=\Z(1^n)$.
\item The simply connected root datum is given by taking $X=\Z^n/\Z(1^n)=L(Q)$
and $X^*=Q^\vee$. Then
$\om_i=(1^i,0^{n-i})+\Z(1^n)$ for $i\in I$.
\item Let $x=x_1+x_2+\dotsm+x_n$.
The adjoint root datum is given by $X=Q$
and $X^*=\Hom(\Z^n,\Z)/\Z x = L(Q^\vee)$.
\end{enumerate}
\end{example}

%

\begin{example} \label{X:Crootdatum} Consider the type $C_n$ Cartan datum
of Example \ref{X:CCartandatum}.
Using $X=\Z^n$ and $e_i$ and $x_i$ as in Example \ref{X:Arootdatum}
we have $\al_i = e_i-e_{i+1}$ and $\alv_i=x_i-x_{i+1}$ for $1\le i\le n-1$, $\al_n=2e_n$,
and $\alv_n=x_n$.
\end{example}

\begin{remark}
The fundamental coweights are elements $\omv_i\in \Q\otimes_\Z X^*$ such that
$$\ip{\omv_i}{\al_j}=\delta_{ij}$$ for $i,j\in I$.
The fundamental coweights need not be in the lattice $X^*$. In Example \ref{X:Crootdatum},
we have $\omega_n^\vee=(1/2,\dotsc,1/2)\not\in X^*$.
\end{remark}

\subsection{Affine root data}
\label{chapter4.section.affine root}
Let $(I,A,X,X^*,\{\al_i\},\{\alv_i\})$ be an irreducible finite simply connected root
datum. We shall describe the associated ``untwisted" simply connected irreducible affine root datum
denoted by $(I_\af,A_\af,X_\af,X_\af^*,\{\al_i\},\{\alv_i\})$.

There is a distinguished node $0\in I_\af$ such that $I_\af = I \cup \{0\}$  \cite{Kac} and
the restriction of $A_\af$ to $I\times I$ is $A$. We write $W_\af=W(I_\af,A_\af)$
for the affine Weyl group.
There is a unique tuple  $(a_i\mid i\in I_\af)$ of relatively prime positive integers
which are the coefficients of a linear dependence relation among the columns of $A_\af$:
\begin{align}\label{E:adef}
  \sum_{j\in I} a_j a_{ij} = 0 \qquad\text{for all $i\in I_\af$.}
\end{align}
The \textit{null root} is
\begin{align}\label{E:nullroot}
  \delta = \sum_{i\in I} a_i \al_i;
\end{align}
It is the unique generator of
\begin{equation}
\label{E:XafWafinv}
\begin{split}
  X_\af^{W_\af} &= \{\la\in X_\af \mid w\cdot \la = \la \text{ for all $w\in W_\af$} \\
  &= \{\la\in X_\af\mid \ip{Q_\af^\vee}{\la}=0\}
\end{split}
\end{equation}
that lies in $\bigoplus_{i\in I_\af} \Z_{\ge0} \al_i$. It satisfies
\begin{align}\label{E:deltatheta}
  \delta = \al_0 + \theta
\end{align}
where $\theta\in \Phi^+$ is the \textit{highest root} of the finite root datum.

There is a unique tuple  $(a_i^\vee\mid i\in I_\af)$ of relatively prime positive integers
which are the coefficients of a linear dependence relation among the \textit{rows} of $A_\af$.
The \textit{canonical central element} $c\in X_\af^*$ is
\begin{align}\label{E:centralelement}
  c = \sum_{i\in I_\af} a_i^\vee \alv_i;
\end{align}
it generates
\begin{equation}
\label{E:Xaf*Wafinv}
\begin{split}
  X_\af^{* W_\af} &= \{\mu\in X_\af^* \mid w\cdot \mu = \mu \text{ for all $w\in W_\af$} \} \\
  &= \{\mu\in X_\af^* \mid \ip{\mu}{Q_\af}=0 \}.
\end{split}
\end{equation}
We have
\begin{align}\label{E:cthetavee}
  c = \alv_0 + \theta^\vee
\end{align}
where $\theta^\vee\in Q^\vee$ is the coroot associated to $\theta$.

We denote the affine fundamental weights by $\{\La_i\mid i\in I_\af\}$ and those
of the finite root datum by $\{\om_i\mid i\in I\}$. Then
\begin{align}\label{E:Xaf}
  X_\af = \Z\delta \oplus \bigoplus_{i\in I_\af} \Z \La_i
\end{align}
and there is an element $d\in X_\af^*$ called the degree generator, defined uniquely mod $\Z c$, such that
$X_\af^*$ has dual basis
\begin{align}
  X_\af^* = \Z d \oplus \bigoplus_{i\in I_\af} \Z \alv_i.
\end{align}
Then
\begin{align}
  \al_i = \delta_{i0} + \sum_{k\in I_\af} a_{ki} \La_k
\end{align}
where $A_\af = (a_{ij})_{i,j\in I_\af}$.
There is an exact sequence
\begin{align*}
  0 \to \Z \delta \oplus \Z \La_0 \to X_\af \to X \to 0.
\end{align*}
There is a section $X\to X_\af$ defined by $\om_i \mapsto \La_i - a_i^\vee\La_0$ for $i\in I$.
Using this section we regard $Q\subset X \subset X_\af$.

The real affine roots $\Phi_\af^\re$ and the positive subset
$\Phi_\af^{\re+}$ are related to the set $\Phi$ of finite roots by
\begin{align}\label{E:reafroots}
\Phi_\af^\re &= \Phi + \Z \delta \\
\label{E:reafposroots}
\Phi_\af^{\re+} &= \Phi^+ \cup (\Phi+\Z_{>0}\delta).
\end{align}

Since $\ip{c}{Q_\af}=0$ (resp. $\ip{Q_\af^\vee}{\delta}=0$),
the action of $W_\af$ on $X_\af$ (resp. $X_\af^*$) preserves the set of weights of a given level
(resp. colevel) where
\begin{align}\label{E:level}
\lev(\la)&=\ip{c}{\la}&\qquad&\text{for $\la\in X_\af$} \\
\colev(\mu)&=\ip{\mu}{\delta}&&\text{for $\mu\in X_\af^*$.}
\end{align}
Therefore for every $m\in\Z$ there is an action of $W_\af$ called the level (colevel) $m$ action,
on the set of elements of $X_\af$ of level zero (elements of $X_\af^*$ of colevel zero)
which may be identified with $\Z\delta \oplus X$ (resp. $Q_\af^\vee = \Z c \oplus Q^\vee$),
defined by
\begin{align}
\label{E:levelmonXdelta}
  w \cdot_m \la &= -m\La_0 + w \cdot (m\La_0 + \la) &\qquad&\text{for $\la\in \Z\delta\oplus X$} \\
\label{E:colevelmonQafv}
  w \cdot_m \mu &= -m d + w \cdot (m d + \mu) &\qquad&\text{for $\mu\in \Z c \oplus Q^\vee$.}
\end{align}
For $\la\in \Z\delta \oplus X$ we have
\begin{align*}
  s_i \cdot_m \la &= s_i \cdot \la \qquad\text{if $i\in I$} \\
  s_0 \cdot_m \la &= s_0\cdot \la - m \al_0.
\end{align*}
For $\mu\in Q_\af^\vee$ we have
\begin{align*}
  s_i \cdot_m \mu &= s_i \cdot \mu \qquad\text{if $i\in I$} \\
  s_0 \cdot_m \mu &= s_0\cdot \mu - m \alv_0.
\end{align*}
Since
\begin{align}
\label{E:deltainvariant}
  W_\af \cdot \delta &= \delta \\
\label{E:cinvariant}
  W_\af \cdot c &= c
\end{align}
the level $m$ action of $W_\af$ on $\Z\delta \oplus X$ (resp. $Q_\af^\vee$)
factors through projection mod $\Z\delta$ (resp. $\Z c$).
So there is a level $m$ action of $W_\af$ on $X$ given by
\begin{equation}\label{E:levelmonX}
\begin{split}
  s_i \cdot_m \la &= s_i \cdot \la \qquad\text{for $i\in I$} \\
  s_0 \cdot_m \la &= s_\theta \cdot \la + m \theta
\end{split}
\end{equation}
for $\la\in X$, and a colevel $m$ action of $W_\af$ on $Q^\vee$ defined by
\begin{equation}\label{E:levelmonQv}
\begin{split}
  s_i \cdot_m \mu &= s_i \cdot \mu \qquad\text{if $i\in I$} \\
  s_0 \cdot_m \mu &= s_\theta \cdot \mu + m \theta^\vee
\end{split}
\end{equation}
for $\mu\in Q^\vee$.
Identifying $W_\af$ with its image in $\Aut(Q^\vee)$ under the faithful
level $1$ action, we have
\begin{align}
  W_\af = W \ltimes Q^\vee
\end{align}
where $Q^\vee = \{t_\mu\mid \mu\in Q^\vee\}$ is the group of
translations by elements of $Q^\vee$. We have
\begin{align*}
  w t_\mu w^{-1} = t_{w\cdot \mu} \qquad\text{for $w\in W$ and $\mu\in Q^\vee$.}
\end{align*}
The reflection $s_{\al+k\delta}$ for $\al+k\delta\in\Phi_\af^\re$ with $\al\in \Phi$ and $k\in\Z$,
satisfies
\begin{align}
  s_{\al+k\delta} = s_\al t_{k\al^\vee}.
\end{align}
In particular, for $\al_0 = \delta - \theta$, we have
\begin{align}\label{E:s0semi}
  s_0 = s_\theta t_{-\theta^\vee}.
\end{align}
Let $\mu\in Q^\vee$. Under the level $m$ action on $\Z\delta\oplus X$ we have
\begin{align}
  t_\mu \cdot_m (\la+k\delta) = (\la+k\delta) - (m + \ip{\mu}{\la})\delta\qquad\text{for $\la\in X$}.
\end{align}
Modding out by $\Z\delta$, the level $0$ action of $W_\af$ on $X$
is given by
\begin{align} \label{E:translevelzero}
  (u t_\mu)\cdot_0 \la = u\cdot \la\qquad\text{for $u\in W$ and $\mu\in Q^\vee$;}
\end{align}
the translations act trivially. Equivalently, by \eqref{E:s0semi}, $s_0$ acts by $s_\theta$.

\begin{example} \label{X:afArootdatum} With the affine $A_{n-1}^{(1)}$
Cartan datum of Example \ref{X:afACartandatum},
we have $a_i=1$ for all $i\in I_\af$, $\delta=\al_0+\al_1+\dotsm+\al_{n-1}$ and
$\theta=e_1-e_n$.
\end{example}

\begin{example} \label{X:afCrootdatum} With the
affine $C_2^{(1)}$ Cartan datum of Example \ref{X:afCCartandatum},
we have $I_\af=\{0,1,2\}$,
\begin{align}
A_\af = \begin{pmatrix}
  2&-1&0\\
  -2&2&-2\\
  0&-1&2
  \end{pmatrix},
\end{align}
$(a_0,a_1,a_2)=(1,2,1)$, $\delta=\al_0+2\al_1+\al_2$,
and $\theta=2\al_1+\al_2$. We let $X=\Z^2$, $\al_1=(1,-1)$,
$\al_2=(0,2)$, $X^*=\Z^2$, $\alv_1=(1,-1)$, $\alv_2=(0,1)$. Then
$\theta=(2,0)=s_1\cdot\al_2$ so that $\theta^\vee = s_1\cdot \alv_2
= (1,0) = \alv_1+\alv_2$. So $c =\alv_0+\alv_1+\alv_2$.
\end{example}

\section{NilHecke ring and Schubert calculus}
\label{sec:MarknilHecke}

In this section we consider the nilHecke ring $\A$ of Kostant and Kumar \cite{KK},
for any Kac-Moody root datum. This general construction  
may be used to compute the cohomological Schubert calculus of the
flag variety for the given root datum, equivariant with respect to the
\textit{maximal} torus. 

The general construction of the Kostant-Kumar nilHecke ring $\A$ of this section,
yields the nilHecke ring denoted $\A$ in Chapter 3
if a finite root datum is used, and yields an affine nilHecke ring $\A_\af$
if an untwisted affine root datum is used. This affine nilHecke ring $\A_\af$
is \textit{not} the one denoted $\A_\af$ in Chapter 3; the latter, which we shall call $\A_\af'$,
is a subquotient of $\A_\af$. Both compute cohomology of the affine flag variety, but $\A_\af$ uses equivariance
for the maximal torus $T_\af$ in the Kac-Moody group $G_\af$, whereas $\A_\af'$
uses equivariance for the ``small torus" $T$ in the simple Lie group $G \subset G_\af$. See Sections \ref{SS:GafPaf} and \ref{SS:Petersonsubalgebra}.

The nilHecke ring $\A$ contains a subring $\A_0$ called the nilCoxeter algebra.
For a finite root datum, $\A_0$ is the finite nilCoxeter algebra, also denoted $\A_0$ 
in Chapter 3, Section \ref{sec:algebra}. For an untwisted affine root datum is of affine type,
$\A_0$ is the affine nilCoxeter algebra, denoted $\aNC$ in
Chapter 3, Sections \ref{sec:algebra}, \ref{chapter3.section.affine nilHecke}, and \ref{chapter3.section.commutation}.

\subsection{NilHecke ring}
Fix a root datum. Let $S=H^T(\pnt)\cong \Sym(X)$ be the polynomial ring having a
variable for each free generator of $X$ and let
$\Fr=\Frac(S)$ be the fraction field. The action of $W$ on $X$
induces actions of $W$ on $S$ and $\Fr$ by ring automorphisms,
and $\Fr$ acts on itself by left multiplication.

\begin{example} For $A_{n-1}$ and the $GL_n$ root datum
we have $X=\bigoplus_{i=1}^n \Z z_n$ and
$S=\Z[z_1,\dotsc,z_n]$, $\Fr=\Q(z_1,\dotsc,z_n)$. $W=S_n$ permutes variables.
\end{example}

Let $v\in W$ and $q,q'\in \Fr$. Viewing both as operators on $\Fr$ we have
\begin{align*}
(v \circ q) \cdot q' &= v \cdot (q\cdot q') \\
&= (v\cdot q)(v\cdot q') \\
&= ((v\cdot q) \circ v) \cdot q',
\end{align*}
that is, $v q = (v\cdot q) v$ as operators on $\Fr$.
Define
\begin{align}\label{E:AQ}
\AQ = \bigoplus_{w\in W} \Fr\, w,
\end{align}
the $\Fr$-vector space with basis
$W$, with product given by
\begin{align}\label{E:QWmult}
(pv)(qw) = (p (v\cdot q)) (vw)\qquad\text{for $p,q\in \Fr$ and $v,w\in W$.}
\end{align}
It is the smash product of $\Q[W]$ and $\Fr$, the
ring generated by the actions of $W$ and $\Fr$ upon $\Fr$.

\begin{example} \label{E:nilHeckeA}
For the $A_{n-1}$ root datum of Example \ref{X:Arootdatum},
we have $S = \Z[x_1,\dotsc,x_n]$ and $\Fr=\Q(x_1,\dotsc,x_n)$.
\end{example}

For any $\al\in \Phi_\re$ define the element $A_\al\in \AQ$ by
\begin{align} \label{E:Aal}
  A_\al = \al^{-1} (1-s_\al).
\end{align}
We write
\begin{align}\label{E:Ai}
  A_i = A_{\al_i} \qquad\text{for $i\in I$.}
\end{align}
For $\al=w\cdot \al_i\in\Phi_\re$ we have
\begin{align}\label{E:Aalconj}
  w A_i w^{-1} &= A_\al \\
  s_\al A_\al &= A_\al \\
  A_\al s_\al &= - A_\al \\
\label{E:A20}
  A_\al^2 &= 0.
\end{align}
The $A_i$ satisfy the braid relations as the $s_i$ in $W$:
\begin{align}
\label{E:Abraid}
  \underbrace{A_iA_j\dotsm}_{\text{$m_{ij}$ times}} &= \underbrace{A_jA_i\dotsm}_{\text{$m_{ij}$ times.}}
\end{align}
Therefore it makes sense to define
\begin{align}
A_w&=A_{i_1}\dotsm A_{i_\ell}&\qquad&\text{where} \\
\notag
w&=s_{i_1}\dotsm s_{i_\ell} &&\text{is a reduced decomposition.}
\end{align}
Using \eqref{E:A20} and \eqref{E:Abraid} one may show that
\begin{align} \label{E:Aadd}
  A_v A_w = \begin{cases}
  A_{vw} & \text{if $\ell(vw)=\ell(v)+\ell(w)$} \\
  0 & \text{otherwise.}
  \end{cases}
\end{align}

Since $A_w\in \AQ$ there are unique $c_{w,v}\in \Fr$ such that
\begin{align} \label{E:AtoW}
  A_w = \sum_{v\in W} c_{w,v} v.
\end{align}

\begin{lem}\label{L:ctri}
\begin{align} \label{E:cmatrixdiag}
c_{w,w} &\ne 0 \\
\label{E:cmatrixsupport} c_{w,v}&=0\qquad\text{unless $w\ge v$.}
\end{align}
In particular $\{A_w\mid w \in W\}$ is a left $\Fr$-basis of $\Fr_W$.
\end{lem}

The subring $\A_0$ of $\AQ$ generated by $\{A_i\mid i\in I\}$ is called the \textit{nilCoxeter algebra}.

\begin{lem} \label{L:A0basis}
\begin{align}\label{E:nilCoxeterbasis}
  \A_0 = \bigoplus_{w\in W} \Z A_w.
\end{align}
\end{lem}

\begin{lem} \label{L:AonS} For every $\al\in \Phi_\re$, $\la\in X$, and $q,q'\in S$, we have
\begin{align}
\label{E:AonX}
  A_\al \cdot \la &=  \ip{\alv}{\la} \\
\label{E:AonSprod}
  A_\al \cdot (qq') &= (A_\al \cdot q) q' + (s_\al\cdot q) (A_\al \cdot q').
\end{align}
In particular, $A_\al \cdot S \subset S$ for all $\al\in\Phi_\re$.
\end{lem}

The \textit{nilHecke ring} $\A$ is the subring of $\AQ$ generated by
$S$ and $\A_0$.

\begin{lem} \label{L:Abasis} $\A$ is a free left $S$-module with basis $\{A_w\mid w\in W\}$:
\begin{align}\label{E:ASfree}
  \A = \bigoplus_{w\in W} S A_w.
\end{align}
\end{lem}
\begin{proof}
By \eqref{E:AonSprod} we have the commutation relation in $\A$:
\begin{align}\label{E:As}
  A_\al\, q = (A_\al \cdot q)A_{\id} + (s_\al \cdot q) A_\al\qquad\text{for all $\al\in\Phi_\re$ and $q\in S$.}
\end{align}
This relation may be used to commute all $A_i$ to the right past elements of $S$.
This shows that $\{A_w\mid w\in W\}$ generates $\A$ as a left $S$-module.
By Lemma \ref{L:ctri} the $A_w$ are independent.
\end{proof}

Let $v\lessdot w$ indicate a Bruhat covering relation: $v<w$ and
$\ell(w)=\ell(v)+1$. In this case there is a unique $\alpha\in \Phi_\re^+$
such that $w = v s_\alpha$.

\begin{prop} \label{P:ALBcomm} For any $v\in W$ and $\la\in X$, we
have
\begin{align}\label{E:ALBcomm}
A_v \la = (v\cdot \la) A_v + \sum_{\substack{\al\in \Phi^+_\re\\ v s_\al
\lessdot v}} \ip{\al^\vee}{\la}\, A_{v s_\al}.
\end{align}
\end{prop}

\begin{exercise} \label{EX:Abasis} Prove Lemmas \ref{L:ctri} and \ref{L:AonS}
and Proposition \ref{P:ALBcomm}.
\end{exercise}

We have
\begin{align}
\label{E:stoA}
s_i = 1 - \al_i A_i \qquad\text{for all $i\in I$}
\end{align}
Therefore
\begin{align}\label{E:WinA}
  W \subset \A
\end{align}
By Lemma \ref{L:Abasis} there exist unique elements $d_{vw}\in S$ such that
\begin{align}\label{E:dmatrix}
  w = \sum_{v\in W}  d_{vw} A_v.
\end{align}

By \eqref{E:cmatrixsupport} it follows that
\begin{align}\label{E:dmatrixsupport}
  d_{vw} = 0 \qquad\text{unless $v\le w$.}
\end{align}

\begin{lem} \label{L:drec}
The $d_{v,w}$ are uniquely defined by:
\begin{enumerate}
\item If $w=\id$ then
\begin{align}\label{E:didentity}
  d_{v,\id} = \delta_{v,\id}.
\end{align}
\item Otherwise let $i\in I$ be such that $ws_i<w$. Then
\begin{align}\label{E:drec}
  d_{v,w} = d_{v,ws_i} + \chi(vs_i<v) (w\cdot \al_i) d_{vs_i,ws_i}.
\end{align}
where
\begin{align}
\label{E:chidef}
\chi(\text{true}) &= 1 \\
\notag
\chi(\text{false}) &= 0
\end{align}
\end{enumerate}
\end{lem}
\begin{proof} The uniqueness holds by induction on $w$ and then on $v$.
Equation \eqref{E:didentity} holds since $\id_W = A_{\id}$. We have
\begin{equation}\label{E:wrec}
\begin{split}
  w &= (ws_i) s_i \\
  &= (ws_i) (1 - \al_i A_i) \\
  &= ws_i - ws_i \al_i A_i \\
  &= ws_i + (w\cdot \al_i) w s_i A_i.
\end{split}
\end{equation}
Taking the coefficient of $A_v$ on both sides and using \eqref{E:Aadd}
we obtain \eqref{E:drec}.
\end{proof}

\begin{prop}\label{P:dmatrix} Let $w=s_{i_1}\dotsm s_{i_\ell}$ be a reduced decomposition.
We have
\begin{align}\label{E:dexpand}
 (-1)^{\ell(v)} d_{v,w} =  \sum_{\substack{(b_1,\dotsc,b_\ell)\in\{0,1\}^\ell \\ \prod_{j=1}^\ell A_{i_j}^{b_j} = A_v}}
   \left(\prod_{j=1}^\ell \al_{i_j}^{b_j} s_{i_j}\right) \cdot 1.
\end{align}
The parenthesized expression is viewed as an element of $\AQ$ acting on $1\in \Fr$
and factors in the products occur in order from left to right as the index $j$ increases.
\end{prop}
Note that the sum runs over subexpressions of the given reduced decomposition of $w$,
which are reduced decompositions of $v$.

\begin{example}
For type $A_2$ (symmetric group $S_3$)
let $v=s_1$ and $w=s_1s_2s_1$. The reduced decomposition $s_1s_2s_1$
has two embedded copies of $s_1$ corresponding to the bit sequences
$(b_1,b_2,b_3)$ given by $(1,0,0)$ and $(0,0,1)$.
Therefore \eqref{E:dexpand} has two summands:
\begin{align*}
  (-1)^{\ell(s_1)} d_{s_1,s_1s_2s_1} &= ((\al_1 s_1)s_2s_1+s_1s_2(\al_1s_1))\cdot1 \\
  &= \al_1 + s_1s_2\al_1 = \al_1+\al_2.
\end{align*}
Here is a table of the values of $(-1)^{\ell(v)}d_{v,w}$ for all $v,w\in S_3$.
\begin{align*}
\begin{array}{|c||c|c|c|c|c|c|} \hline
v\setminus w & 123 & 132 & 213 & 231 & 312 & 321 \\ \hline \hline
123 & 1 & 1 & 1 & 1 & 1 & 1 \\ \hline
132 & 0 & \al_2 & 0 & \al_1+\al_2&\al_2&\al_1+\al_2 \\ \hline
213 & 0&0&\al_1&\al_1&\al_1+\al_2&\al_1+\al_2 \\ \hline
231 & 0 & 0 & 0 & \al_1(\al_1+\al_2) & 0 & \al_1(\al_1+\al_2) \\ \hline
312 & 0 & 0 & 0 & 0 & \al_2(\al_1+\al_2)&\al_2(\al_1+\al_2) \\ \hline
321 & 0&0&0&0&0&\al_1\al_2(\al_1+\al_2) \\ \hline
\end{array}
\end{align*}
\end{example}

\begin{exercise} Prove Proposition \ref{P:dmatrix}.
\end{exercise}

\begin{cor} \label{C:dpos} We have $(-1)^{\ell(v)}d_{v,w}\in\Z_{\ge0}[\al_i\mid i\in I]$.
\end{cor}

\subsection{Coproduct on $\A$}
\label{sec:coproduct}
The nilHecke ring $\A$ has a coproduct structure.
This discussion follows Peterson \cite{Pet} and Kostant and Kumar \cite{KK}.

We first define the coproduct on $\AQ$ over $\Fr$.

Let $M$ and $N$ be left $\AQ$-modules.
Then $M$, $N$, $M\otimes_\Fr N$, and $\Hom_\Fr(M,N)$ are left $\Fr$-modules.
We define $\AQ$-module structures on $M\otimes_\Fr N$ and $\Hom_\Fr(M,N)$.

Recall that $M\otimes_\Fr N$ has relations of left $\Fr$-linearity in each
factor, along with the relation
\begin{align}\label{E:ftimes}
  q\cdot m\otimes n = m \otimes q\cdot n
\end{align}
for all $q\in \Fr$, $m\in M$, and $n\in N$.

Let $\Delta:\AQ\to\AQ\otimes_\Fr\AQ$ be the left $\Fr$-linear map defined by
\begin{align}\label{E:DeltaQdef}
  \Delta(w) = w \otimes w \qquad\text{for all $w\in W$}.
\end{align}
Clearly $\Delta$ has image $\AQ'$, the left $\Fr$-subspace of $\AQ\otimes_\Fr \AQ$ defined by
\begin{align*}
  \AQ' = \bigoplus_{w\in W} \Fr w \otimes w.
\end{align*}
It is obvious that $\AQ'$ is isomorphic to $\AQ$
as a ring under the obvious componentwise product given by $(pv\otimes v)(qw\otimes w)=pvqw\otimes vw=
p(v\cdot q) vw \otimes vw$.

The proof of the following result appears in Appendix \ref{A:coalg}.

\begin{prop} \label{P:welldefined}\
\begin{enumerate}
\item For any expressions of $\Delta(a)$ and $\Delta(b)$ in $\AQ\otimes_\Fr\AQ$
of the form
\begin{align*}
  \Delta(a) &= \sum_{(a)} a_{(1)} \otimes a_{(2)} \\
  \Delta(b) &= \sum_{(b)} b_{(1)} \otimes b_{(2)}
\end{align*}
with $a_{(1)},a_{(2)},b_{(1)},b_{(2)}\in \AQ$ (but it is not assumed that the individual
summands $a_{(1)} \otimes a_{(2)}$ or $b_{(1)} \otimes b_{(2)}$ are in $\AQ'$),
the product of $\Delta(a)$ and $\Delta(b)$
in $\AQ'$ can be computed by the naive componentwise product
\begin{align}\label{E:compwise}
  \Delta(ab) = \Delta(a)\Delta(b) = \sum_{(a),(b)} a_{(1)} b_{(1)} \otimes a_{(2)} b_{(2)}.
\end{align}
\item
Suppose $M$ and $N$ are $\AQ$-modules. With $a\in \AQ$ and $\Delta(a)$ as above,
$\AQ$ acts on $M\otimes_\Fr N$ by the componentwise action of $\Delta(a)$ on $m\otimes n$:
\begin{align}\label{E:ontensor}
  a\cdot (m\otimes n) &= \sum_{(a)} a_{(1)} \cdot m \otimes a_{(2)} \cdot n
\end{align}
\item 
$\AQ$ acts on $\Hom_\Fr(M,N)$ by 
\begin{align}\label{E:onhom}
  (a \cdot f)(m) = \sum_{(a)} a_{(2)} \cdot f(a_{(1)}^t \cdot m)
\end{align}
where $a\mapsto a^t$ is the left $\Fr$-linear automorphism of $\AQ$ given by
$w\mapsto w^{-1}$.
\end{enumerate}
\end{prop}

\begin{remark} Note that in the expression $\Delta(a) = \sum_{(a)} a_{(1)} \otimes a_{(2)}$,
it may be that a particular summand $a_{(1)} \otimes a_{(2)}$ is NOT in $\AQ'$. 
An example is $\Delta(A_i) = A_i \otimes 1 + s_i \otimes A_i$, where neither of the
summands $A_i\otimes 1$ and $s_i \otimes A_i$, is in $\AQ'$. For such individual
summands the componentwise action is ill-defined; see the following paragraph.
Nevertheless if the ill-defined componentwise
action is applied to each summand of an element of $\AQ'$ and the results are added, 
the end result yields a well-defined action.

The naive componentwise product structure, which is well-defined on $\AQ'\subset\AQ\otimes_\Fr \AQ$,
does not extend to $\AQ \otimes_\Fr \AQ$. Take $q\in \Fr$ and $u,v\in W$. We recall that
$q \id\otimes w - \id \otimes qw=0$. Suppose the componentwise product was well-defined on $\AQ\otimes_\Fr\AQ$.
Then we would have
\begin{align*}
  0 &= (u\otimes \id) \cdot (q\id\otimes w - \id\otimes qw) \\
  &\overset{?}{=} uq\id \otimes w - u\otimes \id q w \\
  &= (u\cdot q) u\otimes w - q u \otimes w \\
  &= (-q + u \cdot q ) u\otimes w
\end{align*}
But $u\cdot q\ne q$ in general, giving a contradiction.
\end{remark}

The following result is proved in Appendix \ref{A:coalg}.

\begin{prop} \label{P:coalg} \
\begin{enumerate}
\item
The map $\Delta:\AQ\to\AQ\otimes_\Fr \AQ$
induces the unique left $S$-module homomorphism $\Delta:\A\to \A \otimes_S \A$
such that
\begin{align}
\label{E:DeltaA}
  \Delta(A_i) &= A_i \otimes 1 + s_i \otimes A_i \\
\notag &= 1 \otimes A_i + A_i \otimes s_i &\qquad&\text{for all $i\in I$} \\
  \Delta(ab) &= \Delta(a)\Delta(b) &\qquad&\text{for all $a,b\in \A$},
\end{align}
where the product on $\Image(\Delta)$ is defined by the 
componentwise product \eqref{E:compwise}.
\item
Let $M$ and $N$ be left $\A$-modules. Then there is an action of
$\A$ on $M \otimes_S N$ defined by
\begin{align}
  A_i\cdot (m\otimes n) &= (A_i \cdot m) \otimes n + (s_i\cdot m) \otimes (A_i\cdot n) \\
  &= m \otimes (A_i\cdot n)+ (A_i\cdot m) \otimes(s_i\cdot n) \\
  q \cdot (m\otimes n) &= qm \otimes n
\end{align}
for $i\in i$ and $q\in S$.
\item
There is a left $\A$-action on $\Hom_S(M,N)$ given by
\begin{align}
  (q \cdot f)(m) &= q \cdot f(m) \\
  (A_i \cdot f)(m) &= A_i \cdot f(s_i\cdot m) + f(A_i\cdot m) \\
    &=  s_i \cdot f(A_i\cdot m) + A_i \cdot f(m) \\
 (w \cdot f)(m) &= w \cdot f(w^{-1}\cdot m).
\end{align}
for $q\in S$, $f\in \Hom_S(M,N)$, $m\in M$, $i\in I$, and $w\in W$.
\end{enumerate}
\end{prop}

The Theorem allows us to perform computations such as 
\begin{align*}
  \Delta(A_iA_j) &= \Delta(A_i)\Delta(A_j) \\
  &= (A_i \otimes 1+s_i\otimes A_i)(A_j \otimes 1 + s_j\otimes A_j) \\
  &= A_i A_j \otimes 1 + A_i s_j \otimes A_j + s_i A_j \otimes A_i + s_i s_j \otimes A_iA_j.
\end{align*}

\begin{lem} \label{L:lastincoprod}
\begin{align}
  \Delta(A_w) \in w \otimes A_w + \bigoplus_{u<w} \A \otimes A_u.
\end{align}
\end{lem}

\begin{exercise} Prove Lemma \ref{L:lastincoprod}.
\end{exercise}

\subsection{Duality and the GKM ring}
Let $\AQ^* = \Hom_{\text{$\Fr$-Mod}}(\AQ,\Fr)$ be the $\Fr$-vector space of
left $\Fr$-linear maps $\AQ\to \Fr$. It has two left actions of $\AQ$,
one given by Proposition \ref{P:coalg} and denoted by $\cdot$,
and a commuting $\AQ$-action
\begin{align} \label{E:leftQWonAQ}
  (b \bullet f)(a) = f(ab)
\end{align}
for all $a,b\in \AQ$ and $f\in \AQ^*$. Define the left $\Fr$-bilinear
evaluation pairing
\begin{equation}
\begin{split}
\label{E:Qpair}
\AQ^* \times \AQ &\to \Fr \\
  \ip{f}{a} &= f(a).
\end{split}
\end{equation}
Define the GKM ring $\La\subset \AQ^*$ by
\begin{align}
  \hLa &= \{f \in \AQ^* \mid \text{$f(\A) \subset S$} \} \\
  \notag
   &= \{ f\in \AQ^* \mid \text{$f(A_w)\in S$ for all $w\in W$}\} \\
  \La &= \{ f\in \hLa\mid \text{$f(A_w)\ne0$ for finitely many $w\in
  W$} \}.
\end{align}

The $\cdot$ action of $\AQ$ on $\AQ^*$ restricts to a left $S$-action on
$\La$ and $\hLa$:
\begin{align}\label{E:PhiSaction}
(sf)(a) = s f(a)=f(sa) \qquad\text{for all $f\in \hLa$, $s\in S$, $a\in
\AQ$.}
\end{align}

By Lemma \ref{L:Abasis} any element of $\La$, being left
$S$-linear, is completely determined by its values on the basis
$\{A_w\mid w\in W\}$. The following is immediate.

\begin{prop} \label{P:GKMbasis}
$\La$ is a free $S$-module. It has a unique $S$-basis $\{\xi^v\mid
v\in W\}$ defined by
\begin{align}\label{E:APhidualbases}
  \ip{\xi^v}{A_w} = \delta_{vw}.
\end{align}
\end{prop}

Similarly we have the direct product
\begin{align}\label{E:hPhibasis}
  \hLa = \prod_{w\in W} S \,\xi^v.
\end{align}

By definition, the pairing \eqref{E:Qpair} restricts to a pairing
\begin{align}\label{E:Apair}
\A \times \La &\to S.
\end{align}

\begin{lem} \label{L:AonGKM}
Both actions $\cdot$ and $\bullet$ of $\AQ$ on $\AQ^*$ defined by
\eqref{E:leftQWonAQ} restrict
to an action of $\A$ on $\hLa$ and $\La$.
\end{lem}
\begin{proof} Let $f\in \hLa$ and $w\in W$.
By definition it is easy to see that $S \cdot \hLa\subset \hLa$. We have
\begin{align*}
(A_i \cdot f)(A_w) = s_i \cdot f(A_i A_w) + A_i \cdot f(A_w) \in S
\end{align*}
by Proposition \ref{P:coalg}, \eqref{E:Aadd} and Lemma \ref{L:AonS}.
Since $\A$ is generated by $S$ and the $A_i$, it follows that $\A\cdot \hLa\subset \hLa$.

To show that $\A\bullet \hLa\subset\hLa$,
it suffices to check that $(a\bullet f)(A_w)\in S$
for $a=A_i$ for $i\in I$ and $a=\la\in X\subset S$. Using \eqref{E:Aadd} we have
\begin{align*}
  (A_i\bullet f)(A_w) &= f(A_w A_i) \\
  &= \chi(ws_i>w) f(A_{ws_i}) \in S
\end{align*}
since $f\in \hLa$. For $\la\in X$, by Proposition \ref{P:ALBcomm} we have
\begin{align*}
  (\la \bullet f)(A_w) &= f(A_w \la) \\
  &= (w\cdot\la) f(A_w) + \sum_{\substack{\alpha\in \Phi_\re^+ \\ ws_\al \lessdot w}} \ip{\alv}{\la} f(A_{ws_\al})
\end{align*}
which is in $S$ as required.
\end{proof}

\begin{lem} \label{L:GKM}
For every $f\in \hLa$ we have
\begin{align}\label{E:GKM}
  f(s_\al w)-f(w) \in \al S\qquad\text{for all $w\in W$ and $\al\in \Phi_\re$.}
\end{align}
\end{lem}
We call \eqref{E:GKM} the GKM condition since it is an instance
(for $\tFl$) of a general method due to \cite{GKM:1998} of constructing equivariant cohomology for
suitable spaces. For $\Fl$ this is due to Kostant and Kumar
\cite{KK}.
\begin{proof}
Let $\beta=w^{-1}\cdot\al$. Then $w^{-1}s_\al w = s_\beta$ or $ws_\beta=s_\al w$. We have
\begin{equation}\label{E:Aonfunc}
\begin{split}
(A_\beta\bullet f)(w) &= f(wA_\beta) \\
&= f(w \beta^{-1}(1-s_\beta)) \\
&= f(\al^{-1} w (1-s_\beta)) \\
&= \al^{-1}(f(w)-f(w s_\beta)) \\
&= \al^{-1}(f(w)-f(s_\al w)).
\end{split}
\end{equation}
Since $A_\beta \bullet f \in \hLa$ by Lemma \ref{L:AonGKM} and $w\in \A$ by
\eqref{E:WinA}, the right hand side of \eqref{E:Aonfunc} is in $S$
as required.
\end{proof}

\begin{lem} \label{L:Aonxi} For all $v\in W$ and $i\in I$
\begin{align}
  A_i \cdot \xi^v &= \chi(s_i v<v) \, \xi^{s_i v} \\
  A_i \bullet \xi^v &= \chi(vs_i<v)\, \xi^{vs_i}.
\end{align}
\end{lem}

\begin{exercise} Prove Lemma \ref{L:Aonxi}.
\end{exercise}

\begin{lem}\label{L:xid}
\begin{align}\label{E:xid}
  \xi^v(w) = d_{vw} \qquad\text{for $v,w\in W$.}
\end{align}
\end{lem}
\begin{proof} By \eqref{E:dmatrix} we have
\begin{align*}
  \ip{\xi^v}{w} &= \ip{\xi^v}{\sum_{u\in W} d_{uw} A_u} \\
  &= \sum_u  d_{uw}\, \ip{\xi^v}{A_u} \\
  &= \sum_u  d_{uw} \delta_{uv} \\
  &= d_{vw}.
\end{align*}
\end{proof}

Using the corresponding properties of $d_{vw}$ we have
the following.

\begin{prop} \label{P:xiprop} \
\begin{enumerate}
\item
$\xi^v(w)\in S$ is either $0$ or a homogeneous polynomial of degree $\ell(v)$.
\item
\begin{align}\label{E:xisupport}
\xi^v(w) =  0 \qquad\text{unless $v\le w$.}
\end{align}
\item
\begin{align}\label{E:xibase}
\xi^v(\id) &= \delta_{v,\id}\\
\label{E:xirec}
\xi^v(w) &= \xi^v(ws_i) +
  \chi(vs_i<v) (w\cdot\al_i)\, \xi^{vs_i}(w)\qquad\text{if $ws_i<w$.}
\end{align}
\item
$(-1)^{\ell(v)}\xi^v(w)\in \Z_{\ge0}[\al_i\mid i\in I]$.
\item
\begin{align}
\label{E:xidiag}
(-1)^{\ell(v)}\xi^v(v) =\prod_{\substack{\alpha\in \Phi_\re^+ \\ s_\al v < v }}
\alpha.
\end{align}
\item \cite{AJS} \cite{Bi}
\begin{align}\label{E:billey}
  (-1)^{\ell(v)}\xi^v(w) = \sum_{\substack{b\in\{0,1\}^\ell \\ \prod_{b_j=1} A_{i_j} = A_v}}
  \left(\prod_{j=1}^\ell \al_{i_j}^{b_j} s_{i_j}\right) \cdot 1.
\end{align}
\end{enumerate}
\end{prop}

For $f\in \AQ^*$, define $\Supp(f)$ and $\Omega(f)$ by
\begin{align}
\label{E:supp}
  \Supp(f) &= \{w\in W\mid f(w)\ne0\} \\
\label{E:Omegadef}
  \Omega(f) &= \text{largest subset of $W\setminus \Supp(f)$ such that} \\ \notag
  &\qquad \text{$v\in\Omega(f)$ and $u\le v$ implies $u\in \Omega(f)$ for all $u,v\in W$.}
\end{align}

\begin{prop}\label{P:GKM} Let $f\in \AQ^*$. Then $f\in \hLa$ if and only if
$f$ satisfies the GKM condition \eqref{E:GKM}.
\end{prop}
\begin{proof}
The forward direction holds by Lemma \ref{L:GKM}. Conversely
let $f\in \AQ^*$ satisfy the GKM condition \eqref{E:GKM}. The proof proceeds
by induction on $\Omega(f)$. Let $x\in W$ be minimal
such that $x\not\in\Omega(f)$. Let $\beta\in \Phi_\re^+$ such that $s_\beta x<x$.
Then $s_\beta x \in \Omega(f)$ and $f(s_\beta x)=0$. By \eqref{E:GKM}
$f(x)\in \beta S$. As this holds for all such $\beta$ we have $f(x) \in \xi^x(x) S$
by Lemma \ref{L:inversionset} and Proposition \ref{P:xiprop}. The function $g(w) = f(w) -
(f(x)/\xi^x(x))\xi^x(w)$ satisfies the GKM condition because both
$f$ and $\xi^x$ do. However $\Omega(g) \supsetneq \Omega(f)$.
By induction $g\in\hLa$ and therefore $f\in\hLa$.
\end{proof}

\begin{remark} \label{R:expandalg} The proof of Proposition
\ref{P:GKM} gives an effective algorithm to expand elements of
$\La$ into Schubert classes.
\end{remark}

Let $\la\in X\subset S$. Define $c^\la\in \AQ^*$ by
\begin{align}\label{E:firstChern}
  c^\la(a) = a \cdot \la\qquad\text{for all $a\in \AQ$}.
\end{align}

\begin{lem} \label{L:LBclass} We have
\begin{align}\label{E:LB2Schub}
  c^\la = \la\, \xi^\id + \sum_{i\in I} \ip{\alv_i}{\la} \,\xi^{s_i}
  \, \in \La.
\end{align}
\end{lem}
\begin{proof} Let $\al\in \Phi_\re$ and $w\in W$. We have
$c^\la(s_\al w)-c^\la(w)=s_\al w\cdot \la - w\cdot\la =
-\ip{\al^\vee}{w\cdot\la}\al \in \al S$. By Proposition \ref{P:GKM}
we have $c^\la\in \hLa$. The expansion \eqref{E:LB2Schub} may
be deduced from \eqref{E:AonX}.
\end{proof}

\begin{exercise}
For $w\in S_n$, let $\Schub_w(x;y)$ be the double Schubert
polynomial of Lascoux and Sch\"utzenberger \cite{Las} defined by
$\Schub_{w_0}(x;y)=\prod_{i=1}^{n-1} \prod_{j=1}^{n-i} (x_i-y_j)$
where $w_0\in S_n$ is the longest element (the reversing
permutation), and if $ws_i<w$, $\Schub_{ws_i}(x;y)=\partial^x_i
\Schub_w(x;y)$, where $\partial^x_i = (x_i-x_{i+1})^{-1}(1 - s_i^x)$
where $s_i^x$ acts on the $x$ variables. Find the precise
relationship between $\xi^v(w)$ for the type $A_{n-1}$ root datum
and the specializations $\Schub_v(x;wx)$ 
where $\Schub_v(x;wx)$ means $\Schub_v(x;y)$ with
the $y$ variables replaced by the permuted $x$ variables $wx$.
\end{exercise}

\subsection{Multiplication in $\La$ and coproduct in $\A$}
\label{SS:mult}

Let $\Fun(W,\Fr)$ be the set of functions $W\to \Fr$. There is a
bijection $\AQ^*\to \Fun(W,\Fr)$ given by restriction to $W$. We shall
use this identification without further mention. Then $\AQ^*$ is a
commutative $\Fr$-algebra using the usual left $\Fr$-module structure
and the pointwise product on $\Fun(W,\Fr)$:
\begin{align}\label{E:pointwise}
  (fg)(w) = f(w)g(w) \qquad\text{for $f,g\in \AQ^*$ and $w\in W$.}
\end{align}
The above formula is incorrect if $w$ is replaced by a general
element of $\AQ^*$; the definition of $\AQ^*$ says to first apply
linearity in the argument.

\begin{lem} \label{L:GKMclosed} $\hLa$ is a commutative $S$-algebra.
\end{lem}

For $u,v,w\in W$ define the equivariant Schubert structure constants $p^w_{uv}\in S$ by
\begin{align}\label{E:costructconst}
  \xi^u \xi^v = \sum_{w\in W} p^w_{uv} \xi^w.
\end{align}

\begin{prop} \label{P:costruct} \
\begin{enumerate}
\item $p^w_{uv}$ is $0$ or a homogeneous polynomial of degree
$\ell(u)+\ell(v)-\ell(w)$.
\item $p^w_{uv}=0$ unless $u\le w$ and $v\le w$.
\end{enumerate}
\end{prop}

This implies that the expansion \eqref{E:costructconst} is finite,
so that $\La$ is a commutative $S$-algebra under the pointwise
product.

\begin{exercise} \label{EX:GKMclosed} Prove Lemma \ref{L:GKMclosed}
and Proposition \ref{P:costruct}.
\end{exercise}

The following Theorem requires work. See \cite{Gr,Kum};
the positivity property (aside from the obvious sign) is called
\textit{Graham positivity}.

\begin{theorem}
$(-1)^{\ell(u)+\ell(v)-\ell(w)}p^w_{uv} \in \Z_{\ge0}[\al_i\mid i\in I]$.
\end{theorem}

\begin{prob} Find an explicit manifestly Graham-positive formula for
$p^w_{uv}$.
\end{prob}

The $p^w_{uv}$ have a geometric description as equivariant intersection numbers of
Schubert varieties in $\tFl$. Littlewood-Richardson numbers are a
very special case. The complete answer is unknown even for $G=GL_n$
and for the case $\ell(w)=\ell(u)+\ell(v)$ when $p^w_{uv}$ are integers;
this is the case of multiplying Schubert polynomials.

\begin{prop} \label{P:Chevalley}
\begin{align}
  c^\la \xi^v = (v\cdot \la) \,\xi^v + \sum_{\substack{\al\in \Phi_\re^+ \\ v \lessdot vs_\al}} \ip{\alv}{\la} \,\xi^{ vs_\al}.
\end{align}
\end{prop}

\begin{cor} \label{C:SchubChev}
For all $v\in W$ and $i\in I$,
\begin{align}
  \xi^{s_i} \xi^v = (v\cdot\omega_i-\omega_i) \,\xi^v + \sum_{\substack{\al\in \Phi_\re^+ \\ v\lessdot vs_\al}} \ip{\alv}{\omega_i}\,
  \xi^{vs_\al}.
\end{align}
\end{cor}

\begin{lem} \label{L:locareconst}
\begin{align} \label{E:locareconst}
  p^w_{vw} = \xi^v(w).
\end{align}
\end{lem}

\begin{exercise} \label{EX:locareconst} Prove Proposition
\ref{P:Chevalley}, Corollary \ref{C:SchubChev}, and Lemma
\ref{L:locareconst}.
\end{exercise}

The product in $\La$ can be recovered in $\A$ by duality.
Recall $\Delta$ from \eqref{E:DeltaQdef}.
Define the left $S$-bilinear pairing
\begin{align}
\ip{\cdot}{\cdot}: (\La \otimes_S \La) \times \Image(\Delta)
&\to S\\
  \ip{f\otimes g}{\Delta(a)} &=
  \sum_a f(a_{(1)}) g(a_{(2)})
\end{align}

\begin{lem} \label{L:multcoprod} For $f,g\in\La$ and $a\in \A$, we
have
\begin{align}\label{E:multcoprod}
  \ip{fg}{a} = \ip{f\otimes g}{\Delta(a)}.
\end{align}
\end{lem}
\begin{proof} We give the proof over $\Fr$.
By linearity we may assume $a=w\in W$. We have
\begin{align}
  \ip{f\otimes g}{\Delta(w)} = \ip{f\otimes g}{w \otimes w} =
  f(w)g(w) = \ip{fg}{w}.
\end{align}
\end{proof}

Applying \eqref{E:multcoprod} we see that
\begin{align*}
  p^w_{uv} &= \ip{\xi^u\xi^v}{A_w} \\
  &=\ip{\xi^u\otimes\xi^v}{\Delta(A_w)}.
\end{align*}
Therefore $p^w_{uv}$ is the coefficient of $A_u \otimes A_v$ in
$\Delta(A_w)$. Let $w=s_{i_1}\dotsm s_{i_\ell}$ be a reduced
decomposition. By Proposition \ref{P:coalg}, $\Delta(w)$ can be computed
as the componentwise product of $\Delta(A_{i_1})\dotsm
\Delta(A_{i_\ell})$.

\subsection{Forgetting equivariance}
Recall that $X\subset S$ is the subspace of linear polynomials.
Consider the ring homomorphism
$\epsilon_0^S:S\to\Z$ defined by the property $\epsilon_0^S(1)=1$ and $\epsilon_0^S(X)=0$.
Define the ring homomorphism
$\epsilon_0^\La:\La\to \La_0 := \Z\otimes_S \La $. It may be computed as
follows. Write $f\in\La$ as $f=\sum_{v\in W} f_v \xi^v$ for some
$f_v\in S$. Then $\epsilon_0^\La(f) = \sum_{v\in W} \epsilon_0^S(f_v)\, \xi^v_0$
where $\xi^v_0=\epsilon_0^\La(\xi^v)\in\La_0$.

\begin{prop} $\{\xi^v_0\mid v\in W\}$ is a $\Z$-basis of $\La_0$. This basis has structure constants
given by the degree zero structure constants of $\La$:
\begin{align*}
  \xi^u_0\, \xi^v_0 &= \sum_{w\in W} \phi_0(p^w_{uv})\, \xi^w_0 \\
  &= \sum_{\substack{w\in W \\ \ell(w)=\ell(u)+\ell(v)}} p^w_{uv} \,\xi^w_0.
\end{align*}
\end{prop}

The following is due to Lam~\cite{Lam:2008}; see also Chapter 3 Section~\ref{chapter3.section.commutation}.
Define $\epsilon_0^\A:\A\to \Z \otimes_S \A$  by $\epsilon_0^\A(a)=1\otimes a$.
By abuse of notation we write $A_w$
for its image under $\epsilon_0^\A$. Then for $a_w\in S$ we have
\begin{align}\label{E:phi0A}
\epsilon_0^\A(\sum_{w\in W} a_w A_w)=\sum_{w\in W} \epsilon_0^S(a_w) A_w.
\end{align}
There is a ring isomorphism $\Z\otimes_S \A\cong \A_0$ where $\A_0$
is the nilCoxeter algebra. Using this identification we may view $\epsilon_0^\A:\A\to \A_0$.

\begin{example} \label{X:forget} For the root datum of type $A_1$, $W=\{\id,s\}$, $\Phi_\re^+=\{\al\}$,
$\xi^\id(\id)=\xi^\id(s)=1$, $\xi^s(\id)=0$ and $\xi^s(s)=-\al$. Then
$\xi^\id$ is the identity in $\La$ and $\xi^s \xi^s = -\al \xi^s$.
In $\La_0$, $\xi^\id_0$ is the identity and $(\xi^s_0)^2=0$.
\end{example}

\subsection{Parabolic case}
\label{SS:parabolic}
Let $J\subset I$ be a fixed subset. Let $W_J\subset W$ be the
subgroup generated by $s_i$ for $i\in J$. Let $W^J$ denote the
set of minimum length coset representatives in $W/W_J$. Define
\begin{align}
  \La^J = \{f \in\La\mid \text{$f(wu)=f(w)$ for all $w\in W$, $u\in W_J$} \}.
\end{align}

\begin{prop}\label{P:partialflagbasis}
$\La^J$ is an $S$-subalgebra of $\La$ with
basis $\{\xi^v\mid v\in W^J\}$.
\end{prop}

\begin{remark} There are polynomial or symmetric function
representatives for many cases of cohomological Schubert classes.
Nonequivariant:
\begin{enumerate}
\item Grassmannian $\Gr(r,n)$. Schur functions.
\item Lagrangian Grassmannian $LG(n,2n)$. Schur $Q$-functions. \cite{Pr}.
\item Orthogonal Grassmannians $OG(n,2n+1)$ and $OG(n+1,2n+2)$:
Schur $P$-functions. \cite{PR}.
\item $\Fl$ for $SL_n$: Lascoux-Sch\"utzenberger Schubert polynomials \cite{LS:SchubertPoly}.
\item $\Fl$ for $Sp_{2n}$, $SO_{2n+1}$, $SO_{2n}$: Billey-Haiman Schubert
polynomials \cite{BH}.
\item $\Gr_{SL_{k+1}}$: dual $k$-Schur or affine Schur functions
\cite{LM:2007} \cite{LM:2008} \cite{Lam:2006} \cite{Lam:2008}.
\item $\Gr_{Sp_{2n}}$: \cite{LSS:C}.
\item $\Gr_{SO_n}$: \cite{Pon}.
\end{enumerate}
Equivariant:
\begin{enumerate}
\item $\Fl$ for $SL_n$: Double Schubert polynomials \cite{Las}
\item $\Fl$ in classical type: BCD double Schubert
polynomials \cite{IMN}.
\end{enumerate}
\end{remark}

\subsection{Geometric interpretations}
Fix a root datum. Let $T\subset B\subset G$
be a maximal torus contained in a Borel subgroup $B$ in the
Kac-Moody group $G$ \cite{Kum}. The Kac-Moody flag ind-variety $\Fl=G/B$ \cite{Kum}
is paved by cells $B\dot{w}B/B\cong \C^{\ell(w)}$ whose closures $X_w$
define equivariant fundamental homology classes $[X_w]_T\in H_T(\Fl)$. Finally, 
the homology $H_T(\Fl)$ has a noncommutative ring structure. 
Let $G$ act diagonally on $\Fl\times \Fl$. There are isomorphisms
\begin{align}\label{E:conviso}
H_G(\Fl\times\Fl) \cong H_{G\times (B\times B)}(G \times G) \cong H_{B\times B}(G) \cong H_B(\Fl)\cong H_T(\Fl).
\end{align}
The left hand side has a convolution product \cite{Gin:1997}. There are also isomorphisms
\begin{align}\label{E:convmodiso}
  H_G(\Fl) \cong H_{G\times B}(G) \cong H_B(G\backslash G) \cong H_B(\pnt) \cong H_T(\pnt).
\end{align}
Finally, $H_G(\Fl\times \Fl)$ acts on $H_G(\Fl)$ by convolution \cite{Gin:1997}.
The following is \cite[Prop. 12.8]{Gin:1997}.

\begin{theorem} \label{T:Hom}
\begin{enumerate}
\item
There is an injective ring homomorphism
\begin{align}
  H_T(\pnt) \cong H_G(\Fl) \to H_G(\Fl\times \Fl)
\end{align}
given by composing the isomorphism \eqref{E:convmodiso} with
the embedding induced by the $G$-equivariant
diagonal inclusion $\Fl\to\Fl\times\Fl$.
\item There is an injective ring homomorphism
\begin{align}
  \A_0 &\to H_G(G/B \times G/B) \\
  A_w &\mapsto [O_w]_G \qquad\text{for $w\in W$}
\end{align}
where $[O_w]_G$ is the $G$-equivariant fundamental class of the closure
$O_w$ of the diagonal $G$-orbit of the point $(\dot{e}B/B,\dot{w}B/B)$.
\item The above maps define a ring and left $S$-module isomorphism
\begin{align}\label{E:Hiso}
  \A \cong H_G(\Fl\times\Fl).
\end{align}
\item The following diagram commutes where the horizontal maps are
action maps and the vertical maps are isomorphisms.
\begin{align*}
\begin{diagram}
 \node{H_G(\Fl \times \Fl) \times H_G(\Fl)} \arrow{e,t}{}\arrow{s}{} \node{H_G(\Fl)} \arrow{s}{} \\
 \node{\A \times H_T(\pnt)} \arrow{e,b}{} \node{H_T(\pnt)}
\end{diagram}
\end{align*}
\end{enumerate}
\end{theorem}


There is another kind of Kac-Moody flag variety $\tFl$ with $T$-action called a thick flag scheme \cite{Kas}.
There is a bijection $\tFl^T\cong W$ of the $T$-fixed point set of $\tFl$ with $W$. The thick flag scheme $\tFl$
is paved by finite-codimensional cells, each of which contains a unique $T$-fixed point.
The closure of the cell containing $w$ is denoted $X^w$ and defines a class $[X^w]^T\in H^T(\tFl)$
such that $\{[X_w]_T\mid w\in W\}$ and $\{[X^w]^T\mid w\in W\}$
are dual bases under an intersection pairing
\begin{align}\label{E:pair}
H_T(\Fl) \times H^T(\tFl) \to H^T(\pnt).
\end{align}
For $w\in W$, let $i_w$ be the inclusion of a point into $\tFl$ whose image
is the $w$-th $T$-fixed point.
See \cite{Ar} \cite{KK}\cite{Kum}  for $\Fl$ and \cite{Kas} \cite{KS} for $\tFl$.

\begin{theorem}\label{T:Fl} \
There is a commutative diagram
\begin{align} \label{E:resiso}
\begin{diagram}
\node{H^T(\tFl)} \arrow{e,t}{\res} \arrow{s,t}{\text{For}} \node{\La} \arrow{s,b}{\epsilon_0} \\
\node{H^*(\tFl)} \arrow{e,b}{\res_0} \node{\La_0}
\end{diagram}
\end{align}
where $\text{For}$ is the ring homomorphism that forgets equivariance,
$\res$ is the $S=H^T(\pnt)$-algebra
isomorphism that restricts a $T$-equivariant class to the set $\tFl^T\cong W$
of $T$-fixed points (and thus sends the $T$-equivariant class $[X^v]^T$
to $\xi^v$), and $\res_0$ is a ring isomorphism sending the
ordinary cohomology class $[X^v]$ to $\xi^v_0$.  Moreover
\begin{align}
\label{E:xigeom}
\xi^v(w)=i_w^*([X^v]^T).
\end{align}
\end{theorem}

\begin{prop} \label{P:geompair} The pairing \eqref{E:pair}
corresponds, using the isomorphisms \eqref{E:Hiso} and \eqref{E:resiso},
to the pairing \eqref{E:Apair}.
\begin{align*}
\begin{diagram}
\node{H_T(\Fl)\times H^T(\tFl)} \arrow{e,t}{} \arrow{s,e}{} \node{H^T(\pnt)} \arrow{s,e}{} \\
\node{\A \times \La} \arrow{e,b}{} \node{S}
\end{diagram}
\end{align*}
\end{prop}

For $i\in I$ let $p_i:\tFl\to \tFl^{i}$ be the projection of $\tFl$ onto the
thick partial flag scheme defined by the minimal parabolic subgroup corresponding to $i$.

\begin{prop} \label{P:geomAonLa}
Under the isomorphism $\res$ of Theorem \ref{T:Fl}:
\begin{enumerate}
\item The first Chern class $c_1(\LL_\la)$ of the $T$-equivariant line bundle on $\tFl$ of weight $\la$,
maps to the function $c^\la$, and multiplication by $c_1(\LL_\la)$ in $H^T(\tFl)$ corresponds to
the operation $f\mapsto \la\bullet f$ on $\La$.
\item The push-pull operator $p_i^* p_{i*}$ on $H^T(\tFl)$ corresponds to the operation
$f\mapsto A_i\bullet f$ on $\La$.
\end{enumerate}
\end{prop}

By the definitions \eqref{E:leftQWonAQ} and
\eqref{E:firstChern}, for $f\in\hLa$, $\la\in X$, and $w\in W$
we have
\begin{align}\label{E:LBprod}
  (\la \bullet f)(w) = (w\cdot \la) f(w) = c^\la(w) f(w) = (c^\la f)(w).
\end{align}


\begin{prop} \label{P:partialflag} There is a commutative diagram
\begin{align}
\begin{diagram}
\node{H^T(\tFl^J)} \arrow{e,t}{\res} \arrow{s,t}{} \node{\La^J} \arrow{s,b}{} \\
\node{H^T(\tFl)} \arrow{e,b}{\res} \node{\La}
\end{diagram}
\end{align}
of $S$-algebra homomorphisms, where the horizontal maps are isomorphisms
and the vertical maps are inclusions,
where $\tFl^J$ is the thick partial flag scheme
associated to the subset $J\subset I$. Under the top isomorphism,
the Schubert basis of
$H^T(\tFl^J)$ maps to the elements $\xi^v$ for $v\in W^J$.
\end{prop}

\section{Affine Grassmannian}
\label{sec:affineGrassmannian}

The affine Grassmannian $\Gr_G$ has the form
$\Gr_G = G_\af/P_\af$ where $G_\af$ is the Kac-Moody group of affine type associated
with the semisimple algebraic group $G$ and $P_\af$
is the maximal parabolic subgroup obtained by ``omitting the zero node".
Therefore the previous sections apply and describe the Schubert calculus of $H^{T_\af}(\Gr_G)$
under the cup product, where $T_\af$ is the maximal torus in $B_\af\subset P_\af\subset G_\af$.

For our applications we consider $H^T(\Gr_G)$ where $T\subset B\subset G$ is a maximal torus
in the semisimple algebraic group $G$ whose corresponding affine Kac-Moody group is $G_\af$.
Forgetting from $T_\af$ to $T$ we retain the Schubert calculus of $H^T(\Gr_G)$.
The algebraic model for this ring is denoted $\Lambda'_{\mathrm{af}}$ and is
studied in Section~\ref{chapter4.section.small torus}. It is due to Goresky, Kottwitz, and Macpherson.
This ring also appears in Chapter 3 Section~\ref{chapter3.section.sketch}.

However in this setting there are extra features. Instead of just the above ring,
we get a dual Hopf algebra $H_T(\Gr_G)$ given by the equivariant homology of $\Gr_G$
under the Pontryagin product. The ring $H_T(\Gr_G)$ is very interesting in itself
for several reasons. 

One reason is due to Peterson \cite{Pet} \cite{LS:QH}: the Schubert structure constants of $H_T(\Gr_G)$ 
coincide with those of $QH^T(G/B)$, the equivariant quantum cohomology of the flag manifolds $G/B$. 
Peterson realizes $H_T(\mathrm{Gr}_G)$ algebraically as a commutative subalgebra $\mathbb{B}$
of a variant of the affine nilHecke ring of Kostant and Kumar,
which uses the affine Weyl group but polynomial functions on the Lie algebra of the
finite torus, rather than the affine one. This variant is denoted
$\mathbb{A}_{\mathrm{af}}$ in Chapter 3 Section~\ref{chapter3.section.affine nilHecke} 
and $\mathbb{A}'_{\mathrm{af}}$ in Section~\ref{SS:Petersonsubalgebra}.

For the second reason we forget equivariance and consider $G=SL_{k+1}$.
Then the ring $H_*(\Gr_{SL_{k+1}})$ is Hopf-isomorphic to the subring $\Z[h_1,\dotsc,h_k]$ of symmetric
functions generated by the first $k$ complete symmetric functions, and the Schubert basis is given by
the $k$-Schur functions of Lapointe, Lascoux, and Morse (at $t=1$). 
This basis is connected with Macdonald polynomials.

For $G$ of classical type, $H_*(\Gr_G)$ can be realized by symmetric functions
and the Schubert bases give rise to new families of symmetric functions; see
\cite{LSS:C} \cite{Pon}.

\subsection{Affine Grassmannian as partial affine flags}
\label{SS:GafPaf}

Consider a finite root datum with simple Lie group $G\supset B\supset T$,
containing a Borel subgroup $B$ and maximal torus $T$.
The affine Grassmannian is by definition $\Gr=\Gr_G = G(\C((t)))/G(\C[[t]])$ where
$\C[[t]]$ is the formal power series ring and $\C((t))$ is the
formal Laurent series ring.

Consider the associated affine root datum with Kac-Moody group, Iwahori subgroup,
and maximal torus $G_\af\supset B_\af\supset T_\af$.
Let $\Fl_\af=G_\af/B_\af$ be the affine flag ind-variety where
$B_\af$ is the Iwahori (affine Borel) subgroup. Then (up to a central extension that
gets modded out in the quotient) $G_\af = G(\C((t)))$,
$P_\af=G(\C[[t]])=P_I$ for the subset $I\subset I_\af$, and $\Gr =
G_\af/P_\af = \Fl_\af^I$. Let  $\tFl_\af$ be the
thick affine flag scheme \cite{Kas} and $\tGr=\tFl_\af^I$ the
thick affine Grassmannian, the associated thick partial affine flag scheme.

The previous section describes isomorphisms $$\La_\af\cong
H^{T_\af}(\tGr_G)\qquad\qquad\A_\af\cong H_{T_\af}(\Gr_G).$$

Our applications require working with equivariance with respect to the ``small torus"
$T$ rather than the maximal torus $T_\af$ in $G_\af$. It is possible to modify
the GKM ring $\La_\af$ and the affine nilHecke ring $\A_\af$
to obtain rings $\La_\af'$ and $\A_\af'$ such that
$$\La_\af'\cong
H^T(\tGr_G)\qquad\qquad\A_\af'\cong H_T(\Gr_G).$$
The necessary modifications were obtained in \cite{GKM} and \cite{Pet}
respectively.

Under the small torus equivariance, the homology ring $H_T(\Gr_G)$
becomes a commutative and cocommutative Hopf algebra over $S=H^T(pt)$, with dual Hopf algebra
given by $H^T(\Gr_G)$. Our main interest lies
in obtaining explicit computations involving this Hopf structure.

We write $W_\af^0$ instead of $W_\af^I$ and
$\pi:X_\af\to X$ for the natural projection and also for the induced
map $\pi:S_\af=\Sym(X_\af)\to S=\Sym(X)$.

\subsection{Small torus version of $\La_\af$}
\label{chapter4.section.small torus}
Here we follow \cite{GKM}.
Consider the map $\pi^*:\La_\af \to \Fun(W_\af,S)$ given by $\pi^*(f)=\pi\circ f$.
We wish to characterize the image of $\pi^*$ and the image of its restriction
to $\La_\Gr$.

The usual GKM condition \eqref{E:GKM} holds for functions $f:W_\af\to S$
in the image of $\pi^*$, since it holds in $\La_\af$.
However there are more conditions.

Let $\La'_\af$ be the set of functions $f\in \Fun(W_\af,S)$ that
satisfy \eqref{E:GKM} and the following \textit{small torus GKM conditions}:
\begin{align}
\label{E:smallGKMt}
  f((1-t_{\alv})^d w)&\in \al^d S \\
\label{E:smallGKMs}
  f((1-t_{\alv})^{d-1}(1-s_\al)w) &\in \al^d S
\end{align}
for all $d\in\Z_{>0}$, $w\in W_\af$, and $\al\in \Phi$.

Let $\La'_\Gr$ be the functions in $\La'_\af$ that
are constant on cosets in $W_\af/W$.

\begin{lem} \label{L:smallGKM} Suppose $f\in \Fun(W_\af,S)$ satisfies
\eqref{E:smallGKMt}. Then for all $d\in\Z_{>0}$, $p\in\Z$, $\al\in \Phi$,
and $w\in W_\af$ we have
\begin{align}
  f((1-t_{\alv})^{d-1} w) \equiv f((1-t_{\alv})^{d-1} t_{p\alv} w) \mod \al^d S.
\end{align}
\end{lem}
\begin{proof} By induction we may reduce to the case $p=1$, which
is just \eqref{E:smallGKMt}.
\end{proof}

\begin{lem} \label{L:smallGKMtthens}
If $f\in\Fun(W_\af,S)$
satisfies \eqref{E:GKM} and \eqref{E:smallGKMt} and is constant on
cosets in $W_\af/W$ then
it also satisfies \eqref{E:smallGKMs}.
\end{lem}
\begin{proof}
Let $\al\in \Phi$, $m\in\Z$, $d\in\Z_{>0}$, and
$w=t_\mu u\in W_\af$ with $\mu\in Q^\vee$ and $u\in W$. Let $y=(1-t_{\alv})^{d-1}$.
Using Lemma \ref{L:smallGKM} we have
\begin{align*}
   f(y (1-s_\al)t_\mu u)
&= f(y t_\mu u) - f(y t_{s_\al\cdot\mu} s_\al u) \\
&= f(y t_\mu) - f(y t_{-\ip{\mu}{\al}\alv} t_\mu) \\
&\equiv f(y t_\mu) - f(y t_{\alv} t_\mu) \mod \al^d S \\
&= f((1-t_{\alv})^d t_\mu) \in \al^d S.
\end{align*}
\end{proof}

Define
\begin{align}\label{E:xib}
\xib^v = \pi \circ \xi^v\qquad\text{for $v\in W_\af$.}
\end{align}

The following Theorem is proved in Appendix \ref{A:smalltorusGKM}
and is due to Goresky, Kottwitz, and Macpherson \cite{GKM}.

\begin{theorem} \label{T:smallGKM}
\begin{enumerate}
\item
There is an $S$-algebra isomorphism
\begin{align}
\label{E:HTsmall}
  H^T(\tFl_\af) &\to \La'_\af\cong \bigoplus_{w\in W_\af} S\, \xib^w \\ \notag
  [X^w]^T &\mapsto \xib^w\qquad\text{for $w\in W_\af$.}
\end{align}
\item
The isomorphism \eqref{E:HTsmall} restricts to an $S$-algebra isomorphism
\begin{align}
  H^T(\tGr) \to \La'_\Gr \cong \bigoplus_{w\in W_\af^0} S (\xib^w).
\end{align}
\end{enumerate}
\end{theorem}

The existence of the following map is explained in Appendix \ref{A:HomGr}.

\begin{prop} \label{P:wrongway} The map $\varpi$ defined by
\begin{align}\label{E:varpi}
  \varpi:\La'_\af&\to\La'_\Gr \\ \notag
  f &\mapsto (t_\mu u \mapsto f(t_\mu))\qquad\text{for $\mu\in Q^\vee$ and $u\in W$}
\end{align}
is an $S$-algebra homomorphism which is the identity when restricted to $\La'_\Gr\subset\La'_\af$.
\end{prop}
\begin{proof}[Proof of Prop. \ref{P:wrongway}]
All properties are clear except $\Image(\varpi)\subset\La'_\Gr$.
Let $f\in \La'_\af$ and $g=\varpi(f)$. By Lemma \ref{L:smallGKMtthens} it suffices
to verify \eqref{E:smallGKMt} and \eqref{E:GKM}.
For the former, let $d\in \Z_{>0}$ and $\al\in \Phi$.
Let $w=t_\mu u\in W_\af$ with $\mu\in Q^\vee$ and $u\in W$. We have
\begin{align*}
  g((1-t_{\alv})^d w) &= g((1-t_{\alv})^d t_\mu u) = f((1-t_{\alv})^d t_\mu) \in \al^d S
\end{align*}
by the definition of $\varpi$ and \eqref{E:smallGKMt} for $f$.
For \eqref{E:GKM}, let $\al+m\delta\in \Phi_\re^\af$ for $\al\in\Phi$
and $m\in\Z$. Let $w=t_\mu u$ for $\mu\in Q^\vee$ and $u\in W$. We have
$s_{\al+m\delta} w = s_\al t_{m\alv} t_\mu u = t_{-m\alv} t_{s_\al\cdot\mu} s_\al u=
t_{-(m+\ip{\mu}{\al})\alv} t_\mu s_\al u$. Therefore
\begin{align*}
     g(s_{\al+m\delta} w) - g(w)
  &= f(t_{-(m+\ip{\mu}{\al})\alv} t_\mu) - f(t_\mu) \in \al S
\end{align*}
by Lemma \ref{L:smallGKM} for $d=1$.
\end{proof}

\subsection{Homology of the affine Grassmannian}\ 

For $\mu\in Q^\vee$ define $i_\mu^*\in \Hom_S(\La'_\Gr,S)$ by
\begin{align} \label{E:istar}
  i_\mu^*(f) = f(t_\mu) \qquad\text{for $f\in \La'_\Gr$.}
\end{align}
In the notation of Theorem \ref{T:Fl}, $i_\mu^*=i_{t_\mu}^*$. 

\begin{lem} \label{L:ratHomGr} $\Fr \otimes_S \Hom_S(\La'_\Gr,S)$
has $\Fr$-basis $\{i_\mu^*\mid \mu\in Q^\vee\}$.
\end{lem}
\begin{proof} Let $m_\mu \in W_\af^0$ be defined by $m_\mu W = t_\mu W$ for $\mu\in Q^\vee$.
By \eqref{E:xidiag} the matrix $\xi^v(w)$
is triangular, so the same is true of its restriction $\xib^{m_\mu}(t_\nu)$ for $\nu\in Q^\vee$.
Moreover its diagonal entries are nonvanishing:
since $\xib^{m_\mu}$ is constant on the cosets of $W_\af/W$ we have
$\xib^{m_\mu}(t_\mu)=\xib^{m_\mu}(m_\mu)\ne0$. The Lemma follows.
\end{proof}

The following result is proved in Appendix \ref{A:HomGr}.

\begin{prop} \label{P:convolution}
$H_T(\Gr)$ and $H^T(\tGr)$ are dual Hopf
algebras over $S$.
There is an isomorphism $H_T(\Gr)\cong \Hom_S(\La'_\Gr,S)$
under which the product in $\Hom_S(\La'_\Gr,S)$
is defined over $\Fr$ by $(i_\la^*,i_\mu^*)\mapsto i_{\la+\mu}^*$.
$\Hom_S(\La'_\Gr,S)$ and $\La'_\Gr$ have the structure of dual Hopf algebras
over $S$ where the product in $\Hom_S(\La'_\Gr,S)$
is defined over $\Fr$ by $(i_\la^*,i_\mu^*)\mapsto i_{\la+\mu}^*$.
\end{prop}

Define the $S$-basis $\{\xi_w\mid w\in W_\af^0\}$ of $\Hom_S(\La'_\Gr,S)$
by
\begin{align}\label{E:xihomdef}
\ip{\xi_w}{\xib^v} &= \delta_{wv}\qquad\text{for $v\in W_\af^0$.}
\end{align}

\begin{remark} \label{R:HomGr} There is an isomorphism of Hopf $S$-algebras
\begin{align}\label{E:HomGriso}
H_T(\Gr) &\cong \Hom_S(\La'_\Gr,S) \\
[X_w]_T &\mapsto \xi_w\qquad\text{for $w\in W_\af^0$}
\end{align}
where $X_w = \overline{B_\af \dot{w} P_\af/P_\af}\subset G_\af/P_\af=\Gr$
is the Schubert variety and $[X_w]_T$ is the $T$-equivariant
fundamental class. Moreover there is a perfect intersection pairing
\begin{align} \label{E:HHpair}
  H_T(\Gr) \times H^T(\tGr) \to S
\end{align}
which corresponds to the evaluation pairing
\begin{align}\label{E:fakeHHpair}
  \Hom_S(\La'_\Gr,S) \times \La'_\Gr \to S
\end{align}
under the isomorphisms \eqref{E:HomGriso} and \eqref{T:smallGKM}.
\end{remark}

\subsection{Small torus affine nilHecke ring and Peterson subalgebra}
\label{SS:Petersonsubalgebra}

We denote by $\A_\af$ the maximal torus affine nilHecke ring obtained from an
untwisted affine root datum as in Section \ref{sec:MarknilHecke}. 
The Peterson (small torus) affine nilHecke ring $\A'_\af$ (denoted $\A_\af$ in Chapter 3)
is the subquotient of $\A_\af$ defined as follows.

Let $W_\af$ be the affine Weyl group. It acts naturally on the affine weight lattice $X_\af$
and preserves the level \eqref{E:level} of a weight. Let $X_0 \subset X_\af$ be the
sublattice of level zero weights. By \eqref{E:Xaf} we have $X_\af \cong \Z \La_0 \oplus X_0$
where $\La_0$ is the zero-th affine fundamental weight. There is a natural surjection 
$\pi:X_0\to X$ with kernel $\Z\delta$ where $\delta$ is the null root. Since the action of $W_\af$
fixes $\delta$ there is an action of $W_\af$ on $X$ called the level zero action; it is defined by
\eqref{E:translevelzero}. Let $\Fr_\af$, $\Fr_0$, and $\Fr$ be the fraction fields of the
algebras $S_\af=\Sym(X_\af)$, $S_0=\Sym_(X_0)$, and $S=\Sym(X)$ respectively.
We have $S_\af = S_0[\La_0]$, $S_0/\delta S_0 \cong S$.

In the definition of $\A_\af$ the elements $A_i$ lie in the twisted group algebra $\Fr_\af W_\af$,
but since the variable $\La_0$ does not appear in them, they are also elements of the 
subring $\Fr_0 W_\af$. Let $\A_\af^0$ be the subring of $\Fr_0 W_\af$
generated by $A_i$ for $i\in I_\af$ and $S_0$. Then $\A_\af^0$ is naturally a subring of $\A_\af$
and adjoining the variable $\La_0$ to $\A_\af^0$ yields $\A_\af$: $\A_\af = S_\af \otimes_{S_0} \A_\af^0$.
Define
\begin{align}\label{E:PetnilHeckedef}
  \A_\af' = S \otimes_{S_0} \A_\af^0
\end{align}
where $S_0$ acts on $S$ by the ring homomorphism $\pi: S_0\to S$. In other words, $\A_\af'$ is obtained from
$\A_\af^0$ by setting the variable $\delta$ to zero.
In particular, in $\A_\af'$, since $\pi(\alpha_0)=\pi(\delta-\theta)=-\theta$, we have
\begin{align}\label{E:A0Pet}
A_0 = (1-s_0) /(-\theta).
\end{align}
Alternatively one may define $\A'_\af$ as the subring of $\Fr W_\af$ (where $W_\af$ acts on $\Fr$ by the
level zero action) generated by $A_i$ for $i \in I_\af$ and $S$ with $A_0$ defined as in \eqref{E:A0Pet}.

Any formula in $\A_\af$ that does not involve the
variable $\La_0$, automatically holds in $\A_\af'$ after applying $\pi$.
By Lemma \ref{L:Abasis} (applied for an untwisted affine root datum, so that the
Weyl group is $W_\af$ and the polynomial ring is $S_\af$) we have
\begin{align}
  \A_\af' = \bigoplus_{w\in W_\af} S A_w.
\end{align}
Evaluation yields left $S$-linear perfect pairings
\begin{align}
\A'_\af \times \La'_\af &\to S \\
\Hom_S(\La'_\Gr,S) \times \La'_\Gr &\to S.
\end{align}

Let $j: \Hom_S(\La'_\Gr,S) \to \A'_\af$ be the left $S$-module homomorphism defined by
\begin{equation}\label{E:j}
\ip{j(\xi)}{f} = \ip{\xi}{\varpi(f)}\qquad\text{for $f\in\La'_\af$, $\xi\in \Hom_S(\La'_\Gr)$.}
\end{equation}
It is well-defined since we may vary $f$ over the basis $\{\xib^v\mid v\in W_\af\}$ of $\La'_\af$.
The map $j$ is injective, for $\varpi$ is onto by Proposition \ref{P:wrongway}.

The following diagram relates the maps $\varpi$, $j$, and the pairings.
\begin{align}
\begin{diagram}
\node{\A'_\af} \node{\times} \node{\La'_\af}  \arrow{s,t,A}{\varpi} \arrow{e,t}{\ip{\cdot}{\cdot}}\node{S} \arrow{s,b}{\id_S}\\
\node{\Hom_S(\La'_\Gr,S)} \arrow{n,t,J}{j} \node{\times} \node{\La'_\Gr} \arrow{e,b}{\ip{\cdot}{\cdot}}  \node{S}.
\end{diagram}
\end{align}

The \textit{Peterson subalgebra} $\Pet$ is defined to be the centralizer of $S$ in $\A'_\af$:
\begin{align}\label{E:B}
\Pet := Z_{\A'_\af}(S).
\end{align}

\begin{lem}\label{L:imagej}
We have $\Image(j) = \bigoplus_{\mu\in Q^\vee} \Fr t_\mu \cap \A'_\af =
\Pet$.
\end{lem}
\begin{proof}
For $\mu \in Q^\vee$, we have $\ip{i^*_\mu}{\varpi(f)} = f(t_\mu)$ for all $f\in \A'_\af$,
so that
\begin{align}\label{E:jistar}
j(i^*_\mu) = t_\mu \in \A'_\af \qquad\text{for $\mu\in Q^\vee$.}
\end{align}
The first equality holds by Lemma \ref{L:ratHomGr}.
For the second equality, $\bigoplus_{\mu\in Q^\vee} \Fr t_\mu \cap \A'_\af \subseteq
\Pet$ holds because under the level zero action, $t_\mu$ acts
on $X$ trivially for all $\mu\in Q^\vee$. For the other direction,
let $a=\sum_{w\in W_\af} a_w w\in \Pet$ for $a_w\in \Fr$.
Then for all $\la\in X$ we have
\begin{equation*}
  0 = \la a - a \la = \sum_{w\in W_\af} a_w(\la-w\cdot\la) w.
\end{equation*}
Therefore for all $w\in W_\af$ either $a_w=0$ or $w\la=\la$ for all
$\la \in X$. Taking $\la$ to be $W$-regular, we see that the latter
only holds for $w=t_\mu$ for some $\mu\in Q^\vee$.
\end{proof}

The algebra $\Pet$ inherits a coproduct $\Delta: \Pet \to \Pet
\otimes_S \Pet$ from the coproduct of $\A'_\af$.
That $\Delta(\Pet) \subset \Pet
\otimes_S \Pet$ follows from Lemma \ref{L:imagej} and \eqref{E:DeltaQdef}.
We make $\Pet$ into a Hopf algebra over $S$ by defining the antipode $t_\mu\mapsto t_{-\mu}$
for $\mu\in Q^\vee$.

\begin{theorem}\label{T:jmap}\cite{Pet} \cite{Lam:2008}
The map $j: \Hom_S(\La'_\Gr,S) \to \Pet$ is a Hopf-isomorphism.
\end{theorem}
\begin{proof} We have seen $j$ is injective, and $j$ is surjective
by Lemma \ref{L:imagej}. It suffices to show $j$ is a bialgebra morphism,
since the compatibility with antipodes follows as a consequence.

Since $j$ is $S$-linear, to check that $j$ is compatible with the
Hopf-structure we check the product and coproduct structures on the
$F$-basis $\{i^*_\mu \mid \mu \in Q^\vee\}$ of $\Fr \otimes_S\Hom_S(\La'_\Gr,S)$.

By Proposition \ref{P:convolution} and \eqref{E:jistar}, for $\la,\mu\in Q^\vee$
we have $$j(i^*_\la i^*_\mu) = j(i^*_{\la+\mu}) = t_{\la+\mu} = t_\la t_\mu = j(i^*_\la) j(i^*_\mu).$$
Thus $j$ is an algebra morphism.

To show $j$ is a coalgebra morphism, let $a\in \Hom_S(\La'_\Gr,S)$ and $f,g\in \La'_\Gr$.
Then
\begin{align*}
 \ip{(j\otimes j)\circ\Delta(a)}{f\otimes g}
 &= \sum_a \ip{j(a_{(1)}) \otimes j(a_{(2)})}{f\otimes g} \\
 &= \sum_a \ip{j(a_{(1)})}{f} \ip{j(a_{(2)})}{g} \\
 &= \sum_a \ip{a_{(1)}}{\varpi(f)} \ip{a_{(2)}}{\varpi(g)} \\
 &= \sum_a \ip{a_{(1)} \otimes a_{(2)}}{\varpi(f)\otimes \varpi(g)} \\
 &= \ip{\Delta(a)}{\varpi(f)\otimes \varpi(g)} \\
 &= \ip{a}{\varpi(f)\varpi(g)} \\
 &= \ip{a}{\varpi(fg)} \\
 &= \ip{j(a)}{fg} \\
 &= \ip{\Delta(j(a))}{f\otimes g}.
\end{align*}
\end{proof}

\subsection{The $j$-basis}
\label{SS:jbasis}

Peterson characterized the image $j_w$ of the Schubert basis element $\xi_w$ in $\Pet$
as follows.

\begin{theorem}\label{T:LL}\cite{Pet} \cite{Lam:2008}
For each $w \in W_\af^0$, there is a unique element $j_w\in \Pet$ of the
form
\begin{equation}\label{E:jform}
j_w = A_w + \sum_{v \in W_\af \setminus W_\af^0}  j^v_w A_v
\end{equation}
for some $j^v_w\in S$. Furthermore, $j_w = j(\xi_w)$ and $\Pet =
\bigoplus_{w\in W_\af^0} S \, j_w$.
\end{theorem}
\begin{proof} Since $\{\xi_w \mid w \in W_\af^0 \}$ is an
$S$-basis of $\Hom_S(\La'_\Gr,S)$, by Theorem \ref{T:jmap}, setting $j_w =
j(\xi_w)$ we obtain an $S$-basis of $\Pet$. Let $v,w\in W_\af^0$. We have
\begin{align*}
     \ip{j_w}{\xib^v}
  &= \ip{j(\xi_w)}{\xib^v} \\
  &= \ip{\xi_w}{\varpi(\xib^v)} \\
  &= \ip{\xi_w}{\xib^v} \\
  &= \delta_{vw}
\end{align*}
by \eqref{E:j}, Proposition \ref{P:wrongway}, and \eqref{E:xihomdef}.
On the other hand $\ip{j_w}{\xib^v}$ is the coefficient of $A_v$
in $j_w\in \Pet\subset\A'_\af$. The form \eqref{E:jform} follows.
The element $j_w \in \Pet$ is unique because the set $\{A_w
\mid w \in W_\af^0 \}$ is linearly independent.
\end{proof}

For $w\in W_\af^0$ let $t^w = t_\la$ where $\la\in Q^\vee$ is such that $t^w W = t_\la W$.
Since $\{j_v\mid v\in W_\af^0\}$ is an $S$-basis of $\Pet$ and $t^w\in \Pet$ for all $w\in W_\af^0$,
we have
\begin{equation}
\label{E:tj}
\begin{split}
  t^w &= \sum_{v\in W_\af^0} \xib^v(t^w) j_v \\
  &= \sum_{v\in W_\af^0} \xib^v(w) j_v
\end{split}
\end{equation}
by the definition of $j_v$, \eqref{E:dmatrix}, Lemma \ref{L:xid}, and 
Proposition \ref{P:partialflagbasis}.

Let $D$ be the $W_\af^0\times W_\af^0$ matrix with $D_{v,w} = \xib^v(w)$.
It is invertible over $\Fr$ by \eqref{E:xisupport} and \eqref{E:xidiag}.
Let $C = D^{-1}$. 

\begin{prop}\label{P:jcoefs} For every $u\in W_\af^0$ and $x\in W_\af$,
\begin{align}\label{E:jcoefs}
  j_u^x = \sum_{w\in W_\af^0} C_{wu} \xib^x(t^w).
\end{align}
\end{prop}

\begin{remark}\label{R:jformula}
Note that by Propositions \ref{P:jcoefs} and \ref{P:Homconst} below,
the Schubert structure constants $d_{uv}^w$ for $H_T(\Gr_G)$ 
may be computed using only localizations of Schubert classes.
\end{remark}

\begin{proof} Multiplying \eqref{E:tj} by $C_{wu}$ and summing over $w\in W_\af^0$ we have
\begin{align*}
  \sum_{w\in W_\af^0} C_{wu} t^w &= \sum_{v,w\in W_\af^0} D_{vw} C_{wu} j_v \\
  &= \sum_{v\in W_\af^0} \delta_{vu} j_v \\
  &= j_u.
\end{align*}
Now take the coefficient of $A_x$.
\end{proof}

\begin{prop} \label{P:transexp} \cite{Pet}
 For any antidominant $\la\in Q^\vee$, $t_\la\in W_\af^0$ and
\begin{align}
  j_{t_\la} = \sum_{\mu\in W\la} A_{t_\mu}.
\end{align}
\end{prop}

\begin{cor} \label{C:Schubtrans} For any antidominant $\la\in Q^\vee$
and $x\in W_\af^0$, we have $xt_\la\in W_\af^0$ and
\begin{align}
  j_{xt_\la} = j_x j_{t_\la}.
\end{align}
\end{cor}
\begin{proof} Follows from Proposition \ref{P:transexp} and
Theorem \ref{T:LL}.
\end{proof}

\subsection{Homology structure constants}
Recall that $\Hom_S(\La'_\Gr,S)\cong H_T(\Gr)$.

For $u,v,w\in W_\af^0$ define the $H_T(\Gr)$ Schubert structure
constants $d^w_{uv}\in S$ by
\begin{align}
  j_u j_v = \sum_w d^w_{uv} j_w.
\end{align}

The $H_T(\Gr)$-structure constants all appear in the
expansion of the basis $\{j_v\mid v\in W_\af^0\}$
in terms of $A_w$. The following is due to Peterson \cite{Pet}.

\begin{prop} \label{P:Homconst} For $u,v,w\in W_\af^0$,
\begin{align}\label{E:jHomconst}
  d^w_{uv} = \begin{cases}
    j_u^{wv^{-1}} & \text{if $\ell(wv^{-1})+\ell(v)=\ell(w)$} \\
    0 & \text{otherwise.}
  \end{cases}
\end{align}
\end{prop}
\begin{proof} Using the notation of \eqref{E:jform}, $d_{uv}^w$
is the coefficient of $A_w$ in the expansion of $j_uj_v$
when the latter is written as a left $S$-linear combination of $\{A_z\mid z\in W_\af\}$.
Every product of the form $j^x_u A_x j^y_v A_y$ for $y\not\in W_\af^0$,
is in the left $S$-module $\A A_y$ for $z\in W_\af$,
and by \eqref{E:Aadd} $S A_w \cap \A A_y = 0$. Therefore $d^w_{uv}$ is the coefficient
of $A_w$ in $j_u A_v$, and applying \eqref{E:Aadd} yields \eqref{E:jHomconst}.
\end{proof}

\subsection{Peterson's ``Quantum Equals Affine" Theorems}
\label{SS:quantumequalsaffine}

The set $\mathcal{T}$ of the elements $\xi_{t_\la}$ for $\la\in Q^\vee$ antidominant,
is multiplicatively closed since
\begin{align}
  \xi_{x t_\la} = \xi_x \xi_{t_\la}
\end{align}
for all such $\la$ and $x\in W_\af^-$. This follows from Corollary \ref{C:Schubtrans}
and Theorem \ref{T:LL}. For any $w\in W$ there is an antidominant $\la\in Q^\vee$ such that
$wt_\la\in W_\af^-$. 

Let $QH^T(G/B)$ be the $T$-equivariant quantum cohomology ring of $G/B$.
Linearly it is isomorphic to $\Z[q_i\mid i\in I] \otimes_\Z H^T(G/B)$.
For $\alpha^\vee = \sum_{i\in I} c_i \alpha_i^\vee\in Q^\vee$
define $q_{\alpha^\vee} = \prod_{i\in I} q_i^{c_i} \in \Z[q_i^{\pm}\mid i\in I]$.
For $w\in W$ let $\sigma^w$ be the quantum Schubert class
defined by the $w$-th opposite Schubert variety $\overline{B_-wB/B}$.

\begin{theorem} \cite{Pet} \cite{LS:QH} Let $QH^T(G/B)_{(q)}$ be the localization of $QH^T(G/B)$
at the quantum parameters. Denote by $H_T(\Gr_G)_{\mathcal{T}}$ the localization of $H_T(\Gr_G)$
at the multiplicatively closed set $\mathcal{T}$. Then there is an $S$-algebra isomorphism
$QH^T(G/B)_{(q)} \to H_T(\Gr_G)_{\mathcal{T}}$ defined as follows.
Given $\mu\in Q^\vee$ there is an antidominant $\la\in Q^\vee$ such that $wt_{\mu+\la}\in W_\af^-$.
Then $q_\mu \sigma^w \mapsto \xi_{t_\la}^{-1} \xi_{w t_{\mu+\la}}$.
\end{theorem}

\begin{remark}
Since the Schubert bases of $QH^T(G/B)$ and $H_T(\Gr_G)$ correspond,
it follows that the Schubert structure constants for these rings agree. By Proposition 
\ref{P:Homconst}, the equivariant Gromov-Witten invariants all have the form $j_v^x$.
\end{remark}

There is also a parabolic version which covers all homogeneous spaces $G/P$.

\begin{theorem} \label{T:QH} \cite{Pet} \cite{LS:QH}
For any parabolic subgroup $P\subset G$, there is an ideal $J_P\subset H_T(\Gr)$
and a multiplicatively closed subset $\mathcal{T}_P$ of $H_T(\Gr)/J_P$, such
that there is a ring isomorphism
\begin{align}
  QH^T(G/P)_{(q)} \cong (H_T(\Gr)/J_P)_{\mathcal{T}_P}
\end{align}
where $QH^T(G/P)$ is the $T$-equivariant small quantum cohomology ring of $G/P$
and $(q)$ is the set of quantum parameters.

Moreover, Schubert classes correspond to Schubert classes. In particular
every $T$-equivariant Gromov-Witten invariant for any $G/P$,
occurs as a Schubert structure constant for $H_T(\Gr)$,
and vice versa for $P=B$.
\end{theorem}

\appendix

\section{Appendix: Proof of coalgebra properties}
\label{A:coalg}

Let $M$ and $N$ be left $\AQ$-modules. Then $M$, $N$, $M \otimes_\Fr N$,
and $\Hom_\Fr(M,N)$ are left $\Fr$-modules.
We define an $\AQ$-module structure on $M \otimes_\Fr N$ and $\Hom_\Fr(M,N)$.

\begin{proof}[Proof of Proposition \ref{P:welldefined}]
We first check the well-definedness of the formula for $a\cdot m \otimes n$.
Let $a\in \AQ$ with $\Delta(a) = \sum_{(a)} a_{(1)} \otimes a_{(2)}$.
We further expand $a_{(k)} = \sum_w a_{(k)}^w w$ for $k=1,2$ where $a_{(k)}^w\in \Fr$.
The condition of membership in $\Image(\Delta)$ of the right hand side of \eqref{E:ontensor},
is that only terms of the form $v\otimes v$ survive:
\begin{align}
  \sum_a a_{(1)}^v a_{(2)}^w = 0  \qquad\text{for all $v\ne w$.}
\end{align}
We take a typical generator of the relations in $M \otimes_Q N$:
$qm \otimes n - m\otimes qn$. We compute the componentwise action of $\Delta(a)$ on
$qm \otimes n$ and $m \otimes qn$.
\begin{align*}
  \Delta(a) \cdot (qm \otimes n)
  &= \sum_a a_{(1)} \cdot qm \otimes a_{(2)} \cdot n \\
  &= \sum_{a,v,w} a_{(1)}^v a_{(2)}^w v \cdot qm \otimes w \cdot n \\
  &= \sum_{a,v,w} a_{(1)}^v a_{(2)}^w (v\cdot q)(v\cdot m) \otimes w\cdot n \\
  &= \sum_{a,v,w} a_{(1)}^v a_{(2)}^w (v\cdot q) ( v\cdot m \otimes w\cdot n).
\end{align*}
Similarly
\begin{align*}
  \Delta(a) \cdot (m \otimes qn)
  &= \sum_{a,v,w} a_{(1)}^v a_{(2)}^w (w\cdot q)( v\cdot m \otimes w\cdot n).
\end{align*}
The difference of these two expressions is
\begin{align*}
  &\quad\,\, \sum_a \sum_{v\ne w} a_{(1)}^v a_{(2)}^w (v\cdot q - w\cdot q) (v\cdot m \otimes w\cdot n) \\
  &= \sum_{v\ne w} (v\cdot q-w\cdot q)(v\cdot m \otimes w\cdot n) \sum_a a_{(1)}^v a_{(2)}^w \\
  &= 0.
\end{align*}
Thus the formula for $a\cdot (m\otimes n)$ is well-defined.

Applying this to the special case of the action of $\AQ$ on $\AQ \otimes_\Fr \AQ$,
we recover part (1), including \eqref{E:compwise}. For $a,b\in \AQ$ we have
\begin{align*}
  a\cdot (b \cdot (m\otimes n)) &= \sum_{(a)} \sum_{(b)} a_{(1)} \cdot (b_{(1)} \cdot m)\otimes a_{(2)}\cdot (b_{(2)}\cdot n) \\
  &= \sum_{(a)} \sum_{(b)} (a_{(1)}b_{(1)})\cdot m \otimes (a_{(2)} b_{(2)}) \cdot n \\
  &= (ab) \cdot (m \otimes n)
\end{align*}
where the last step holds because of \eqref{E:compwise}.
Hence we have an action of $\AQ$ on $M\otimes_\Fr N$.

For the left $\AQ$-module $M$, the dual $M^* = \Hom_{\Fr}(M,\Fr)$
has a left $\AQ$-module structure defined by
$w \cdot m^* = m^* \circ w^{-1}$ or more generally $a \cdot m^* = m^* \circ a^t$
for $w\in W$ and $a\in \AQ$. Consider the left $\Fr$-linear isomorphism $M^* \otimes_\Fr N \cong \Hom_\Fr(M,N)$
given by $m^* \otimes n \mapsto (x\mapsto m^*(x)n)$. We define a left $\AQ$-module structure on $\Hom_\Fr(M,N)$
by declaring that the above map is an isomorphism of left $\AQ$-modules. It is enough to consider the action
of $a=w$: $w \cdot m^* \otimes n= w \cdot m^* \otimes w \cdot n = m^*\circ w^{-1} \otimes w\cdot n$.
This corresponds to the function $x\mapsto m^*(w^{-1}x) w\cdot n = w\cdot m^*(w^{-1}x) n$.
If $f\in \Hom_\Fr(M,N)$ corresponds to $m^* \otimes n$ then the above function corresponds to
$x\mapsto w \cdot f(w^{-1} x)$, which is the required action.
\end{proof}

\begin{proof}[Proof of Proposition \ref{P:coalg}]
We first compute
\begin{align*}
  \Delta(A_i) 
 &= \Delta(\al_i^{-1}(1-s_i)) \\
 &= \al_i^{-1} (1\otimes 1-s_i \otimes s_i) \\
 &= \al_i^{-1} (1 \otimes 1 - s_i \otimes 1 + s_i\otimes 1 - s_i\otimes s_i) \\
 &= A_i \otimes 1 + s_i \otimes A_i \\
 &= \al_i^{-1}(1 \otimes 1 - 1 \otimes s + 1\otimes s_i - s_i\otimes s_i) \\
 &= 1 \otimes A_i + A_i \otimes s_i.
\end{align*}
This yields \eqref{E:DeltaA}. It follows that the restriction of $\Delta$ to $\A$ has image in $\A \otimes_S \A$
and inherits the required properties by Proposition \ref{P:welldefined}. All other assertions follow directly.
\end{proof}

\section{Appendix: Small torus GKM proofs}
\label{A:smalltorusGKM}

\begin{proof}[Proof of Theorem \ref{T:smallGKM}] We prove (1) as (2) follows from it.
There is a commutative diagram of ring homomorphisms
\begin{align}
\begin{diagram}
\node{H^{T_\af}(\tFl_\af)} \arrow{e,t}{\res} \arrow{s,t}{\text{For}^{T_\af}_{T}} \node{\La_\af} \arrow{s,b}{\pi^*} \\
\node{H^T(\tFl_\af)} \arrow{e,b}{\res'} \node{\Fun(W_\af,S)}
\end{diagram}
\end{align}
The horizontal maps are restriction to torus-fixed points. $\text{For}^{T_\af}_T$ is the map that forgets from
$T_\af$-equivariance to $T$-equivariance. The top map is an isomorphism by Theorem \ref{T:Fl}.
It is not hard to show that $\text{For}$ is surjective and that $H^T(\tFl_\af)$ has an $H^T(\pnt)$-basis
given by the $T$-equivariant classes of Schubert varieties $[X^v]^T\in H^T(\tFl_\af)$ for $v\in W_\af$.
By commutativity of the diagram,
\begin{align}\label{E:imageres'}
  \Image(\res')=\Image(\pi^*) = \bigoplus_{v\in W_\af} S \xib^v.
\end{align}
The functions $\xib^v$ are independent since the matrix $(\pi(d_{v,w}))_{v,w\in W_\af}$
is triangular with nonvanishing diagonal entries.

It remains to show that
\begin{align}\label{E:Lap}
\Image(\pi^*)=\La'_\af.
\end{align}
Let $v\in W_\af$. Certainly $\xib^v$ satisfies \eqref{E:GKM} since $\xi^v$ does.
We must check the conditions \eqref{E:smallGKMt} and \eqref{E:smallGKMs}.
Let $w\in W_\af$, $\al\in\Phi$, and $d\in \Z_{>0}$. Let $W'\subset W_\af$ be
the subgroup generated by $t_{\al^\vee}$ and $r_\al$; it is
isomorphic to the affine Weyl group of $SL_2$. Define the function
$f:W'\to S$ by $f(x)=\xib^v(x w)$. Since $\xi^v$ satisfies \eqref{E:GKM}
for $\Fl_\af$, $f$ satisfies \eqref{E:GKM} for the affine flag variety $\Fl'$
corresponding to $\al$. It follows that $f$ is an $S$-linear combination
of Schubert classes in $\Fl'$. By Propositions \ref{P:sl2smallGKMt} and \ref{P:sl2smallGKMs}
(proved below), $\pi \circ f$ satisfies \eqref{E:smallGKMt} and \eqref{E:smallGKMs},
so that $\xib^v\in \La'_\af$, as required.

Conversely, suppose $\xi \in \La'_\af$. We show that
\begin{align}\label{E:xiinxivs}
\xi\in \bigoplus_{v\in W_\af} S \,\xib^v.
\end{align}
Let $x=t_\la u$ be of minimal length in the support of $\xi$, with $u\in W$
and $\la\in Q^\vee$. It suffices to show that
\begin{equation}\label{E:xidivdiag}
    \xi(x) \in \xib^x(x) S.
\end{equation}
Suppose \eqref{E:xidivdiag} holds.
Define $\xi':W_\af\to S$ by $\xi'=\xi-(\xi(x)/\xib^x(x))\xib^x$.
Since $\La'_\af$ is an $S$-module, $\xi'\in \La'_\af$.
Moreover $\Omega(\xi')\supsetneq\Omega(\xi)$ where $\Omega(\xi)$ is defined
by \eqref{E:Omegadef}. By induction \eqref{E:xiinxivs} holds for $\xi'$ and therefore it
holds for $\xi$.

We now show \eqref{E:xidivdiag}.
The elements $\{\al \mid \al\in \Phi^+\}$ are relatively prime in
$S$. Letting $\al\in\Phi^+$,
by \eqref{E:xidiag} it suffices to show that $\xi(x)\in J := \al^d S$
where $d=|\Inv_\al(x^{-1})|$ and $\Inv_\al(x^{-1})$ is the set of roots in $\Inv(x^{-1})$
(see \eqref{E:Inv}) of the form $\pm \al + k \delta$
for some $k\in\Z_{\ge0}$.

Note that for $\beta\in\Phi_\af^{+\re}$, $\beta\in\Inv(x^{-1})$ if and only if
$x^{-1}\cdot \beta\in - \Phi_\af^{+\re}$. We have
\begin{align*}
  x^{-1} \cdot (\pm \al + k \delta) &= u^{-1} t_{-\la} \cdot (\pm \al + k\delta) = \pm u^{-1} \al
  + (k \pm \ip{\la}{\al}) \delta.
\end{align*}
Letting $\chi=\chi(\al\in\Inv(u^{-1}))$ we have
\begin{align}
\label{E:alphainv}
  \Inv_\al(x^{-1}) = \begin{cases}
\{\al,\al+\delta,\dotsc,\al-(\ip{\la}{\al} +1-\chi)\delta \} &\text{if $\ip{\la}{\al} \le 0$} \\
\{-\al+\delta,-\al+2\delta,\dotsc,-\al+(\ip{\la}{\al}-\chi) \delta \} &\text{if $\ip{\la}{\al} > 0$.}
\end{cases}
\end{align}
Suppose first that $\ip{\la}{\al}>0$.  Then
$d=\ip{\la}{\al}-\chi(\al\in\Inv(u^{-1}))$.  Applying~\eqref{E:smallGKMs} to $y=t_{(1-d)\al^\vee}x$,
we have $Z_1\in J$ where
\begin{align*}
  Z_1&=\xi((1-t_{\al^\vee})^{d-1}(1-s_\al) y) \\
  &=(-1)^{d-1} \xi((1-t_{-\al^\vee})^{d-1}x) - \xi((1-t_{\al^\vee})^{d-1} s_\al y) \\
  &= (-1)^{d-1} \xi(x) - \xi((1-t_{\al^\vee})^{d-1} s_\al y)
\end{align*}
where the last equality holds by the assumption on $\Supp(\xi)$
and a calculation of $\Inv_\al((s_{\alpha+d\delta} x)^{-1})$, giving $x > s_{\alpha+d\delta} x > t_{-k\al^\vee}x$
for all $1\le k \le d-1$. By Lemma~\ref{L:smallGKM} we have $Z_2\in J$ where
\begin{align*}
Z_2 = \xi((1-t_{\al^\vee})^{d-1} s_\al y) - \xi((1-t_{\al^\vee})^{d-1} t_{p\al^\vee} s_\al y)
\end{align*}
for any $p\in\Z$. Thus
\begin{equation*}
 Z_1+Z_2 = (-1)^{d-1} \xi(x) -  \xi((1-t_{\al^\vee})^{d-1} t_{p\al^\vee}s_\al y) \in J.
\end{equation*}
By the assumption on $\Supp(x)$ and \eqref{E:alphainv},
\begin{align*}
  \xi((1-t_{\al^\vee})^{d-1} t_{p\al^\vee} s_\al y) = 0
\end{align*}
for $p=2-d$. It follows that $\xi(x)\in J$.

Otherwise $\ip{\la}{\al}\le 0$. By the previous case we may assume
that $t_{d\al^\vee}x\not\in\Supp(\xi)$.  Thus
\begin{align*}
  \xi(x) = \xi((1-t_{\al^\vee})^d x) \in J
\end{align*}
by induction on $\Supp(\xi)$ and \eqref{E:smallGKMt}.
This proves \eqref{E:xidivdiag} and \eqref{E:xiinxivs} as required.
\end{proof}

\subsection{Small torus GKM condition for $\slh_2$}
\label{SS:smallGKMsl2}
In this section we consider the root datum for $\slh_2$.
Let $\Phi^+=\{\al\}$ where $\al=\al_1$. Also $\delta=\al_0+\al_1$
so that $\pi(\al_0)=-\pi(\al_1)$ and the level zero action of $W_\af$
is given by $s_0\cdot\al=s_1\cdot\al=-\al$. For $i\in \Z_{\ge 0}$ let
\begin{equation} \label{E:sl2gens}
\begin{aligned}
\sigma_{2i} &= (s_1s_0)^i &\qquad \sigma_{-2i}&=(s_0s_1)^i,\\
\sigma_{2i+1}&=s_0\sigma_{2i} & \sigma_{-(2i+1)}&=s_1 \sigma_{-2i}.
\end{aligned}
\end{equation}
Then we have $\ell(\sigma_j) = |j|$ for $j\in\Z$, $W_\af^I=\{\sigma_j\mid j\in\Z_{\ge0}\}$, and
\begin{equation*}
  \sigma_{2i} = t_{-i\al^\vee}\qquad\text{for $i\in\Z$.}
\end{equation*}

Let $\xi^i_j:=\pi(\xi^{\sigma_i}(\sigma_j))$ for $i,j\in\Z$.
For $m,a\in\Z_{\ge0}$,
\begin{align}
\label{E:xisl2}
\xi^m_j &=  (-1)^{mj} \binom{m+a}{m} \al^m
\end{align}
where $j\in \{m+2a,m+2a+1,-m-2a-1,-m-2a-2\}$.
This is easily proved by induction using \eqref{E:xirec}.
The rest of the values for $\xi^m$ are zero by \eqref{E:xisupport}.
For $m<0$ we may use
\begin{equation}\label{E:sl2xineg}
    \xi^m_j = (-1)^m \xi^{-m}_{-j} \qquad\text{for $m,j\in\Z$}
\end{equation}
which follows from the Dynkin symmetry $0\leftrightarrow 1$.

\begin{prop} \label{P:sl2smallGKMt}
For all $d\ge 1$, $m\in \Z$, and $w\in W_\af$ we have
\begin{equation}\label{E:sl2smallGKMt}
\xi^m((1-t_{\alpha^\vee})^d w) \in \al^d \Z[\al].
\end{equation}
\end{prop}
\begin{proof} One may reduce to $m\ge0$ using \eqref{E:sl2xineg}.
Equation \eqref{E:sl2smallGKMt} is proved for
$m=2i$, $w=t_{(-i-a)\alpha^\vee}$ and $d=(i+a)+(i+1+b)=2i+a+b+1$ for $a,b\in\Z_{\ge0}$,
as the other possibilities are similar or easier. Equation \eqref{E:sl2smallGKMt} can be rewritten as
\begin{equation}\label{E:rewrite}
  Z := \sum_{k=0}^d (-1)^k \binom{d}{k} \xi^{2i}_{-2i-2-2b+2k} \in \al^d \Z[\al].
\end{equation}
By \eqref{E:xisupport} $\xi^{2i}_{2p}=0$ for $-2i-2<2p<2i$. Using \eqref{E:xisl2},
\begin{align*}
  Z &= \left(\sum_{k=0}^b + \sum_{k=2i+1+b}^d\right) (-1)^k \binom{d}{k} \xi^{2i}_{-2i-2-2b+2k} \\
  &=\sum_{k=0}^b (-1)^k \binom{d}{k} \xi^{2i}_{-2i-2-2b+2k}
   + \sum_{k=0}^a (-1)^{k+2i+1+b} \binom{d}{2i+1+b+k} \xi^{2i}_{2i+2k} \\
  &=(-1)^b \sum_{k=0}^b (-1)^k \binom{d}{b-k} \xi^{2i}_{-2i-2-2k}
   - (-1)^b \sum_{k=0}^a (-1)^k \binom{d}{a-k} \xi^{2i}_{2i+2k} \\
   &= (-1)^b \sum_{k=0}^b (-1)^k \binom{d}{b-k} \binom{2i+k}{2i} \al^{2i}
  - (-1)^b \sum_{k=0}^a (-1)^k \binom{d}{a-k} \binom{2i+k}{2i} \al^{2i}.
\end{align*}
Taking the coefficient of  $(-x)^b$ in $(1-x)^d/(1-x)^{2i+1}=(1-x)^{a+b}$ we have
\begin{align*}
  \sum_{k=0}^b (-1)^k \binom{d}{b-k} \binom{2i+k}{2i} = \binom{a+b}{b}.
\end{align*}
Exchanging the roles of $a$ and $b$ we see that $Z=0$ which implies \eqref{E:rewrite}.
\end{proof}

\begin{prop} \label{P:sl2smallGKMs}
For all $d\ge 1$, $m\in \Z$, and $w\in W_\af$ we have
\begin{equation*}
    \xi^m((1-t_{\alpha^\vee})^{d-1}(1-s_\alpha) w) \in \al^d \Z[\al].
\end{equation*}
\end{prop}

\begin{proof}
Let $y=(1-t_{\alpha^\vee})^{d-1}$. Without loss of generality
we may assume $w=t_{p\al^\vee}$ for some $p\in\Z$.
By Proposition~\ref{P:sl2smallGKMt}, $\xi^m$ satisfies \eqref{E:smallGKMt}.
We have
\begin{align*}
    \xi^m(y(1-s_\al) t_{p\al^\vee}) &= \xi^m(y t_{p\al^\vee}) - \xi^m(y t_{-p\al^\vee}s_\al) \\
    &\equiv \xi^m(y) - \xi^m(y s_\al) \mod \al^d \Z[\al],
\end{align*}
using Lemma \ref{L:smallGKM} twice. However
\begin{align*}
  \xi^m(y)-\xi^m(ys_\al) &= \xi^m(y \al A_\al) \\
  &= (y\cdot\al) (A_\al\bullet \xi^m)(y).
\end{align*}
Now $y\cdot \al=\pm\al$ and $A_\al\bullet \xi^m$ is $0$ if $m\ge0$
and is equal to $\xi^{m+1}$ if $m<0$. Assuming the latter,
by \eqref{E:smallGKMt} for $d-1$, $\xi^{m+1}(y)\in \al^{d-1}S$,
so that $\xi^m(y)-\xi^m(ys_\al)\in \al^d S$ as required.
\end{proof}

\section{Appendix: Homology of $\Gr$}
\label{A:HomGr}

Let $K$ be the maximal compact form of $G$ and $T_\R=K\cap T$.
The based loop group $\Omega K$ of continuous maps $(S^1,1)\to(K,1)$,
has a $T_\R$-equivariant multiplication map
$\Omega K \times \Omega K \to \Omega K$ given by
pointwise multiplication on $K$, and this induces the structure of a
commutative and co-commutative Hopf-algebra on $H^{T_\R}(\Omega K)$.
The co-commutativity of $H^{T_\R}(\Omega K)$
follows from the fact that $\Omega K$ is a homotopy double-loop space.

$\Gr=\Gr_G$ and $\Omega K$ are weakly homotopy equivalent \cite{Q}.
By ``fattening loops", it follows that $H_T(\Gr)$ and $H^T(\tGr)$ are
dual Hopf algebras over $S=H^T(\pnt)$ with duality
given by an intersection pairing
\begin{align}
  \ip{\cdot}{\cdot}: H_T(\Gr) \times H^T(\tGr) \to S.
\end{align}
We may regard an element $\xi\in H_T(\Gr)$ as
an $S$-module homomorphism $H^T(\tGr) \to S$ by
$\xi(f) = \ip{\xi}{f}$.

\begin{lem}\label{L:Hmult}
Let $\la,\mu \in Q^\vee$, and consider the maps $i^*_\la,
i^*_\mu: H^T(\tGr) \to S$ as elements of $H_T(\Gr)$.  Then in
$H_T(\Gr)$, we have
$$
i^*_\la \; i^*_\mu = i^*_{\la + \mu}.
$$
\end{lem}
\begin{proof}
It suffices to work in $H^{T_\R}(\Omega K)$.  The map $i^*_\la
\; i^*_\mu$ is induced by the map ${\rm pt} \to \Omega K \times
\Omega K \to \Omega K$ where the image of the first map is the pair
$(t_\la,t_\mu) \in \Omega K \times \Omega K$ of fixed points,
and the latter map is multiplication.  Treating $t_\la, t_\mu:S^1
\to K$ as homomorphisms into $K$, the pointwise
multiplication of $t_\la$ and $t_\mu$ gives $t_{\la + \mu}$. Thus
$i^*_\la \; i^*_\mu = i^*_{\la + \mu}$.
\end{proof}

The antipode of $H_T(\Gr_G)$ is given by $i^*_\la\mapsto
i^*_{-\la}$, since the fixed points satisfy $t^{-1}_\la =
t_{-\la}$ in $\Omega K$.

For $w\in W_\af^0$ denote by $[X_w]_T\in H_T(\Gr)$ the
equivariant fundamental homology class of the Schubert variety
$X_w:=\overline{B_\af \dot{w} P_\af/P_\af} \subset G_\af/P_\af= \Gr$.

Using the intersection pairing we have that $\{[X_w]\mid w\in W_\af^0\}$
is the basis of $H_T(\Gr_G)$ that is dual to the basis $\{[X^w]^T\mid w\in W_\af^0\}$
of $H^T(\tGr)$, which corresponds to the basis $\{\xib^w\mid w\in W_\af^0\}$ of $\La'_\Fl$ under
the isomorphism of Theorem \ref{T:smallGKM}.

\begin{proof}[Proof of Prop. \ref{P:convolution}]
Follows from the above discussion.
\end{proof}

%



\begin{thebibliography}{99}

\bibitem{[Ag]} 
S. Agnihotri, 
\textit{Quantum cohomology and the Verlinde algebra}, 
Ph.D. Thesis, University of Oxford, 1995.

\bibitem{AJS} 
H.~H.~Andersen, J.~C.~Jantzen, W.~Soergel,
\textit{Representations of quantum groups at a $p$th root of unity and of
semisimple groups in characteristic $p$: independence of $p$},
Ast\'erisque No. 220 (1994), 321 pp.

\bibitem{Ar} 
A. Arabia,
\textit{Cycles de Schubert et cohomologie \'equivariante de $K/T$},
Invent. Math. 85 (1986), no. 1, 39--52.

\bibitem{Assaf:2008} 
S. Assaf, 
\textit{A combinatorial realization of Schur-Weyl duality via crystal graphs and dual equivalence graphs},
20th Annual International Conference on Formal Power Series and Algebraic Combinatorics (FPSAC 2008), pp. 141--152, 
Discr. Math. Theor. Comput. Sci. Proc., AJ, Assoc. Discrete Math. Theor. Comput. Sci., Nancy, France, 2008.

\bibitem{Assaf:2010}
S.~Assaf,
\textit{Dual Equivalence Graphs I: A combinatorial proof of LLT and Macdonald positivity},
preprint {\tt arXiv:1005.3759}.

\bibitem{AssafBilley} 
S. Assaf, S. Billey, 
\textit{Affine dual equivalence and $k$-Schur functions},
J. Comb. {\bf 3} (2012), no. 3, 343--399.

\bibitem{BandlowSchillingZabrocki} 
J. Bandlow, A. Schilling, M. Zabrocki,
\textit{The Murnaghan--Nakayama rule for $k$-Schur functions}, 
J. of Comb. Th., Ser. A Volume 118, Issue 5, June 2011, 1588--1607.

\bibitem{[BMW]} 
L. B\'egin, P. Mathieu, M.A. Walton,
\textit{${\hat su}(3)_k$ fusion coefficients}, 
Mod. Phys. Lett. A {\bf 7} (1995), no. 35, 3255--3265.

\bibitem{[BKMW]} 
L. B\'egin, A. Kirillov, P. Mathieu, M.A. Walton,
\textit{Berenstein-Zelevinsky triangles, elementary couplings and fusion rules}, 
Lett. Math. Phys. 28 (1993), 257--268.

\bibitem{BK} 
E.~Bender, D.E.~Knuth, 
\textit{Enumeration of plane partitions}, 
J. Comb. Theory. Ser. A. 13 (1972) 40--54.

\bibitem{BBTZ:2011}
C.~{Berg}, N.~{Bergeron}, H.~{Thomas}, M.~{Zabrocki},
\textit {Expansion of $k$-Schur functions for maximal $k$-rectangles within the affine nilCoxeter algebra},
J. Comb. {\bf 3} (2012), no. 3, 563--589.

\bibitem{BJV:2009} 
C.~Berg, B.~C.~Jones, M.~Vazirani, 
\textit{A bijection on core partitions and a parabolic quotient of the affine symmetric group},
J. Combin. Theory Ser. A 116 (8) (2009) 1344--1360.

\bibitem{BSS:2012a}
C.~Berg, F.~Saliola, L.~Serrano,
\textit{Pieri operators on the affine nilCoxeter algebra},
preprint {\tt arXiv:1203.4465}.

\bibitem{BSS:2012}
C.~Berg, F.~Saliola, L.~Serrano,
\textit{Combinatorial expansions for families of non-commutative k-Schur functions},
preprint {\tt arXiv:1208.4857}.

\bibitem{BS} 
N.~Bergeron, F.~Sottile, 
\textit{Skew Schubert functions and the Pieri formula for flag manifolds}, 
Trans. Amer. Math. Soc.  354  (2002),  no. 2, 651--673.

\bibitem{BCF} A. Bertram, I. Ciocan-Fontanine, W. Fulton,
\textit{Quantum multiplication of Schur polynomials},
J. Algebra, 219(2) (1999), 728--746.

\bibitem{Bi}
S.~Billey,
\textit{Kostant Polynomials and the Cohomology Ring for $G/B$},
Duke Math. J. 96 (1999), 205--224.

\bibitem{BJS} 
S.~Billey, W. Jockusch, R.~Stanley, 
\textit{Some combinatorial properties of Schubert polynomials}, 
J. Algebraic Combin. 2 (1993), 345--374. 

\bibitem{BH} 
S.~Billey, M.~Haiman, 
\textit{Schubert polynomials for the classical groups},
J. Amer. Math. Soc.  8  (1995),  no. 2, 443--482.

\bibitem{BL} 
S.~Billey, T.K.~Lam, 
\textit{Vexillary elements in the hyperoctahedral group},
J. Algebraic Combin.  8 (1998),  no. 2, 139--152.

\bibitem{BB} 
A.~Bj\"{o}rner, F.~Brenti, 
\textit{Affine permutations of type $A$}, 
Electronic Journal of Combinatorics \textbf{3}, (1996), Research paper 18.

\bibitem{Bott}
R. Bott, 
\textit{The space of loops on a Lie group},
Michigan Math Journal, {\bf 5} (1958), 35--61.

\bibitem{BravoLapointe:2009}
D.~Bravo, L.~Lapointe,
\textit{A recursion formula for k-Schur functions},
J. Combin. Theory Ser. A 116 (2009), no. 4, 918--935. 

\bibitem{BFP} 
F.~Brenti, S.~Fomin, A.~Postnikov,
\textit{Mixed Bruhat operators and Yang-Baxter equations for Weyl groups},
Internat. Math. Res. Notices 1999, no. 8, 419--441.

\bibitem{Bri} 
J.~Brichard, 
\textit{The Center of the Nilcoxeter and 0-Hecke Algebras}, 
preprint {\tt arXiv:0811.2590}.

\bibitem{Buc2} 
A.~Buch, 
\textit{Stanley symmetric functions and quiver varieties},
J. Algebra 235 (2001),  no. 1, 243--260.

\bibitem{Buc} 
A.~Buch, 
\textit{A Littlewood-Richardson rule for the $K$-theory of Grassmannians}, 
Acta Math. 189 (2002), no. 1, 37--78. 

\bibitem{BKSTY} 
A.~Buch, A.~Kresch, M.~Shimozono, H.~Tamvakis, A.~Yong, 
\textit{Stable Grothendieck polynomials and $K$-theoretic factor sequences}, 
Math. Annalen 340 (2008), 359--382.

\bibitem{BKT:2003}
A.~S.~Buch, A.~Kresch, H.~Tamvakis,
\textit{Gromov-Witten invariants on Grassmannians},
J. Amer. Math. Soc. \textbf{16} (2003), no. 4, 901--915.

\bibitem{CMP}
P.-E. Chaput, L.~Manivel, N.~Perrin,
\textit{Affine symmetries of the equivariant cohomology ring of rational homogeneous spaces}, 
Math. Res. Letters 16 (2009), 7--21.

\bibitem{Che}
I.~Cherednik, 
\textit{Double affine Hecke algebras},
London Mathematical Society Lecture Note Series, 319. 
Cambridge University Press, Cambridge, 2005.

\bibitem{CH:2008}
L.-C.~Chen, M.~Haiman,
\textit{A representation-theoretic model for $k$-atoms},
Talk 1039-05-169 at the AMS meeting in Claremont, McKenna, May 2008.

\bibitem{Coskun:2009}
I.~Coskun,
\textit{A Littlewood-Richardson rule for two-step flag varieties},
Invent. Math. \textbf{176} (2009), no. 2, 325--395. 

\bibitem{[Cu]} 
C.J. Cummins, 
\textit{$su(n)$ and $sp(2n)$ WZE fusion rules},
J. Phys. A {\bf 24} (1991), no. 2, 391--400.

\bibitem{DalalMorse:2012}
A.~J.~Dalal, J.~Morse,
\textit{The ABC's of affine Grassmannians and Hall--Littlewood polynomials},
DMTCS proc. \textbf{AR} (2012) 935--946.

\bibitem{DalalMorse:2013}
A.~J.~Dalal, J.~Morse,
\textit{A $t$-generalization for Schubert representatives of the affine Grassmannian},
preprint (2013).
 
\bibitem{Dem}
M. Demazure,
\textit{D\'esingularization des vari\'et\'es de Schubert},
Annales E.N.S, 6 (1974) 53--88.

\bibitem{Denton:2012} 
T.~Denton, 
\textit{Canonical decompositions of affine permutations, affine codes, and split $k$-Schur functions},
Electron. J. Combin. {\bf 19} (2012), no. 4, Paper 19, 41 pp.

\bibitem{DM:2007} 
F.~Descouens, H.~Morita,
\textit{Factorization formulas for Macdonald polynomials},
Europ. J. of Comb., Volume 29, Issue 2, February 2008, pp. 395--410.

\bibitem{EG} 
P.~Edelman, C.~Greene, 
\textit{Balanced tableaux}, 
Adv. Math. 63 (1987), 42--99.

\bibitem{ErikssonEriksson:1998} 
H. Eriksson, K. Eriksson, 
\textit{Affine Weyl groups as infinite permutations}, 
Electron. J. Combin. 5 (1998) Research Paper 18, 32 pp. (electronic).

\bibitem{Fomin:1995} 
S.~Fomin, 
\textit{Schensted algorithms for dual graded graphs}, 
J. Algebraic Combin., 4 (1995), no. 1, pp. 5--45.

\bibitem{FGP}
S.~Fomin, S.~Gelfand, A.~Postnikov, 
\textit{Quantum Schubert polynomials},
J. Amer. Math. Soc. 10 (1997), no. 3, 565--596. 

\bibitem{FG:1998}
S.~Fomin, C.~Greene,
\textit{Noncommutative Schur functions and their applications},
Discrete Mathematics 193 (1998), 179--200. 
Reprinted in  the Discrete Math Anniversary Volume 306 (2006), 1080--1096. 

\bibitem{FK:K} 
S.~Fomin, A.N.~Kirillov, 
\textit{Grothendieck polynomials and the Yang-Baxter equation},
Proc. 6th Intern. Conf. on Formal Power Series and Algebraic Combinatorics, DIMACS, 1994, 183--190.

\bibitem{FK:C} 
S.~Fomin, A.N.~Kirillov, 
\textit{Combinatorial $B_n$-analogues of Schubert polynomials},
Trans. Amer. Math. Soc.  348  (1996),  no. 9, 3591--3620.

\bibitem{FR} 
S.~Fomin, N.~Reading, 
\textit{Root systems and generalized associahedra},
Lecture notes for the IAS/Park City Graduate Summer School in Geometric Combinatorics.

\bibitem{FS} 
S.~Fomin, R.~Stanley, 
\textit{Schubert polynomials and the nilCoxeter algebra}, 
Advances in Math. 103 (1994), 196--207. 

\bibitem{FultonWoodward:2004}
W.~Fulton, C.~Woodward,
\textit{On the quantum product of Schubert classes},
J. Algebraic Geom. 13 (2004), 641--661.

\bibitem{GH:1996} 
A.~Garsia, M.~Haiman, 
\textit{Some natural bigraded $S_n$-Modules and $q,t$-Kostka coefficients},
Elect. J. of Comb., R24, Volume 3(2), 1996.  60 pages.

\bibitem{GP:1992} 
A.~Garsia, C.~Procesi, 
\textit{On certain graded Sn-modules and the q-Kostka polynomials},
Adv. Math. 94 (1992), no. 1, 82--138.

\bibitem{GR:1996} 
A.~Garsia, J.~Remmel, 
\textit{Plethystic formulas and positivity for q,t-Kostka coefficients},
Mathematical essays in honor of Gian-Carlo Rota (Cambridge, MA, 1996), 245--262, 
Progr. Math., 161, Birkh\"auser Boston, Boston, MA, 1998.

\bibitem{GT:1996} 
A.~Garsia, G.~Tesler, 
\textit{Plethystic formulas for Macdonald q,t-Kostka coefficients}, 
Adv. Math. 123 (1996), no. 2, 144--222.

\bibitem{Gin} 
V.~Ginzburg, 
\textit{Perverse sheaves on a loop group and Langlands' duality}, 
preprint \texttt{math.AG/9511007}.

\bibitem{Gin:1997} 
V. Ginzburg,
\textit{Geometric methods in the representation theory of Hecke algebras and quantum groups.
Notes by Vladimir Baranovsky.},
NATO Adv. Sci. Inst. Ser. C Math. Phys. Sci., 514,
Representation theories and algebraic geometry
(Montreal, PQ, 1997),
127--183, Kluwer Acad. Publ., Dordrecht, 1998.

\bibitem{[GW]} 
F. Goodman, H. Wenzl, 
\textit{Littlewood-Richardson coefficients for Hecke algebras at roots of unity},
Adv of Math, {\bf 82} (1990) 244--265.

\bibitem{GKM:1998}
M.~Goresky, R.~Kottwitz, R.~MacPherson,
\textit{Equivariant cohomology, Koszul duality, and the localization theorem},
Invent. Math.  131  (1998),  no. 1, 25--83.

\bibitem{GKM}
M.~Goresky, R.~Kottwitz, R.~MacPherson,
\textit{Homology of affine Springer fibers in the unramified case},
Duke Math. J. 121 (2004), 509--561.

\bibitem{Gr} 
W. Graham, 
\textit{Positivity in equivariant Schubert calculus},
Duke Math. J. 109 (2001), no. 3, 599--614.

\bibitem{Haiman:1999} 
M.~Haiman, 
\textit{Macdonald polynomials and geometry}, 
New Perspectives in Geometric Combinatorics, MSRI Publications 37 (1999),  207--254.

\bibitem{Haiman:2001} 
M.~Haiman, 
\textit{Hilbert schemes, polygraphs, and the Macdonald positivity conjecture},
J. Amer. Math. Soc. 14 (2001), 941--1006.

\bibitem{Hai} 
M.~Haiman, 
\textit{Dual equivalence with applications}, 
including a conjecture of Proctor, 
Discrete Mathematics 99 (1992), 79--113.

\bibitem{HY}
Z.~Hamaker, B.~Young, 
\textit{Relating Edelman-Greene insertion to the Little map}, 
preprint {\tt arXiv:1210.7119}.

\bibitem{Hum}
J.~E.~Humphreys,
\textit{Reflection groups and Coxeter groups},
Cambridge Studies in Advanced Mathematics, 29.
Cambridge University Press, Cambridge, 1990. xii+204 pp.

\bibitem{IMN} 
T.~Ikeda, L.~Mihalcea, H.~Naruse,
\textit{Double Schubert polynomials for the classical groups}, 
Adv. Math. 226 (2011), no. 1, 840--886.

\bibitem{Jing} 
N.~Jing, 
\textit{Vertex operators and Hall-Littlewood symmetric functions}, 
Adv. Math. 87 (1991) 226--248.

\bibitem{Kac}
V.G.~Kac,
\textit{Infinite-dimensional Lie algebras},
Third edition. Cambridge University Press, Cambridge, 1990. xxii+400 pp. ISBN: 0-521-37215-1.

\bibitem{Kas} 
M. Kashiwara,
\textit{The flag manifold of Kac-Moody Lie algebra},
Algebraic analysis, geometry, and number theory (Baltimore, MD, 1988),
161--190, Johns Hopkins Univ. Press, Baltimore, MD, 1989.

\bibitem{KS}
M.~Kashiwara, M.~Shimozono, 
\textit{Equivariant K-theory of affine flag manifolds and affine Grothendieck polynomials},
Duke Math. J. 148 (2009),  no. 3, 501--538.

\bibitem{KirillovNoumi:1999} 
A.~N.~Kirillov, M.~Noumi, 
\textit{$q$-difference raising operators for Macdonald polynomials and the integrality of 
transition coefficients}. 
Algebraic methods and $q$-special functions (Montr\'eal, QC, 1996), 
227--243, CRM Proc. Lecture Notes, 22, Amer. Math. Soc., Providence, RI, 1999.

\bibitem{Knop:1996} 
F.~Knop, 
\textit{Integrality of two variable Kostka functions}, 
J. Reine Angew. Math. 482 (1997), 177--189.

\bibitem{Knuth} 
D.~Knuth, 
\textit{ Permutations, matrices, and generalized Young tableaux}, 
Pac. J. Math. 34 (1970), 709--727.

\bibitem{KLS}
A.~Knutson, T.~Lam, D.~Speyer, 
\textit{Positroid varieties I: juggling and geometry}, 
preprint {\tt arXiv:0903.3694}.

\bibitem{KnutsonTao:2003}
A.~Knutson, T.~Tao,
\textit{Puzzles and (equivariant) cohomology of Grassmannians},
Duke Math. J. \textbf{119} (2003), no. 2, 221--260.

\bibitem{KK}
B.~Kostant, S.~Kumar,
\textit{The nil Hecke ring and cohomology of $G/P$ for a Kac--Moody group $G$},
Adv. in Math. 62 (1986),  no. 3, 187--237.

\bibitem{KK:K}
B.~Kostant, S.~Kumar,
\textit{T-equivariant K-theory of generalized flag varieties},
J. Differential Geom. 32 (1990), no. 2, 549--603.

\bibitem{Kum}
S.~Kumar, 
\textit{Kac-Moody groups, their flag varieties and representation theory},
Progress in Mathematics, 204. Birkh\"{a}user Boston, Inc., Boston, MA, 2002. xvi+606 pp.

\bibitem{Lam:2006}
T.~Lam,
\textit{Affine Stanley symmetric functions},
Amer. J. Math.  128  (2006),  no. 6, 1553--1586.

\bibitem{Lam:2008} 
T.~Lam, 
\textit{Schubert polynomials for the affine Grassmannian}, 
J. Amer. Math. Soc. 21 (2008), no. 1, 259--281.

\bibitem{Lam:ASP}
T.~Lam, 
\textit{Affine Schubert classes, Schur positivity, and combinatorial Hopf algebras},
Bull. Lond. Math. Soc. 43 (2011), no. 2, 328--334. 

\bibitem{LLMS:2006} 
T.~Lam, L.~Lapointe, J.~Morse, M.~Shimozono,
\textit{Affine insertion and Pieri rules for the affine Grassmannian},
Mem. Amer. Math. Soc. 208 (2010), no. 977, xii+82 pp. ISBN: 978-0-8218-4658-2.

\bibitem{LLMS:2010} 
T.~Lam, L.~Lapointe, J.~Morse, M.~Shimozono,
\textit{$k$-shape poset and branching of $k$-Schur functions}, 
Mem. Amer. Math. Soc., posted October 16, 2012.
ISSN 1947-6221(online) ISSN 0065-9266(print) 
{\tt arXiv:1007.5334}.

\bibitem{LaLiMiSh}
T.~Lam, C.~Li, L.~Mihalcea, M.~Shimozono,
\textit{Quantum $K$-theory of $G/B$ and $K$-homology of affine Grassmannians},
in preparation.

\bibitem{LSS:C}
T.~Lam, A.~Schilling, M.~Shimozono, 
\textit{Schubert Polynomials for the affine Grassmannian of the symplectic group},
Math. Z. 264, (2010), 765--811.

\bibitem{LSS}
T.~Lam, A.~Schilling, M.~Shimozono, 
\textit{$K$-theoretic Schubert calculus of the affine Grassmannian},
Compositio Mathematica 146 Issue 4 (2010) 811--852. 

\bibitem{LS:affineLittle} 
T.~Lam, M.~Shimozono, 
\textit{A Little bijection for affine Stanley symmetric functions},
Seminaire Lotharingien de Combinatoire 54A (2006), B54Ai.

\bibitem{LS:DGG} 
T.~Lam, M.~Shimozono, 
\textit{Dual graded graphs for Kac-Moody algebras},
Algebra and Number Theory 1 (2007), 451--488.

\bibitem{LS:QH} 
T.~Lam, M.~Shimozono, 
\textit{Quantum cohomology of $G/P$ and homology of affine Grassmannian},
Acta. Math. 204 (2010), 49--90.

\bibitem{LS:kdoubleSchur}
T.~Lam, M.~Shimozono,
\textit{k-Double Schur functions and equivariant (co)homology of the affine Grassmannian},
Math. Ann. 356 (2013), no. 4, 1379--1404.

\bibitem{LS:EQToda}
T.~Lam, M.~Shimozono, 
\textit{From double quantum Schubert polynomials to $k$-double Schur functions via the Toda lattice}, 
preprint {\tt arXiv:1109.2193}.

\bibitem{LM:1998} 
L.~Lapointe, J.~Morse, 
\textit{Tableaux statistics for two part Macdonald polynomials}, 
Algebraic combinatorics and quantum groups, 61--84, World Sci. Publ., River Edge, NJ, 2003.

\bibitem{LLM:2003}
L.~Lapointe, A.~Lascoux, J.~Morse,
\textit{Tableau atoms and a new Macdonald positivity conjecture},
Duke Math. J.  116  (2003),  no. 1, 103--146.

\bibitem{LM:2003}
L.~Lapointe, J. Morse, 
\textit{Schur function analogs for a filtration of the symmetric function space}, 
J. Comb. Th. A 101 (2)  (2003) pp. 191--224.

\bibitem{LM2:2003} 
L.~Lapointe, J.~Morse, 
\textit{Schur function identities, their $t$-analogs, and $k$-Schur irreducibility},
Advances in Mathematics 180 (2003), no. 1, 222--247.

\bibitem{LM:2005} 
L.~Lapointe, J.~Morse,
\textit{Tableaux on $k+1$-cores, reduced words for affine permutations, and $k$-Schur expansions}, 
J. Combin. Theory Ser. A 112 (2005), no. 1, 44--81.

\bibitem{LM:2007}
L.~Lapointe, J.~Morse,
\textit{A $k$-tableau characterization of $k$-Schur functions},
Adv. Math.  213 (2007),  no. 1, 183--204.

\bibitem{LM:2008} 
L.~Lapointe, J.~Morse,
\textit{Quantum cohomology and the $k$-Schur basis}, 
Trans. Amer. Math. Soc. 360 (2008), pp. 2021--2040. 

\bibitem{LMW} 
L.~Lapointe, J.~Morse, M.~Wachs, 
\textit{Type $A$-affine Weyl group and the $k$-Schur functions}, 
unpublished.

\bibitem{LapointePinto}
L.~Lapointe, M.-E.~Pinto,
\textit{Charge on tableaux and the poset of $k$-shapes},
preprint (2013), {\tt arXiv:1305.2438}.

\bibitem{LapointeVinet:1996} 
L.~Lapointe, L.~Vinet, 
\textit{A short proof of the integrality of the Macdonald (q,t)-Kostka coefficients,},
Duke Math. J. Volume 91, Number 1 (1998), 205--214.

\bibitem{Lascoux:1990}
A.~Lascoux,
\textit{Anneau de Grothendieck de la variet\'e de drapeaux},
In ``The Grothendieck Festschrift,volume III'' of Progr Math, Boston, Birkh\"auser,
(1990) 1--34.

\bibitem{Las} 
A.~Lascoux,
\textit{Classes de Chern des vari\'et\'es de drapeaux},
C. R. Acad. Sci. Paris S\'er. I Math. 295 (1982), no. 5, 393--398.

\bibitem{Lascoux:2001} 
A.~Lascoux, 
\textit{Ordering the affine symmetric group}, 
in Algebraic Combinatorics and Applications (Gossweinstein, 1999), 219--231, Springer, Berlin (2001).

\bibitem{LLT:1993} 
A.~Lascoux, B.~Leclerc, J.-Y.~Thibon, 
\textit{Fonctions de Hall-Littlewood et Polyn\^omes de Kostka-Foulkes aux racines de l'unit\'e}, 
Comptes Rendus de l'Academie des Sciences, Paris, 316 (1993), pp. 1--6. 

\bibitem{LLT:1995}
A.~Lascoux, B.~Leclerc, J.-Y.~Thibon,
\textit{Crystal graphs and $q$-analogues of weight multiplicities for the root system $A_n$},
Lett. Math. Phys. 35 (1995), no. 4, 359--374.

\bibitem{LLT} 
A.~Lascoux, B.~Leclerc, J.-Y.~Thibon, 
``The plactic monoid'' in Algebraic Combinatorics on Words, M. Lothaire ed., 
Cambridge University Press, 2002. 

\bibitem{LS:1981} 
A.~Lascoux, M.-P.~Sch\"utzenberger, 
\textit{Le monoide plaxique},
Quaderni della Ricerca scientifica 109  (1981), 129--156.

\bibitem{LS:Schub} 
A.~Lascoux, M.-P.~Sch\"{u}tzenberger, 
\textit{Structure de Hopf de l'anneau de cohomologie et de l'anneau 
de Grothendieck d'une vari\'{e}t\'{e} de drapeaux}, 
C. R. Acad. Sci. Paris Ser. I Math. 295 (1982), no. 11, 629--633.

\bibitem{LS:SchubertPoly}
A.~Lascoux, M.-P.~Sch\"{u}tzenberger, 
\textit{Polyn\^omes de Schubert}, 
C. R. Acad. Sci. Paris S\'er. I Math. 294 (1982), no. 13, 447--450. 

\bibitem{LS} 
A.~Lascoux, M.-P.~Sch\"{u}tzenberger, 
\textit{Schubert polynomials and the Littlewood-Richardson rule}, 
Lett. Math. Phys. 10(2-3) (1985), 111--124.

\bibitem{LT:2000}
B.~Leclerc, J.-Y.~Thibon,
\textit{Littlewood-Richardson coefficients and Kazhdan--Lusztig polynomials},
Combinatorial methods in representation theory (Kyoto, 1998), 155--220, 
Adv. Stud. Pure Math., 28, Kinokuniya, Tokyo, 2000.

\bibitem{Leeuwen:1999} 
M.~van~Leeuwen, 
\textit{Edge sequences, ribbon tableaux, and an action of affine permutations}, 
Europ. J. Combinatorics 20 (1999), 179--195.

\bibitem{Lenart:2000}
C.~Lenart, 
\textit{Combinatorial aspects of the $K$-theory of Grassmannians},
Ann. Combin. 4 (2000) 67--82.

\bibitem{LP:2007}
C.~Lenart, A.~Postnikov,
\textit{Affine Weyl groups in K-theory and representation theory},
Int. Math. Res. Not. (2007), no. 12, Art. ID rnm038, 65 pp.

\bibitem{LenartSchilling:2012}
C.~Lenart, A.~Schilling,
\textit{Crystal energy functions via the charge in types $A$ and $C$},
Math. Zeitschrift \textbf{273(1)} (2013), 401--426.

\bibitem{LL} 
N.C.~Leung, C.~Li, 
\textit{Gromov-Witten invariants for G/B and Pontryagin product for $\Omega K$},
preprint {\tt arXiv:0810.4859}.

\bibitem{Lit} 
D.~Little, 
\textit{Combinatorial aspects of the Lascoux-Sch\"{u}tzenberger tree},
Adv. Math. 174  (2003),  no. 2, 236--253.

\bibitem{LR:1934}
D.~E.~Littlewood, A.~R.~Richardson, 
\textit{Group characters and algebra},
Phil. Trans. Royal Soc. A (London) 233 (1934) 99--141.

\bibitem{Mac:1988} 
I.~G.~Macdonald, 
\textit{A new class of symmetric functions},
Publ. I.R.M.A. Strasbourg, Actes $20^{e}$ S\'{e}minar Lotharingien,131--171.

\bibitem{Mac:1995}
I.~G.~Macdonald,
\textit{Symmetric Functions and Hall Polynomials, Second Edition}, 
Oxford University Press, 1995.

\bibitem{McN} 
P.~McNamara, 
\textit{Cylindric skew Schur functions}, 
Adv. in Math. 205 (1) (2006), 275--312.

\bibitem{Meszaros:2011}
K.~M\'esz\'aros, private communication, October 2011.

\bibitem{MisraMiwa:1990} 
K.C. Misra, T. Miwa, 
\textit{Crystal base of the basic representation of $U_q ( {\mathfrak sl}_n )$}, 
Commun. Math. Phys. 134 (1990), 79--88.

\bibitem{Morse} 
J.~Morse,
\textit{Combinatorial aspects of affine $K$-theory},
Adv. Math. 229, No. 5 (2012), pp. 2950--2984.

\bibitem{MS:2012}
J.~Morse, A.~Schilling,
\textit{A combinatorial formula for fusion coefficients},
DMTCS proc AR (2012) 735--744.

\bibitem{MS:2013}
J.~Morse, A.~Schilling,
\textit{Crystal operators and flag Gromov-Witten invariants},
preprint

\bibitem{Mur:1937}
D.~F.~Murnaghan,
\textit{The Characters of the Symmetric Group},
Amer. J. Math. 59 (1937), no. 4, 739--753. 

\bibitem{Nakashima:1993}
T.~Nakashima,
\textit{Crystal base and a generalization of the Littlewood-Richardson rule for the classical Lie algebras},
Comm. Math. Phys. 154 (1993), no. 2, 215--243. 

\bibitem{Nak:1941}
T.~Nakayama, 
\textit{On some modular properties of irreducible representations of symmetric groups. I and II.},
Jap. J. Math.  17  (1941) 411--423 and
Jap. J. Math. 18 (1941) 89--108.

\bibitem{NY:1997}
A.~Nakayashiki, Y.~Yamada,
\textit{Kostka polynomials and energy functions in solvable lattice models},
Selecta Math. (N.S.) 3 (1997), no. 4, 547--599.

\bibitem{Pet}
D.~Peterson,
Lecture Notes at MIT, 1997.

\bibitem{Pon}
S.~Pon,
\textit{Affine Stanley symmetric functions for classical types},
J. Algebraic Combin. {\bf 36} (2012), no. 4, 595--622.
and Ph.D Thesis, UC Davis, 2010.

\bibitem{Pos} 
A.~Postnikov, 
\textit{Affine approach to quantum Schubert calculus}, 
Duke Math. J. 128 (2005),  no. 3, 473--509.

\bibitem{Pr} 
P.~Pragacz,
\textit{Algebro-geometric applications of Schur $S$- and $Q$-polynomials}, 
Topics in invariant theory (Paris, 1989/1990),
130--191, Lecture Notes in Math., 1478, Springer, Berlin,
1991.

\bibitem{PR} 
P.~Pragacz, J.~Ratajski,
\textit{Formulas for Lagrangian and orthogonal degeneracy loci;
$\tilde Q$-polynomial approach.}, 
Compositio Math. 107 (1997), no. 1, 11--87.

\bibitem{Q} D. Quillen, unpublished.

\bibitem{RamYip:2011}
A.~Ram, M.~Yip.,
\textit{A combinatorial formula for {M}acdonald polynomials},
Adv. Math. 226 (2011), 309--331.

\bibitem{Robinson:1938} 
G. de B. Robinson, 
\textit{On the Representations of the Symmetric Group},
Amer. J. Math. 60 (3), (1938), 745--760.

\bibitem{Ruij}
S. N. M. Ruijsenaars, \textit{Complete integrability of the relativistic
Calogero-Moser system and elliptic function identities},
Comm. Math.  Phys. {\bf 110} (1987), 191--213.

\bibitem{Sagan} 
B.~Sagan, 
\textit{The Symmetric Group}, 
Graduate texts in mathematics, 203, Springer-Verlag, New York, 2001, ISBN 0387950672.

\bibitem{sage-combinat}
The {S}age-{C}ombinat community.
\newblock {S}age-{C}ombinat: enhancing {S}age as a toolbox for computer
  exploration in algebraic combinatorics, 2008.
\newblock {\tt http://combinat.sagemath.org}.

\bibitem{Sahi:1996} 
S.~Sahi, 
\textit{Interpolation, integrality, and a generalization of Macdonald's polynomials}, 
Internat. Math. Res. Notices (1996), no. 10, 457--471.

\bibitem{Schensted:1961} 
C.~Schensted, 
\textit{Longest increasing and decreasing subsequences}, 
Can. J. Math. 13 (1961), 179--191.

\bibitem{SchillingWarnaar:1999} 
A.~Schilling, S.~O.~Warnaar, 
\textit{Inhomogeneous lattice paths, generalized Kostka-Foulkes polynomials, and $A_{n-1}$-supernomials},
Comm. Math. Phys.  202 (1999), 359--401.

\bibitem{Schubert:1879}
H.~Schubert, 
\textit{Kalk\"ul der abz\"ahlenden Geometrie}. (German) 
[Calculus of enumerative geometry] Reprint of the 1879 original. With an introduction by Steven L. Kleiman. 
Springer-Verlag, Berlin-New York, 1979. 349 pp. ISBN: 3-540-09233-1 01A75.

\bibitem{Shimozono:2001} 
M.~Shimozono, 
\textit{A cyclage poset structure for Littlewood-Richardson tableaux}, 
European J. Combin.  22 (2001), 365--393. 

\bibitem{Shimozono2:2001} 
M. Shimozono, 
\textit{Multi-atoms and monotonicity of generalized Kostka polynomials}, 
European J. Combin.  22 (2001), 395--414. 

\bibitem{ShimozonoWeyman} 
M.~Shimozono, J.~Weyman, 
\textit{Graded characters of modules supported in the closure of a nilpotent conjugacy class}, 
European J. Combin. 21 (2000), 257--288.

\bibitem{ShimozonoZabrocki} 
M.~Shimozono, M.~Zabrocki, 
\textit{Hall-Littlewood vertex operators and generalized Kostka polynomials}, 
Adv. in Math.,  158 (2001), pp. 66--85.

\bibitem{Sta} 
R.~Stanley, 
\textit{On the number of reduced decompositions of elements of Coxeter groups}, 
European J. Combinatorics 5 (1984), 359--372. 

\bibitem{EC2} 
R.~Stanley, 
\textit{Enumerative Combinatorics}, 
Vol 2, Cambridge Studies in Advanced Mathematics, Cambridge University Press, 2001.

\bibitem{sage}
W.\thinspace{}A. Stein et~al.
\newblock {\em {S}age {M}athematics {S}oftware ({V}ersion 5.4)}.
\newblock The Sage~Development Team, 2012.
\newblock {\tt http://www.sagemath.org}.

\bibitem{TY} 
H.~Thomas, A.~Yong, 
\textit{A jeu de taquin theory for increasing tableaux, with applications to $K$-theoretic Schubert calculus}, 
Algebra Number Theory 3 (2009),  no. 2, 121--148.

\bibitem{TY1} 
H.~Thomas, A.~Yong, 
\textit{Longest strictly increasing subsequences, Plancherel measure and the Hecke insertion algorithm}, 
Advances in Applied Math, to appear.

\bibitem{Waugh:1999} 
D.~Waugh, 
\textit{Upper bounds in affine Weyl groups under the weak order}, 
Order 16 (1999), 77--87.

\bibitem{[Wi]} 
E. Witten, 
\textit{The Verlinde algebra and the cohomology of the Grassmanian}, 
``Geometry, Topology, and Physics'', 357--422, Conf. Proc. Lecture Notes Geom. Topology, IV, Internat Press,
Cambridge MA, 1995.

\bibitem{[Wa]} 
M. Walton, 
\textit{Fusion rules in Wess-Zumino-Witten models},
Nuclear Phys. B, {\bf 340} (1990), no. 2-3, 777-790.

\bibitem{ZabrockiThesis}
M.~Zabrocki, \textit{PhD Thesis}, On the action of the Hall-Littlewood vertex operator
(1998).


\bibitem{Zabrocki:1998} 
M.~Zabrocki, 
\textit{A Macdonald vertex operator and standard tableaux statistics for the two-Column $(q,t)$-Kostka coefficients}, 
Elect. J. of Comb. 5 (1998),~46 pages.

\bibitem{Zelevinsky} 
A. V.~Zelevinsky,
\textit{Representations of finite classical groups: a Hopf algebra approach.}  
Springer Lecture Notes, 869, (1981).

\end{thebibliography}
\end{document}